\documentclass[11pt,reqno]{amsart}
\usepackage{amsmath,amsfonts,amssymb}

\theoremstyle{theorem}
\newtheorem{theorem}{Theorem}[section]
\newtheorem{proposition}[theorem]{Proposition}
\newtheorem{lemma}[theorem]{Lemma}
\newtheorem{condition}[theorem]{Condition}
\newtheorem{remark}[theorem]{Remark}
\newtheorem{corollary}[theorem]{Corollary}
\newtheorem{example}[theorem]{Example}
\numberwithin{equation}{section}

\theoremstyle{plain}

\usepackage[T2A]{fontenc}
\usepackage[cp1251]{inputenc}
\usepackage{amsfonts}
\usepackage[english]{babel}
\usepackage{amsmath,amsthm,amssymb}

\def\Ran{\mathrm{Ran}\,}
\def\Ker{\mathrm{Ker}\,}
\def\rank{\mathrm{rank}\,}
\def\Dom{\mathrm{Dom}\,}
\def\clos{\mathrm{clos}\,}
\def\sgn{\mathrm{sign}\,}
\def\mes{\mathrm{meas}\,}
\def\diag{\mathrm{diag}\,}
\def\esssup{\mathrm{ess}\text{-}\mathrm{sup}\,}

\def\le{\leqslant}

\def\ge{\geqslant}

\begin{document}

\title[Homogenization of hyperbolic equations]{Homogenization of hyperbolic equations \\ with periodic coefficients}

\author{M.~A.~Dorodnyi and T.~A.~Suslina}


\address{St. Petersburg State University, Universitetskaya nab. 7/9, St.~Petersburg, 199034, Russia}

\email{mdorodni@yandex.ru}

\email{t.suslina@spbu.ru}


\begin{abstract}
In $L_2(\mathbb{R}^d;\mathbb{C}^n)$ we consider selfadjoint strongly elliptic second order differential operators ${\mathcal A}_\varepsilon$
with periodic coefficients depending on ${\mathbf x}/ \varepsilon$, $\varepsilon>0$. 
We study the behavior of the operator cosine
$\cos( {\mathcal A}^{1/2}_\varepsilon \tau)$, $\tau \in \mathbb{R}$, for small $\varepsilon$. Approximations for this operator in the
$(H^s\to L_2)$-operator norm with a suitable $s$ are obtained. The results are used to study the behavior of
the solution ${\mathbf v}_\varepsilon$ of the Cauchy problem for the hyperbolic equation
$\partial^2_\tau {\mathbf v}_\varepsilon = - \mathcal{A}_\varepsilon {\mathbf v}_\varepsilon +\mathbf{F}$.
General results are applied to the acoustics equation and the system of elasticity theory.
\end{abstract}

\thanks{Supported by Russian Foundation for Basic Research (grant no.~16-01-00087)}

\maketitle

\section*{Introduction\label{Intr}}

The paper concerns  homogenization for periodic differential operators (DOs). A broad literature is devoted to homogenization problems;
first, we mention the books \cite{BeLP, BaPa, ZhKO}.
For homogenization problems in  ${\mathbb R}^d$, one of the methods is the spectral approach based on the Floquet-Bloch theory.
The spectral approach to homogenization was used in many works; see, e.~g.,
\cite[Chapter~4]{BeLP}, \cite[Chapter~2]{ZhKO}, \cite{Se}, \cite{COrVa}, \cite{APi}.

\subsection{The class of operators\label{sec0.1}}
We consider selfadjoint elliptic second order DOs in $L_2(\mathbb{R}^d;\mathbb{C}^n)$ admitting a factorization of the form
\begin{equation}
\label{intro_A}
\mathcal{A} = f(\mathbf{x})^* b(\mathbf{D})^* g(\mathbf{x}) b(\mathbf{D}) f(\mathbf{x}).
\end{equation}
Here $b(\mathbf{D})=\sum_{l=1}^d b_l D_l$ is the $(m\times n)$-matrix first order DO with constant coefficients.
We assume that $m\ge n$ and that the symbol $b(\boldsymbol{\xi})$ has maximal rank.
It is assumed that the matrix-valued functions  $g(\mathbf{x})$ (of size $m\times m$) and $f(\mathbf{x})$
(of size $n\times n$) are periodic with respect to some lattice $\Gamma$ and such that
$$
g(\mathbf{x}) >0; \quad g,g^{-1} \in L_\infty; \quad f,f^{-1} \in L_\infty.
$$
It is convenient to start with a simpler class of operators
\begin{equation}
\label{intro_hatA}
\widehat{\mathcal{A}} = b(\mathbf{D})^* g(\mathbf{x}) b(\mathbf{D}),
\end{equation}
corresponding to the case where $f = \mathbf{1}$.
Many operators of mathematical physics can be represented in the form \eqref{intro_A} or~\eqref{intro_hatA};
the simplest example is the acoustics operator $\widehat{\mathcal{A}} = - \hbox{\rm div}\, g(\mathbf{x}) \nabla =
\mathbf{D}^* g(\mathbf{x})\mathbf{D}$.
This and other examples are discussed in \cite{BSu1} in detail.

Now we introduce the small parameter  $\varepsilon > 0$ and
denote $\varphi^\varepsilon(\mathbf{x}):= \varphi(\varepsilon^{-1}\mathbf{x})$ for any $\Gamma$-periodic function $\varphi(\mathbf{x})$.
Consider the operators
\begin{gather}
\label{intro_A_eps}
\mathcal{A}_{\varepsilon} = f^{\varepsilon}(\mathbf{x})^* b(\mathbf{D})^* g^{\varepsilon}(\mathbf{x}) b(\mathbf{D}) f^{\varepsilon}(\mathbf{x}),\\
\label{intro_hatA_eps}
\widehat{\mathcal{A}}_{\varepsilon} = b(\mathbf{D})^* g^{\varepsilon}(\mathbf{x}) b(\mathbf{D}),
\end{gather}
whose coefficients oscillate rapidly as $\varepsilon \to 0$.

\subsection{Operator error estimates for elliptic and parabolic problems in $\mathbb{R}^d$\label{sec0.2}}
 In a series of papers \cite{BSu1,BSu2,BSu3,BSu4} by Birman and Suslina, an operator-theoretic approach to homogenization of elliptic equations in
$\mathbb{R}^d$  was suggested and developed. This approach is based on the scaling transformation, the Floquet-Bloch theory, and the analytic perturbation theory.

The homogenization problem for elliptic equations in $\mathbb{R}^d$ can be regarded as
 a problem of asymptotic description of the resolvent of $\mathcal{A}_{\varepsilon}$ as $\varepsilon \to 0$.
For definiteness, let us talk about the simpler operator \eqref{intro_hatA_eps}.
 In \cite{BSu1}, it was shown that the resolvent  $(\widehat{\mathcal{A}}_{\varepsilon} + I )^{-1}$
converges to $(\widehat{\mathcal{A}}^0 + I )^{-1}$  in the  $L_2$-operator norm, as $\varepsilon \to 0$.
Here $\widehat{\mathcal{A}}^0 = b(\mathbf{D})^* g^0 b(\mathbf{D})$ is the \textit{effective operator} with the constant \textit{effective matrix} $g^0$.
The formula for the effective matrix is well known in homogenization theory; in the case under consideration it is described below in
Subsection \ref{germ_and_eff_g0_section}.
 In \cite{BSu1}, it was proved that
\begin{equation}
\label{intro_resolv_est}
\| (\widehat{\mathcal{A}}_{\varepsilon} + I )^{-1} - (\widehat{\mathcal{A}}^0 + I )^{-1} \|_{L_2(\mathbb{R}^d) \to L_2(\mathbb{R}^d)} \le C \varepsilon.
\end{equation}
In \cite{BSu2, BSu3}, a more accurate approximation for the resolvent of $\widehat{\mathcal{A}}_{\varepsilon}$
in the $L_2 (\mathbb{R}^d; \mathbb{C}^n)$-operator norm with an error $O(\varepsilon^2)$ was obtained,
and in \cite{BSu4} an approximation of the same resolvent in the norm of operators acting from
$L_2 (\mathbb{R}^d; \mathbb{C}^n)$ to the Sobolev space $H^1 (\mathbb{R}^d; \mathbb{C}^n)$
with an error $O(\varepsilon)$ was found. In these approximations, some correction terms of first order (the \textit{correctors}) were taken into account.

Similarly, the homogenization problem for parabolic equations in $\mathbb{R}^d$ can be regarded as a problem of asymptotic description of the
semigroup $\exp(-\tau \widehat{\mathcal{A}}_{\varepsilon} )$ for $\tau >0$ and small $\varepsilon$.
The operator-theoretic approach was applied to such problems in \cite{Su1, Su2, Su3, V, VSu}.
In \cite{Su1,Su2}, it was proved that
\begin{equation}
\label{intro_parab_exp_est}
\| \exp(-\tau \widehat{\mathcal{A}}_{\varepsilon} ) - \exp(-\tau \widehat{\mathcal{A}}^0) \|_{L_2(\mathbb{R}^d) \to L_2(\mathbb{R}^d)} \le C \varepsilon (\tau + \varepsilon^2)^{-1/2}.
\end{equation}
In \cite{V}, a more accurate approximation of the operator $\exp(-\tau \widehat{\mathcal{A}}_{\varepsilon} )$
in the $L_2 (\mathbb{R}^d; \mathbb{C}^n)$-operator norm with an error $O(\varepsilon^2)$ for  fixed $\tau$ was obtained, and in \cite{Su3}
approximation of the same operator in the norm of operators acting from  $L_2 (\mathbb{R}^d; \mathbb{C}^n)$ to $H^1 (\mathbb{R}^d; \mathbb{C}^n)$ with an error $O (\varepsilon)$ for  fixed $\tau$ was proved. In these approximations, the first order correctors were taken into account.

Even more accurate approximations for the semigroup and the resolvent of $\widehat{\mathcal{A}}_{\varepsilon}$
with the first and second correctors taken into account were found in \cite{VSu}.
In the papers cited above, similar (but more complicated) results were obtained also for more general operator \eqref{intro_A_eps}; we will not dwell on this.

Estimates of the form \eqref{intro_resolv_est}, \eqref{intro_parab_exp_est} are called \textit{operator error estimates} in homogenization theory.
A different approach to operator error estimates (the so called ``modified method of the first appproximation'')
was suggested by Zhikov. In \cite{Zh, ZhPas1}, the acoustics operator and the operator of elasticity theory
(which have the form \eqref{intro_hatA_eps}) were studied;
approximations for the resolvents in the $(L_2 \to L_2)$-norm with an error $O(\varepsilon)$
and in the $(L_2 \to H^1)$-norm with an error $O(\varepsilon)$ were obtained.
In \cite{ZhPas2},  estimate \eqref{intro_parab_exp_est} was proved for the scalar elliptic operator $-{\rm div}\,g^\varepsilon(\mathbf{x}) \nabla$.

\subsection{Operator error estimates for nonstationary Schr\"odinger-type and hyperbolic equations\label{intro_BSu5_Su5_rewiew_section}}
So, in the case of elliptic and parabolic problems the operator error estimates are well studied.
The situation with  homogenization of nonstationary  Schr\"odinger-type and hyperbolic equations
is different. The paper \cite{BSu5} is devoted to such problems. Again, we dwell on the results for the simpler operator \eqref{intro_hatA_eps}.
In operator terms, the behavior of  the operator exponential $\exp(-i \tau \widehat{\mathcal{A}}_{\varepsilon})$ and the operator cosine
$\cos(\tau \widehat{\mathcal{A}}_{\varepsilon}^{1/2})$  (where $\tau \in \mathbb{R}$) for small $\varepsilon$  is studied.
For these operators it is impossible to obtain approximations in the $L_2 (\mathbb{R}^d; \mathbb{C}^n)$-operator norm,
and we are forced to consider the norm of operators acting from the Sobolev space $H^s (\mathbb{R}^d; \mathbb{C}^n)$ (with appropriate $s$)
to $L_2 (\mathbb{R}^d; \mathbb{C}^n)$. In \cite{BSu5}, the following estimates were proved:
\begin{gather}
\label{intro_exp_est}
\| \exp(-i \tau \widehat{\mathcal{A}}_{\varepsilon}) - \exp(-i \tau \widehat{\mathcal{A}}^0) \|_{H^3 (\mathbb{R}^d) \to L_2 (\mathbb{R}^d)} \le (\widetilde{C}_1 + \widetilde{C}_2 |\tau|)\varepsilon, \\
\label{intro_cos_est}
\| \cos(\tau \widehat{\mathcal{A}}_{\varepsilon}^{1/2}) - \cos(\tau (\widehat{\mathcal{A}}^0)^{1/2}) \|_{H^2 (\mathbb{R}^d) \to L_2 (\mathbb{R}^d)} \le (C_1 + C_2 |\tau|)\varepsilon.
\end{gather}
By interpolation, it is easily seen that the operator in \eqref{intro_exp_est} in the $(H^s \to L_2)$-norm is of order
$O(\varepsilon^{s/3})$ (where $0\le s \le 3$)
and the operator in \eqref{intro_cos_est} in the $(H^s \to L_2)$-norm is of order $O(\varepsilon^{s/2})$ (where $0\le s \le 2$).
In \cite{BSu5}, similar results for
more general operator \eqref{intro_A_eps} were also obtained.

Let us explain the method of \cite{BSu5}; we comment on the proof of estimate \eqref{intro_cos_est}.
 Denote ${\mathcal H}_0 := -\Delta$.
Clearly, \eqref{intro_cos_est} is equivalent to
\begin{equation}
\label{intro_est2}
\| ( \cos(\tau \widehat{\mathcal{A}}_{\varepsilon}^{1/2}) - \cos(\tau (\widehat{\mathcal{A}}^0)^{1/2})) (\mathcal{H}_0 + I )^{-1} \|_{L_2(\mathbb{R}^d) \to L_2(\mathbb{R}^d)} \le (C_1 + C_2 |\tau|)\varepsilon.
\end{equation}
Next, the scaling transformation shows that \eqref{intro_est2} is equivalent to the estimate
\begin{equation}
\label{intro_est3}
\| \bigl( \cos(\tau \varepsilon^{-1} \widehat{\mathcal{A}}^{1/2}) - \cos(\tau \varepsilon^{-1} (\widehat{\mathcal{A}}^0)^{1/2})\bigr)
 \varepsilon^{2} (\mathcal{H}_0 + \varepsilon^{2} I )^{-1} \|_{L_2 \to L_2} \le (C_1 + C_2 |\tau|) \varepsilon.
\end{equation}

Next, using the unitary Gelfand transformation, we expand $\widehat{\mathcal{A}}$
in the direct integral of the operators $\widehat{\mathcal{A}}(\mathbf{k})$ acting in $L_2 (\Omega; \mathbb{C}^n)$
(where $\Omega$ is the cell of the lattice $\Gamma$).
Here $\widehat{\mathcal{A}}(\mathbf{k})$ is given by the differential expression
$b(\mathbf{D} + \mathbf{k})^* g(\mathbf{x}) b(\mathbf{D} + \mathbf{k})$ with periodic boundary conditions;
the spectrum of  $\widehat{\mathcal{A}}(\mathbf{k})$ is discrete. The family of operators $\widehat{\mathcal{A}}(\mathbf{k})$ is studied by means of the analytic perturbation theory
(with respect to the onedimensional parameter $t = |\mathbf{k}|$). For the operators $\widehat{\mathcal{A}}(\mathbf{k})$
the analog of estimate \eqref{intro_est3} is proved with
the constants independent of  $\mathbf{k}$. Then the inverse Gelfand transformation leads to \eqref{intro_est3}.
A good deal of considerations in the study of the family $\widehat{\mathcal A}({\mathbf k})$
is fulfilled in the abstract operator-theoretic setting.

The question about the sharpness of the results~\eqref{intro_exp_est}, \eqref{intro_cos_est} with respect to the type of the
operator norm (i.~e., the order of the Sobolev space) remained open until recent time.
In the recent paper~\cite{Su4} (see also brief communication \cite{Su5}) it was shown that estimate~\eqref{intro_exp_est} is sharp in the following sense:
there exist operators for which estimate
\begin{equation*}
\| \exp(-i \tau \widehat{\mathcal{A}}_{\varepsilon}) - \exp(-i \tau \widehat{\mathcal{A}}^0) \|_{H^s (\mathbb{R}^d) \to L_2 (\mathbb{R}^d)} \le C(\tau) \varepsilon
\end{equation*}
is false if $s < 3$. On the other hand, under some additional assumptions on the operator the result can be improved and we have
\begin{equation*}
\| \exp(-i \tau \widehat{\mathcal{A}}_{\varepsilon}) - \exp(-i \tau \widehat{\mathcal{A}}^0) \|_{H^2 (\mathbb{R}^d) \to L_2 (\mathbb{R}^d)} \le (\check{C}_1 + \check{C}_2 |\tau| )\varepsilon.
\end{equation*}

\subsection{Main results of the paper\label{intro_main_results_section}}
In the present paper, we study the behavior of the operator cosine  $\cos(\tau \widehat{\mathcal{A}}_{\varepsilon}^{1/2})$ for small $\varepsilon$
by the same approach as in \cite{Su4}. Main results are formulated in the operator terms and next applied
to study the behavior of the solutions of the Cauchy problem for hyperbolic equations.
On the one hand, we confirm the sharpness of estimate \eqref{intro_cos_est} in the following sense.
We find a condition on the operator, under which the estimate
\begin{equation*}
\| \cos(\tau \widehat{\mathcal{A}}_{\varepsilon}^{1/2}) - \cos(\tau (\widehat{\mathcal{A}}^0)^{1/2}) \|_{H^s (\mathbb{R}^d) \to L_2 (\mathbb{R}^d)} \le  C(\tau)\varepsilon
\end{equation*}
is false if $s<2$.
It is easy to formulate this condition in the spectral terms. We consider the operator family
$\widehat{\mathcal{A}}(\mathbf{k})$ and put
 $\mathbf{k} = t \boldsymbol{\theta}$, $t = |\mathbf{k}|$, $\boldsymbol{\theta} \in \mathbb{S}^{d-1}$.
This family is analytic with respect to the parameter~$t$. For $t=0$ the point $\lambda_0=0$ is an eigenvalue of multiplicity $n$ of the ``unperturbed''
operator $\widehat{\mathcal{A}}(0)$. Then for small $t$ there exist the real-analytic branches of the eigenvalues and the eigenvectors of
$\widehat{\mathcal{A}}(t \boldsymbol{\theta})$.
For small $t$ the eigenvalues $\lambda_l (t, \boldsymbol{\theta})$, $l=1,\dots,n,$ admit the convergent power series expansions
\begin{equation}
\label{intro_eigenvalues_series}
\lambda_l(t, \boldsymbol{\theta}) = \gamma_l (\boldsymbol{\theta}) t^2 + \mu_l (\boldsymbol{\theta}) t^3 + \ldots,
\quad \; \quad   l = 1, \ldots, n,
\end{equation}
where $\gamma_l(\boldsymbol{\theta})>0$ and $\mu_l(\boldsymbol{\theta})\in \mathbb{R}$.
The condition is that $\mu_l (\boldsymbol{\theta}_0) \ne 0$ for at least one $l$ and at least one point  $\boldsymbol{\theta}_0 \in \mathbb{S}^{d-1}$.
Examples of the operators satisfying this condition are provided; in particular, one example is
of the form $- \operatorname{div} g^\varepsilon (\mathbf{x}) \nabla$, where $g(\mathbf{x})$ is  Hermitian matrix with complex entries.
Another example is the matrix operator with real-valued coefficients (the operator of elasticity theory).

On the other hand, under some additional assumptions on the operator we improve the result and obtain the estimate
\begin{equation*}
\| \cos(\tau \widehat{\mathcal{A}}_{\varepsilon}^{1/2}) - \cos(\tau (\widehat{\mathcal{A}}^0)^{1/2}) \|_{H^{3/2} (\mathbb{R}^d) \to L_2 (\mathbb{R}^d)} \le (\tilde{C}_1 + \tilde{C}_2 |\tau|)\varepsilon.
\end{equation*}
In the case where  $n=1$, for this it suffices that the coefficient
$\mu (\boldsymbol{\theta}) = \mu_1 (\boldsymbol{\theta})$ in \eqref{intro_eigenvalues_series} is identically zero.
In particular, this is the case for the operator $- \operatorname{div} g^\varepsilon (\mathbf{x}) \nabla$, where $g(\mathbf{x})$
is symmetric matrix with real entries.
In the matrix case (i.~e., for $n\ge 2$), besides the condition that all the coefficients $\mu_l (\boldsymbol{\theta})$
in \eqref{intro_eigenvalues_series}  are equal to zero, we impose one more condition in terms of the coefficients
$\gamma_l (\boldsymbol{\theta})$, $l=1,\dots,n$. The simplest version of this condition is that the branches
$\gamma_l (\boldsymbol{\theta})$ must not intersect:
for each pair $j\ne l$ either $\gamma_j (\boldsymbol{\theta})$ and $\gamma_l (\boldsymbol{\theta})$ are separated from each other or
they coincide for all $\boldsymbol{\theta} \in \mathbb{S}^{d-1}$.

It turns out that for more general operator \eqref{intro_A_eps} it is convenient to study the operator cosine sandwiched between appropriate rapidly oscillating factors.
Namely, we study the operator $f^\varepsilon \cos(\tau \mathcal{A}_{\varepsilon}^{1/2}) (f^\varepsilon)^{-1}$ and
obtain analogs of the results described above for this operator.

Next, we apply the results given in operator terms to study the behavior of the solution
$\mathbf{v}_\varepsilon (\mathbf{x}, \tau)$, $\mathbf{x} \in \mathbb{R}^d$, $\tau \in \mathbb{R}$, of the following problem
\begin{equation*}
\left\{
\begin{aligned}
&\frac{\partial^2 \mathbf{v}_\varepsilon (\mathbf{x}, \tau)}{\partial \tau^2} = - (\widehat{\mathcal{A}}_{\varepsilon} \mathbf{v}_\varepsilon) (\mathbf{x}, \tau) + \mathbf{F} (\mathbf{x}, \tau), \\
& \mathbf{v}_\varepsilon (\mathbf{x}, 0) = \boldsymbol{\phi} (\mathbf{x}), \quad \frac{\partial \mathbf{v}_\varepsilon }{\partial \tau} (\mathbf{x}, 0) = \boldsymbol{\psi}(\mathbf{x}),
\end{aligned}
\right.
\end{equation*}
A more general problem with the operator $\mathcal{A}_{\varepsilon}$ is also studied.
We apply the general results to specific equations of mathematical physics.
In particular, we consider the  nonstationary acoustics equation and the system of elasticity theory.

\subsection{Method\label{method}}
The results are obtained by further development of the operator-theoretic approach.
We follow the plan described above in Subsection~\ref{intro_BSu5_Su5_rewiew_section}.
Considerations are based on the abstract operator-theoretic scheme.
In the abstract setting, we study the family of operators $A(t)=X(t)^* X(t)$ acting
in some Hilbert space $\mathfrak{H}$. Here $X(t)=X_0 + tX_1$. (The family $A(t)$ is modelling the operator family $\mathcal{A}(\mathbf{k}) = \mathcal{A}(t \boldsymbol{\theta})$, but the parameter $\boldsymbol{\theta}$ is absent in the abstract setting.) It is assumed that
the point $\lambda_0=0$ is an eigenvalue of $A(0)$ of finite multiplicity $n$.
Then for $|t|\le t^0$ the perturbed operator $A(t)$ has exactly $n$ eigenvalues (counted with multiplicities) on the interval $[0,\delta]$
(here $\delta$ and $t^0$ are controlled explicitly).
These eigenvalues and the corresponding eigenvectors are real-analytic functions of $t$.
The coefficients of the corresponding power series expansions for the eigenvalues and the eigenvectors are called \textit{threshold characteristics} of the operator $A(t)$.
We distinguish a finite rank operator $S$ (the so called  \textit{spectral germ} of the operator family $A(t)$) which acts in the space
$\mathfrak{N} = \Ker A(0)$.
The spectral germ (see the definition in Subsection \ref{abstr_Z_R_S_op_section} below)
contains the information about the threshold characteristics of principal order.
 Let $F(t)$ be the spectral projection of the operator $A(t)$ for the interval $[0,\delta]$.
 We rely on the threshold approximations for the projection $F(t)$ and for the operator $A(t)F(t)$ obtained in
\cite[Chapter~1]{BSu1}  and \cite{BSu2}.
  Note that in \cite{BSu5} the threshold approximations of principal order from \cite{BSu1} were applied:
$F(t)$ was approximated by the projection $P$ onto the subspace $\mathfrak{N}$,
  and the operator   $A(t)F(t)$ was approximated by $t^2SP$.
   It turns out that, in order to obtain more subtle results described above, we need to use
  more accurate threshold approximations obtained in \cite{BSu2}. Moreover,
  we need to divide the eigenvalues of $A(t)$ into clusters and find more detailed
  threshold approximations associated with this division (see Section \ref{abstr_cluster_section}).

 In terms of the spectral germ, it is possible to approximate the operator cosine $\cos(\varepsilon^{-1} \tau A(t)^{1/2})$
 multiplied by an appropriate ``smoothing factor''.
  Application of the abstract results leads to the required estimates for differential operators.
  However, at this step additional difficulties arise. They concern the improvement of the results in the case where
  all the coefficients $\mu_l (\boldsymbol{\theta})$ are equal to zero. These difficulties are related to the fact that
  in the general (matrix) case we are not always able to make our constructions and estimates uniform in $\boldsymbol{\theta}$, and we are forced to
  impose the additional assumptions of isolation of the branches $\gamma_l (\boldsymbol{\theta})$, $l=1,\dots,n$.

\subsection{Plan of the paper}
The paper consists of three chapters. Chapter~1 (Sections 1--5) contains the necessary operator-theoretic material.
In Chapter~2 (Sections~6--12), the periodic differential operators of the form~\eqref{intro_A}, \eqref{intro_hatA} are studied.
Some preliminary material is given in Section~6.
In Section~7 we describe the class of operators and the direct integral expansion for periodic operators of the form~\eqref{intro_A};
the corresponding family of operators ${\mathcal A}(\mathbf{k})$ acting in $L_2(\Omega;\mathbb{C}^n)$ is incorporated in the framework of the abstract scheme.
In Section~8 we describe the effective characteristics for the operator~\eqref{intro_hatA}.
In Section~9, using the abstract results, we obtain approximation for the operator
$\cos( \varepsilon^{-1} \tau \widehat{\mathcal A}(\mathbf{k})^{1/2})$. The operator ${\mathcal A}(\mathbf{k})$ is considered in Section~10.
In Section~11, with the help of the abstract results we find approximation for the sandwiched operator
 $\cos( \varepsilon^{-1} \tau {\mathcal A}(\mathbf{k})^{1/2})$. Next, in Section~12 we return to the operators~\eqref{intro_A}, \eqref{intro_hatA} acting in $L_2(\mathbb{R}^d;\mathbb{C}^n)$;
 applying the results of Sections~9 and~11, we obtain approximations for the operator
 $\cos( \varepsilon^{-1} \tau \widehat{\mathcal A}^{1/2})$ and the sandwiched operator
 $\cos( \varepsilon^{-1} \tau {\mathcal A}^{1/2} )$. Chapter~3 (Sections 13--16) is devoted to homogenization problems.
In Section~13, by the scaling transformation, the results of Chapter~2 imply main results of the paper: approximations
for the operator  $\cos(\tau \widehat{\mathcal A}_\varepsilon^{1/2})$
  and the sandwiched operator $\cos(\tau {\mathcal A}_\varepsilon^{1/2})$ in the $(H^s \to L_2)$-norm.
  In Section~14, the results are applied to study the behavior of the solution of the Cauchy problem for
hyperbolic equations. The last Sections~15 and~16 are devoted to applications of the general results to the particular equations, namely, to
  the acoustics  equation and the system of elasticity theory.

\subsection{Notation}
Let $\mathfrak{H}$ and $\mathfrak{H}_{*}$ be complex separable Hilbert spaces.
The symbols $(\cdot, \cdot)_{\mathfrak{H}}$ and $ \| \cdot \|_{\mathfrak{H}}$ stand for the inner product and the norm in $\mathfrak{H}$, respectively;
the symbol $\| \cdot \|_{\mathfrak{H} \to \mathfrak{H}_*}$ stands for the norm of a linear continuous operator from $\mathfrak{H}$ to $\mathfrak{H}_{*}$. Sometimes we omit the indices. By $I = I_{\mathfrak{H}}$ we denote the identity operator in $\mathfrak{H}$.
If $\mathfrak{N}$ is a subspace of $\mathfrak{H}$, then $\mathfrak{N}^{\perp} : = \mathfrak{H} \ominus \mathfrak{N}$.
If $P$ is the orthogonal projection of $\mathfrak{H}$ onto $\mathfrak{N}$, then $P^\perp$ is the orthogonal projection onto
 $\mathfrak{N}^{\perp}$.
If $A: \mathfrak{H} \to \mathfrak{H}_*$ is a linear operator, then $\Dom A$ stands for its domain and $\Ker A$ stands for its kernel.

The symbols $\langle\cdot, \cdot \rangle$ and $|\cdot|$ denote the inner product and the norm in $\mathbb{C}^n$;
$\boldsymbol{1}_n$ is the unit $(n \times n)$-matrix. If $a$ is an $(n\times n)$-matrix, then the symbol $|a|$
denotes the norm of $a$ viewed as a linear operator in $\mathbb{C}^n$.
 Next, we use the notation $\mathbf{x} = (x_1,\dots, x_d) \in \mathbb{R}^d$, $i D_j = \partial_j = \partial / \partial x_j$,
$j=1,\dots,d$, $\mathbf{D} = -i \nabla = (D_1,\dots, D_d)$.

The $L_p$-classes of $\mathbb{C}^n$-valued functions in a domain $\mathcal{O} \subset \mathbb{R}^d$
are denoted by $L_p (\mathcal{O}; \mathbb{C}^n), 1 \le p \le \infty$.
The Sobolev classes of $\mathbb{C}^n$-valued functions in a domain $\mathcal O$ of order $s$ and integrability index $p$
are denoted by $W^s_p (\mathcal{O}; \mathbb{C}^n)$. For $p=2$, we denote this space by $H^s (\mathcal{O}; \mathbb{C}^n)$.
If $n=1$, we write simply $L_p({\mathcal O})$, $H^s({\mathcal O})$, but sometimes we use such abbreviated notation also for the spaces
of vector-valued or matrix-valued functions.

By $c, C, {\mathcal C}, {\mathfrak C}$ (possibly, with indices and marks) we denote various constants in estimates.

\smallskip
A brief communication on the results of the present paper was published in \cite{DSu}.

\section*{Chapter~1. Abstract operator-theoretic scheme}

\section{Quadratic operator pencils\label{abstr_section_1}}
Our approach to homogenization problems is based on the abstract operator-theoretic scheme.

\subsection{The operators $X(t)$ and $A(t)$\label{abstr_X_A_section}}
Let $\mathfrak{H}$ and $\mathfrak{H}_{*}$ be complex separable Hilbert spaces. Suppose that
$X_0: \mathfrak{H} \to \mathfrak{H}_{*}$ is a densely defined and closed operator, and
$X_1: \mathfrak{H} \to \mathfrak{H}_{*}$ is a bounded operator.
Then the operator $X(t) := X_0 + tX_1$, $t \in \mathbb{R}$, is closed on the domain $\Dom X(t) = \Dom X_0$.
Consider the family of selfadjoint (and nonnegative) operators
$A(t) := X(t)^* X(t)$ in $\mathfrak{H}$.
The operator $A(t)$ is generated by the closed quadratic form $\| X(t) u \|^{2}_{\mathfrak{H}_*}$,
$u \in \Dom X_0$. Denote $A(0) = X_0^*X_0 =: A_0$ and
$\mathfrak{N} := \Ker  A_0 = \Ker X_0$.
We impose the following condition.

\begin{condition}\label{cond1}
Suppose that the point $\lambda_0 = 0$ is an isolated point of the spectrum of $A_0$, and $0 < n := {\rm{dim}}\, \mathfrak{N} < \infty$.
\end{condition}

By $d^0$ we denote the \textit{distance from the point $\lambda_0 = 0$ to the rest of the spectrum of} $A_0$.
We put $\mathfrak{N}_*:= \Ker X_0^*$, $n_* := \dim \mathfrak{N}_*$, and \textit{assume that} $n \le n_* \le \infty$.
Let $P$ and $P_*$  be the orthogonal projections of $\mathfrak{H}$ onto $\mathfrak{N}$ and of
$\mathfrak{H}_*$ onto  $\mathfrak{N}_*$, respectively.

The operator family $A(t)$ has been studied in \cite[Chapter~1]{BSu1}, \cite{BSu2}, \cite[Chapter~1]{BSu4}  in detail.

Denote by $F(t;[a,b])$ the spectral projection of $A(t)$ for the interval $[a,b]$, and put
${\mathfrak F}(t;[a,b]):=F(t; [a,b]) \mathfrak{H}$. \textit{We fix a number $\delta>0$ such that} $8 \delta < d^0$.
We often write $F(t)$ in place of $F(t;[0,\delta])$ and ${\mathfrak F}(t)$ in place of ${\mathfrak F}(t;[0,\delta])$.
Next, we choose a number $t^0 >0$ such that
\begin{equation}
\label{abstr_t0_fixation}
t^0 \le \delta^{1/2} \|X_1\|^{-1}.
\end{equation}
According to \cite[Chapter~1, (1.3)]{BSu1},
\begin{equation*}
F(t; [0,\delta]) = F(t;[0, 3 \delta]), \quad \rank F(t; [0,\delta]) = n, \quad |t| \le t^0.
\end{equation*}

\subsection{The operators $Z$, $R$, and $S$\label{abstr_Z_R_S_op_section}}
Now we introduce some operators appearing in the analytic perturbation theory considerations; see \cite[Chapter~1, Section~1]{BSu1},
\cite[Section~1]{BSu2}.

Denote $\mathcal{D} := \Dom X_0 \cap \mathfrak{N}^{\perp}$. Since $\lambda_0 = 0$ is an isolated eigenvalue of $A_0$,
the form $(X_0 \varphi, X_0 \zeta), \; \varphi, \zeta \in \mathcal{D}$, defines an inner product in $\mathcal{D}$,
converting $\mathcal{D}$ into Hilbert space. Let $\omega \in \mathfrak{N}$.
Consider the following equation for $\phi \in \mathcal{D}$:
\begin{equation*}
X^*_0 (X_0 \phi + X_1 \omega) = 0,
\end{equation*}
which is understood in the weak sense:
\begin{equation*}
(X_0 \phi, X_0 \zeta )_{\mathfrak{H}_*} = -(X_1 \omega, X_0 \zeta )_{\mathfrak{H}_*}, \quad \forall \zeta \in \mathcal{D}.
\end{equation*}
Here the right-hand side is a continuous antilinear functional of $\zeta \in \mathcal{D}$.
Then, by the Riesz theorem, there exists a unique solution $\phi = \phi (\omega)$.
We define a bounded operator $Z : \mathfrak{H} \to \mathfrak{H}$ by the following relation
\begin{equation*}
Zu = \phi (P u), \quad u \in \mathfrak{H}.
\end{equation*}
Note that $Z$  takes $\mathfrak{N}$ to $\mathfrak{N}^\perp$ and $\mathfrak{N}^\perp$ to $\{0\}$.
Let $R : \mathfrak{N} \to \mathfrak{N}_*$ be the operator defined by
$$
R \omega = X_0 \phi(\omega) +X_1 \omega = (X_0 Z + X_1)\omega, \ \ \omega \in \mathfrak{N}.
$$
Another representation for $R$ is given by $R = P_* X_1\vert_{\mathfrak{N}}$.

According to \cite[Chapter~1]{BSu1}, the operator  $S:=R^*R: \mathfrak{N} \to \mathfrak{N}$
is called the \textit{spectral germ of the operator family $A(t)$ at} $t=0$.
The germ $S$ can be represented as
$S= P X_1^* P_* X_1\vert_{\mathfrak{N}}.$
The spectral germ is said to be \textit{nondegenerate} if ${\rm Ker}\, S = \{ 0\}$.

Note that
\begin{equation}
\label{abstr_Z_R_S_est}
\| Z \| \le (8 \delta)^{-1/2} \| X_1 \|, \quad \| R \| \le \| X_1 \|, \quad   \| S \| \le \| X_1 \|^2.
\end{equation}

\subsection{The analytic branches of eigenvalues and eigenvectors  of $A(t)$\label{branches}}
According to the general analytic perturbation theory (see \cite{Ka}), for $|t|\le t^0$ there exist real-analytic functions $\lambda_l(t)$
(the branches of the eigenvalues) and real-analytic $\mathfrak{H}$-valued functions $\varphi_l(t)$ (the branches of the eigenvectors) such that
\begin{equation}
\label{abstr_A(t)_eigenvectors_branches}
A(t) \varphi_l(t) = \lambda_l (t) \varphi_l(t), \quad l = 1, \dots, n,\quad |t| \le t^0,
\end{equation}
and the $\varphi_l(t)$, $l=1,\dots,n$, form an \textit{orthonormal basis} in ${\mathfrak F}(t)$.
Moreover, for $|t|\le t_*$, where $0< t_* \le t^0$ \textit{is sufficiently small}, we have the following convergent power series expansions:
\begin{align}
\label{abstr_A(t)_eigenvalues_series}
\lambda_l(t) &= \gamma_l t^2 + \mu_l t^3 + \dots, \quad \gamma_l \ge 0, \quad \; \mu_l \in \mathbb{R}, \quad   l = 1, \dots, n, \\
\label{abstr_A(t)_eigenvectors_series}
\varphi_l (t) &= \omega_l + t \psi_l^{(1)} + \dots, \quad	l = 1, \dots, n.
\end{align}
The elements $\omega_l= \varphi_l(0)$, $l=1,\dots,n,$ form an orthonormal basis in  $\mathfrak{N}$.
We agree to enumerate the branches so that $\gamma_1 \le \gamma_2 \le \dots \le \gamma_n$.

Substituting expansions \eqref{abstr_A(t)_eigenvalues_series} and \eqref{abstr_A(t)_eigenvectors_series}
in~\eqref{abstr_A(t)_eigenvectors_branches} and comparing the coefficients of  powers $t$ and $t^2$, we see that 
\begin{gather}
\label{omegatilde}
\widetilde{\omega}_l := \psi_l^{(1)} - Z \omega_l \in \mathfrak{N}, \quad l = 1, \ldots, n, \\
\label{abstr_S_eigenvectors}
S \omega_l = \gamma_l \omega_l , \quad l = 1, \ldots, n.
\end{gather}
(Cf. \cite[Chapter~1, Section~1]{BSu1}, \cite[Section~1]{BSu2}.)
Thus, the \textit{numbers $\gamma_l$ and the elements $\omega_l$ defined by}
\eqref{abstr_A(t)_eigenvectors_branches}--\eqref{abstr_A(t)_eigenvectors_series}
\textit{are eigenvalues and eigenvectors of the germ} $S$.
We have
\begin{gather}
\label{abstr_P_repr_omega}
P = \sum_{l=1}^{n} (\cdot, \omega_l) \omega_l, \\
\label{abstr_SP_repr_gamma_omega}
SP = \sum_{l=1}^{n} \gamma_l (\cdot, \omega_l) \omega_l.
\end{gather}

   Note that
\begin{equation}
\label{1.12}
(\widetilde{\omega}_l , \omega_j) + (\omega_l, \widetilde{\omega}_j)=0,\quad l,j =1,\dots,n,
\end{equation}
   which follows from the relations $(\varphi_l(t),\varphi_j(t))=\delta_{lj}$
by substituting  \eqref{abstr_A(t)_eigenvectors_series}, comparing the coefficients of power $t$, and taking \eqref{omegatilde} into account.

\subsection{Threshold approximations\label{threshold}}
The spectral projection $F(t)$ and the operator $A(t)F(t)$ are real-analytic operator-valued functions for $|t|\le t^0$.
We have
\begin{gather*}
F(t) = \sum_{l=1}^{n} (\cdot, \varphi_l (t))\varphi_l (t), \\
A(t) F(t) = \sum_{l=1}^{n} \lambda_l (t)  (\cdot, \varphi_l (t))\varphi_l (t).
\end{gather*}
Together with
\eqref{abstr_A(t)_eigenvalues_series},~\eqref{abstr_A(t)_eigenvectors_series},~\eqref{abstr_P_repr_omega}, and~\eqref{abstr_SP_repr_gamma_omega} this  
yields the power series expansions
$F(t) = P + tF_1+\dots$ and $A(t)F(t)= t^2 SP+ t^3 K +\dots$, convergent for $|t|\le t_*$.
However, for our purposes not expansions, but
approximations (with one or several first terms)
 with error estimates on the whole interval $|t|\le t^0$ are needed.

The following statement was obtained in \cite[Chapter~1, Theorems~4.1~and~4.3]{BSu1}.
In what follows, \emph{we agree
to denote by $\beta_j$ various  absolute constants} (which can be controlled explicitly) \emph{assuming that} $\beta_j \ge 1$.

\begin{theorem}[\cite{BSu1}]
Suppose that the assumptions of Subsection~\emph{\ref{abstr_X_A_section}} are satisfied. Then we have

    \begin{align}
    \label{abstr_F(t)_threshold_1}
    \| F(t) - P \| &\le C_1 |t|, \quad |t| \le t^0, \\
    \label{abstr_A(t)_threshold_1}
    \| A(t)F(t) - t^2 SP \| &\le C_2 |t|^3, \quad |t| \le t^0.
    \end{align}
    Here $t^0$ is subject to~\emph{\eqref{abstr_t0_fixation}}, and the constants $C_1$, $C_2$ are given by

    \begin{equation}
    \label{abstr_C1_C2}
     C_1 = \beta_1 \delta^{-1/2} \| X_1 \|, \quad   C_2 = \beta_2 \delta^{-1/2}\| X_1 \|^3.
    \end{equation}
\end{theorem}

From \eqref{abstr_t0_fixation}, \eqref{abstr_Z_R_S_est}, and~\eqref{abstr_A(t)_threshold_1} it follows that
\begin{equation}
\label{abstr_A(t)F(t)_est}
\| A ( t) F(t) \| \le (\|X_1\|^2 + C_2 t^0) t^2 \le C_3 t^2, \quad | t | \le  t^0,
\end{equation}
where
\begin{equation}
\label{abstr_C3}
C_3 = \beta_3 \|X_1\|^2, \quad \beta_3 = 1 + \beta_2.
\end{equation}

We also need more precise approximations obtained in  \cite[Theorem~4.1]{BSu2}.

\begin{theorem}[\cite{BSu2}] Suppose that the assumptions of Subsection~\emph{\ref{abstr_X_A_section}}
are satisfied. Then for $|t|\le t^0$ we have
    \begin{align}
     \label{abstr_A(t)_threshold_2}
    A(t) F(t) &= t^2 SP + t^3 K + \Psi (t),
\\
    \label{abstr_C5}
    \| \Psi (t) \| &\le C_5 t^4, \quad |t|\le t^0, \quad C_5 = \beta_5 \delta^{-1}\| X_1 \|^4.
    \end{align}
The operator  $K$ is represented as
\begin{equation}
\label{abstr_K_N_def}
K = K_0 + N = K_0 + N_0 + N_*,
\end{equation}
where $K_0$ takes $\mathfrak{N}$ to $\mathfrak{N}^{\perp}$ and $\mathfrak{N}^{\perp}$ to $\mathfrak{N}$, while $N = N_0 + N_*$
takes $\mathfrak{N}$ to itself and $\mathfrak{N}^{\perp}$ to $\{ 0 \}$.
 In terms of the power series coefficients, the operators $ K_0, N_0, N_*$ are given by
\begin{gather}
\notag
K_0 = \sum_{l=1}^{n} \gamma_l \left( (\cdot, Z \omega_l) \omega_l + (\cdot, \omega_l) Z \omega_l \right) , \\
\label{abstr_N_0_N_*}
N_0 = \sum_{l=1}^{n} \mu_l (\cdot, \omega_l) \omega_l, \quad N_* = \sum_{l=1}^{n} \gamma_l \left( (\cdot, \widetilde{\omega}_l) \omega_l + (\cdot, \omega_l) \widetilde{\omega}_l\right) .
\end{gather}
In the invariant terms, we have
\begin{equation}
\label{abstr_F1_K0_N_invar}
K_0 = Z S P + S P Z^*, \quad 
N = Z^*X_1^* R P + (RP)^* X_1 Z.
\end{equation}
The operators $N$ and $K$ satisfy the following estimates
\begin{equation}
\label{abstr_K_N_estimates}
\|N\| \le (2 \delta)^{-1/2} \| X_1 \|^3, \quad \|K\| \le 2(2 \delta)^{-1/2} \| X_1 \|^3.
\end{equation}
\end{theorem}

\begin{remark}
    \label{abstr_N_remark}

 $1^\circ$. If $Z = 0$, then $K_0 = 0$, $N = 0$, and $K = 0$.

 $2^\circ$. In the basis $\{\omega_l\}$, the operators $N$, $N_0$, and $N_*$
 \emph{(}restricted to $\mathfrak{N}$\emph{)}
are represented by $(n\times n)$-matrices. The operator $N_0$ is diagonal:
        \begin{equation*}
             (N_0 \omega_j, \omega_k ) = \mu_j \delta_{jk}, \quad  j, k = 1, \ldots ,n.
        \end{equation*}
The matrix entries of $N_*$ are given by
        \begin{equation*}
         (N_* \omega_j, \omega_k) = \gamma_k (\omega_j, \widetilde{\omega}_k) + \gamma_j (\widetilde{\omega}_j, \omega_k ) = ( \gamma_j - \gamma_k)(\widetilde{\omega}_j, \omega_k ), \quad j, k = 1,\ldots, n.
        \end{equation*}
Here we have taken \eqref{1.12} into account.
 It is seen that the diagonal elements of $N_*$ are equal to zero:
$(N_* \omega_j, \omega_j)=0$, $j=1,\dots,n$. Moreover,
        \begin{equation*}
           (N_* \omega_j, \omega_k) = 0\quad  \text{if}  \  \gamma_j = \gamma_k.
        \end{equation*}

 $3^\circ$.
If $n=1$, then $N_*=0$, i.~e., $N=N_0$.

\end{remark}

\subsection{The nondegeneracy condition\label{abstr_nondegenerated_section}}
Below we impose the following additional condition on the operator $A(t)$ (cf.~\cite[Chapter~1, Subsection~5.1]{BSu1}.

\begin{condition}\label{nondeg}
There exists a constant $c_*>0$ such that
\begin{equation}
\label{abstr_A(t)_nondegenerated}
A(t) \ge c_* t^2 I, \quad |t| \le t^0.
\end{equation}
\end{condition}

From \eqref{abstr_A(t)_nondegenerated} it follows that $\lambda_l(t) \ge c_* t^2$, $l=1,\dots,n$, for $|t|\le t^0$.
By \eqref{abstr_A(t)_eigenvalues_series}, this implies
\begin{equation}
\label{abstr_gamma_ge_c_*}
\gamma_l \ge c_* > 0, \quad l= 1, \ldots, n,
\end{equation}
i.~e., the germ $S$ is nondegenerate:
\begin{equation}
\label{abstr_S_nondegenerated}
S \ge c_* I_{\mathfrak{N}}.
\end{equation}

\section{The clusters of eigenvalues of $A(t)$\label{abstr_cluster_section}}

The content of this section is borrowed from ~\cite[Section~2]{Su4} and concerns the case where $n \ge 2$.

\subsection{Renumbering of eigenvalues}

Suppose that Condition~\ref{nondeg} is satisfied.
Then the spectrum of the operator $SP$ consists of the eigenvalue $\lambda_0=0$
(with the eigenspace $\mathfrak{N}^{\perp}$) and the eigenvalues $\gamma_1,\dots, \gamma_n$  satisfying \eqref{abstr_gamma_ge_c_*}.
Now it is convenient to change the notation tracing the multiplicities of the eigenvalues.
Let $p$ be the number of different eigenvalues of the germ.
We enumerate these eigenvalues in the increasing order
and denote them by $\gamma_j^\circ$, $j=1,\dots,p$.
Their multiplicities are denoted by $k_1,\dots, k_p$ (obviously, $k_1+\dots+k_p =n$).
Then, in the previous notation,
$$
\gamma_1 = \cdots = \gamma_{k_1} < \gamma_{k_1 + 1} = \cdots = \gamma_{k_1 + k_2} < \cdots < \gamma_{n-k_p+1} = \cdots = \gamma_n.
$$
We have
$\gamma^{\circ}_1 = \gamma_1 = \cdots  = \gamma_{k_1}$, $\gamma^{\circ}_2 = \gamma_{k_1 +1} = \cdots = \gamma_{k_1+k_2}$, etc. Denote $\mathfrak{N}_j := \Ker(S - \gamma^{\circ}_j I_\mathfrak{N})$, $j = 1 ,\ldots, p$. Then
\begin{equation*}
\mathfrak{N} = \sum_{j=1}^{p} \oplus \mathfrak{N}_j.
\end{equation*}
Let $P_j$ be the orthogonal projection of $\mathfrak{H}$ onto $\mathfrak{N}_j$. Then
\begin{equation}
\label{abstr_P_Pj}
P = \sum_{j=1}^{p} P_j, \qquad P_j P_l = 0 \quad \text{if} \  j \ne l.
\end{equation}

We also change the notation for the eigenvectors of the germ
(which are the ``embrios'' in \eqref{abstr_A(t)_eigenvectors_series}) dividing them in $p$ parts so that
$\omega_1^{(j)},\dots, \omega^{(j)}_{k_j}$ correspond to the eigenvalue
 $\gamma_j^\circ$ and form an orthonormal basis in $\mathfrak{N}_j$.
 (In the previous notation, these are $\omega_{k_1+\dots+k_{j-1}+1}, \dots, \omega_{k_1+\dots+k_j}$.)

We also change the notation for the analytic branches of the eigenvalues and the eigenvectors of $A(t)$.
The eigenvalue and the eigenvector whose expansions \eqref{abstr_A(t)_eigenvalues_series} and \eqref{abstr_A(t)_eigenvectors_series}
start with the terms $\gamma_j^\circ t^2$ and $\omega^{(j)}_q$
are denoted by $\lambda^{(j)}_q(t)$ and $\varphi^{(j)}_q(t)$, respectively.
 For $|t|\le t_*$ we have
\begin{align*}
\lambda^{(j)}_q (t) &= \gamma^{\circ}_j t^2 + \mu^{(j)}_q t^3 + \ldots, \quad q = 1, \ldots, k_j, \\
\varphi^{(j)}_q (t) &= \omega^{(j)}_q + t \psi^{(j)}_q + \ldots, \quad	q = 1, \ldots, k_j.
\end{align*}

 \subsection{Refinement of threshold approximations\label{abstr_cluster_treshold_section}}
 Suppose that Condition~\ref{nondeg} is satisfied. For $|t|\le t^0$ we consider
 the following bounded selfadjoint operator in $\mathfrak{H}$:
$$
{\mathfrak A}(t) =  \begin{cases}
 t^{-2} A(t)F(t), & t\ne 0,\cr
SP, & t=0.
\end{cases}
$$
We apply the spectral perturbation theory arguments,
treating ${\mathfrak A}(t)$ as a perturbation of the operator $SP$. By \eqref{abstr_A(t)_threshold_1},
\begin{equation}
\label{abstr_frakA_est}
\| \mathfrak{A}(t) - SP\| \le C_2 |t|, \quad |t| \le t^0.
\end{equation}

The results about $A(t)$ (see Section \ref{abstr_section_1}) show that for $|t|\le t^0$
the point $\lambda_0=0$ is an eigenvalue of the perturbed operator ${\mathfrak A}(t)$
(with the eigenspace ${\mathfrak F}(t)^\perp$), and ${\mathfrak A}(t)$ has
positive eigenvalues of total multiplicity $n$.
We divide these positive eigenvalues in $p$ clusters
which for small $|t|$ are located near the eigenvalues $\gamma_1^\circ,\dots, \gamma_p^\circ$
of the unperturbed operator. Clearly, the $j$-th  cluster
consists of the eigenvalues $\nu^{(j)}_q(t) = t^{-2} \lambda^{(j)}_q(t)$ of ${\mathfrak A}(t)$, $q=1,\dots,k_j$,
since $\nu^{(j)}_q(t)$ are continuous (and even analytic) in $t$
and $\nu^{(j)}_q(0)=\gamma_j^\circ$.
The corresponding orthonormal eigenvectors are $\varphi^{(j)}_q(t)$, $q=1,\dots,k_j$.

For sufficiently small $|t|$ these clusters are separated from each other.
However, it will be more convenient for our purposes, for each pair of indices $j\ne l$, to divide
the clusters in two parts separated  from each other and such that one part contains
the $j$-th cluster and another part contains the  $l$-th cluster.
For each pair of indices $(j,l)$,  $1\le j,  l \le p$, $j \ne l$, we denote
\begin{equation}
\label{abstr_c_circ_jl}
c^{\circ}_{jl} := \min \{c_*, n^{-1} |\gamma^{\circ}_l - \gamma^{\circ}_j|\}.
\end{equation}
Clearly, there exists a number $i_0 = i_0(j,l)$, where $j \le i_0 \le l-1$ if $j<l$
and $l \le i_0 \le j-1$ if $l<j$, such that $\gamma^\circ_{i_0+1} - \gamma_{i_0}^\circ \ge c^\circ_{jl}$.
It means that on the interval between $\gamma_{j}^\circ$ and $\gamma_{l}^\circ$
there is a gap in the spectrum of $S$ of length at least $c^\circ_{jl}$.
If such $i_0$ is not unique, we agree to take the minimal possible $i_0$ (for definiteness).

  We choose a number $t^{00}_{jl}\le t^0$ such that  (see \eqref{abstr_C1_C2})
\begin{equation}
\label{abstr_t00_jl}
t^{00}_{jl} \le (4C_2)^{-1} c^{\circ}_{jl} = (4 \beta_2)^{-1} \delta^{1/2} \|X_1\|^{-3 } c^{\circ}_{jl}.
\end{equation}
By \eqref{abstr_frakA_est}, for $|t| \le t^{00}_{jl}$ we have $\| {\mathfrak A}(t) - SP\| \le c^\circ_{jl}/4$.
Hence, the segments $[\gamma_1^\circ - c_{jl}^\circ/4, \gamma^\circ_{i_0} + c_{jl}^\circ/4]$ and
$[\gamma_{i_0+1}^\circ - c_{jl}^\circ/4, \gamma^\circ_{p} + c_{jl}^\circ/4]$
are disjoint, and  the distance between them is at least $c_{jl}^\circ/2$.
Consequently, for $|t| \le t^{00}_{jl}$ the perturbed operator ${\mathfrak A}(t)$
has exactly $k_1+\dots + k_{i_0}$ eigenvalues (counted with multiplicities)
in the segment \hbox{$[\gamma_1^\circ - c_{jl}^\circ/4, \gamma^\circ_{i_0} + c_{jl}^\circ/4]$}.
These are $\nu^{(1)}_1(t),\dots, \nu^{(1)}_{k_1}(t); \dots; \nu^{(i_0)}_1(t),\dots, \nu^{(i_0)}_{k_{i_0}}(t)$.
We denote the corresponding eigenspace by ${\mathfrak F}_{jl}^{(1)}(t)$;
for $t\ne 0$ it coincides with the eigenspace \hbox{${\mathfrak F}(t;[0, (\gamma^\circ_{i_0} + c^\circ_{jl}/4)t^2])$}
 of $A(t)$. The elements  $\varphi^{(1)}_{1}(t), \dots, \varphi^{(1)}_{k_1}(t);\dots; \varphi^{(i_0)}_{1}(t), \dots, \varphi^{(i_0)}_{k_{i_0}}(t)$ form an orthonormal basis in ${\mathfrak F}_{jl}^{(1)}(t)$.
Similarly, for $|t| \le t^{00}_{jl}$ the perturbed operator ${\mathfrak A}(t)$
has exactly $k_{i_0+1}+\dots + k_{p}$ eigenvalues (counted with multiplicities)
in the segment  $[\gamma_{i_0+1}^\circ - c_{jl}^\circ/4, \gamma^\circ_{p} + c_{jl}^\circ/4]$.
These are $\nu^{(i_0+1)}_1(t),\dots, \nu^{(i_0+1)}_{k_{i_0+1}}(t); \dots; \nu^{(p)}_1(t),\dots, \nu^{(p)}_{k_{p}}(t)$.
The corresponding eigenspace is denoted by ${\mathfrak F}_{jl}^{(2)}(t)$; for $t\ne 0$ it coincides with the eigenspace
${\mathfrak F}(t;[(\gamma^\circ_{i_0+1} - c^\circ_{jl}/4 )t^2, (\gamma^\circ_{p} + c^\circ_{jl}/4 )t^2 ])$
of $A(t)$. The elements  $\varphi^{(i_0+1)}_{1}(t), \dots, \varphi^{(i_0+1)}_{k_{i_0+1}}(t);\dots; \varphi^{(p)}_{1}(t), \dots, \varphi^{(p)}_{k_p}(t)$ form an orthonormal basis in ${\mathfrak F}_{jl}^{(2)}(t)$.
Let $F^{(r)}_{jl}(t)$ be the orthogonal projections onto ${\mathfrak F}_{jl}^{(r)}(t)$,  $r=1,2$.
Then the spectral projection $F(t)$ of the operator $A(t)$ for the interval $[0,\delta]$
can be represented as
\begin{equation}
\label{abstr_F(t)_F(1)_F(2)}
F(t) = F^{(1)}_{jl} (t) + F^{(2)}_{jl} (t), \quad |t| \le t^{00}_{jl}.
\end{equation}

The following statement was proved in \cite[Proposition 2.1]{Su4}.

\begin{proposition}[\cite{Su4}]
For $|t|\le t^{00}_{jl}$ we have
    \begin{equation}
    \label{abstr_cluster_thresold}
    \begin{split}
    \| F^{(1)}_{jl} (t) - (P_1 + \cdots + P_{i_0}) \| &\le C_{6,jl} |t|, \\
    \| F^{(2)}_{jl} (t) - (P_{i_0+1} + \cdots + P_p) \| &\le C_{6,jl} |t|.
    \end{split}
    \end{equation}
    The number $t^{00}_{jl}$ is subject to~\emph{\eqref{abstr_c_circ_jl}}, \emph{\eqref{abstr_t00_jl}},
and the constant~$C_{6,jl}$ is given by
    \begin{equation}
    \label{abstr_C6jl}
    C_{6,jl} = \beta_6 \delta^{-1/2} \|X_1\|^5 (c^{\circ}_{jl})^{-2}.
    \end{equation}
\end{proposition}

\section{Threshold approximations for the operator $\cos(\tau A(t)^{1/2})$}

\subsection{Approximation of the operator $A(t)^{1/2} F(t)$}

The following result was proved in~\cite[Theorem~2.4]{BSu5}.

\begin{theorem}
    Under the assumptions of Subsections~\emph{\ref{abstr_X_A_section}} and~\emph{\ref{abstr_nondegenerated_section}}, we have
    \begin{equation}
    \label{abstr_A_sqrt_threshold_1}
    \| A(t)^{1/2} F(t) - (t^2 S)^{1/2} P \| \le C_7 t^2, \quad |t| \le t^0,
    \end{equation}
    where
    \begin{equation}
    \label{abstr_C7}
    C_7 := C_1 C_3^{1/2} + C_1 \| X_1 \| +  \frac{1}{2} C_2 c_*^{-1/2} = \beta_7 \delta^{-1/2} \|X_1\|^2 + \frac{\beta_2}{2} \delta^{-1/2} c_*^{-1/2} \|X_1\|^3.
    \end{equation}
\end{theorem}

We need to find more precise approximation for the operator $A(t)^{1/2} F(t)$.
For this, we use the following representation (see, e.~g.,~\cite[Chapter~3, Section~3, Subsection~3]{Kr})
\begin{equation}
\label{abstr_A_sqrt_repres}
A(t)^{1/2} F(t)  = \frac{1}{\pi} \int_{0}^{\infty} s^{-1/2} (A(t) + sI)^{-1} A(t) F(t) \, ds, \quad 0 < |t| \le t^0.
\end{equation}
Also, we need the following threshold approximation for the resolvent \hbox{$(A(t) + sI)^{-1}$} obtained in~\cite[(5.19)]{BSu2}:
\begin{multline}
\label{abstr_resolv_threshold}
(A(t) + sI)^{-1} F(t) = \Xi(t,s) + t (Z \Xi(t,s) + \Xi(t,s) Z^*) - \\- t^3 \Xi(t,s) N \Xi(t,s) + \mathcal{J}(t,s), \quad |t| \le t^0, \; s > 0.
\end{multline}
Here $\Xi(t,s): = (t^2 SP + sI)^{-1} P$, and the operator $\mathcal{J}(t,s) $ satisfies the following estimate (see~\cite[Subsection~5.2]{BSu2})
\begin{equation}
\label{abstr_J_est}
\| \mathcal{J}(t,s) \| \le C_8 t^4 (c_* t^2 + s)^{-2} + C_9 t^2 (c_* t^2 + s)^{-1},
\end{equation}
where
\begin{equation}
\label{abstr_C8_C9}
\begin{split}
C_8 &= \beta_8 \delta^{-1} \|X_1\|^4 + \check{\beta}_8 c_*^{-1} \delta^{-1} \|X_1\|^6,\\
C_9 &= \beta_9 \delta^{-1} \|X_1\|^2 + \check{\beta}_{9} c_*^{-1} \delta^{-1} \|X_1\|^4.
\end{split}
\end{equation}

By~\eqref{abstr_A(t)_threshold_2} and~\eqref{abstr_resolv_threshold},
\begin{equation}
\label{abstr_sum}
\begin{aligned}
(A(t) + sI)^{-1} A(t) F(t) &= t^2 \Xi(t,s)  SP + t^3 Z \Xi(t,s)   SP \\
&+    t^3 \Xi(t,s) K -  t^5 \Xi(t,s) N \Xi(t,s) SP + Y(t,s),
\end{aligned}
\end{equation}
where
\begin{equation*}
\begin{aligned}
Y(t,s)&:=
 \Xi(t,s) \Psi(t) + t (Z \Xi(t,s) + \Xi(t,s) Z^*)(t^3 K + \Psi(t))
\\
 &- t^3 \Xi(t,s) N \Xi(t,s) (t^3 K + \Psi(t)) + \mathcal{J} (t,s) \left( t^2 SP + t^3 K + \Psi (t) \right).
\end{aligned}
\end{equation*}
We have taken into account that $Z^* P =0$.
Substituting \eqref{abstr_sum} in~\eqref{abstr_A_sqrt_repres}, we obtain
\begin{equation}
\label{abstr_1/2}
A(t)^{1/2} F(t)  = \sum_{j=1}^4 I_j(t) + \Phi(t), \quad |t| \le t^0,
\end{equation}
where
\begin{align}
\label{I1}
I_1(t)& := \frac{1}{\pi} \int_{0}^{\infty} s^{-1/2} t^2 \Xi(t,s)  SP \, ds,
\\
\nonumber
I_2(t) &:= \frac{1}{\pi} \int_{0}^{\infty} s^{-1/2}  t^3  Z \Xi(t,s)  SP\, ds= tZ I_1(t),
\\
\nonumber
I_3(t) &:= \frac{1}{\pi} \int_{0}^{\infty} s^{-1/2}  t^3 \Xi(t,s) K \, ds = t I_1(t) S^{-1}PK,
\\
\nonumber
I_4(t) &:= - \frac{1}{\pi} \int_{0}^{\infty} s^{-1/2}   t^5 \Xi(t,s) N \Xi(t,s) SP\, ds,
\\
\label{Phi}
\Phi(t) &:= \frac{1}{\pi} \int_{0}^{\infty} s^{-1/2}  Y(t,s) \, ds.
\end{align}
For $t=0$ we put $I_j(0):=0$, $j=1,2,3,4,$ and $\Phi(0):=0$.
Then representation \eqref{abstr_1/2} for $t=0$ is obvious.

Using representation of the form \eqref{abstr_A_sqrt_repres}
for the operator $(t^2SP)^{1/2}P=|t| S^{1/2}P$, we see that
\begin{equation}
\label{abstr_I1-3}
I_1(t)=  |t| S^{1/2} P,\quad I_2(t)= t |t| Z S^{1/2} P, \quad I_3(t) = t |t| S^{-1/2} P K.
\end{equation}
Next, recalling that $N = N_0 + N_*$, we represent $I_4(t)$ as
\begin{align}
\label{abstr_I(t)_I0(t)_I*(t)}
I_4(t) &= I_0(t) + I_*(t),
\\
\nonumber
I_0(t) &:= -\frac{t^5}{\pi} \int_{0}^{\infty} s^{-1/2} \Xi (t,s) N_0 \Xi (t,s) \, SP \, ds,
\\
\label{I*}
I_*(t) &:= -\frac{t^5}{\pi} \int_{0}^{\infty} s^{-1/2} \Xi (t,s) N_* \Xi (t,s) \, SP \, ds.
\end{align}
Let us calculate the operator $I_0(t)$  in the basis $\{\omega_l\}_{l=1}^n$.
Since $N_0 S = S N_0$ (see \eqref{abstr_S_eigenvectors} and \eqref{abstr_N_0_N_*}), then
\begin{equation*}
I_0 (t) \omega_l = -\frac{t^5}{\pi} \int_{0}^{\infty} s^{-1/2} \frac{\gamma_l}{( \gamma_l t^2 + s)^{2}} N_0 \omega_l \, ds = - \frac{1}{2} t |t| \gamma_l^{-1/2} N_0 \omega_l,\quad l=1,\dots,n.
\end{equation*}
In the invariant terms, we have
\begin{equation}
\label{I0}
I_0(t) = - \frac{1}{2} t |t| N_0 S^{-1/2} P.
\end{equation}

Now, substituting \eqref{abstr_I1-3}, \eqref{abstr_I(t)_I0(t)_I*(t)}, and \eqref{I0} in \eqref{abstr_1/2}
and recalling that $K = Z S P + S P Z^* + N_0 + N_*$ (see~\eqref{abstr_K_N_def},~\eqref{abstr_F1_K0_N_invar}), we obtain
\begin{equation}
\label{repr}
\begin{aligned}
A(t)^{1/2} F(t) &= |t| S^{1/2} P +  t |t| \left( Z S^{1/2} P +  S^{1/2} P Z^* \right)
+ \frac{1}{2} t |t| N_0 S^{-1/2} P \\
&+ t |t| S^{-1/2} P N_* + I_*(t) + \Phi(t).
\end{aligned}
\end{equation}
We have taken into account that $PZ=0$.

Consider the term
\begin{equation}
\label{II*}
{\mathcal I}_*(t) := I_*(t) - t |t| N_* S^{-1/2} P.
\end{equation}
Since $|t| S^{-1/2} P = S^{-1} I_1(t)$, it follows from \eqref{I1} and \eqref{I*} that
\begin{equation}
\label{III*}
\begin{aligned}
{\mathcal I}_*(t) &:= -\frac{t^3}{\pi} \int_{0}^{\infty} s^{-1/2} \left( \Xi (t,s) N_* \Xi (t,s) \,t^2 SP + N_* \Xi (t,s)\right) \, ds
\\
&= -\frac{t^3}{\pi} \int_{0}^{\infty} s^{-1/2} \left( \Xi (t,s) N_*  + N_* \Xi (t,s) - s \,\Xi (t,s) N_* \Xi (t,s)\right) \, ds.
\end{aligned}
\end{equation}
From the last representation it is clear that the operator ${\mathcal I}_*(t)$
is selfadjoint. Is is easily seen that
\begin{equation}
\label{IIII*}
{\mathcal I}_*(t)= t|t| {\mathcal I}_*(1).
\end{equation}
Relations \eqref{abstr_S_nondegenerated} and \eqref{III*} imply the estimate
\begin{equation*}
\| \mathcal{I}_*(1) \| \le \frac{1}{\pi} \|N\|  \int_{0}^{\infty} s^{-1/2} \left( 2(s+ c_*)^{-1} + s(s+ c_*)^{-2}\right)\,  ds.
\end{equation*}
Combining this with \eqref{abstr_K_N_estimates}, we obtain
\begin{equation}
\label{I*est}
\| \mathcal{I}_* (1)\| \le \frac{5}{2}(2\delta)^{-1/2} c_*^{-1/2} \|X_1\|^3.
\end{equation}

By~\eqref{abstr_t0_fixation}, \eqref{abstr_Z_R_S_est}, \eqref{abstr_C5}, \eqref{abstr_K_N_estimates}, \eqref{abstr_S_nondegenerated}, \eqref{abstr_J_est}, and \eqref{abstr_C8_C9}, it is easy to check that the term
\eqref{Phi} satisfies
\begin{equation}
\label{abstr_A_sqrt_threshold_pre}
 \| \Phi (t) \| \le C_{10} |t|^3, \quad |t| \le t^0,
\end{equation}
where
\begin{equation}
\label{abstr_C10}
C_{10} = \beta_{10} \|X_1\|^4 \delta^{-1} c_*^{-1/2} + \beta_{11} \|X_1\|^6 \delta^{-1} c_*^{-3/2} + \beta_{12} \|X_1\|^8 \delta^{-1} c_*^{-5/2}.
\end{equation}

Finally, combining \eqref{repr}--\eqref{abstr_C10}, we arrive at the following result.

\begin{theorem}
    Suppose that the assumptions of Subsections~\emph{\ref{abstr_X_A_section}} and~\emph{\ref{abstr_nondegenerated_section}}
are satisfied. For $|t| \le t^0$ we have
     \begin{equation}
     \label{abstr_A_sqrt_threshold_2}
     A(t)^{1/2} F(t) = |t| S^{1/2} P  + t |t| G + \Phi (t),
     \end{equation}
 where
     \begin{equation}
     \label{abstr_G}
     G := Z S^{1/2} P + S^{1/2}P Z^* + \frac{1}{2} N_0 S^{-1/2} P  + S^{-1/2} N_* P +   N_* S^{-1/2} P + \mathcal{I}_* (1).
     \end{equation}
    The operators $S$ and $Z$ are defined in Subsection~\emph{\ref{abstr_Z_R_S_op_section}},
the operators $N_0$ and $N_*$ are defined by~\emph{\eqref{abstr_N_0_N_*}}. The operator $\mathcal{I}_*(1)$ is given by~\emph{\eqref{III*}} with $t=1$ and satisfies estimate \emph{\eqref{I*est}}. The operator $\Phi(t)$ satisfies estimate \eqref{abstr_A_sqrt_threshold_pre},
   where $C_{10}$ is given by \emph{\eqref{abstr_C10}}.
\end{theorem}

\subsection{Approximation of the operator $\cos (\tau A(t)^{1/2}) P$}

In this subsection, we approximate the operator $\cos (\tau A(t)^{1/2}) P$ by $\cos (\tau (t^2 S)^{1/2}P) P$
for $\tau \in \mathbb{R}$ and $|t|\le t^0$.
Such approximation was found in \cite[Section 2]{BSu5}.
Now we repeat the proof of this result tracing more carefully how different terms are estimated.
This will be needed in what follows to confirm the sharpness of the result.
On the other hand,  we will distinguish an important case where the result of \cite{BSu5} can be refined.

Consider the operator
\begin{equation}
    \label{abstr_E}
    E (t, \tau) := e^{-i\tau A(t)^{1/2}} P - e^{-i\tau (t^2 S)^{1/2}P} P.
\end{equation}
We have
\begin{gather}
    \label{abstr_E_E1_E2}
    E (t, \tau) = E_1 (t, \tau) + E_2 (t, \tau), \\
    \label{abstr_E1}
    E_1 (t, \tau) := e^{-i\tau A(t)^{1/2}} F(t)^{\perp} P - F(t)^{\perp} e^{-i\tau (t^2 S)^{1/2}P} P, \\
    \label{abstr_E2}
    E_2 (t, \tau) := e^{-i\tau A(t)^{1/2}} F(t) P - F(t) e^{-i\tau (t^2 S)^{1/2}P} P.
\end{gather}
Since  $F(t)^\perp P = (P - F(t))P$, \eqref{abstr_F(t)_threshold_1}  implies the following estimate for the operator \eqref{abstr_E1}:
\begin{equation}
    \label{abstr_E1_estimate}
    \| E_1 (t, \tau) \| \le 2 C_1 |t|, \quad |t| \le t^0.
\end{equation}

The operator \eqref{abstr_E2}  can be written as
\begin{gather}
    \label{abstr_E2_Sigma}
    E_2 (t, \tau) = e^{-i \tau A(t)^{1/2}} \Sigma (t, \tau), 
\\
\notag
\Sigma (t, \tau) :=  F(t) P - e^{i \tau A(t)^{1/2}}F(t) e^{-i\tau (t^2 S)^{1/2}P} P.
\end{gather}
Obviously,  $\Sigma(t,0)=0$ and     $\Sigma'(t, \tau) = \frac {d \Sigma} {d \tau} (t, \tau)$ takes the form
\begin{equation}
    \label{abstr_Sigma'}
    \Sigma'(t, \tau) = -i e^{i \tau A(t)^{1/2}}F(t) \left( A(t)^{1/2}F(t) - (t^2 S)^{1/2} P \right) e^{-i\tau (t^2 S)^{1/2}P} P.
\end{equation}
Since $\Sigma (t, \tau) = \int_{0}^{\tau} \Sigma'(t, \tilde{\tau}) \, d \tilde{\tau}$,
from~\eqref{abstr_Sigma'} and~\eqref{abstr_A_sqrt_threshold_1} it follows that
\begin{equation}
    \label{abstr_Sigma_estimate}
    \| \Sigma (t, \tau) \| \le C_7 |\tau| t^2, \quad |t| \le t^0.
\end{equation}

Relations~\eqref{abstr_E},~\eqref{abstr_E_E1_E2},~\eqref{abstr_E1_estimate},~\eqref{abstr_E2_Sigma}, and~\eqref{abstr_Sigma_estimate}
 imply the following result which is similar  to Theorem~2.5 from~\cite{BSu5}.

\begin{theorem}
    \label{abstr_cos_general_thrm_wo_eps}
    Under the assumptions of Subsections~\emph{\ref{abstr_X_A_section}} and~\emph{\ref{abstr_nondegenerated_section}},
for $\tau \in \mathbb{R}$ and  $|t| \le t^0$ we have
    \begin{equation}
        \label{abstr_cos_general_est_wo_eps}
        \| \cos (\tau A(t)^{1/2}) P - \cos (\tau (t^2 S)^{1/2}P) P \| \le 2 C_1 |t| + C_7 |\tau| t^2.
    \end{equation}
The number $t^0$ is subject to~\emph{\eqref{abstr_t0_fixation}}, and the constants $C_1$ and $C_7$ are defined
by~\emph{\eqref{abstr_C1_C2}} and~\emph{\eqref{abstr_C7}}, respectively.
\end{theorem}

Now we proceed to more subtle considerations that will allow us to improve the result under the additional assumptions.
Using representation~\eqref{abstr_A_sqrt_threshold_2}, from \eqref{abstr_Sigma'} we obtain
\begin{equation}
    \label{abstr_Sigma}
    \Sigma (t, \tau) = -i \int_0^{\tau} e^{i \tilde{\tau} A(t)^{1/2}}F(t) ( t |t| G  + \Phi (t)) e^{-i \tilde{\tau} (t^2 S)^{1/2}P} P \, d \tilde{\tau}.
\end{equation}
Taking~\eqref{abstr_G} into account and recalling that $Z^* P = 0$, we represent the operator~\eqref{abstr_Sigma} as
\begin{gather}
    \label{abstr_Sigma_tildeSigma_hatSigma}
    \Sigma (t, \tau) = \widetilde{\Sigma}(t, \tau) + \widehat{\Sigma}(t, \tau),\\
    \label{abstr_tildeSigma}
    \widetilde{\Sigma}(t, \tau) := -i \int_0^{\tau} e^{i \tilde{\tau} A(t)^{1/2}}F(t) \left( t |t|  Z S^{1/2} P + \Phi (t) \right) e^{-i \tilde{\tau} (t^2 S)^{1/2}P} P \, d \tilde{\tau}, \\
    \label{abstr_hatSigma}
    \widehat{\Sigma}(t, \tau) :=  -i \, t |t| \int_0^{\tau} e^{i \tilde{\tau} A(t)^{1/2}}F(t) \widehat{G} e^{-i \tilde{\tau} (t^2 S)^{1/2}P} P \, d \tilde{\tau},
\end{gather}
where
\begin{equation}
\label{abstr_hatG}
\widehat{G} = \frac{1}{2} N_0 S^{-1/2} P  + S^{-1/2} N_* P +   N_* S^{-1/2} P + \mathcal{I}_*(1).
\end{equation}
Since $P Z = 0$, then $F(t)Z S^{1/2} P = (F(t) - P ) Z S^{1/2} P$. Hence, relations~\eqref{abstr_Z_R_S_est},~\eqref{abstr_F(t)_threshold_1}, and~\eqref{abstr_A_sqrt_threshold_pre} imply the following estimate for the term~\eqref{abstr_tildeSigma}:
\begin{equation}
    \label{abstr_tildeSigma_estimate}
    \| \widetilde{\Sigma}(t, \tau) \| \le C_{11} | \tau | |t|^3, \quad |t| \le t^0,
\end{equation}
where $C_{11} = C_1 (8 \delta)^{-1/2} \| X_1 \|^2 + C_{10}$.
Taking~\eqref{abstr_C1_C2} and~\eqref{abstr_C10} into account, we see that
\begin{equation}
    \label{abstr_C11}
\begin{aligned}
    C_{11}  &= 8^{-1/2} \beta_1  \|X_1\|^3  \delta^{-1}+ \beta_{10} \|X_1\|^4 \delta^{-1} c_*^{-1/2}
\\
&+ \beta_{11} \|X_1\|^6 \delta^{-1} c_*^{-3/2} + \beta_{12} \|X_1\|^8 \delta^{-1} c_*^{-5/2}.
\end{aligned}
\end{equation}

Now, \eqref{abstr_E2_Sigma} and~\eqref{abstr_Sigma_tildeSigma_hatSigma} imply that
\begin{equation}
    \label{abstr_E2_tildeE2_hatE2}
    E_2 (t, \tau) = \widetilde{E}_2(t, \tau) + \widehat{E}_2(t, \tau),
\end{equation}
where $\widetilde{E}_2(t, \tau) = e^{-i \tau A(t)^{1/2}} \widetilde{\Sigma} (t, \tau)$, $\widehat{E}_2(t, \tau) = e^{-i \tau A(t)^{1/2}} \widehat{\Sigma} (t, \tau)$. By \eqref{abstr_tildeSigma_estimate},
\begin{equation}
    \label{abstr_tildeE2_esimate}
    \| \widetilde{E}_2(t, \tau) \| \le C_{11} | \tau | |t|^3, \quad |t| \le t^0.
\end{equation}

Finally, relations~\eqref{abstr_E}, \eqref{abstr_E_E1_E2}, \eqref{abstr_E1_estimate}, \eqref{abstr_hatSigma}, \eqref{abstr_hatG}, \eqref{abstr_E2_tildeE2_hatE2}, and~\eqref{abstr_tildeE2_esimate}, together with~\eqref{abstr_K_N_estimates}, \eqref{abstr_S_nondegenerated},
and~\eqref{I*est} imply the following result.
\begin{theorem}
    \label{abstr_cos_enchanced_thrm_1_wo_eps}
    Under the assumptions of Subsections~\emph{\ref{abstr_X_A_section}} and~\emph{\ref{abstr_nondegenerated_section}},
for $\tau \in \mathbb{R}$ and $|t| \le t^0$ we have
    \begin{equation*}
        e^{-i\tau A(t)^{1/2}} P - e^{-i\tau (t^2 S)^{1/2}P} P = E_1(t, \tau) + \widetilde{E}_2(t, \tau) + \widehat{E}_2(t, \tau),
    \end{equation*}
   where the first two terms satisfy estimates~\emph{\eqref{abstr_E1_estimate}} and~\emph{\eqref{abstr_tildeE2_esimate}}, respectively.
The third term admits the following representation
    \begin{equation}
        \label{abstr_hatE2}
        \widehat{E}_2(t, \tau) = -i \, t |t| e^{-i \tau A(t)^{1/2}} \int_0^{\tau} e^{i \tilde{\tau} A(t)^{1/2}}F(t) \widehat{G}  e^{-i \tilde{\tau} (t^2 S)^{1/2}P} P \, d \tilde{\tau},
    \end{equation}
    where the operator~$\widehat{G}$ is defined by~\emph{\eqref{abstr_hatG}}. We have
    \begin{gather*}
        \| \widehat{E}_2(t, \tau) \| \le C_{12} | \tau| t^2, \quad |t| \le t^0, \quad 
        C_{12} = 5 (2 \delta)^{-1/2} c_*^{-1/2} \| X_1 \|^3.
    \end{gather*}
\end{theorem}

Note that, if $N=0$, then also $N_0=N_*=0$, ${\mathcal I}_*(1)=0$, whence
$\widehat{G}=0$ and $\widehat{E}_2(t,\tau)=0$. We obtain the following corollary.

\begin{corollary}
    \label{abstr_coroll_enchd_est_wo_eps_N=0}
Suppose that the assumptions of Subsections~\emph{\ref{abstr_X_A_section}} and~\emph{\ref{abstr_nondegenerated_section}}
are satisfied. Suppose that $N=0$. Then for $\tau \in \mathbb{R}$ and $|t|\le t^0$ we have
    \begin{equation}
        \label{abstr_cos_enchanced_est_1_wo_eps}
        \| \cos(\tau A(t)^{1/2}) P - \cos(\tau (t^2 S)^{1/2}P) P \| \le 2 C_1 |t| + C_{11} | \tau | |t|^3.
    \end{equation}
\end{corollary}

\subsection{Estimates for the terms containing $N_*$}

We will use the notation and the results of Section~\ref{abstr_cluster_section}.
Recall that $N=N_0 + N_*$. By Remark~\ref{abstr_N_remark}, we have
\begin{equation}
    \label{PNP}
    P_j N_* P_j = 0,\quad j = 1, \ldots ,p;\qquad P_l N_0 P_j  =0\quad \text{if} \  l \ne j.
\end{equation}
Thus, the operators $N_0$ and $N_*$ admit the following invariant representations:
\begin{equation}
    \label{abstr_N_invar_repers}
    N_0 = \sum_{j=1}^{p} P_j N P_j, \quad N_* = \sum_{\substack{1 \le j, l \le p: \\ j \ne l}} P_j N P_l.
\end{equation}

The term~(\ref{abstr_hatE2}) can be written as
\begin{equation}
    \label{abstr_hatE2_E0_E*}
    \widehat{E}_2(t, \tau) = E_0(t, \tau) + E_*(t, \tau),
\end{equation}
where
\begin{align}
    \label{abstr_E0}
    E_0(t, \tau) &= -\frac{i}{2} t |t| e^{-i \tau A(t)^{1/2}} \int_0^{\tau} e^{i \tilde{\tau} A(t)^{1/2}}F(t) S^{-1/2} N_0 e^{-i \tilde{\tau} (t^2 S)^{1/2}P} P \, d \tilde{\tau}, \\
    \label{abstr_E*}
    E_*(t, \tau) &= -i  t |t| e^{-i \tau A(t)^{1/2}} \int_0^{\tau} e^{i \tilde{\tau} A(t)^{1/2}}F(t)
G_*  e^{-i \tilde{\tau} (t^2 S)^{1/2}P} P \, d \tilde{\tau},
\end{align}
where
\begin{equation}
    \label{G*}
G_* := S^{-1/2} N_*P  +   N_* S^{-1/2}P  +  \mathcal{I}_* (1).
\end{equation}

In this subsection, we obtain the analog of estimate \eqref{abstr_cos_enchanced_est_1_wo_eps} under the weaker assumption that $N_0=0$.
For this, we have to estimate the operator~\eqref{abstr_E*}. However, we are able to do this only for a
smaller interval of $t$. By \eqref{abstr_P_Pj}, \eqref{III*}, \eqref{PNP}, and \eqref{G*},
it is seen that 
\begin{equation*}
G_* = \sum_{\substack{1 \le j, l \le p: \\ j \ne l}} P_j G_* P_l.
\end{equation*}
Hence, the term \eqref{abstr_E*} can be represented as
\begin{gather}
    \label{abstr_E*_sumJjl}
    E_*(t, \tau) = -i e^{-i \tau A(t)^{1/2}} \sum_{\substack{1 \le j, l \le p: \\  j \ne l}} J_{jl} (t, \tau),\\
    \label{abstr_Jjl}
    J_{jl} (t, \tau) = t |t| \int_0^{\tau} e^{i \tilde{\tau} A(t)^{1/2}}F(t) P_j G_* P_l e^{-i \tilde{\tau} (t^2 S)^{1/2}P} P \, d \tilde{\tau}.
\end{gather}

We have to estimate only those terms in \eqref{abstr_E*_sumJjl} for which $P_j N P_l \ne 0$.
So, let $j\ne l$, and let $P_j N P_l \ne 0$. Suppose that $c^\circ_{jl}$ is defined by~\eqref{abstr_c_circ_jl},
and $t^{00}_{jl}$ is subject to~\eqref{abstr_t00_jl}.
By~\eqref{abstr_F(t)_F(1)_F(2)}, for $|t|\le t^{00}_{jl}$ the operator~\eqref{abstr_Jjl} can be represented as
\begin{gather}
    \label{abstr_J_jl_sumJ^r_jl}
    J_{jl}(t, \tau) = J^{(1)}_{jl}(t, \tau) + J^{(2)}_{jl}(t, \tau), \\
    \label{abstr_J^r_jl}
    \begin{split}
    J^{(r)}_{jl} (t, \tau) =  t |t| \int_0^{\tau} e^{i \tilde{\tau} A(t)^{1/2}}F^{(r)}_{jl}(t)  P_j G_*  P_l  e^{-i \tilde{\tau} (t^2 S)^{1/2}P} P \, d \tilde{\tau},\  r=1,2.
    \end{split}
\end{gather}

For definiteness, assume that $j<l$. Then $j < i_0+1$ and, by~\eqref{abstr_P_Pj} and~\eqref{abstr_cluster_thresold},
\begin{equation}
    \label{abstr_F^(2)_jl_P_j}
    \|F^{(2)}_{jl}(t) P_j \| = \| ( F^{(2)}_{jl}(t) - (P_{i_0+1} + \ldots + P_p) ) P_j \| \le C_{6,jl}|t|, \quad |t|\le t^{00}_{jl}.
\end{equation}

Combining~\eqref{G*},~\eqref{abstr_J^r_jl},~\eqref{abstr_F^(2)_jl_P_j}, and taking~\eqref{abstr_K_N_estimates}, \eqref{abstr_S_nondegenerated}, \eqref{abstr_C6jl}, and~\eqref{I*est} into account, we obtain
\begin{equation}
    \label{abstr_J^2_jl_estimate}
    \| J^{(2)}_{jl} (t, \tau)\| \le C_{6,jl} |t|^3 | \tau | (2 c_*^{-1/2} \|N\| + \|\mathcal{I}_* (1)\|) \le C_{13,jl}  | \tau | |t|^3, \quad |t| \le t^{00}_{jl},
\end{equation}
where $C_{13,jl} =\beta_{13} \delta^{-1} c_*^{-1/2} \|X_1\|^{8} (c^{\circ}_{jl})^{-2}$.

It remains to consider the term $J_{jl}^{(1)}(t,\tau)$.
Obviously, $P_l e^{-i \tilde{\tau} (t^2 S)^{1/2}P} P = e^{-i \tilde{\tau} |t| \sqrt{\gamma_l^{\circ}} } P_l$.
Recall that (see Subsection~\ref{abstr_cluster_treshold_section})
 the  (nonzero) eigenvalues of the operator $A(t)F_{jl}^{(1)}(t)$ are
$t^2\nu^{(1)}_1 (t), \ldots,  t^2\nu^{(1)}_{k_1} (t); \ldots ; t^2\nu^{(i_0)}_1 (t) , \ldots ,t^2\nu^{(i_0)}_{k_{i_0}} (t)$, and
$\nu^{(r)}_{q} (t) \in [\gamma^{\circ}_1 - c^{\circ}_{jl}/4, \gamma^{\circ}_{i_0} + c^{\circ}_{jl}/4]$.
The corresponding orthonormal eigenvectors are
$\varphi^{(1)}_1 (t), \ldots,\varphi^{(1)}_{k_1}(t); \ldots;  \varphi^{(i_0)}_1 (t), \ldots ,\varphi^{(i_0)}_{k_{i_0}} (t)$.
Then
\begin{equation*}
    e^{i \tilde{\tau} A(t)^{1/2}}F^{(1)}_{jl}(t) = \sum_{r=1}^{i_0} \sum_{q=1}^{k_r} e^{i \tilde{\tau} |t| \sqrt{\nu^{(r)}_q (t)}} (\cdot, \varphi^{(r)}_{q}(t))\varphi^{(r)}_{q}(t).
\end{equation*}
As a result, the operator  $J_{jl}^{(1)}(t,\tau)$ can be written as
\begin{equation}
    \label{abstr_J^1_jl}
    J^{(1)}_{jl} (t, \tau) = t |t| \sum_{r=1}^{i_0} \sum_{q=1}^{k_r} \left(  \int_0^{\tau} e^{i \tilde{\tau} |t| \bigl( \sqrt{ \mathstrut \nu^{(r)}_q (t)}-\sqrt{\mathstrut \gamma^{\circ}_l}\bigr) } d \tilde{\tau} \right) ( P_j G_* P_l \cdot, \varphi^{(r)}_{q}(t))\varphi^{(r)}_{q}(t).
\end{equation}
Calculating the integral in~\eqref{abstr_J^1_jl} and taking into account that
\begin{align*}
    \left|( \nu^{(r)}_q (t))^{1/2} - (\gamma^{\circ}_l)^{1/2}\right| &\ge (2\| X_1 \|)^{-1} | \nu^{(r)}_q (t)-\gamma^{\circ}_l |, \\
    | \nu^{(r)}_q (t)-\gamma^{\circ}_l | &\ge 3 c^{\circ}_{jl}/4, \quad |t| \le t^{00}_{jl},
\end{align*}
we obtain
\begin{multline}
    \label{abstr_J^1_jl_est_a}
    \left| \int_0^{\tau} e^{i \tilde{\tau} |t| \left( (\nu^{(r)}_q (t))^{1/2} - ( \gamma^{\circ}_l)^{1/2}\right) } d \tilde{\tau} \right|
\\
= |t|^{-1} \left| (\nu^{(r)}_q (t))^{1/2} - ( \gamma^{\circ}_l)^{1/2} \right|^{-1} \left| e^{i \tau t \left( (\nu^{(r)}_q (t))^{1/2} - ( \gamma^{\circ}_l)^{1/2} \right) } - 1 \right|
\\
\le 16 \| X_1 \| (3 c^{\circ}_{jl})^{-1}  |t|^{-1}.
\end{multline}
Now, relations~\eqref{abstr_J^1_jl} and~\eqref{abstr_J^1_jl_est_a} together with~\eqref{abstr_K_N_estimates},~\eqref{abstr_S_nondegenerated}, and~\eqref{I*est}
imply that
\begin{equation}
    \label{abstr_J^1_jl_estimate}
    \| J^{(1)}_{jl} (t, \tau) \| \le 16 \| X_1 \| (3 c^{\circ}_{jl})^{-1}  |t| (2 c_*^{-1/2}\|N\| + \|  \mathcal{I}_* (1)\|) \le C_{14,jl} |t|, \quad |t| \le  t^{00}_{jl},
\end{equation}
where
\begin{equation*}
C_{14,jl} = \beta_{14} \delta^{-1/2} c_*^{-1/2} \| X_1 \|^4 (c^{\circ}_{jl})^{-1}.
\end{equation*}
The case where $j>l$ can be treated similarly.

Denote
\begin{equation*}
    \mathcal{Z} = \{ (j, l) \colon 1 \le j, l \le p, \; j \ne l, \; P_j N P_l \ne 0\}.
\end{equation*}
We put
\begin{equation}
    \label{abstr_c^circ}
    c^{\circ} := \min_{(j, l) \in \mathcal{Z}} c^{\circ}_{jl}
\end{equation}
and choose a number $t^{00} \le t^0$ such that
\begin{equation}
    \label{abstr_t00}
    t^{00} \le (4 \beta_2 )^{-1} \delta^{1/2} \|X_1\|^{-3} c^{\circ}.
\end{equation}
We may assume that  $t^{00} \le t^{00}_{jl}$ for all $(j,l)\in {\mathcal Z}$ (see~\eqref{abstr_t00_jl}).

Now, relations~\eqref{abstr_E*_sumJjl},~\eqref{abstr_J_jl_sumJ^r_jl},~\eqref{abstr_J^2_jl_estimate}, and~\eqref{abstr_J^1_jl_estimate},
together with expressions for the constants~$C_{13,jl}$ and~$C_{14,jl}$ yield
\begin{equation}
    \label{abstr_E*_estimate}
    \| E_*(t, \tau) \| \le C_{13} | \tau | |t|^3 + C_{14} |t|, \quad |t| \le t^{00}.
\end{equation}
Here
\begin{equation}
    \label{abstr_C13_C14}
    \begin{aligned}
        C_{13} &= \beta_{13} n^2   \delta^{-1}  c_*^{-1/2} \| X_1 \|^{8}  (c^{\circ})^{-2},
\\
        C_{14} &= \beta_{14} n^2   \delta^{-1/2} c_*^{-1/2} \| X_1 \|^4  ( c^{\circ})^{-1}.
    \end{aligned}
\end{equation}

Finally, combining Theorem~\ref{abstr_cos_enchanced_thrm_1_wo_eps} and relations
~\eqref{abstr_hatE2_E0_E*},~\eqref{abstr_E0},~ \eqref{abstr_E*_estimate}, and denoting
$\check{E}(t,\tau):= E_1(t,\tau) + \widetilde{E}_2(t,\tau) + {E}_*(t,\tau)$, we arrive at the following result.

\begin{theorem}
    \label{abstr_cos_enchanced_thrm_2_wo_eps}
    Suppose that the assumptions of Subsections~\emph{\ref{abstr_X_A_section}} and~\emph{\ref{abstr_nondegenerated_section}}
are satisfied. Suppose that the number $t^{00}\le t^0$ is subject to~\emph{\eqref{abstr_c^circ}},~\emph{\eqref{abstr_t00}}.
Then for $\tau \in \mathbb{R}$ and $|t| \le t^{00}$ we have
  \begin{equation*}
        e^{-i\tau A(t)^{1/2}} P - e^{-i\tau (t^2 S)^{1/2}P} P = E_0 (t, \tau) + \check{E}(t, \tau),
    \end{equation*}
   where the second term satisfies
    \begin{equation*}
        \|\check{E}(t, \tau) \| \le C_{15} |t| + C_{16} |\tau| |t|^3, \quad |t| \le t^{00}.
    \end{equation*}
  The constants $C_{15}$ and $C_{16}$ are given by $C_{15} = 2C_1 + C_{14}$, $C_{16} = C_{11} + C_{13}$,
where $C_1$, $C_{11}$ are defined by~\emph{\eqref{abstr_C1_C2}},~\emph{\eqref{abstr_C11}}, and $C_{13}$, $C_{14}$ are defined
by~\emph{\eqref{abstr_C13_C14}}. The operator $E_0 (t, \tau)$ is given by~\emph{\eqref{abstr_E0}}
and satisfies the estimate
    \begin{equation*}
        \| E_0 (t, \tau) \| \le \frac{1}{2} c_*^{-1/2} (2 \delta)^{-1/2} \|X_1\|^3   |\tau | t^2.
    \end{equation*}
\end{theorem}

\begin{corollary}
    \label{abstr_coroll_enchd_est_wo_eps_N*=0}
    Suppose that the assumptions of Subsections~\emph{\ref{abstr_X_A_section}} and~\emph{\ref{abstr_nondegenerated_section}} are satisfied. If $N_0=0$, then for $\tau \in \mathbb{R}$ and $|t| \le t^{00}$ we have
    \begin{equation*}
        \| \cos(\tau A(t)^{1/2}) P - \cos(\tau (t^2 S)^{1/2}P) P \| \le C_{15} |t| + C_{16} | \tau | |t|^3.
    \end{equation*}
\end{corollary}

\begin{remark}
Let $\mu_l$, $l=1,\dots,n,$  be the coefficients at $t^3$ in the expansions \emph{\eqref{abstr_A(t)_eigenvalues_series}}.
By Remark~\emph{\ref{abstr_N_remark}}, the condition $N_0=0$ is equivalent to the relations $\mu_l=0$ for all $l=1,\dots,n$.
\end{remark}

\section{Approximation of the operator $\cos(\varepsilon^{-1} \tau A(t)^{1/2})$}
\label{abstr_aprox_thrm_section}

\subsection{Approximation of the operator $\cos(\varepsilon^{-1} \tau A(t)^{1/2})$ in the general case}

Let $\varepsilon > 0$. We study the behavior of the operator $\cos(\varepsilon^{-1} \tau A(t)^{1/2})$
for small $\varepsilon$. We multiply this operator by the ``smoothing factor''
\hbox{$\varepsilon^s (t^2 + \varepsilon^2)^{-s/2}P$}, where $s > 0$.
(The term is explained by the fact that in applications to differential operators  this factor turns into the smoothing operator.)
Our goal is to find approximation for the smoothed operator cosine with an error $O (\varepsilon)$ for minimal possible $s$.

Let $|t|\le t^0$. We apply Theorem~\ref{abstr_cos_general_thrm_wo_eps}. By~\eqref{abstr_cos_general_est_wo_eps} (with $\tau$ replaced by $\varepsilon^{-1} \tau$),
\begin{multline*}
    \| \cos (\varepsilon^{-1} \tau A(t)^{1/2}) P - \cos (\varepsilon^{-1} \tau (t^2 S)^{1/2}P) P \| \varepsilon^2 (t^2 + \varepsilon^2)^{-1}  
\\ 
\le \left( 2 C_1 |t| + C_7 \varepsilon^{-1}  |\tau| t^2\right) \varepsilon^2 (t^2 + \varepsilon^2)^{-1} \le (C_1 + C_7 |\tau|) \varepsilon.
\end{multline*}
Here we take $s=2$. We arrive at the following result which has been proved before in \cite[Theorem 2.7]{BSu5}.

\begin{theorem}
    \label{abstr_cos_general_thrm}
    Suppose that the assumptions of Subsections~\emph{\ref{abstr_X_A_section}} and~\emph{\ref{abstr_nondegenerated_section}}
    are satisfied. Then for~$\varepsilon > 0, \tau \in \mathbb{R}$, and $|t| \le t^0$ we have
    \begin{equation}
        \label{abstr_cos_general_est}
        \| \cos (\varepsilon^{-1} \tau A(t)^{1/2}) P - \cos (\varepsilon^{-1} \tau (t^2 S)^{1/2}P) P \| \varepsilon^2 (t^2 + \varepsilon^2)^{-1} \le (C_1 + C_7 |\tau|) \varepsilon.
    \end{equation}
    The number $t^0$ is subject to~\emph{\eqref{abstr_t0_fixation}}, and the constants $C_1$, $C_7$ are given by~\emph{\eqref{abstr_C1_C2}}
    and~\emph{\eqref{abstr_C7}}.
\end{theorem}

\subsection{ Refinement of approximation under the additional assumptions}

Corollary~\ref{abstr_coroll_enchd_est_wo_eps_N=0} allows us to improve the result of Theorem~\ref{abstr_cos_general_thrm} in the case where~$N = 0$.

\begin{theorem}
    \label{abstr_cos_enchanced_thrm_1}
    Suppose that the assumptions of Theorem~\emph{\ref{abstr_cos_general_thrm}} are satisfied.
Suppose that the operator $N$ defined in~\emph{\eqref{abstr_F1_K0_N_invar}} is equal to zero: $N = 0$.
Then for~$\varepsilon > 0, \tau \in \mathbb{R}$, and~$|t| \le t^0$ we have
    \begin{equation}
        \label{abstr_cos_enchanced_est_1}
        \begin{aligned}
        \|& \cos(\varepsilon^{-1} \tau A(t)^{1/2}) P - \cos(\varepsilon^{-1} \tau (t^2 S)^{1/2}P) P \| \varepsilon^{3/2} (t^2 + \varepsilon^2)^{-3/4} \\
        &\le ( C'_1 + C_{11} | \tau | ) \varepsilon.
        \end{aligned}
    \end{equation}
Here $C'_1 = \max\{2, 2 C_1\}$, $C_1$ and $C_{11}$ are defined by~\emph{\eqref{abstr_C1_C2}} and~\emph{\eqref{abstr_C11}}.
\end{theorem}

\begin{proof} Note that for $|t| \ge \sqrt[3]{\varepsilon}$ we have $\varepsilon^{3/2} (t^2 + \varepsilon^2)^{-3/4} \le \varepsilon$,
whence the left-hand side of~(\ref{abstr_cos_enchanced_est_1}) does not exceed $2 \varepsilon$.
Thus, we may assume that~$|t| < \sqrt[3]{\varepsilon}$. Using~\eqref{abstr_cos_enchanced_est_1_wo_eps} with $\tau$ replaced by
$\varepsilon^{-1} \tau$, for $|t| < \sqrt[3]{\varepsilon}$ we obtain
\begin{multline*}
    \| \cos(\varepsilon^{-1} \tau A(t)^{1/2}) P - \cos(\varepsilon^{-1} \tau (t^2 S)^{1/2}P) P \| \varepsilon^{3/2} (t^2 + \varepsilon^2)^{-3/4}  \\ 
    \le \left( 2 C_1 |t| + C_{11} \varepsilon^{-1}  | \tau | |t|^3 \right) \varepsilon^{3/2} (t^2 + \varepsilon^2)^{-3/4}
    \\
    \le 2 C_1 \varepsilon + C_{11} |\tau| \varepsilon^{1/2} |t|^{3/2} \le (2 C_1  + C_{11} |\tau|) \varepsilon.
\end{multline*}
The required statement follows.
\end{proof}

Similarly, Corollary~\ref{abstr_coroll_enchd_est_wo_eps_N*=0} yields the following result.

\begin{theorem}
    \label{abstr_cos_enchanced_thrm_2}
    Suppose that the assumptions of Subsections~\emph{\ref{abstr_X_A_section}} and~\emph{\ref{abstr_nondegenerated_section}}
    are satisfied. Suppose that the operator $N_0$ defined in~\emph{\eqref{abstr_N_0_N_*}} is equal to zero: $N_0 = 0$.
    Then for $\varepsilon > 0, \tau \in \mathbb{R}$, and $|t| \le t^{00}$ we have
    \begin{equation*}
    \begin{aligned}
        \|& \cos(\varepsilon^{-1} \tau A(t)^{1/2}) P - \cos(\varepsilon^{-1} \tau (t^2 S)^{1/2}P) P \| \varepsilon^{3/2} (t^2 + \varepsilon^2)^{-3/4} \\
        &\le (C'_{15} + C_{16} | \tau |) \varepsilon.
    \end{aligned}
    \end{equation*}
    Here the number $t^{00}\le t^0$ is subject to~\emph{\eqref{abstr_t00}}, the constant~$C'_{15}$ is given by~$C'_{15} = \max \{2, C_{15}\}$, and $C_{15}$, $C_{16}$ are as in Theorem~\emph{\ref{abstr_cos_enchanced_thrm_2_wo_eps}}.
\end{theorem}

\subsection{Sharpness of the result in the general case}

Now we show that the result of Theorem~\ref{abstr_cos_general_thrm} is sharp in the general case.
Namely, if~$N_0 \ne 0$, the exponent $s$ in the smoothing factor can not be taken smaller than 2.

\begin{theorem}
    \label{abstr_s<2_general_thrm}
    Suppose that the assumptions of Subsections~\emph{\ref{abstr_X_A_section}} and~\emph{\ref{abstr_nondegenerated_section}}
    are satisfied. Let $N_0 \ne 0$. Let $\tau \ne 0$ and $0 \le s < 2$. Then there does not exist a constant $C(\tau)>0$
    such that the estimate
        \begin{equation}
        \label{abstr_s<2_est_imp}
        \| \cos(\varepsilon^{-1} \tau A(t)^{1/2}) P - \cos(\varepsilon^{-1} \tau (t^2 S)^{1/2}P) P \| \varepsilon^{s} (t^2 + \varepsilon^2)^{-s/2} \le C(\tau) \varepsilon
    \end{equation}
    holds for all sufficiently small $|t|$ and $\varepsilon > 0$.
\end{theorem}

\begin{proof}
It suffices to assume that $1 \le s < 2$.
We start with a preliminary remark.
Since $F(t)^{\perp}P = (P - F(t))P$, from~(\ref{abstr_F(t)_threshold_1}) it follows that
\begin{multline}
    \label{abstr_s<2_cos_Fperp_est}
    \begin{aligned}
    \| \cos(\varepsilon^{-1} \tau A(t)^{1/2}) F(t)^{\perp} P \| \varepsilon (t^2 + \varepsilon^2)^{-1/2} \le C_1 |t| \varepsilon (t^2 + \varepsilon^2)^{-1/2} \le C_1 \varepsilon,
    \cr
    |t| \le t^0.
    \end{aligned}
\end{multline}

Fix $0 \ne \tau \in \mathbb{R}$.
We prove by contradiction. Suppose that for some $1\le s<2$ there exists a constant $C(\tau)>0$ such that~\eqref{abstr_s<2_est_imp}
 is valid for all sufficiently small~$|t|$ and~$\varepsilon$.
 By~\eqref{abstr_s<2_cos_Fperp_est}, this assumption is equivalent to
the existence of a constant $\widetilde{C}(\tau)$ such that
\begin{equation}
    \label{abstr_s<2_est_a}
    \left\| \left( \cos(\varepsilon^{-1} \tau A(t)^{1/2}) F(t) - \cos(\varepsilon^{-1} \tau (t^2 S)^{1/2}P) P\right)  P \right\| \varepsilon^{s} (t^2 + \varepsilon^2)^{-s/2} \le \widetilde{C} (\tau) \varepsilon
\end{equation}
for all sufficiently small~$|t|$ and~$\varepsilon$.

Recall that for $|t| \le t_*$ the power series expansions~\eqref{abstr_A(t)_eigenvalues_series} and~\eqref{abstr_A(t)_eigenvectors_series} are convergent.
Together with~\eqref{abstr_gamma_ge_c_*} and the Taylor formula for
   $\sqrt{1+x}, \; |x| < 1,$ this implies that for some $0 < t_{**} \le t_*$ we have the following convergent power series expansions:
   \begin{equation}
    \label{abstr_sqrtA(t)_eigenvalues_series}
    \sqrt{\lambda_l(t)} = \sqrt{\gamma_l} |t| \left( 1 + \frac{\mu_l}{2 \gamma_l} t + \ldots\right) , \qquad l = 1, \ldots, n, \quad  |t| \le t_{**}.
\end{equation}

For $ |t| \le t^0$ we have
\begin{equation}
\label{s<2_a}
    \cos(\varepsilon^{-1} \tau A(t)^{1/2})F (t) = \sum_{l=1}^{n} \cos(\varepsilon^{-1} \tau \sqrt{\lambda_l(t)})(\cdot,\varphi_l(t))\varphi_l(t).
\end{equation}
From the convergence of the power series expansions~\eqref{abstr_A(t)_eigenvectors_series} it follows that
\begin{equation}
    \label{abstr_s<2_phi_perturb_est}
    \| \varphi_l(t) - \omega_l \| \le c_1 |t|, \quad |t| \le t_*, \quad l = 1,\ldots,n.
\end{equation}
By \eqref{abstr_s<2_est_a}, \eqref{s<2_a}, and \eqref{abstr_s<2_phi_perturb_est}, there exists a constant
$\widehat{C} (\tau)$ such that
\begin{equation}
    \label{abstr_s<2_est_b}
    \left\| \sum_{l=1}^{n} \left( \cos(\varepsilon^{-1} \tau \sqrt{\mathstrut \lambda_l(t)}) - \cos(\varepsilon^{-1} \tau |t| \sqrt{\mathstrut \gamma_l}) \right) (\cdot,\omega_l)\omega_l  \right\| \varepsilon^{s} (t^2 + \varepsilon^2)^{-s/2} \le \widehat{C} (\tau) \varepsilon
\end{equation}
for all sufficiently small $|t|$ and $\varepsilon$.

The condition $N_0 \ne 0$ means that $\mu_j \ne 0$ at least for one $j$.
Applying the operator under the norm sign in~\eqref{abstr_s<2_est_b} to $\omega_j$, we obtain
\begin{equation}
    \label{abstr_s<2_est_c}
    \left| \cos \bigl(\varepsilon^{-1} \tau \sqrt{\mathstrut \smash{\lambda_j(t)}}\bigr)
    - \cos\bigl(\varepsilon^{-1} \tau |t| \sqrt{\mathstrut \smash{\gamma_j}}\bigr) \right| \varepsilon^{s} (t^2 + \varepsilon^2)^{-s/2} \le \widehat{C} (\tau) \varepsilon
\end{equation}
for all sufficiently small $|t|$ and $\varepsilon$.

Now, for a fixed $\tau \ne 0$, we put
\begin{equation*}
t = t(\varepsilon) = (2 \pi)^{1/2} \left| \frac{\sqrt{\gamma_j}}{\mu_j \tau}\right|^{1/2} \varepsilon^{1/2} = c \varepsilon^{1/2}.
\end{equation*}
If $\varepsilon$ is so small that $t(\varepsilon) \le t_{**}$,
from~\eqref{abstr_sqrtA(t)_eigenvalues_series} it follows that
\begin{multline}
\label{abstr_s<2_|cos-cos|}
\left| \cos \bigl(\varepsilon^{-1} \tau \sqrt{\mathstrut \smash{\lambda_j(t(\varepsilon))}}\bigr) -
\cos\bigl(\varepsilon^{-1} \tau |t(\varepsilon)| \sqrt{\mathstrut \smash{\gamma_j}}\bigr) \right|
\\
=
\left| \cos \bigl(\pi \, \sgn \,\mu_j  + \alpha_j \varepsilon^{-1/2} + O(\varepsilon^{1/2}) \bigr) -
\cos  \bigl( \alpha_j \varepsilon^{-1/2}\bigr) \right|
\\
=
\left| \cos \bigl( \alpha_j \varepsilon^{-1/2} + O(\varepsilon^{1/2})\bigr) +
\cos\bigl(\alpha_j \varepsilon^{-1/2}\bigr) \right|,
\end{multline}
where
$\alpha_j := (2 \pi)^{1/2} \gamma_j^{3/4} |\tau|^{1/2} |\mu_j|^{-1/2}$.

Next, consider the sequence $\{\varepsilon_k\}_{k=1}^{\infty}, \; \varepsilon_k \to 0$, given by
\begin{equation*}
\varepsilon_k := \frac{1}{2\pi k^2} \frac{\gamma_j^{3/2} |\tau|}{|\mu_j| }.
\end{equation*}
Obviously, for $\varepsilon =\varepsilon_k$ with large $k$ expression~\eqref{abstr_s<2_|cos-cos|} is separated from zero.
Combining this with~\eqref{abstr_s<2_est_c}, we see
that the sequence $\varepsilon_k^{s/2-1}(c^2 +  \varepsilon_k)^{-s/2}$ is uniformly bounded for large $k$.
But this is not true provided that $s < 2$.
This contradiction completes the proof.
\end{proof}

\section{Approximation of the sandwiched operator cosine}
\label{abstr_sandwiched_section}

\subsection{The operator family $A(t) = M^* \widehat{A} (t) M$}
\label{abstr_A_and_Ahat_section}

Let $\widehat{\mathfrak{H}}$ be yet another separable Hilbert space.
Let $\widehat{X} (t) = \widehat{X}_0 + t \widehat{X}_1 \colon \widehat{\mathfrak{H}} \to \mathfrak{H}_* $~be a family of operators of the same form as  $X(t)$, and suppose that~$\widehat{X} (t)$ satisfies the assumptions of Subsection~\ref{abstr_X_A_section}.
Let $M \colon \mathfrak{H} \to \widehat{\mathfrak{H}}$~be an isomorphism. Suppose that
\begin{equation*}
M \Dom X_0 = \Dom \widehat{X}_0, \; X(t) = \widehat{X} (t) M,
\end{equation*}
and then also $X_0 = \widehat{X}_0 M$, $X_1 = \widehat{X}_1 M$.
In $\widehat{\mathfrak{H}}$ we consider the family of selfadjoint operators $\widehat{A} (t) = \widehat{X} (t)^* \widehat{X} (t)$.
Then, obviously,
\begin{equation}
\label{abstr_A_and_Ahat}
A(t) = M^* \widehat{A} (t) M.
\end{equation}
In what follows, all the objects corresponding to the family $\widehat{A}(t)$ are marked by~\textquotedblleft$\, \widehat{\phantom{\_}} \,$\textquotedblright.  Note that $\widehat{\mathfrak{N}} = M \mathfrak{N}$, $\widehat{n} = n$, $\widehat{\mathfrak{N}}_* =  \mathfrak{N}_*$, $\widehat{n}_* = n_*$, and $\widehat{P}_* = P_*$.

In $\widehat{\mathfrak{H}}$ we consider the positive definite operator
\begin{equation*}
Q := (M M^*)^{-1} \colon \widehat{\mathfrak{H}} \to \widehat{\mathfrak{H}}.
\end{equation*}
Let $Q_{\widehat{\mathfrak{N}}}$~be the block of  $Q$ in the subspace~$\widehat{\mathfrak{N}}$, i.~e.,
\begin{equation*}
Q_{\widehat{\mathfrak{N}}} = \widehat{P} Q|_{\widehat{\mathfrak{N}}} \colon \widehat{\mathfrak{N}} \to \widehat{\mathfrak{N}}.
\end{equation*}
Obviously, $Q_{\widehat{\mathfrak{N}}}$~is an isomorphism in~$\widehat{\mathfrak{N}}$.

According to \cite[Proposition~1.2]{Su2}, the orthogonal projection $P$ of~$\mathfrak{H}$ onto $\mathfrak{N}$ and the orthogonal
projection~$\widehat{P}$ of $\widehat{\mathfrak{H}}$ onto $\widehat{\mathfrak{N}}$ satisfy the following relation
\begin{equation}
\label{abstr_P_and_P_hat_relation}
P = M^{-1} (Q_{\widehat{\mathfrak{N}}})^{-1} \widehat{P} (M^*)^{-1}.
\end{equation}
Let $\widehat{S} \colon \widehat{\mathfrak{N}} \to \widehat{\mathfrak{N}}$~be the spectral germ of~$\widehat{A} (t)$ at $t = 0$,
and let~$S$~be the germ of $A (t)$. According to \cite[Chapter~1, Subsection~1.5]{BSu1}, we have
\begin{equation}
\label{abstr_S_and_S_hat_relation}
S = P M^* \widehat{S} M |_\mathfrak{N}.
\end{equation}

\subsection{The operators $\widehat{Z}_Q$ and $\widehat{N}_Q$}
\label{abstr_hatZ_Q_and_hatN_Q_section}

For the operator family $\widehat{A} (t)$ we introduce the operator~$\widehat{Z}_Q$ acting in~$\widehat{\mathfrak{H}}$
and taking an element~$\widehat{u} \in \widehat{\mathfrak{H}}$ to the solution $\widehat{\phi}_Q$ of the problem
\begin{equation*}
\widehat{X}^*_0 (\widehat{X}_0 \widehat{\phi}_Q + \widehat{X}_1 \widehat{\omega}) = 0, \quad Q \widehat{\phi}_Q \perp \widehat{\mathfrak{N}},
\end{equation*}
where $\widehat{\omega} = \widehat{P} \widehat{u}$. As shown in~\cite[Section~6]{BSu2},
the operator $Z$ for $A(t)$ and the operator~$\widehat{Z}_Q$ introduced above
satisfy~$\widehat{Z}_Q =M Z M^{-1} \widehat{P}$.  Next, we put
\begin{equation}
\label{abstr_hatN_Q}
\widehat{N}_Q := \widehat{Z}_Q^* \widehat{X}_1^* \widehat{R} \widehat{P} + (\widehat{R} \widehat{P})^* \widehat{X}_1  \widehat{Z}_Q.
\end{equation}
According to~\cite[Section~6]{BSu2}, the operator~$N$ for $A(t)$ and the operator~\eqref{abstr_hatN_Q} satisfy
\begin{equation}
\label{abstr_N_and_hatN_Q_relat}
\widehat{N}_Q = \widehat{P} (M^*)^{-1} N M^{-1} \widehat{P}.
\end{equation}
Recall that $N=N_0 + N_*$ and introduce the operators
\begin{equation}
\label{abstr_N0*_and_hatN0*_Q_relat}
\widehat{N}_{0,Q} = \widehat{P} (M^*)^{-1} N_0 M^{-1} \widehat{P}, \qquad \widehat{N}_{*,Q} = \widehat{P} (M^*)^{-1} N_* M^{-1} \widehat{P}.
\end{equation}
Then $\widehat{N}_Q = \widehat{N}_{0,Q} + \widehat{N}_{*,Q}$.

We need the following lemma proved in~\cite[Lemma~5.1]{Su4}.

\begin{lemma}[\cite{Su4}]
    \label{abstr_N_and_Nhat_lemma}
    Suppose that the assumptions of Subsection~\emph{\ref{abstr_A_and_Ahat_section}} are satisfied.
    Suppose that the operators $N$ and $N_0$ are defined by~\emph{\eqref{abstr_F1_K0_N_invar}} and~\emph{\eqref{abstr_N_0_N_*}},
    and the operators  $\widehat{N}_Q$ and $\widehat{N}_{0,Q}$ are defined in Subsection~\emph{\ref{abstr_hatZ_Q_and_hatN_Q_section}}.
    Then relation~$N = 0$ is equivalent to the relation~$\widehat{N}_Q = 0$.
    The relation~$N_0 = 0$ is equivalent to the relation~$\widehat{N}_{0,Q} = 0$.
\end{lemma}

\subsection{Relations between the operators and the coefficients of the power series expansions}

Now we describe relations between the coefficients of the power series expansions~\eqref{abstr_A(t)_eigenvalues_series}, \eqref{abstr_A(t)_eigenvectors_series} and the operators~$\widehat{S}$ and~$Q_{\widehat{\mathfrak{N}}}$.
(See~\cite[Subsections~1.6,~1.7]{BSu3}.) We denote $\zeta_l := M \omega_l \in \widehat{\mathfrak{N}}, \, l = 1, \ldots, n$.
Then relations~\eqref{abstr_S_eigenvectors} and~\eqref{abstr_P_and_P_hat_relation}, \eqref{abstr_S_and_S_hat_relation} show that
\begin{equation}
\label{abstr_hatS_gener_spec_problem}
\widehat{S} \zeta_l  = \gamma_l Q_{\widehat{\mathfrak{N}}} \zeta_l, \quad l = 1, \ldots, n.
\end{equation}
The set $\zeta_1, \ldots, \zeta_n$ forms a basis in~$\widehat{\mathfrak{N}}$ orthonormal with the weight~$Q_{\widehat{\mathfrak{N}}}$:
\begin{equation}
\label{abstr_sndwchd_zeta_basis}
(Q_{\widehat{\mathfrak{N}}} \zeta_l, \zeta_j) = \delta_{lj}, \qquad l,j = 1,\ldots,n.
\end{equation}

The operators $\widehat{N}_{0,Q}$ and $\widehat{N}_{*,Q}$ can be described in terms of the coefficients
of the expansions~\eqref{abstr_A(t)_eigenvalues_series} and~\eqref{abstr_A(t)_eigenvectors_series}; cf.~\eqref{abstr_N_0_N_*}.
We put~$\widetilde{\zeta}_l := M \widetilde{\omega}_l \in \widehat{\mathfrak{N}}, \; l = 1, \ldots,n $. Then
\begin{equation}
\label{abstr_hatN_0Q_N_*Q}
\begin{split}
&\widehat{N}_{0,Q} = \sum_{k=1}^{n} \mu_k (\cdot, Q_{\widehat{\mathfrak{N}}} \zeta_k) Q_{\widehat{\mathfrak{N}}} \zeta_k, \\
\widehat{N}_{*,Q} = \sum_{k=1}^{n} &\gamma_k \left( (\cdot, Q_{\widehat{\mathfrak{N}}} \widetilde{\zeta}_k) Q_{\widehat{\mathfrak{N}}} \zeta_k + (\cdot, Q_{\widehat{\mathfrak{N}}} \zeta_k) Q_{\widehat{\mathfrak{N}}} \widetilde{\zeta}_k \right).
\end{split}
\end{equation}

\begin{remark}
    By~\emph{\eqref{abstr_sndwchd_zeta_basis}} and~\emph{\eqref{abstr_hatN_0Q_N_*Q}}, we have
    \begin{gather*}
    (\widehat{N}_{0,Q} \zeta_j, \zeta_l) = \mu_l \delta_{jl}, \qquad j,l=1,\ldots,n,\\
    (\widehat{N}_{*,Q} \zeta_j, \zeta_l) = \gamma_l (\zeta_j, Q_{\widehat{\mathfrak{N}}} \widetilde{\zeta}_l) + \gamma_j (Q_{\widehat{\mathfrak{N}}} \widetilde{\zeta}_j, \zeta_l), \qquad j,l=1,\ldots,n.
    \end{gather*}
\end{remark}

Relations~\eqref{1.12} imply that
\begin{equation*}
(Q_{\widehat{\mathfrak{N}}} \widetilde{\zeta}_j, \zeta_l) + (\zeta_j, Q_{\widehat{\mathfrak{N}}} \widetilde{\zeta}_l) = 0, \qquad j, l = 1, \ldots,n.
\end{equation*}
Hence,
\begin{equation*}
(\widehat{N}_{*,Q} \zeta_j, \zeta_l) = 0, \quad \text{if} \  \gamma_j = \gamma_l.
\end{equation*}

Now we return to the notation of Section~\ref{abstr_cluster_section}.
Recall that the different eigenvalues of the germ $S$ are denoted by
$\gamma^{\circ}_j,\, j = 1,\ldots,p$, and the corresponding eigenspaces by~$\mathfrak{N}_j$.
The vectors~$\omega^{(j)}_i, \, i = 1,\ldots, k_j,$ form an orthonormal basis in~$\mathfrak{N}_j$.
Then the same numbers~$\gamma^{\circ}_j, \, j = 1,\ldots,p$,~are different eigenvalues of the problem~\eqref{abstr_hatS_gener_spec_problem}, and
$M \mathfrak{N}_j$~are the corresponding eigenspaces. The vectors~$\zeta^{(j)}_i = M\omega^{(j)}_i, \, i = 1,\ldots, k_j,$
 form a basis in~$M \mathfrak{N}_j$ (orthonormal with the weight~$Q_{\widehat{\mathfrak{N}}}$).
 By~$\mathcal{P}_j$ we denote the ``skew'' projection onto~$M \mathfrak{N}_j$ that  is orthogonal with respect to the inner product~$(Q_{\widehat{\mathfrak{N}}} \cdot, \cdot)$, i.~e.,
\begin{equation*}
\mathcal{P}_j = \sum_{i=1}^{k_j} (\cdot, Q_{\widehat{\mathfrak{N}}} \zeta^{(j)}_i) \zeta^{(j)}_i, \quad j = 1, \ldots, p.
\end{equation*}
It is easily seen that~$\mathcal{P}_j =M P_j M^{-1} \widehat{P}$.
Using~\eqref{abstr_N_invar_repers}, \eqref{abstr_N_and_hatN_Q_relat}, and~\eqref{abstr_N0*_and_hatN0*_Q_relat},
it is easy to check that
\begin{equation}
\label{abstr_hatN_0Q_N_*Q_invar_repr}
\widehat{N}_{0,Q} = \sum_{j=1}^{p} \mathcal{P}_j^* \widehat{N}_Q \mathcal{P}_j, \quad \widehat{N}_{*,Q} = \sum_{\substack{1 \le l,j \le p: \\ l \ne j}} \mathcal{P}_l^* \widehat{N}_Q \mathcal{P}_j.
\end{equation}
These relations are similar to \eqref{abstr_N_invar_repers};
they  give the invariant representations for the operators~$\widehat{N}_{0,Q}$ and~$\widehat{N}_{*,Q}$.

\subsection{Approximation of the sandwiched operator cosine}

In this subsection, we find an approximation for the operator $\cos(\tau A(t)^{1/2})$, where 
$A(t)$ is given by~\eqref{abstr_A_and_Ahat},
in terms of the germ~$\widehat{S}$ of $\widehat{A}(t)$ and the isomorphism~$M$.
It is convenient to border the operator cosine by appropriate factors.

We put $M_0 := (Q_{\widehat{\mathfrak{N}}})^{-1/2}$.  According to~\cite[Proposition~3.3]{BSu5},
we have
\begin{equation}
\label{abstr_sandwiched_cos_S_relation}
M \cos(\tau (t^2 S)^{1/2} P) P M^* = M_0 \cos (\tau (t^2 M_0 \widehat{S} M_0)^{1/2}) M_0 \widehat{P}.
\end{equation}
Denote
\begin{equation}
\label{Jdef}
{\mathcal J}(t,\tau) := M \cos(\tau A(t)^{1/2})M^{-1} \widehat{P} - M_0 \cos (\tau (t^2 M_0 \widehat{S} M_0)^{1/2}) M_0^{-1} \widehat{P}.
\end{equation}

\begin{lemma}
    \label{abstr_cos_sandwiched_est_lemma}
   Under the assumptions of Subsection~\emph{\ref{abstr_A_and_Ahat_section}}, we have
   \begin{gather}
   \label{abstr_cos_sandwiched_est_1}
   \begin{split}
   	\| {\mathcal J}(t,\tau)\|
       \le \| M \|^2 \| M^{-1} \|^2 \| \cos (\tau A(t)^{1/2}) P -  \cos ( \tau (t^2 S)^{1/2} P) P \|,
   \end{split}
   \\
   \label{abstr_cos_sandwiched_est_2}
   \begin{split}
   \| \cos (\tau  A(t)^{1/2}) P -  \cos ( \tau (t^2 S)^{1/2} P) P \|
   \le \| M \|^2 \| M^{-1} \|^2 \| {\mathcal J}(t,\tau)\|.
   \end{split}
   \end{gather}
\end{lemma}

\begin{proof}
Since $M_0 = (Q_{\widehat{\mathfrak{N}}})^{-1/2}$ and
$M^{-1} Q_{\widehat{\mathfrak{N}}}^{-1} \widehat{P} = P M^*$ (see~\eqref{abstr_P_and_P_hat_relation}),
using~\eqref{abstr_sandwiched_cos_S_relation}, we obtain
\begin{equation*}
\| {\mathcal J}(t,\tau)\| = \left\| M \left( \cos (\tau A(t)^{1/2}) P -
\cos ( \tau (t^2 S)^{1/2} P) P \right) M^* Q_{\widehat{\mathfrak{N}}} \widehat{P}  \right\|.
\end{equation*}
Hence,
\begin{equation*}
\| {\mathcal J}(t,\tau)\| \le \| M \|^2 \| Q_{\widehat{\mathfrak{N}}} \widehat{P}\| \| \cos (\tau A(t)^{1/2}) P -  \cos ( \tau (t^2 S)^{1/2} P) P \|.
\end{equation*}
Since $\|Q_{\widehat{\mathfrak{N}}}\| \le \|Q\| = \|M^{-1}\|^2$, we arrive at~(\ref{abstr_cos_sandwiched_est_1}).

Estimate~\eqref{abstr_cos_sandwiched_est_2} can be checked similarly in the ``inverse way''.
Obviously,
\begin{multline*}
\| \cos (\tau A(t)^{1/2}) P -  \cos ( \tau (t^2 S)^{1/2} P) P\| \le \\ \le \| M^{-1} \|^2 \|M\cos (\tau A(t)^{1/2}) PM^* -  M\cos ( \tau (t^2 S)^{1/2} P)P M^* \|.
\end{multline*}
By the identity $P M^* = M^{-1} Q_{\widehat{\mathfrak{N}}}^{-1} \widehat{P}$ and~\eqref{abstr_sandwiched_cos_S_relation},
the right-hand side can be written as
$\| M^{-1} \|^2 \| {\mathcal J}(t,\tau) Q_{\widehat{\mathfrak{N}}}^{-1} \widehat{P} \|$.
 Together with the inequality \hbox{$\| Q_{\widehat{\mathfrak{N}}}^{-1} \widehat{P}\| \le \| M \|^2$}
(which follows from the identity $Q_{\widehat{\mathfrak{N}}}^{-1} \widehat{P} = M P M^*$),
this implies~\eqref{abstr_cos_sandwiched_est_2}.
\end{proof}

Now, Theorem~\ref{abstr_cos_general_thrm_wo_eps} and inequality~\eqref{abstr_cos_sandwiched_est_1}
directly imply the following result (which has been obtained before in~\cite[(3.19)]{BSu5}).

\begin{theorem}\label{th5.4}
Suppose that the assumptions of Subsections~\emph{\ref{abstr_A_and_Ahat_section}} and~\emph{\ref{abstr_nondegenerated_section}} are satisfied.
Let ${\mathcal J}(t,\tau)$ be defined by~\emph{\eqref{Jdef}}. Then for
    $\tau \in \mathbb{R}$ and $|t| \le t^0$ we have
    \begin{equation}
    \label{abstr_cos_sandwiched_general_est_wo_eps}
    \| {\mathcal J}(t,\tau)\|
    \le \| M \|^2 \| M^{-1} \|^2 (2 C_1 |t| + C_7 |\tau| t^2).
    \end{equation}
    The number $t^0$ is subject to~\emph{(\ref{abstr_t0_fixation})}, and the constants $C_1$, $C_7$ are given by~\emph{(\ref{abstr_C1_C2})} and~\emph{(\ref{abstr_C7})}.
\end{theorem}

Similarly,  combining Corollary~\ref{abstr_coroll_enchd_est_wo_eps_N=0}, Lemma~\ref{abstr_N_and_Nhat_lemma}, and Lemma~\ref{abstr_cos_sandwiched_est_lemma}, we arrive at the following result.

\begin{theorem}
    \label{abstr_cos_sandwiched_ench_thrm_wo_eps_1}
    Suppose that the assumptions of Theorem~\emph{\ref{th5.4}} are satisfied.
     Suppose that the operator $\widehat{N}_Q$ defined in Subsection~\emph{\ref{abstr_hatZ_Q_and_hatN_Q_section}}
      is equal to zero: $\widehat{N}_Q = 0$. Then for $\tau \in \mathbb{R}$ and $|t| \le t^0$ we have
      \begin{equation*}
    \| {\mathcal J}(t,\tau) \|
    \le \| M \|^2 \| M^{-1} \|^2 (2 C_1 |t| + C_{11} | \tau | |t|^3).
    \end{equation*}
    The number $t^0$ is subject to~\emph{\eqref{abstr_t0_fixation}}, and the constants
 $C_1$, $C_{11}$ are defined by~\emph{\eqref{abstr_C1_C2}} and~\emph{\eqref{abstr_C11}}.
\end{theorem}

Finally,  from Corollary~\ref{abstr_coroll_enchd_est_wo_eps_N*=0}, Lemma~\ref{abstr_N_and_Nhat_lemma}, and Lemma~\ref{abstr_cos_sandwiched_est_lemma} we deduce the following statement.

\begin{theorem}
    \label{abstr_cos_sandwiched_ench_thrm_wo_eps_2}
    Suppose that the assumptions of Theorem~\emph{\ref{th5.4}}
     are satisfied. Suppose that the operator $\widehat{N}_{0,Q}$ defined in Subsection~\emph{\ref{abstr_hatZ_Q_and_hatN_Q_section}}
      is equal to zero: $\widehat{N}_{0,Q}=0$. Then for $\tau \in \mathbb{R}$ and $|t| \le t^{00}$ we have
    \begin{equation*}
    \| {\mathcal J}(t,\tau) \|
    \le \| M \|^2 \| M^{-1} \|^2 (C_{15} |t| + C_{16} | \tau | |t|^3).
    \end{equation*}
   The number $t^{00} \le t^0$ is subject to~\emph{\eqref{abstr_t00}}, and the constants $C_{15}$, $C_{16}$ are as in Theorem~\emph{\ref{abstr_cos_enchanced_thrm_2_wo_eps}}.
\end{theorem}

\subsection{Approximation of the sandwiched operator cosine}

Writing down~\eqref{abstr_cos_sandwiched_general_est_wo_eps} with $\tau$ replaced by $\varepsilon^{-1} \tau$
and multiplying it by the ``smoothing factor'', we arrive at the following result,
which has been proved before in~\cite[Theorem~3.4]{BSu5}.

\begin{theorem}
    \label{abstr_cos_sandwiched_general_thrm}
    Under the assumptions of Theorem~\emph{\ref{th5.4}}, for $\tau \in \mathbb{R}$, $\varepsilon > 0$, and $|t| \le t^0$ we have
    \begin{equation}
    \label{abstr_cos_sandwiched_general_est}
    \| {\mathcal J}(t,\varepsilon^{-1} \tau) \| \varepsilon^{2} (t^2 + \varepsilon^2)^{-1}
    \le \| M \|^2 \| M^{-1} \|^2 (C_1  + C_7 |\tau| ) \varepsilon.
    \end{equation}
    The number $t^0$ is subject to~\emph{(\ref{abstr_t0_fixation})}, and the constants $C_1$, $C_7$ are defined by~\emph{\eqref{abstr_C1_C2}} and~\emph{\eqref{abstr_C7}}.
\end{theorem}

Similarly to the proof of Theorem~\ref{abstr_cos_enchanced_thrm_1}, from Theorem~\ref{abstr_cos_sandwiched_ench_thrm_wo_eps_1}
we deduce the following statement.

\begin{theorem}
    \label{abstr_cos_sandwiched_ench_thrm_1}
    Suppose that the assumptions of Theorem~\emph{\ref{abstr_cos_sandwiched_ench_thrm_wo_eps_1}}
    are satisfied. Then for $\tau \in \mathbb{R}$, $\varepsilon > 0$, and $|t| \le t^0$ we have
    \begin{equation*}
    \| {\mathcal J}(t,\varepsilon^{-1} \tau) \| \varepsilon^{3/2} (t^2 + \varepsilon^2)^{-3/4}
    \le \| M \|^2 \| M^{-1} \|^2 ( C'_1  + C_{11} | \tau | ) \varepsilon.
    \end{equation*}
    The number $t^0$ is subject to~\emph{(\ref{abstr_t0_fixation})}, and the constants $C'_1$ and $C_{11}$ are as in
    Theorem~\emph{\ref{abstr_cos_enchanced_thrm_1}}.
\end{theorem}

Finally, Theorem~\ref{abstr_cos_sandwiched_ench_thrm_wo_eps_2} implies the following result.

\begin{theorem}
    \label{abstr_cos_sandwiched_ench_thrm_2}
    Suppose that the assumptions of Theorem~\emph{\ref{abstr_cos_sandwiched_ench_thrm_wo_eps_2}}
    are satisfied. Then for $\tau \in \mathbb{R}$ and $|t| \le t^{00}$ we have
    \begin{equation*}
    \| {\mathcal J}(t,\varepsilon^{-1} \tau) \| \varepsilon^{3/2} (t^2 + \varepsilon^2)^{-3/4}
    \le \| M \|^2 \| M^{-1} \|^2 (C'_{15} + C_{16} | \tau | ) \varepsilon.
    \end{equation*}
    The number $t^{00} \le t^0$ is subject to~\emph{\eqref{abstr_t00}}, and the constants $C'_{15}$, $C_{16}$ are as in
    Theorem~\emph{\ref{abstr_cos_enchanced_thrm_2}}.
\end{theorem}

\subsection{The sharpness of the result}

Now we confirm that the result of Theorem~\ref{abstr_cos_sandwiched_general_thrm} is sharp in the general case.
Namely, if $N_{0,Q} \ne 0$,  then the exponent $s$ in the smoothing factor can not be taken smaller than 2.

\begin{theorem}
    \label{abstr_sndwchd_s<2_general_thrm}
    Suppose that the assumptions of Subsections~\emph{\ref{abstr_A_and_Ahat_section}} and~\emph{\ref{abstr_nondegenerated_section}}
    are satisfied. Let $\widehat{N}_{0,Q} \ne 0$. Let $\tau \ne 0$ and $0 \le s < 2$.
    Then there does not exist a constant $C(\tau)>0$ such that the estimate
    \begin{equation}
    \label{abstr_sndwchd_s<2_est_imp}
    \| {\mathcal J}(t,\varepsilon^{-1} \tau) \| \varepsilon^{s} (t^2 + \varepsilon^2)^{-s/2} \le C(\tau) \varepsilon
    \end{equation}
    holds for all sufficiently small $|t|$ and $\varepsilon >0$.
\end{theorem}

\begin{proof} By Lemma~\ref{abstr_N_and_Nhat_lemma}, under our assumptions we have $N_0 \ne 0$.
We prove by contradiction. Fix $\tau \ne 0$.
Suppose that for some $0 \le s < 2$ there exists a constant $C (\tau) > 0$ such that~\eqref{abstr_sndwchd_s<2_est_imp}
 holds for all sufficiently small $|t|$ and $\varepsilon$.
 By~\eqref{abstr_cos_sandwiched_est_2}, this means that the inequality of the form~\eqref{abstr_s<2_est_imp} also holds
 (with some other constant).
But this contradicts the statement of Theorem~\ref{abstr_s<2_general_thrm}.
\end{proof}

\section*{Chapter 2. Periodic differential operators in $L_2(\mathbb{R}^d; \mathbb{C}^n)$}

\section{Preliminaries}

\subsection{Lattices $\Gamma$ and $\widetilde \Gamma$}

Let $\Gamma$ be a lattice in $\mathbb{R}^d$ generated by
the basis $\mathbf{a}_1, \ldots , \mathbf{a}_d$:
\begin{equation*}
\Gamma = \left\{ \mathbf{a} \in \mathbb{R}^d \colon \mathbf{a} = \sum_{j=1}^{d} n_j \mathbf{a}_j, \; n_j \in \mathbb{Z} \right\},
\end{equation*}
and let $\Omega$ be the (elementary) cell of this lattice:
\begin{equation*}
\Omega := \left\{ \mathbf{x} \in \mathbb{R}^d \colon \mathbf{x} = \sum_{j=1}^{d} \xi_j \mathbf{a}_j, \; 0 < \xi_j < 1 \right\}.
\end{equation*}
The basis $\mathbf{b}_1, \ldots , \mathbf{b}_d$ dual to $\mathbf{a}_1, \ldots , \mathbf{a}_d$ is defined by the relations
   $\langle \mathbf{b}_l, \mathbf{a}_j \rangle = 2 \pi \delta_{jl}$.
This basis generates the \textit{lattice  $\widetilde \Gamma$ dual to} $\Gamma$:
\begin{equation*}
\widetilde \Gamma = \left\{ \mathbf{b} \in \mathbb{R}^d \colon \mathbf{b} = 
\sum_{j=1}^{d} m_j \mathbf{b}_j, \; m_j \in \mathbb{Z} \right\} .
\end{equation*}
Let $\widetilde \Omega$   be the central \emph{Brillouin zone} of the lattice $\widetilde \Gamma $:
\begin{equation}
\label{Brillouin_zone}
\widetilde \Omega = \left\{ \mathbf{k} \in \mathbb{R}^d \colon | \mathbf{k} | < | \mathbf{k} - \mathbf{b} |, \; 0 \ne \mathbf{b} \in \widetilde \Gamma \right\}.
\end{equation}
Denote $| \Omega | = \mes \Omega$, $| \widetilde \Omega | = \mes \widetilde \Omega$.
Note that $| \Omega |  | \widetilde \Omega | = (2 \pi)^d$. Let $r_0$ be the maximal radius of the ball containing in
$\clos \widetilde \Omega$. We have
\begin{equation}
\label{r_0}
2 r_0 = \min|\mathbf{b}|,  \quad 0 \ne \mathbf{b} \in \widetilde \Gamma.
\end{equation}

With the lattice $\Gamma$, we associate the discrete Fourier transformation  $ \{ \hat{\mathbf{v}}_{\mathbf{b}}\} \mapsto \mathbf{v}$:
\begin{equation}
\label{fourier}
\mathbf{v}(\mathbf{x}) = | \Omega |^{-1/2} \sum_{\mathbf{b} \in \widetilde \Gamma} \hat{\mathbf{v}}_{\mathbf{b}} \exp (i \left<\mathbf{b}, \mathbf{x} \right>), \quad \mathbf{x} \in \Omega,
\end{equation}
which is a unitary mapping of $l_2 (\widetilde \Gamma; \mathbb{C}^n) $ onto $L_2 (\Omega; \mathbb{C}^n)$:
\begin{equation}
\label{fourier_unitary}
\intop_{\Omega} |\mathbf{v}(\mathbf{x})|^2 d \mathbf{x} = \sum_{\mathbf{b} \in \widetilde \Gamma} | \hat{\mathbf{v}}_{\mathbf{b}} |^2.
\end{equation}

\textit{By $\widetilde H^1(\Omega; \mathbb{C}^n)$ we denote the subspace in $H^1(\Omega; \mathbb{C}^n)$
consisting of the functions whose $\Gamma$-periodic extension to
$\mathbb{R}^d$ belongs to $H^1_{\mathrm{loc}}(\mathbb{R}^d; \mathbb{C}^n)$}. We have
\begin{equation}
\label{D_and_fourier}
\int_{\Omega} |(\mathbf{D} + \mathbf{k}) \mathbf{u}|^2\, d\mathbf{x} = \sum_{\mathbf{b} \in \widetilde{\Gamma}} |\mathbf{b} + \mathbf{k} |^2 |\hat{\mathbf{u}}_{\mathbf{b}}|^2, \quad \mathbf{u} \in \widetilde{H}^1(\Omega; \mathbb{C}^n), \; \mathbf{k} \in \mathbb{R}^d,
\end{equation}
and convergence of the series in the right-hand side of~\eqref{D_and_fourier} is equivalent to the relation
$\mathbf{u} \in \widetilde{H}^1(\Omega; \mathbb{C}^n)$. From~\eqref{Brillouin_zone},~\eqref{fourier_unitary}, and~\eqref{D_and_fourier} it follows that
\begin{equation}
\label{AO}
\int_{\Omega} |(\mathbf{D} + \mathbf{k}) \mathbf{u}|^2 d\mathbf{x} \ge \sum_{\mathbf{b} \in \widetilde{\Gamma}} | \mathbf{k} |^2 |\hat{\mathbf{u}}_{\mathbf{b}}|^2 = | \mathbf{k} |^2 \int_{\Omega} |\mathbf{u}|^2 d\mathbf{x}, \quad \mathbf{u} \in \widetilde{H}^1(\Omega; \mathbb{C}^n), \; \mathbf{k} \in \widetilde{\Omega}.
\end{equation}

\subsection{The Gelfand transformation}

Initially, the Gelfand transformation $\mathcal{U}$ is defined on the functions of the Schwartz class by the formula
\begin{multline*}
\widetilde{\mathbf{v}} ( \mathbf{k}, \mathbf{x}) = (\mathcal{U} \- \mathbf{v}) (\mathbf{k}, \mathbf{x}) = | \widetilde \Omega |^{-1/2} \sum_{\mathbf{a} \in \Gamma} \exp (- i \left< \mathbf{k}, \mathbf{x} + \mathbf{a} \right> ) \mathbf{v} ( \mathbf{x} + \mathbf{a}),
\\
\mathbf{v} \in \mathcal{S} (\mathbb{R}^d; \mathbb{C}^n), \quad \mathbf{x} \in \Omega, \quad \mathbf{k} \in \widetilde \Omega.
\end{multline*}
Since
\begin{equation*}
\int_{\widetilde \Omega} \int_{\Omega} | \widetilde{\mathbf{v}} ( \mathbf{k}, \mathbf{x}) |^2 d  \mathbf{x} \, d  \mathbf{k} = \int_{\mathbb{R}^d} | \mathbf{v} ( \mathbf{x} ) |^2 d \mathbf{x},
\end{equation*}
the transformation $\mathcal{U}$ extends by continuity up to a \textit{unitary mapping}
\begin{equation*}
\mathcal{U} \colon L_2 (\mathbb{R}^d; \mathbb{C}^n) \to \int_{\widetilde \Omega} \oplus  L_2 (\Omega; \mathbb{C}^n) d \mathbf{k} =: \mathcal{K}.
\end{equation*}
The relation $\mathbf{v} \in H^1 (\mathbb{R}^d; \mathbb{C}^n)$ is equivalent to the fact that
$\widetilde{\mathbf{v}} (\mathbf{k}, \cdot) \in \widetilde H^1 (\Omega; \mathbb{C}^n)$ for almost~every $\mathbf{k} \in \widetilde \Omega $ and
\begin{equation*}
\int_{\widetilde \Omega} \int_{\Omega} \left( | (\mathbf{D} + \mathbf{k}) \widetilde{\mathbf{v}} (\mathbf{k}, \mathbf{x}) |^2 + | \widetilde{\mathbf{v}} (\mathbf{k}, \mathbf{x}) |^2 \right) \, d \mathbf{x} \,  d \mathbf{k} < \infty.
\end{equation*}
Under the Gelfand transformation $\mathcal{U}$, the operator of multiplication by a bounded periodic function in
$L_2 (\mathbb{R}^d; \mathbb{C}^n)$ turns into multiplication by the same function on the fibers of the direct integral
$\mathcal{K}$. The operator $b(\mathbf{D})$ applied to
$\mathbf{v} \in H^1 (\mathbb{R}^d; \mathbb{C}^n)$ turns into  the operator $b(\mathbf{D} + \mathbf{k})$
applied to $\widetilde{\mathbf{v}} (\mathbf{k}, \cdot) \in \widetilde H^1 (\Omega; \mathbb{C}^n)$.

\section{The class of differential operators in $L_2(\mathbb{R}^d; \mathbb{C}^n)$}

\subsection{Factorized second order operators $\mathcal{A}$}
\label{A_oper_subsect}

Let $b(\mathbf{D})$ be a matrix first order differential operator of the form $b(\mathbf{D}) = \sum_{l=1}^d b_l D_l$; here $b_l$ are constant  $(m \times n)$-matrices
(in general, with complex entries). \emph{Assume that $m \ge n$}.
Consider the symbol $b(\boldsymbol{\xi})= \sum^d_{l=1} b_l \xi_l$, $ \boldsymbol{\xi} \in \mathbb{R}^d$, and \emph{assume that}
\hbox{$\rank b( \boldsymbol{\xi} ) = n$} for $0 \ne  \boldsymbol{\xi} \in \mathbb{R}^d $.
This condition is equivalent to the inequalities
\begin{equation}
\label{rank_alpha_ineq}
\alpha_0 \mathbf{1}_n \le b( \boldsymbol{\theta} )^* b( \boldsymbol{\theta} ) \le \alpha_1 \mathbf{1}_n, \quad  \boldsymbol{\theta} \in \mathbb{S}^{d-1}, \quad 0 < \alpha_0 \le \alpha_1 < \infty,
\end{equation}
with some positive constants $\alpha_0, \alpha_1 > 0$.

Suppose that an $(m\times m)$-matrix-valued function $h(\mathbf{x})$ and an $(n\times n)$-matrix-valued function $f(\mathbf{x})$
(in general, with complex entries) are $\Gamma$-periodic and such that
\begin{equation}
\label{h_f_L_inf}
f, f^{-1} \in L_{\infty} (\mathbb{R}^d); \quad h, h^{-1} \in L_{\infty} (\mathbb{R}^d).
\end{equation}
Consider the DO
\begin{gather*}
\mathcal{X} = h b( \mathbf{D} ) f \colon  L_2 (\mathbb{R}^d ; \mathbb{C}^n) \to  L_2 (\mathbb{R}^d ; \mathbb{C}^m), \\\Dom \mathcal{X} = \left\lbrace \mathbf{u} \in L_2 (\mathbb{R}^d ; \mathbb{C}^n) \colon f \mathbf{u} \in H^1  (\mathbb{R}^d ; \mathbb{C}^n) \right\rbrace .
\end{gather*}
The operator $\mathcal{X}$ is closed.
In $L_2 (\mathbb{R}^d ; \mathbb{C}^n)$, consider the selfadjoint operator
$\mathcal{A} = \mathcal{X}^* \mathcal{X}$ generated by the closed quadratic form
$\mathfrak{a}[\mathbf{u}, \mathbf{u}] = \| \mathcal{X} \mathbf{u} \|^2_{L_2(\mathbb{R}^d)}, \; \mathbf{u} \in \Dom \mathcal{X}$.
 Formally, we have
\begin{equation}
\label{A}
\mathcal{A} = f (\mathbf{x})^* b( \mathbf{D} )^* g( \mathbf{x} )  b( \mathbf{D} ) f(\mathbf{x}),
\end{equation}
where $g(\mathbf{x}) = h(\mathbf{x} )^* h( \mathbf{x})$.
Note that the Hermitian matrix-valued function $g(\mathbf{x})$ is bounded and uniformly positive definite.
Using the Fourier transformation and \eqref{rank_alpha_ineq},~\eqref{h_f_L_inf}, it is easy to check that
\begin{equation*}
\alpha_0 \| g^{-1} \|_{L_{\infty}}^{-1} \| \mathbf{D} (f \mathbf{u}) \|_{L_2({\mathbb R}^d)}^2 \le \mathfrak{a}[\mathbf{u}, \mathbf{u}] \le \alpha_1 \| g \|_{L_{\infty}} \| \mathbf{D} (f \mathbf{u}) \|_{L_2({\mathbb R}^d)}^2, \  \mathbf{u} \in \Dom \mathcal{X}.
\end{equation*}

\subsection{The operators $\mathcal{A}(\mathbf{k})$}

Putting
\begin{equation}
\label{Spaces_H}
\mathfrak{H} = L_2 (\Omega; \mathbb{C}^n), \quad \mathfrak{H}_* = L_2 (\Omega; \mathbb{C}^m),
\end{equation}
we consider the closed operator $\mathcal{X} (\mathbf{k}) \colon \mathfrak{H} \to \mathfrak{H}_*, \; \mathbf{k} \in \mathbb{R}^d$,
defined by
\begin{equation*}
\mathcal{X} (\mathbf{k}) = hb(\mathbf{D} + \mathbf{k})f, \quad \Dom \mathcal{X} (\mathbf{k}) = \left\lbrace \mathbf{u} \in \mathfrak{H} \colon   f \mathbf{u} \in \widetilde{H}^1 (\Omega; \mathbb{C}^n)\right\rbrace =: \mathfrak{d}.
\end{equation*}
The selfadjoint operator
\begin{equation*}
\mathcal{A} (\mathbf{k}) =\mathcal{X} (\mathbf{k})^* \mathcal{X} (\mathbf{k}) \colon \mathfrak{H} \to \mathfrak{H}
\end{equation*}
is generated by the closed quadratic form
\begin{equation*}
\mathfrak{a}(\mathbf{k})[\mathbf{u}, \mathbf{u}] = \| \mathcal{X}(\mathbf{k}) \mathbf{u} \|_{\mathfrak{H}_*}^2, \quad \mathbf{u} \in \mathfrak{d}.
\end{equation*}
Using the Fourier series expansion~\eqref{fourier} for $\mathbf{v}= f \mathbf{u}$ and conditions~\eqref{rank_alpha_ineq},~\eqref{h_f_L_inf},
it is easy to check that
\begin{multline}
\label{a(k)_form_est}
\alpha_0 \|g^{-1} \|_{L_\infty}^{-1} \|(\mathbf{D} + \mathbf{k}) f \mathbf{u} \|_{L_2 (\Omega)}^2 \le \mathfrak{a}(\mathbf{k})[\mathbf{u}, \mathbf{u}] \le \alpha_1 \|g \|_{L_\infty} \|(\mathbf{D} + \mathbf{k}) f \mathbf{u} \|_{L_2 (\Omega)}^2,
\\
\mathbf{u} \in \mathfrak{d}.
\end{multline}

By the lower estimate~\eqref{a(k)_form_est} and~\eqref{AO}, we have
\begin{align}
\label{A(k)_nondegenerated_and_c_*}	
\mathcal{A} (\mathbf{k}) &\ge c_* |\mathbf{k}|^2 I, \quad \mathbf{k} \in \widetilde{\Omega},
\\
\label{c*}
 c_* &= \alpha_0\|f^{-1} \|_{L_\infty}^{-2} \|g^{-1} \|_{L_\infty}^{-1} .
\end{align}

We put
\begin{equation}
\label{Ker1}
\mathfrak{N} := \Ker \mathcal{A} (0) = \Ker \mathcal{X} (0).
\end{equation}
Relations~\eqref{a(k)_form_est} with $\mathbf{k} = 0$ show that
\begin{equation}
\label{Ker2}
\mathfrak{N} = \left\lbrace \mathbf{u} \in L_2 (\Omega; \mathbb{C}^n) \colon f \mathbf{u} = \mathbf{c} \in \mathbb{C}^n \right\rbrace, \quad \dim \mathfrak{N} = n.
\end{equation}

By~\eqref{r_0} and~\eqref{D_and_fourier} with $\mathbf{k} = 0$,
\begin{equation}
\label{AP}
\| \mathbf{D} \mathbf{v} \|_{L_2 (\Omega)}^2 \ge 4 r_0^2 \| \mathbf{v} \|_{L_2 (\Omega)}^2, \quad \mathbf{v} \in \widetilde{H}^1 (\Omega; \mathbb{C}^n), \; \int_{\Omega} \mathbf{v} \, d \mathbf{x} = 0.
\end{equation}
From~\eqref{AP} and the lower estimate~\eqref{a(k)_form_est} with $\mathbf{k} = 0$ it follows that
the distance $d^0$ from the point $\lambda_0 = 0$ to the rest of the spectrum of $\mathcal{A}(0)$ satisfies
\begin{equation}
\label{d0_est}
d^0 \ge 4 c_* r_0^2.
\end{equation}

\subsection{The band functions}

The consecutive eigenvalues $E_j(\mathbf{k})$, $j \in \mathbb{N}$,
of the operator $\mathcal{A}(\mathbf{k})$ (counted with multiplicities) are called \textit{band functions}:
\begin{equation*}
E_1(\mathbf{k}) \le E_2(\mathbf{k}) \le \ldots \le E_j(\mathbf{k}) \le \ldots, \quad \mathbf{k} \in \mathbb{R}^d.
\end{equation*}
The band functions $E_j(\mathbf{k})$ are continuous and $\widetilde{\Gamma}$-periodic.
As shown in~\cite[Chapter~2, Subsection~2.2]{BSu1} (by simple variational arguments),
the band functions satisfy the following estimates:
\begin{align}
\notag
E_j(\mathbf{k}) &\ge c_* | \mathbf{k} |^2, \quad  \mathbf{k} \in \clos \widetilde{\Omega}, \quad j = 1, \ldots, n,
\\
\label{E_n+1}
E_{n+1}(\mathbf{k}) &\ge c_* r_0^2, \quad \mathbf{k} \in \clos \widetilde{\Omega},
\\
\notag
E_{n+1}(0) &\ge 4 c_* r_0^2.
\end{align}

\subsection{The direct integral expansion for the operator $\mathcal{A}$}
With the help of the Gelfand transformation, the operator $\mathcal A$ is represented as the direct integral of the operators
$\mathcal{A} (\mathbf{k})$:
\begin{equation}
\label{decompose}
\mathcal{U} \mathcal{A}  \mathcal{U}^{-1} = \int_{\widetilde \Omega} \oplus \mathcal{A} (\mathbf{k}) \, d \mathbf{k}.
\end{equation}
This means the following.
If $\mathbf{v} \in \Dom \mathcal{X}$, then
\begin{gather}
\label{AA}
\widetilde{\mathbf{v}}(\mathbf{k}, \cdot) \in \mathfrak{d} \qquad \text{for a.~e.} \; \mathbf{k} \in \widetilde \Omega,
\\
\label{AB}
\mathfrak{a}[\mathbf{v}, \mathbf{v}] = \int_{\widetilde{\Omega}} \mathfrak{a}(\mathbf{k}) [\widetilde{\mathbf{v}}(\mathbf{k}, \cdot), \widetilde{\mathbf{v}}(\mathbf{k}, \cdot)] \, d \mathbf{k} .
\end{gather}
Conversely, if $\widetilde{\mathbf{v}} \in \mathcal{K}$ satisfies~\eqref{AA} and the integral in~\eqref{AB} is finite, then
$\mathbf{v} \in \Dom \mathcal{X}$ and~\eqref{AB} is valid.

From~\eqref{decompose} it follows that the spectrum of $\mathcal{A}$
is the union of segments  (spectral bands)
$\Ran E_j,\; j \in \mathbb{N}$. By~\eqref{Ker1} and~\eqref{Ker2},
\begin{equation*}
\min_{\mathbf{k}} E_j (\mathbf{k}) = E_j (0) = 0,\quad j = 1, \ldots, n,
\end{equation*}
i.~e., the first $n$ spectral bands of $\mathcal A$ overlap and have the common bottom $\lambda_0=0$,
while the $(n+1)$-th band is separated from zero (see \eqref{E_n+1}).

\subsection{Incorporation of the operators $\mathcal{A} (\mathbf{k})$ in the pattern of Chapter~1}
If $ d > 1 $, the operators $\mathcal{A} (\mathbf{k})$ depend on the multidimensional parameter $\mathbf{k}$.
As in \cite[Chapter~2]{BSu1}, we distinguish the onedimensional parameter $t = | \mathbf{k}|$.
We will apply the scheme of Chapter~1.
Then all constructions and estimates will depend on the additional parameter $\boldsymbol{\theta} = \mathbf{k} / | \mathbf{k}| \in \mathbb{S}^{d-1}$. We have to make our estimates uniform in $\boldsymbol{\theta}$.
The spaces $\mathfrak{H}$ and $\mathfrak{H}_*$ are defined by~\eqref{Spaces_H}.
We put $X(t) = X(t, \boldsymbol{\theta}) := \mathcal{X}(t \boldsymbol{\theta})$.
Then $X(t, \boldsymbol{\theta}) = X_0 + t  X_1 (\boldsymbol{\theta})$, where
$X_0 = h(\mathbf{x}) b (\mathbf{D}) f(\mathbf{x}), \; \Dom X_0 = \mathfrak{d}$,
and $ X_1 (\boldsymbol{\theta})$  is a bounded operator of multiplication by the matrix-valued function
$h(\mathbf{x}) b(\boldsymbol{\theta}) f(\mathbf{x})$.
Next, we put $A(t) = A(t, \boldsymbol{\theta}) := \mathcal{A}(t \boldsymbol{\theta})$.
Then $A(t, \boldsymbol{\theta}) = X(t, \boldsymbol{\theta})^* X(t, \boldsymbol{\theta})$.
According to~\eqref{Ker1},~\eqref{Ker2}, we have $\mathfrak{N} = \Ker X_0 = \Ker \mathcal{A} (0), \; \dim \mathfrak{N} = n$.
The number $d^0$ satisfies estimate~(\ref{d0_est}).
As shown in \cite[Chapter~2,~Section~3]{BSu1}, condition $n \le n_* = \dim \Ker X^*_0$ is also satisfied.
Moreover, either $n_* = n$ (if $m = n$), or $n_* = \infty$ (if $m > n$).

Thus, all assumptions of the abstract scheme are satisfied.

   In Subsection~\ref{abstr_X_A_section}, it was required to choose a number $\delta \in (0, d^0/8)$.
   Taking~\eqref{c*} and \eqref{d0_est} into account, we fix $\delta$ as follows:
\begin{equation}
\label{delta_fixation}
\delta = \frac{1}{4} c_* r^2_0 = \frac{1}{4} \alpha_0\|f^{-1} \|_{L_\infty}^{-2} \|g^{-1} \|_{L_\infty}^{-1} r^2_0.
\end{equation}
Next, by~\eqref{rank_alpha_ineq} and~\eqref{h_f_L_inf}, we have
\begin{equation}
\label{X_1_estimate}
\| X_1 (\boldsymbol{\theta}) \| \le  \alpha^{1/2}_1 \| h \|_{L_{\infty}} \| f \|_{L_{\infty}}, \quad \boldsymbol{\theta} \in \mathbb{S}^{d-1}.
\end{equation}
This allows us to choose $t^0$ (see~\eqref{abstr_t0_fixation})
equal to the following number independent of $\boldsymbol{\theta} \in \mathbb{S}^{d-1}$:
\begin{equation}
\label{t0_fixation}
\begin{aligned}
t^0 &= \delta^{1/2} \alpha_1^{-1/2} \| h \|_{L_{\infty}}^{-1} \|f\|_{L_{\infty}}^{-1}
\\
&= \frac{r_0}{2} \alpha_0^{1/2} \alpha_1^{-1/2} \left( \| h \|_{L_{\infty}} \| h^{-1} \|_{L_{\infty}} \|f \|_{L_\infty} \|f^{-1} \|_{L_\infty} \right)^{-1}.
\end{aligned}
\end{equation}
Obviously, $t^0 \le r_0/2$. Thus, the ball $|\mathbf{k}| \le t^0$ lies inside $\widetilde{\Omega}$.
 It is important that $c_*$, $\delta$, and $t^0$ (see~\eqref{c*}, \eqref{delta_fixation}, and
 \eqref{t0_fixation}) do not depend on $\boldsymbol{\theta}$.

By~\eqref{A(k)_nondegenerated_and_c_*},
Condition~\ref{nondeg} is now satisfied.
The germ $S(\boldsymbol{\theta})$ of the operator $A(t, \boldsymbol{\theta})$ is nondegenerate uniformly in
$\boldsymbol{\theta}$: we have $S(\boldsymbol{\theta}) \ge c_* I_{\mathfrak{N}}$ (cf.~\eqref{abstr_S_nondegenerated}).

\section{The effective characteristics \\ of the operator $\widehat{\mathcal{A}} = b(\mathbf{D})^* g(\mathbf{x}) b(\mathbf{D})$}

\subsection{The case where $f = \mathbf{1}_n$}
In the case where $f = \mathbf{1}_n$, the operator $A(t, \boldsymbol{\theta})$ plays a special role.
In this case, all the objects will be marked by \textquotedblleft$\, \widehat{\phantom{\_}} \,$\textquotedblright.
For instance, for the operator
\begin{equation}
\label{hatA}
\widehat{\mathcal{A}} = b(\mathbf{D})^* g(\mathbf{x}) b(\mathbf{D})
\end{equation}
the family $\widehat{\mathcal{A}} (\mathbf{k})$ is denoted by
$\widehat{A} (t, \boldsymbol{\theta})$. The kernel~(\ref{Ker2}) takes the form
\begin{equation}
\label{Ker3}
\widehat{\mathfrak{N}} = \left\lbrace \mathbf{u} \in L_2 (\Omega; \mathbb{C}^n) \colon \mathbf{u} = \mathbf{c} \in \mathbb{C}^n \right\rbrace,
\end{equation}
i.~e., $\widehat{\mathfrak{N}}$ consists of constant vector-valued functions.
The orthogonal projection $\widehat{P}$ of the space $L_2 (\Omega; \mathbb{C}^n)$ onto the subspace \eqref{Ker3}
is the operator of averaging over the cell:
\begin{equation}
    \label{Phat_projector}
	\widehat{P} \mathbf{u} = |\Omega|^{-1} \int_{\Omega} \mathbf{u} (\mathbf{x}) \, d\mathbf{x}.
\end{equation}

If $f = \mathbf{1}_n$, the constants \eqref{c*}, \eqref{delta_fixation}, and~\eqref{t0_fixation}
take the form
\begin{align}
\label{hatc_*}
&\widehat{c}_*  = \alpha_0 \|g^{-1} \|_{L_\infty}^{-1}, \\
\label{hatdelta_fixation}
&\widehat{\delta} =  \frac{1}{4} \alpha_0 \|g^{-1} \|_{L_\infty}^{-1} r^2_0,\\
\label{hatt0_fixation}
&\widehat{t}^{\,0}  = \frac{r_0}{2} \alpha_0^{1/2} \alpha_1^{-1/2} \left( \| h \|_{L_{\infty}} \| h^{-1} \|_{L_{\infty}} \right)^{-1}.
\end{align}

Inequality \eqref{X_1_estimate} takes the form
\begin{equation}
\label{hatX_1_estmate}
\| \widehat{X}_1 (\boldsymbol{\theta}) \| \le \alpha_1^{1/2} \| g \|_{L_{\infty}}^{1/2}.
\end{equation}

\subsection{The germ of the operator $\widehat{A}(t, \boldsymbol{\theta})$}
\label{germ_and_eff_g0_section}
According to~\cite[Chapter~3,~Section~1]{BSu1},
the spectral germ $\widehat{S} (\boldsymbol{\theta})$ of the family
$\widehat{A}(t, \boldsymbol{\theta})$ acts in $\widehat{\mathfrak{N}}$ and is represented as
\begin{equation*}
\widehat{S} (\boldsymbol{\theta}) = b(\boldsymbol{\theta})^* g^0 b(\boldsymbol{\theta}), \quad  \boldsymbol{\theta} \in \mathbb{S}^{d-1},
\end{equation*}
where $b(\boldsymbol{\theta})$ is the symbol of the operator $b(\mathbf{D})$, and $g^0$ is the so called \emph{effective matrix}. The constant
       ($m \times m$)-matrix $g^0$ is defined as follows. Suppose that a $\Gamma$-periodic $(n \times m)$-matrix-valued function
       $\Lambda \in \widetilde{H}^1 (\Omega)$ is the weak solution of the problem
       \begin{equation}
\label{equation_for_Lambda}
b(\mathbf{D})^* g(\mathbf{x}) (b(\mathbf{D}) \Lambda (\mathbf{x}) + \mathbf{1}_m) = 0, \quad \int_{\Omega} \Lambda (\mathbf{x}) \, d \mathbf{x} = 0.
\end{equation}
The effective matrix $g^0$ is given by
\begin{equation}
\label{g0}
g^0 = | \Omega |^{-1} \int_{\Omega} \widetilde{g} (\mathbf{x}) \, d \mathbf{x},
\end{equation}
where
\begin{equation}
\label{g_tilde}
\widetilde{g} (\mathbf{x}) := g(\mathbf{x})( b(\mathbf{D}) \Lambda (\mathbf{x}) + \mathbf{1}_m).
\end{equation}
It turns out that the matrix $g^0$ is positive definite.

\subsection{The effective operator}
Consider the symbol
\begin{equation}
\label{effective_oper_symb}
\widehat{S} (\mathbf{k}) := t^2 \widehat{S} (\boldsymbol{\theta}) = b(\mathbf{k})^* g^0 b(\mathbf{k}), \quad \mathbf{k} \in \mathbb{R}^{d}.
\end{equation}
Expression~(\ref{effective_oper_symb}) is the symbol of the DO
\begin{equation}
\label{hatA0}
\widehat{\mathcal{A}}^0 = b(\mathbf{D})^* g^0 b(\mathbf{D})
\end{equation}
acting in $L_2(\mathbb{R}^d; \mathbb{C}^n)$ and called the \emph{effective operator} for the operator~$\widehat{\mathcal{A}}$.

Let $\widehat{\mathcal{A}}^0 (\mathbf{k})$ be the operator family in~$L_2(\Omega; \mathbb{C}^n)$ corresponding to the operator~\eqref{hatA0}.
Then $\widehat{\mathcal{A}}^0 (\mathbf{k})$ is given by the expression
$b(\mathbf{D} + \mathbf{k})^* g^0 b(\mathbf{D} + \mathbf{k})$ with periodic boundary conditions. By~\eqref{Phat_projector} and~\eqref{effective_oper_symb}, we have
\begin{equation}
\label{hatS_P=hatA^0_P}
\widehat{S} (\mathbf{k}) \widehat{P} = \widehat{\mathcal{A}}^0 (\mathbf{k}) \widehat{P}.
\end{equation}

\subsection{Properties of the effective matrix}

The following properties of the matrix $g^0$ were checked in~\cite[Chapter~3, Theorem~1.5]{BSu1}.
\begin{proposition}
    The effective matrix satisfies the estimates
    \begin{equation}
    \label{Voigt_Reuss}
    \underline{g} \le g^0 \le \overline{g},
    \end{equation}
    where
    \begin{equation*}
    \overline{g} := | \Omega |^{-1} \int_{\Omega} g (\mathbf{x}) \, d \mathbf{x},
    \quad \underline{g} := \left( | \Omega |^{-1} \int_{\Omega} g (\mathbf{x})^{-1} \, d \mathbf{x}\right)^{-1}.
    \end{equation*}
    If $m = n$, then $g^0 = \underline{g}$.
\end{proposition}

For specific DOs, estimates~\eqref{Voigt_Reuss} are known in homogenization theory as the Voigt-Reuss bracketing.
Now we distinguish the cases where one of the inequalities in~\eqref{Voigt_Reuss} becomes an identity.
    The following statements were obtained in~\cite[Chapter~3, Propositions~1.6, 1.7]{BSu1}.

\begin{proposition}
    The identity $g^0 = \overline{g}$ is equivalent to the relations
    \begin{equation}
    \label{g0=overline_g_relat}
    b(\mathbf{D})^* \mathbf{g}_k (\mathbf{x}) = 0, \quad k = 1, \ldots, m,
    \end{equation}
    where $\mathbf{g}_k (\mathbf{x}), \; k = 1, \ldots,m,$~are the columns of the matrix~$g (\mathbf{x})$.
\end{proposition}

\begin{proposition}
    The identity $g^0 = \underline{g}$ is equivalent to the representations
    \begin{equation}
    \label{g0=underline_g_relat}
    \mathbf{l}_k (\mathbf{x}) = \mathbf{l}^0_k + b(\mathbf{D}) \mathbf{w}_k(\mathbf{x}), \quad \mathbf{l}^0_k \in \mathbb{C}^m, \quad \mathbf{w}_k \in \widetilde{H}^1 (\Omega; \mathbb{C}^n), \quad k = 1, \ldots,m,
    \end{equation}
    where $\mathbf{l}_k (\mathbf{x}), \; k = 1, \ldots,m,$~are the columns of the matrix~$g (\mathbf{x})^{-1}$.
\end{proposition}

\subsection{The analytic branches of eigenvalues and eigenvectors}

The analytic (in $t$) branches of the eigenvalues~$\widehat{\lambda}_l (t, \boldsymbol{\theta})$ and the analytic branches of
the eigenvectors $\widehat{\varphi}_l (t, \boldsymbol{\theta})$ of $\widehat{A} (t, \boldsymbol{\theta})$ admit the power series expansions of the form~\eqref{abstr_A(t)_eigenvalues_series} and~\eqref{abstr_A(t)_eigenvectors_series}
with the coefficients depending on $\boldsymbol{\theta}$:
\begin{gather}
\label{hatA_eigenvalues_series}
\widehat{\lambda}_l (t, \boldsymbol{\theta}) = \widehat{\gamma}_l (\boldsymbol{\theta}) t^2 + \widehat{\mu}_l (\boldsymbol{\theta}) t^3 + \ldots, \quad l = 1, \ldots, n, \\
\label{hatA_eigenvectors_series}
\widehat{\varphi}_l (t, \boldsymbol{\theta}) = \widehat{\omega}_l (\boldsymbol{\theta}) + t \widehat{\psi}^{(1)}_l (\boldsymbol{\theta}) + \ldots, \quad l = 1, \ldots, n.
\end{gather}
(However,  we do not control the interval of convergence $t = |\mathbf{k}| \le t_* (\boldsymbol{\theta})$.)
According to~\eqref{abstr_S_eigenvectors}, the numbers $\widehat{\gamma}_l (\boldsymbol{\theta})$ and the elements $\widehat{\omega}_l (\boldsymbol{\theta})$ are eigenvalues and eigenvectors of the germ:
\begin{equation*}
b(\boldsymbol{\theta})^* g^0 b(\boldsymbol{\theta}) \widehat{\omega}_l (\boldsymbol{\theta}) = \widehat{\gamma}_l (\boldsymbol{\theta}) \widehat{\omega}_l (\boldsymbol{\theta}), \quad l = 1, \ldots, n.
\end{equation*}

\subsection{The operator $\widehat{N} (\boldsymbol{\theta})$}

We need to describe the operator~$N$ (which in abstract terms is defined in~\eqref{abstr_F1_K0_N_invar}).
According to~\cite[Section~4]{BSu3}, for the family $\widehat{A} (t, \boldsymbol{\theta})$ this operator takes the form
\begin{equation}
\label{N(theta)}
\widehat{N} (\boldsymbol{\theta}) = b(\boldsymbol{\theta})^* L(\boldsymbol{\theta}) b(\boldsymbol{\theta}) \widehat{P},
\end{equation}
where the  $(m\times m)$-matrix $L (\boldsymbol{\theta})$ is given by
\begin{equation}
\label{L(theta)}
L (\boldsymbol{\theta}) = | \Omega |^{-1} \int_{\Omega} (\Lambda (\mathbf{x})^* b(\boldsymbol{\theta})^* \widetilde{g}(\mathbf{x}) + \widetilde{g}(\mathbf{x})^* b(\boldsymbol{\theta}) \Lambda (\mathbf{x}) ) \, d \mathbf{x}.
\end{equation}
Here $\Lambda (\mathbf{x})$ is the $\Gamma$-periodic solution of problem~\eqref{equation_for_Lambda}, and $\widetilde{g}(\mathbf{x})$~is
given by~\eqref{g_tilde}.

Observe that $L (\mathbf{k}) := t L (\boldsymbol{\theta}), \; \mathbf{k} \in \mathbb{R}^d,$
 is a Hermitian matrix-valued function first order homogeneous in $\mathbf{k}$.
 We put $\widehat{N}(\mathbf{k}) := t^3 \widehat{N} (\boldsymbol{\theta}), \; \mathbf{k} \in \mathbb{R}^d$.
 Then $\widehat{N}(\mathbf{k}) = b(\mathbf{k})^* L(\mathbf{k}) b(\mathbf{k}) \widehat{P}$.
 The matrix-valued function $b(\mathbf{k})^* L(\mathbf{k}) b(\mathbf{k})$ is a homogeneous third order polynomial of~$\mathbf{k} \in \mathbb{R}^d$. Therefore, either $\widehat{N} (\boldsymbol{\theta}) = 0$ for all $\boldsymbol{\theta} \in \mathbb{S}^{d-1}$, or $\widehat{N} (\boldsymbol{\theta}) \ne 0$ at \textquotedblleft most\textquotedblright \ points $\boldsymbol{\theta}$
 (except for the zeroes of this polynomial).

Some cases where the operator~\eqref{N(theta)} is equal to zero were distinguished in~\cite[Section~4]{BSu3}.

\begin{proposition}
    \label{N=0_proposit}
    Suppose that at least one of the following conditions is fulfilled:

 $1^\circ$. The operator $\widehat{\mathcal{A}}$ has the form
  $\widehat{\mathcal{A}} = \mathbf{D}^* g(\mathbf{x}) \mathbf{D}$, where $g(\mathbf{x})$~is a symmetric matrix with real entries.

 $2^\circ$. Relations~\emph{\eqref{g0=overline_g_relat}} are satisfied, i.~e., $g^0 = \overline{g}$.

 $3^\circ$. Relations~\emph{\eqref{g0=underline_g_relat}} are satisfied, i.~e., $g^0 = \underline{g}$.
\emph{(}In particular, this is true if $m = n$.\emph{)}

\noindent
    Then $\widehat{N} (\boldsymbol{\theta}) = 0$ for any $\boldsymbol{\theta} \in \mathbb{S}^{d-1}$.
\end{proposition}

On the other hand, there are examples (see~\cite[Subsections~10.4,~13.2,~14.6]{BSu3}) showing that, in general, the operator
   $\widehat{N} (\boldsymbol{\theta})$ is not equal to zero.
   One example is the scalar operator of the form ${\mathbf D}^* g({\mathbf x}) {\mathbf D}$, where
    $g({\mathbf x})$ is a Hermitian matrix with complex entries (see Example~\ref{model_exmpl} below).
Other examples are matrix operators
     with real-valued coefficients; see Example~\ref{elast_exmpl_N0_ne_0} and Subsection~\ref{example} below.

Recall (see Remark~\ref{abstr_N_remark}) that
$\widehat{N} (\boldsymbol{\theta}) = \widehat{N}_0 (\boldsymbol{\theta}) + \widehat{N}_* (\boldsymbol{\theta})$, where the operator $\widehat{N}_0 (\boldsymbol{\theta})$ is diagonal in the basis $\{ \widehat{\omega}_l (\boldsymbol{\theta})\}_{l=1}^n$,
while  the diagonal elements of the operator $\widehat{N}_* (\boldsymbol{\theta})$
are equal to zero. We have
\begin{equation}
\label{hatN_0(theta)_matrix_elem}
(\widehat{N} (\boldsymbol{\theta}) \widehat{\omega}_l (\boldsymbol{\theta}), \widehat{\omega}_l (\boldsymbol{\theta}))_{L_2 (\Omega)} = (\widehat{N}_0 (\boldsymbol{\theta}) \widehat{\omega}_l (\boldsymbol{\theta}), \widehat{\omega}_l (\boldsymbol{\theta}))_{L_2 (\Omega)} = \widehat{\mu}_l (\boldsymbol{\theta}), \quad l=1, \ldots, n.
\end{equation}

In~\cite[Subsection~4.3]{BSu3}, the following argument was given. Suppose that
$b(\boldsymbol{\theta})$ and $g (\mathbf{x})$~are matrices with \emph{real entries}.
Then the matrix $\Lambda (\mathbf{x})$ (see~\eqref{equation_for_Lambda}) has purely imaginary entries, while
$\widetilde{g} (\mathbf{x})$ and $g^0$~are matrices with real entries. In this case $L (\boldsymbol{\theta})$ (see~\eqref{L(theta)}) and $b(\boldsymbol{\theta})^* L (\boldsymbol{\theta}) b(\boldsymbol{\theta})$~are Hermitian matrices with purely imaginary entries.
Hence, for any \emph{real} vector $\mathbf{q} \in \widehat{\mathfrak{N}}$ we have $(\widehat{N} (\boldsymbol{\theta}) \mathbf{q}, \mathbf{q}) = 0$. If the analytic branches of the eigenvalues $\widehat{\lambda}_l (t, \boldsymbol{\theta})$
and the analytic branches of the eigenvectors $\widehat{\varphi}_l (t, \boldsymbol{\theta})$ of $\widehat{A} (t, \boldsymbol{\theta})$
can be chosen so that the vectors $\widehat{\omega}_1 (\boldsymbol{\theta}), \ldots, \widehat{\omega}_n (\boldsymbol{\theta})$
are real, then, by~\eqref{hatN_0(theta)_matrix_elem}, we have $\widehat{\mu}_l (\boldsymbol{\theta}) = 0$, $l = 1, \ldots, n$, i.~e.,  $\widehat{N}_0 (\boldsymbol{\theta}) = 0$. We arrive at the following statement.

\begin{proposition}
    \label{N_0=0_proposit}
    Suppose that $b(\boldsymbol{\theta})$ and $g (\mathbf{x})$~have real entries. Suppose that in the expansions~\emph{\eqref{hatA_eigenvectors_series}} for the analytic branches of the eigenvectors of $\widehat{A} (t, \boldsymbol{\theta})$ the    \textquotedblleft embrios\textquotedblright \ $\widehat{\omega}_l (\boldsymbol{\theta}), \; l = 1, \ldots, n,$ can be chosen to be real. Then in~\emph{\eqref{hatA_eigenvalues_series}} we have $\widehat{\mu}_l (\boldsymbol{\theta}) = 0$, $l=1, \ldots, n,$ i.~e.,
    $\widehat{N}_0 (\boldsymbol{\theta}) = 0$.
\end{proposition}

  In the \textquotedblleft real\textquotedblright \ case under consideration, the germ $\widehat{S} (\boldsymbol{\theta})$
  is a symmetric matrix with real entries. Clearly, if the eigenvalue $\widehat{\gamma}_j (\boldsymbol{\theta})$ of the germ is simple,
then the embrio $\widehat{\omega}_j (\boldsymbol{\theta})$ is defined uniquely up to a phase factor, and we can always choose
$\widehat{\omega}_j (\boldsymbol{\theta})$ to be real.  We arrive at the following corollary.

\begin{corollary}
    \label{S_spec_simple_coroll}
    Suppose that $b(\boldsymbol{\theta})$ and $g (\mathbf{x})$
    have real entries. Suppose that the spectrum of the germ~$\widehat{S} (\boldsymbol{\theta})$ is simple. Then
    $\widehat{N}_0 (\boldsymbol{\theta}) = 0$.
\end{corollary}

However, as is seen from Example~\ref{elast_exmpl_N0_ne_0} considered below, even in the \textquotedblleft real\textquotedblright \ case
 it is not always possible to choose the vectors $\widehat{\omega}_l (\boldsymbol{\theta})$ to be real.
  Moreover, it can happen that
 $\widehat{N}_0 (\boldsymbol{\theta}) \ne 0$ at some points $\boldsymbol{\theta}$.

\subsection{Multiplicities of the eigenvalues of the germ}
\label{eigenval_multipl_section}
Considerations of this subsection concern the case where $n\ge 2$. Now we return to the notation of Section~\ref{abstr_cluster_section},
tracing the multiplicities of the eigenvalues of the spectral germ $\widehat{S} (\boldsymbol{\theta})$.
In general, the number $p(\boldsymbol{\theta})$ of different eigenvalues $\widehat{\gamma}^{\circ}_1 (\boldsymbol{\theta}), \ldots, \widehat{\gamma}^{\circ}_{p(\boldsymbol{\theta})} (\boldsymbol{\theta})$ of the spectral germ $\widehat{S}(\boldsymbol{\theta})$
and their multiplicities $k_1 (\boldsymbol{\theta}), \ldots, k_{p(\boldsymbol{\theta})} (\boldsymbol{\theta})$ depend on the parameter $\boldsymbol{\theta} \in \mathbb{S}^{d-1}$. For a fixed $\boldsymbol{\theta}$ denote by
$\widehat{P}_j (\boldsymbol{\theta})$ the orthogonal projection of $L_2 (\Omega; \mathbb{C}^n)$ onto the eigenspace of the germ
$\widehat{S}(\boldsymbol{\theta})$ corresponding to the eigenvalue $\widehat{\gamma}_j^{\circ} (\boldsymbol{\theta})$. According to~\eqref{abstr_N_invar_repers}, the operators $\widehat{N}_0 (\boldsymbol{\theta})$ and $\widehat{N}_* (\boldsymbol{\theta})$
admit the following invariant representations:
\begin{gather}
\label{N0_invar_repr}
\widehat{N}_0 (\boldsymbol{\theta}) = \sum_{i=1}^{p(\boldsymbol{\theta})} \widehat{P}_j (\boldsymbol{\theta}) \widehat{N} (\boldsymbol{\theta}) \widehat{P}_j (\boldsymbol{\theta}), \\
\label{N*_invar_repr}
\widehat{N}_* (\boldsymbol{\theta}) = \sum_{\substack{1 \le j, l \le p(\boldsymbol{\theta}):\\ j \ne l}} \widehat{P}_j (\boldsymbol{\theta}) \widehat{N} (\boldsymbol{\theta}) \widehat{P}_l (\boldsymbol{\theta}).
\end{gather}

In conclusion of this section, we give Example~8.7 from~\cite{Su4},
which shows that for matrix operators even with real-valued coefficients the eigenvalues of the germ may be multiple, and the coefficients
$\widehat{\mu}_l(\boldsymbol{\theta})$ in the expansions~\eqref{hatA_eigenvalues_series} may be nonzero.

\begin{example}[\cite{Su4}]
    \label{elast_exmpl_N0_ne_0}
Let $d=2$, $n=2$, and $m=3$.
For simplicity, assume that $\Gamma = (2\pi \mathbb{Z})^2$. Suppose that the operator $b(\mathbf{D})$ and the matrix $g(\mathbf{x})$
are given by
    \begin{gather*}
    b(\mathbf{D})= \begin{pmatrix}
    D_1 & 0 \\
    \frac{1}{2}D_2 & \frac{1}{2}D_1 \\
    0 & D_2
    \end{pmatrix}, \quad
    g(\mathbf{x})= \begin{pmatrix}
    1 & 0 & 0 \\
    0 & g_2 (x_1) & 0\\
    0 & 0 & g_3 (x_1)
    \end{pmatrix},
    \end{gather*}
where $g_2(x_1)$ and $g_3(x_1)$~are $(2\pi)$-periodic bounded and positive definite functions of $x_1$, and  $\overline{g_3}=1$.
Then, under a suitable choice of $g_2$ and $g_3$,  $\widehat{N}_0(\boldsymbol{\theta})\ne 0$ for some $\boldsymbol{\theta}$.
\end{example}

In this example, the matrices $b(\boldsymbol{\theta})$ and $g(\mathbf{x})$  have real entries.
It is easy to find the  $\Gamma$-periodic solution of problem~\eqref{equation_for_Lambda}:
\begin{equation*}
\Lambda(\mathbf{x})= \begin{pmatrix}
0 & 0 & 0 \\
0 & \Lambda_{22}(x_1) & 0
\end{pmatrix}.
\end{equation*}
Here $\Lambda_{22}(x_1)$~is the $(2\pi)$-periodic solution of the problem
\begin{equation*}
\frac{1}{2} D_1 \Lambda_{22}(x_1) + 1 = \underline{g_2} (g_2 (x_1))^{-1}, \quad \int_{0}^{2\pi} \Lambda_{22}(x_1)\, d x_1 = 0.
\end{equation*}
Obviously, $\Lambda_{22}(x_1)$ is purely imaginary. Then
$\widetilde{g}(\mathbf{x}) = \diag \{1, \underline{g_2}, g_3(x_1)\}$, and the effective matrix is given by~$g^0 = \diag \{1, \underline{g_2}, 1\}$.  The spectral germ~$\widehat{S} (\boldsymbol{\theta}) = b(\boldsymbol{\theta})^* g^0 b(\boldsymbol{\theta})$ takes the form
\begin{equation}
\label{example_S}
\widehat{S} (\boldsymbol{\theta}) = \begin{pmatrix}
\theta_1^2 + \frac{1}{4} \theta_2^2 \underline{g_2} & \frac{1}{4} \theta_1 \theta_2 \underline{g_2} \\
 \frac{1}{4} \theta_1 \theta_2 \underline{g_2} & \frac{1}{4} \theta_1^2 \underline{g_2} + \theta_2^2
\end{pmatrix}, \quad \boldsymbol{\theta} = (\theta_1, \theta_2) \in \mathbb{S}^1.
\end{equation}
It is easily seen that the matrix~\eqref{example_S}  has a multiple eigenvalue (for some~$\boldsymbol{\theta}$)
only if $\underline{g_2} = 4$.

  So, let $\underline{g_2} = 4$. Then the eigenvalues of the germ
 \begin{equation*}
\widehat{S} (\boldsymbol{\theta}) = \begin{pmatrix}
1  &  \theta_1 \theta_2 \\
\theta_1 \theta_2 & 1
\end{pmatrix}
\end{equation*}
are given by $\widehat{\gamma}_1(\boldsymbol{\theta}) = 1 + \theta_1 \theta_2$, $\widehat{\gamma}_2(\boldsymbol{\theta}) = 1 - \theta_1 \theta_2$.
They coincide at four points~$\boldsymbol{\theta}^{(1)} = (0, 1)$, $\boldsymbol{\theta}^{(2)} = (0, -1)$, $\boldsymbol{\theta}^{(3)} = (1, 0)$, and $\boldsymbol{\theta}^{(4)} = (-1, 0)$.

Next, we calculate the $(3 \times 3)$-matrix~$L(\boldsymbol{\theta})$~(see~\eqref{L(theta)}):
\begin{equation*}
L(\boldsymbol{\theta}) = \begin{pmatrix}
0 & 0 & 0 \\
0 & 0 & \theta_2 \overline{\Lambda_{22}^* g_3} \\
0 & \theta_2 \overline{\Lambda_{22} g_3} & 0
\end{pmatrix}.
\end{equation*}
Hence,
\begin{equation*}
\widehat{N}(\boldsymbol{\theta}) = b(\boldsymbol{\theta})^* L(\boldsymbol{\theta}) b(\boldsymbol{\theta}) = \frac{1}{2} \theta_2^3 \begin{pmatrix}
0 & \overline{\Lambda_{22}^* g_3} \\
\overline{\Lambda_{22} g_3} & 0
\end{pmatrix}.
\end{equation*}
Below we assume that~$\overline{\Lambda_{22} g_3} \ne 0$.
(It is easy to give a concrete example: if
$$
g_2(x_1) = 4\left( 1 + \tfrac{1}{2} \sin x_1\right)^{-1}, \quad g_3(x_1) = 1 + \tfrac{1}{2} \cos x_1,
$$
all the conditions are fulfilled.)

For $\boldsymbol{\theta} \ne \boldsymbol{\theta}^{(j)}$, $j = 1, 2, 3, 4$, we have
$\widehat{\gamma}_1 (\boldsymbol{\theta}) \ne \widehat{\gamma}_2 (\boldsymbol{\theta})$ and then  $\widehat{N}(\boldsymbol{\theta}) = \widehat{N}_*(\boldsymbol{\theta}) \ne 0$. At the points $\boldsymbol{\theta}^{(1)}$ and $\boldsymbol{\theta}^{(2)}$ we have
\begin{equation*}
\widehat{\gamma}_1 (\boldsymbol{\theta}^{(j)}) = \widehat{\gamma}_2 (\boldsymbol{\theta}^{(j)}) = 1, \quad \widehat{N}(\boldsymbol{\theta}^{(j)}) = \widehat{N}_0(\boldsymbol{\theta}^{(j)}) \ne 0, \quad j = 1, 2.
\end{equation*}
Obviously, the numbers $\pm \mu$, where $\mu = \frac{1}{2} \left| \overline{\Lambda_{22} g_3} \right| $,
 are the eigenvalues of the operator~$\widehat{N}(\boldsymbol{\theta}^{(j)})$ for~$j = 1, 2$.
 In the expansions~\eqref{hatA_eigenvalues_series} there are nonzero coefficients at~$t^3$:
\begin{equation*}
\widehat{\lambda}_1 (t, \boldsymbol{\theta}^{(j)}) = t^2 + \mu t^3 + \ldots, \quad \widehat{\lambda}_2 (t, \boldsymbol{\theta}^{(j)}) = t^2 - \mu t^3 + \ldots, \quad j = 1,2.
\end{equation*}
In this case, the embrios~$\widehat{\omega}_1 (\boldsymbol{\theta}^{(j)})$, $\widehat{\omega}_2 (\boldsymbol{\theta}^{(j)})$
 in the expansions~(\ref{hatA_eigenvectors_series})  can not be real (see Proposition~\ref{N_0=0_proposit}).

 At the points $\boldsymbol{\theta}^{(3)}$ and $\boldsymbol{\theta}^{(3)}$  the situation is different. We have
\begin{equation*}
\widehat{\gamma}_1 (\boldsymbol{\theta}^{(j)}) = \widehat{\gamma}_2 (\boldsymbol{\theta}^{(j)}) = 1, \quad \widehat{N}(\boldsymbol{\theta}^{(j)}) = 0, \quad j = 3, 4.
\end{equation*}

This example also shows that, though the operator $\widehat{N}(\boldsymbol{\theta})$ is always continuous in~$\boldsymbol{\theta}$
(it is a polynomial of the third degree), its \textquotedblleft blocks\textquotedblright \ $\widehat{N}_0(\boldsymbol{\theta})$ and $\widehat{N}_*(\boldsymbol{\theta})$ can be  discontinuous:
at the points where the branches of the eigenvalues of the germ intersect,
$\widehat{N}_0(\boldsymbol{\theta})$ and $\widehat{N}_*(\boldsymbol{\theta})$ may have jumps. Moreover, it may
happen that $\widehat{N}_0(\boldsymbol{\theta})$  is not equal to zero only at some isolated points.

\section{Approximation of the operator $\cos(\varepsilon^{-1} \tau \widehat{\mathcal{A}}(\mathbf{k})^{1/2})$}

\subsection{The case where $|\mathbf{k}| \le \widehat{t}^{\,0} $}

Consider the operator $\mathcal{H}_0 = -\Delta$ in $L_2 (\mathbb{R}^d; \mathbb{C}^n)$.
Under the Gelfand transformation, this operator expands in the direct integral of the operators
$\mathcal{H}_0 (\mathbf{k})$ acting in $L_2 (\Omega; \mathbb{C}^n)$. The operator~$\mathcal{H}_0 (\mathbf{k})$
is given by the differential expression $| \mathbf{D} + \mathbf{k} |^2$  with periodic boundary conditions.
Denote
\begin{equation}
\label{R(k, epsilon)}
\mathcal{R}(\mathbf{k}, \varepsilon) := \varepsilon^2 (\mathcal{H}_0 (\mathbf{k}) + \varepsilon^2 I)^{-1}.
\end{equation}
Obviously,
\begin{equation}
\label{R_P}
\mathcal{R}(\mathbf{k}, \varepsilon)^{s/2}\widehat{P} = \varepsilon^s (t^2 + \varepsilon^2)^{-s/2} \widehat{P}, \qquad s > 0.
\end{equation}

We will apply theorems of Section~\ref{abstr_aprox_thrm_section} to the operator
 $\widehat{A}(t, \boldsymbol{\theta}) = \widehat{\mathcal{A}}(\mathbf{k})$ starting with Theorem~\ref{abstr_cos_general_thrm}.
First, we need to specify the constants. By~\eqref{hatc_*}, \eqref{hatdelta_fixation}, and~\eqref{hatX_1_estmate},
 instead of the precise values of the constants~\eqref{abstr_C1_C2} and \eqref{abstr_C7} (which now depend on $\boldsymbol{\theta}$)
we can take
\begin{align}
\label{hatC1}
\widehat{C}_1 &= 2 \beta_1 r_0^{-1} \alpha_1^{1/2} \alpha_0^{-1/2} \|g\|_{L_\infty}^{1/2} \|g^{-1}\|_{L_\infty}^{1/2},\\
\label{hatC7}
\widehat{C}_7 &= 2 \beta_7 r_0^{-1} \alpha_1 \alpha_0^{-1/2} \|g\|_{L_\infty} \|g^{-1}\|_{L_\infty}^{1/2} + \beta_2 r_0^{-1} \alpha_1^{3/2} \alpha_0^{-1} \|g\|_{L_\infty}^{3/2} \|g^{-1}\|_{L_\infty}.
\end{align}

Combining~\eqref{abstr_cos_general_est}, \eqref{hatS_P=hatA^0_P}, and~\eqref{R_P},  we arrive at the inequality
\begin{multline}
\label{cos_main_est_k<hat_t^0}
\| ( \cos (\varepsilon^{-1} \tau \widehat{\mathcal{A}}(\mathbf{k})^{1/2})  - \cos (\varepsilon^{-1} \tau \widehat{\mathcal{A}}^0(\mathbf{k})^{1/2}) ) \mathcal{R}(\mathbf{k}, \varepsilon)  \widehat{P} \|_{L_2(\Omega) \to L_2 (\Omega) }
 \\
\le (\widehat{C}_1 + \widehat{C}_7 |\tau|) \varepsilon, \quad
\tau \in \mathbb{R}, \; \varepsilon > 0, \; |\mathbf{k}| \le \widehat{t}^{\,0}.
\end{multline}

\subsection{The case where $ | \mathbf{k} | > \widehat{t}^{\,0}$}
For $\mathbf{k} \in \widetilde{\Omega}$ and $ | \mathbf{k} | > \widehat{t}^{\,0}$ estimates are trivial. By~\eqref{R_P}, we have
\begin{equation}
\label{R_hatP_est}
\| \mathcal{R}(\mathbf{k}, \varepsilon)^{1/2}\widehat{P} \|_{L_2 (\Omega) \to L_2 (\Omega)} \le (\widehat{t}^{\,0})^{-1} \varepsilon, \qquad \varepsilon > 0, \; \mathbf{k} \in \widetilde{\Omega}, \; | \mathbf{k} | > \widehat{t}^{\,0}.
\end{equation}
Therefore,
\begin{multline}
\label{cos_main_est_k>hat_t^0}
\| ( \cos (\varepsilon^{-1} \tau \widehat{\mathcal{A}}(\mathbf{k})^{1/2}) - \cos (\varepsilon^{-1} \tau \widehat{\mathcal{A}}^0(\mathbf{k})^{1/2})) \mathcal{R}(\mathbf{k}, \varepsilon)^{1/2}  \widehat{P} \|_{L_2(\Omega) \to L_2 (\Omega) }
\\
 \le 2 (\widehat{t}^{\,0})^{-1} \varepsilon,\quad
\tau \in \mathbb{R}, \; \varepsilon > 0, \; \mathbf{k} \in \widetilde{\Omega}, \; |\mathbf{k}| > \widehat{t}^{\,0}.
\end{multline}
Note that here the smoothing operator is  $\mathcal{R}(\mathbf{k}, \varepsilon)^{1/2}$ (i.~e.,  $s = 1$).
Of course, the left-hand side of~\eqref{cos_main_est_k<hat_t^0} also satisfies the same estimate.

From~\eqref{cos_main_est_k<hat_t^0} and~\eqref{cos_main_est_k>hat_t^0}, using expressions for
$\widehat{t}^{\,0}$ and $\widehat{C}_1$, we obtain
\begin{multline}
\label{cos_main_est_w_hatP}
\| ( \cos (\varepsilon^{-1} \tau \widehat{\mathcal{A}}(\mathbf{k})^{1/2})  - \cos (\varepsilon^{-1} \tau \widehat{\mathcal{A}}^0(\mathbf{k})^{1/2})) \mathcal{R}(\mathbf{k}, \varepsilon)  \widehat{P} \|_{L_2(\Omega) \to L_2 (\Omega) }
\\
\le (\widehat{C}^*_1 + \widehat{C}_7 |\tau|) \varepsilon, \quad
\tau \in \mathbb{R}, \; \varepsilon > 0, \; \mathbf{k} \in \widetilde{\Omega},
\end{multline}
where $\widehat{C}^*_1 = \max\{\widehat{C}_1, 2 (\widehat{t}^0)^{-1}\} = \beta^*_1 r_0^{-1} \alpha_1^{1/2} \alpha_0^{-1/2} \|g\|_{L_\infty}^{1/2} \|g^{-1}\|_{L_\infty}^{1/2}$.

\subsection{Removal of the operator $\widehat{P}$}
Now we show that, up to an admissible error, the projection $\widehat{P}$ can be replaced by the identity operator  under the norm sign in~\eqref{cos_main_est_w_hatP}. For this, we estimate the norm of the operator~$\mathcal{R}(\mathbf{k}, \varepsilon)^{s/2}(I - \widehat{P})$.
Under the discrete Fourier transformation (see~(\ref{fourier})), the operator $\mathcal{R}(\mathbf{k}, \varepsilon)^{s/2}$
 turns into multiplication of the Fourier coefficients by the symbol
$\varepsilon^s (|\mathbf{b} + \mathbf{k}|^2 + \varepsilon^2)^{-s/2}$. The operator $I - \widehat{P}$
makes the zero Fourier coefficient equal to zero. Therefore,
\begin{multline}
\label{R(k,eps)(I-P)_est}
\| \mathcal{R}(\mathbf{k}, \varepsilon)^{s/2} (I - \widehat{P}) \|_{L_2(\Omega) \to L_2 (\Omega) }  = \sup_{0 \ne \mathbf{b} \in \widetilde{\Gamma}} \varepsilon^s (|\mathbf{b} + \mathbf{k}|^2 + \varepsilon^2)^{-s/2} \le r_0^{-s} \varepsilon^s, \\
\varepsilon > 0, \; \mathbf{k} \in \widetilde{\Omega}.
\end{multline}
Hence,
\begin{multline}
\label{cos_main_est_(I - hatP)}
\| ( \cos (\varepsilon^{-1} \tau \widehat{\mathcal{A}}(\mathbf{k})^{1/2}) - \cos (\varepsilon^{-1} \tau \widehat{\mathcal{A}}^0(\mathbf{k})^{1/2})) \mathcal{R}(\mathbf{k}, \varepsilon)^{1/2}  (I - \widehat{P}) \|_{L_2(\Omega) \to L_2 (\Omega) }
\\
\le 2 r_0^{-1} \varepsilon,\quad
\tau \in \mathbb{R}, \; \varepsilon > 0, \; \mathbf{k} \in \widetilde{\Omega}.
\end{multline}
Note that here the smoothing operator is~$\mathcal{R}(\mathbf{k}, \varepsilon)^{1/2}$ (i.~e., $s = 1$).

Finally, from~\eqref{cos_main_est_w_hatP} and~\eqref{cos_main_est_(I - hatP)},  using the obvious inequality
$\| \mathcal{R}(\mathbf{k}, \varepsilon)\| \le 1$ and expressions for the constants,
we obtain the following result which has been proved before in ~\cite[Theorem 7.2]{BSu5}.

\begin{theorem}
    \label{cos_general_thrm}
For $\tau \in \mathbb{R}$, $\varepsilon > 0$, and $\mathbf{k} \in \widetilde{\Omega}$ we have
    \begin{equation*}
    \| ( \cos (\varepsilon^{-1} \tau \widehat{\mathcal{A}}(\mathbf{k})^{1/2})  - \cos (\varepsilon^{-1} \tau \widehat{\mathcal{A}}^0(\mathbf{k})^{1/2})) \mathcal{R}(\mathbf{k}, \varepsilon)\|_{L_2(\Omega) \to L_2 (\Omega) }  \le (\widehat{\mathcal{C}}_1 + \widehat{\mathcal{C}}_2 |\tau|) \varepsilon,
    \end{equation*}
   where
    \begin{equation}
    \label{calC1_C2}
    \widehat{\mathcal{C}}_1 = \widehat{\beta}_1 r_0^{-1} \alpha_1^{1/2} \alpha_0^{-1/2} \|g\|_{L_\infty}^{1/2} \|g^{-1}\|_{L_\infty}^{1/2}, \quad \widehat{\mathcal{C}}_2 = \widehat{C}_7.
    \end{equation}
The constant $\widehat{C}_7$ is given by~\emph{\eqref{hatC7}}.
\end{theorem}

\subsection{Refinement of approximation 
in the case where~$\widehat{N}(\boldsymbol{\theta}) = 0$}
Now we apply Theorem~\ref{abstr_cos_enchanced_thrm_1}, assuming that~$\widehat{N}(\boldsymbol{\theta}) = 0$ for any~$\boldsymbol{\theta} \in \mathbb{S}^{d-1}$. Taking~\eqref{hatS_P=hatA^0_P} and~\eqref{R_P} into account, we have
\begin{multline*}
\| ( \cos (\varepsilon^{-1} \tau \widehat{\mathcal{A}}(\mathbf{k})^{1/2}) - \cos (\varepsilon^{-1} \tau \widehat{\mathcal{A}}^0(\mathbf{k})^{1/2})) \mathcal{R}(\mathbf{k}, \varepsilon)^{3/4} \widehat{P} \|_{L_2 (\Omega) \to L_2 (\Omega) }
\\
\le ( \widehat{C}'_1 + \widehat{C}_{11} | \tau | ) \varepsilon, \quad
\tau \in \mathbb{R}, \; \varepsilon > 0, \; |\mathbf{k}| \le \widehat{t}^{\,0}.
\end{multline*}
Here $\widehat{C}'_1 = \max \{2, 2\widehat{C}_1\}$, and the constant $\widehat{C}_{11}$ is given by
\begin{multline}
\label{hatC11}
\widehat{C}_{11} =
2^{1/2} \beta_1 r^{-2}_0 \alpha_1^{3/2} \alpha_0^{-1} \|g\|_{L_\infty}^{3/2}  \|g^{-1} \|_{L_\infty}   +  4 \beta_{10} r^{-2}_0 \alpha_1^{2} \alpha_0^{-3/2} \|g\|_{L_\infty}^{2}  \|g^{-1}\|_{L_\infty}^{3/2}
\\
+ 4 \beta_{11} r^{-2}_0 \alpha_1^{3} \alpha_0^{-5/2} \|g\|_{L_\infty}^{3} \|g^{-1}\|_{L_\infty}^{5/2}
+ 4 \beta_{12} r^{-2}_0 \alpha_1^{4} \alpha_0^{-7/2} \|g\|_{L_\infty}^{4} \|g^{-1}\|_{L_\infty}^{7/2}.
\end{multline}

Together with~\eqref{cos_main_est_k>hat_t^0} and~\eqref{cos_main_est_(I - hatP)} this implies the following result.

\begin{theorem}
    \label{cos_enchanced_thrm_1}
Let   $\widehat{N}(\boldsymbol{\theta})$ be the operator defined by~\emph{(\ref{N(theta)})}.
Suppose that~$\widehat{N}(\boldsymbol{\theta}) = 0$ for any $\boldsymbol{\theta} \in \mathbb{S}^{d-1}$.
Then
    \begin{multline*}
    \| ( \cos (\varepsilon^{-1} \tau \widehat{\mathcal{A}}(\mathbf{k})^{1/2})  - \cos (\varepsilon^{-1} \tau \widehat{\mathcal{A}}^0(\mathbf{k})^{1/2}) )\mathcal{R}(\mathbf{k}, \varepsilon)^{3/4}\|_{L_2(\Omega) \to L_2 (\Omega) }
\\
\le (\widehat{\mathcal{C}}_3 + \widehat{\mathcal{C}}_4 |\tau|) \varepsilon,\quad
\tau \in \mathbb{R}, \; \varepsilon > 0, \; \mathbf{k} \in \widetilde{\Omega},
    \end{multline*}
where $\widehat{\mathcal{C}}_3 = \max \{2, 2\widehat{C}_1\} + 2 r_0^{-1}$ and $\widehat{\mathcal{C}}_4 = \widehat{C}_{11}$.
The constants $\widehat{C}_1$ and $\widehat{C}_{11}$ are given by~\emph{\eqref{hatC1}} and \emph{\eqref{hatC11}}, respectively.
\end{theorem}

Recall that some sufficient conditions ensuring that~$\widehat{N}(\boldsymbol{\theta}) = 0$ for any \hbox{$\boldsymbol{\theta} \in \mathbb{S}^{d-1}$}
 are given in Proposition~\ref{N=0_proposit}.

\subsection{Refinement of approximation 
in the case where~$\widehat{N}_0(\boldsymbol{\theta}) = 0$}
\label{ench_approx2_section}

Now, we reject the assumption of Theorem~\ref{cos_enchanced_thrm_1}, but instead we assume that~$\widehat{N}_0(\boldsymbol{\theta}) = 0$
for any~$\boldsymbol{\theta}$. We may also assume that~$\widehat{N}(\boldsymbol{\theta}) = \widehat{N}_*(\boldsymbol{\theta}) \ne 0$ for some   $\boldsymbol{\theta}$,  and then at \textquotedblleft most\textquotedblright \ points~$\boldsymbol{\theta}$
 (otherwise, one can apply Theorem~\ref{cos_enchanced_thrm_1}.)
We would like to apply the~\textquotedblleft abstract\textquotedblright \ result (namely, Theorem~\ref{abstr_cos_enchanced_thrm_2}).
However, there is an additional difficulty related to the fact that the multiplicities of the eigenvalues of the germ~$\widehat{S} (\boldsymbol{\theta})$
 may change at some points~$\boldsymbol{\theta}$.
 Near such points the distance between some pair of different eigenvalues
 tends to zero, and we are not able to choose the parameters $\widehat{c}^{\circ}_{jl}$ and $\widehat{t}^{\,00}_{jl}$ to be independent of~$\boldsymbol{\theta}$. Therefore, we are forced to impose an additional condition.
We have to take care only about those pairs of eigenvalues for which the corresponding term in~\eqref{N*_invar_repr}
is not zero.  Since the number of different eigenvalues of the germ and their multiplicities may depend on~$\boldsymbol{\theta}$,
now it is more convenient to use the initial enumeration of the eigenvalues
$\widehat{\gamma}_1 (\boldsymbol{\theta}), \ldots , \widehat{\gamma}_n (\boldsymbol{\theta})$ of the germ $\widehat{S} (\boldsymbol{\theta})$
(each eigenvalue is repeated according to its multiplicity). We enumerate them in the nondecreasing order:
\begin{equation*}
\widehat{\gamma}_1 (\boldsymbol{\theta}) \le \widehat{\gamma}_2 (\boldsymbol{\theta}) \le \ldots \le \widehat{\gamma}_n (\boldsymbol{\theta}).
\end{equation*}
For each $\boldsymbol{\theta}$ denote by $\widehat{P}^{(k)} (\boldsymbol{\theta})$ the orthogonal projection of~$L_2 (\Omega; \mathbb{C}^n)$
onto the eigenspace of $\widehat{S} (\boldsymbol{\theta})$ corresponding to the eigenvalue $\widehat{\gamma}_k (\boldsymbol{\theta})$.
Clearly, for each $\boldsymbol{\theta}$ the operator $\widehat{P}^{(k)} (\boldsymbol{\theta})$ coincides with one of the projections
$\widehat{P}_j (\boldsymbol{\theta})$       introduced in Subsection~\ref{eigenval_multipl_section} (but the number $j$ may depend on~$\boldsymbol{\theta}$).

\begin{condition}
    \label{cond9}

    $1^\circ$. The operator $\widehat{N}_0(\boldsymbol{\theta})$ defined by~\emph{\eqref{N0_invar_repr}} is equal to zero: $\widehat{N}_0(\boldsymbol{\theta})=0$ for any $\boldsymbol{\theta} \in \mathbb{S}^{d-1}$.

$2^\circ$.  For any pair of indices $(k,r), 1 \le k,r \le n, k \ne r$, such that $\widehat{\gamma}_k (\boldsymbol{\theta}_0) = \widehat{\gamma}_r (\boldsymbol{\theta}_0) $ for some $\boldsymbol{\theta}_0 \in \mathbb{S}^{d-1}$, we have
$\widehat{P}^{(k)} (\boldsymbol{\theta}) \widehat{N} (\boldsymbol{\theta}) \widehat{P}^{(r)} (\boldsymbol{\theta}) = 0$ for any
\hbox{$\boldsymbol{\theta} \in \mathbb{S}^{d-1}$}.
\end{condition}

  Condition $2^\circ$ can be reformulated
    as follows: we assume that, for the
 \textquotedblleft blocks\textquotedblright \ $\widehat{P}^{(k)} (\boldsymbol{\theta}) \widehat{N} (\boldsymbol{\theta}) \widehat{P}^{(r)} (\boldsymbol{\theta})$ of the operator $\widehat{N} (\boldsymbol{\theta})$  that are not
    identically zero, the corresponding branches of the eigenvalues  $\widehat{\gamma}_k (\boldsymbol{\theta})$ and  $\widehat{\gamma}_r (\boldsymbol{\theta})$   do not intersect.

   Obviously, Condition~\ref{cond9} is ensured by the following more restrictive condition.

\begin{condition}
     \label{cond99}

         $1^\circ$.
The operator $\widehat{N}_0(\boldsymbol{\theta})$ defined by~\emph{\eqref{N0_invar_repr}} is equal to zero: $\widehat{N}_0(\boldsymbol{\theta})=0$ for any $\boldsymbol{\theta} \in \mathbb{S}^{d-1}$.

$2^\circ$.  Assume that the number $p$ of different eigenvalues of the spectral germ $\widehat{S}(\boldsymbol{\theta})$
does not depend on $\boldsymbol{\theta} \in \mathbb{S}^{d-1}$.
        \end{condition}

Under Condition~\ref{cond99},  denote different eigenvalues of the germ enumerated in the increasing order
 by $\widehat{\gamma}^{\circ}_1(\boldsymbol{\theta}), \ldots, \widehat{\gamma}^{\circ}_p(\boldsymbol{\theta})$.
Then their multiplicities $k_1, \ldots, k_p$ do not depend on $\boldsymbol{\theta} \in \mathbb{S}^{d-1}$.

\begin{remark}
Assumption $2^\circ$ of Condition~\emph{\ref{cond99}} is a fortiori satisfied, if the spectrum of the germ $\widehat{S}(\boldsymbol{\theta})$
is simple for any $\boldsymbol{\theta} \in \mathbb{S}^{d-1}$.
\end{remark}

So, we assume that Condition~\ref{cond9}  is satisfied.
We are interested only in the pairs of indices from the set
\begin{equation*}
\widehat{\mathcal{K}} := \{ (k,r) \colon 1 \le k,r \le n, \; k \ne r, \;  \widehat{P}^{(k)} (\boldsymbol{\theta}) \widehat{N} (\boldsymbol{\theta}) \widehat{P}^{(r)} (\boldsymbol{\theta}) \not\equiv 0 \}.
\end{equation*}
Denote
\begin{equation*}
\widehat{c}^{\circ}_{kr} (\boldsymbol{\theta}) := \min \{\widehat{c}_*, n^{-1} |\widehat{\gamma}_k (\boldsymbol{\theta}) - \widehat{\gamma}_r (\boldsymbol{\theta})| \}, \quad (k,r) \in \widehat{\mathcal{K}}.
\end{equation*}
Since $\widehat{S} (\boldsymbol{\theta})$  is continuous in $\boldsymbol{\theta} \in \mathbb{S}^{d-1}$ (this is a polynomial of the second degree),
then the perturbation theory of discrete spectrum implies that the functions
  $\widehat{\gamma}_j (\boldsymbol{\theta})$~are continuous  on the sphere~$\mathbb{S}^{d-1}$.
By Condition~\ref{cond9}($2^\circ$), for $(k,r) \in \widehat{\mathcal{K}}$ we
have~$|\widehat{\gamma}_k (\boldsymbol{\theta}) - \widehat{\gamma}_r (\boldsymbol{\theta})| > 0$ for any $\boldsymbol{\theta} \in \mathbb{S}^{d-1}$, whence
\begin{equation*}
\widehat{c}^{\circ}_{kr}:= \min_{\boldsymbol{\theta} \in \mathbb{S}^{d-1}} \widehat{c}^{\circ}_{kr} (\boldsymbol{\theta}) > 0, \quad (k,r) \in \widehat{\mathcal{K}}.
\end{equation*}
We put
\begin{equation}
\label{hatc^circ}
\widehat{c}^{\circ} := \min_{(k,r) \in \widehat{\mathcal{K}}} \widehat{c}^{\circ}_{kr}.
\end{equation}
Clearly, the number~\eqref{hatc^circ}~ is a realization of~\eqref{abstr_c^circ} chosen independently of $\boldsymbol{\theta}$.

Under Condition~\ref{cond9}, the number $\widehat{t}^{\,00}$ subject to~\eqref{abstr_t00}
also can be chosen independently of $\boldsymbol{\theta} \in \mathbb{S}^{d-1}$. Taking~\eqref{hatdelta_fixation} and~\eqref{hatX_1_estmate} into account, we put
\begin{equation}
\label{hatt00_fixation}
\widehat{t}^{\,00} = (8 \beta_2)^{-1} r_0 \alpha_1^{-3/2} \alpha_0^{1/2} \| g\|_{L_{\infty}}^{-3/2} \| g^{-1}\|_{L_{\infty}}^{-1/2} \widehat{c}^{\circ},
\end{equation}
where $\widehat{c}^{\circ}$ is defined by~\eqref{hatc^circ}. (The condition $\widehat{t}^{\,00} \le \widehat{t}^{\,0}$
 is valid automatically since $\widehat{c}^{\circ} \le \| \widehat{S} (\boldsymbol{\theta}) \| \le \alpha_1 \|g\|_{L_{\infty}}$.)

\begin{remark}

$1^\circ$.  Unlike $\widehat{t}^{\,0}$ \emph{(}see~\emph{\eqref{hatt0_fixation}}\emph{)}  that is controlled only in terms of
$r_0$, $\alpha_0$, $\alpha_1$, $\|g\|_{L_{\infty}}$, and $\|g^{-1}\|_{L_{\infty}}$, the number $\widehat{t}^{\,00}$
depends on the spectral characteristics of the germ, namely, on the  minimal distance between its different eigenvalues
$\widehat{\gamma}_k (\boldsymbol{\theta})$ and $ \widehat{\gamma}_r (\boldsymbol{\theta})$ \emph{(}where $(k,r)$
runs through $\widehat{\mathcal{K}}$\emph{)}.

$2^\circ$.  If we reject Condition~\emph{\ref{cond9}} and admit intersection of the branches
$\widehat{\gamma}_k (\boldsymbol{\theta})$ and $\widehat{\gamma}_r (\boldsymbol{\theta})$ \emph{(}for some $(k,r) \in \widehat{\mathcal{K}}$\emph{)}, then $\widehat{c}^{\circ}$ will be not positive definite, and we
will be not able to choose the number~$\widehat{t}^{\,00}$ independently of $\boldsymbol{\theta}$.
\end{remark}

Under Condition~\ref{cond9}, we apply Theorem~\ref{abstr_cos_enchanced_thrm_2} and obtain
\begin{multline}
\label{cos_enched_est_k<hat_t^00}
\| ( \cos (\varepsilon^{-1} \tau \widehat{\mathcal{A}}(\mathbf{k})^{1/2}) - \cos (\varepsilon^{-1} \tau \widehat{\mathcal{A}}^0(\mathbf{k})^{1/2})) \mathcal{R}(\mathbf{k}, \varepsilon)^{3/4} \widehat{P} \|_{L_2 (\Omega) \to L_2 (\Omega) }
\\
\le (\widehat{C}'_{15} + \widehat{C}_{16} | \tau |) \varepsilon, \quad
\tau \in \mathbb{R}, \; \varepsilon > 0, \; |\mathbf{k}| \le \widehat{t}^{\,00},
\end{multline}
where
\begin{equation}
\label{hatC'_15_16}
\widehat{C}'_{15} = \max\{2, 2 \widehat{C}_1 + \widehat{C}_{14}\}, \quad
\widehat{C}_{16} = \widehat{C}_{11} + \widehat{C}_{13}.
\end{equation}
The constants $\widehat{C}_{13}$ and $\widehat{C}_{14}$ are given by
\begin{align*}
\widehat{C}_{13} &= 4 \beta_{13} n^2    \alpha_0^{-3/2} \alpha_1^{4} \|g^{-1} \|_{L_\infty}^{3/2}    \|g\|_{L_\infty}^{4} r^{-2}_0  (\widehat{c}^{\circ})^{-2},
\\
\widehat{C}_{14} &= 2 \beta_{14} n^2   \alpha_0^{-1}  \alpha_1^{2} \|g^{-1}\|_{L_\infty}  \|g\|_{L_\infty}^{2} r_0^{-1}
 ( \widehat{c}^{\circ})^{-1} ,
\end{align*}
the constants $\widehat{C}_{1}$ and $\widehat{C}_{11}$ are defined by~\eqref{hatC1} and \eqref{hatC11}, respectively.

Similarly to~\eqref{cos_main_est_k>hat_t^0}, we have
\begin{multline}
\label{cos_enched_est_k>hat_t^00}
\| ( \cos (\varepsilon^{-1} \tau \widehat{\mathcal{A}}(\mathbf{k})^{1/2}) - \cos (\varepsilon^{-1} \tau \widehat{\mathcal{A}}^0(\mathbf{k})^{1/2}) ) \mathcal{R}(\mathbf{k}, \varepsilon)^{1/2}  \widehat{P} \|_{L_2(\Omega) \to L_2 (\Omega) }
\\
\le 2 (\widehat{t}^{\,00})^{-1} \varepsilon,\quad
\tau \in \mathbb{R}, \; \varepsilon > 0, \; \mathbf{k} \in \widetilde{\Omega}, \; |\mathbf{k}| > \widehat{t}^{\,00}.
\end{multline}

Now, relations~\eqref{cos_main_est_(I - hatP)}, \eqref{cos_enched_est_k<hat_t^00}, and \eqref{cos_enched_est_k>hat_t^00}
directly imply the following result.
\begin{theorem}
    \label{cos_enchanced_thrm_2}
     Suppose that Condition~\emph{\ref{cond9}} \emph{(}or more restrictive Condition~\emph{\ref{cond99}}\emph{)}
is satisfied. Then
     \begin{multline}
     \label{cos_enchanced_est_2}
     \| ( \cos (\varepsilon^{-1} \tau \widehat{\mathcal{A}}(\mathbf{k})^{1/2})  - \cos (\varepsilon^{-1} \tau \widehat{\mathcal{A}}^0(\mathbf{k})^{1/2}) ) \mathcal{R}(\mathbf{k}, \varepsilon)^{3/4}\|_{L_2(\Omega) \to L_2 (\Omega) }
\\
\le (\widehat{\mathcal{C}}_5 + \widehat{\mathcal{C}}_6 |\tau|) \varepsilon,
\quad \tau \in \mathbb{R}, \; \varepsilon > 0, \; \mathbf{k} \in \widetilde{\Omega},
     \end{multline}
where $\widehat{\mathcal{C}}_5 = \max \{\widehat{C}'_{15}, 2 (\widehat{t}^{\,00})^{-1}\} + 2r_0^{-1}$, $\widehat{\mathcal{C}}_6 = \widehat{C}_{16}$, and the constants $\widehat{C}'_{15}$, $\widehat{C}_{16}$, and $\widehat{t}^{\,00}$
are defined by~\emph{\eqref{hatC'_15_16}} and~\emph{\eqref{hatt00_fixation}}.
\end{theorem}

The assumptions of Theorem~\ref{cos_enchanced_thrm_2}  are a fortiori satisfied in the \textquotedblleft real\textquotedblright \ case, if the
spectrum of the germ is simple (see Corollary~\ref{S_spec_simple_coroll}).
We arrive at the following corollary.

\begin{corollary}
    \label{cos_enchanced_2_coroll}
    Suppose that the matrices $b (\boldsymbol{\theta})$ and $g (\mathbf{x})$ have real entries.
Suppose that the spectrum of the germ $\widehat{S}(\boldsymbol{\theta})$  is simple for any $\boldsymbol{\theta} \in \mathbb{S}^{d-1}$.
Then estimate~\emph{\eqref{cos_enchanced_est_2}} holds for any $\tau \in \mathbb{R}, \varepsilon > 0$, and $\mathbf{k} \in \widetilde{\Omega}$.
\end{corollary}

\subsection{The sharpness of the result in the general case}

Application of Theorem~\ref{abstr_s<2_general_thrm} allows us to confirm the sharpness of the result of Theorem~\ref{cos_general_thrm}
 in the general case.

\begin{theorem}
    \label{hat_s<2_thrm}
 Let $\widehat{N}_0 (\boldsymbol{\theta})$ be the operator defined by~\emph{\eqref{N0_invar_repr}}.
Suppose that  $\widehat{N}_0 (\boldsymbol{\theta}_0) \ne 0$ for some $\boldsymbol{\theta}_0 \in \mathbb{S}^{d-1}$.
Let $0 \ne \tau \in \mathbb{R}$ and $0 \le s < 2$. Then there does not exist a constant $\mathcal{C} (\tau) > 0$
such that the estimate
    \begin{equation}
    \label{hat_s<2_est_imp}
     \| ( \cos (\varepsilon^{-1} \tau \widehat{\mathcal{A}}(\mathbf{k})^{1/2})  - \cos (\varepsilon^{-1} \tau \widehat{\mathcal{A}}^0(\mathbf{k})^{1/2})) \mathcal{R}(\mathbf{k}, \varepsilon)^{s/2}\|_{L_2(\Omega) \to L_2 (\Omega) }  \le \mathcal{C} (\tau) \varepsilon
    \end{equation}
    holds for almost all $\mathbf{k} = t \boldsymbol{\theta} \in \widetilde{\Omega}$ and sufficiently small $\varepsilon > 0$.
\end{theorem}

For the proof we need the following lemma which is similar to Lemma~9.10 from~\cite{Su4}.

\begin{lemma}
    \label{Lipschitz_lemma}
Let $\widehat{\delta}$ and $\widehat{t}^{\,0}$  be given by~\emph{\eqref{hatdelta_fixation}} and~\emph{\eqref{hatt0_fixation}}, respectively.
Let $\widehat{F} (\mathbf{k}) = \widehat{F} (t, \boldsymbol{\theta})$~be the spectral projection of the operator $\widehat{\mathcal{A}}(\mathbf{k})$ for the interval $[0, \widehat{\delta}]$. Then for $|\mathbf{k}| \le \widehat{t}^{\,0}$ and $|\mathbf{k}_0| \le \widehat{t}^{\,0}$ we have
    \begin{gather}
    \label{F(k) - F(k_0)}
    \| \widehat{F} (\mathbf{k}) - \widehat{F} (\mathbf{k}_0)\|_{L_2(\Omega) \to L_2 (\Omega) } \le \widehat{C}' | \mathbf{k} - \mathbf{k}_0|, \\
    \label{hatA^1/2 (k) - hatA^1/2 (k_0)}
    \| \widehat{\mathcal{A}}(\mathbf{k})^{1/2} \widehat{F} (\mathbf{k}) - \widehat{\mathcal{A}}(\mathbf{k}_0)^{1/2} \widehat{F} (\mathbf{k}_0)\|_{L_2(\Omega) \to L_2 (\Omega) } \le \widehat{C}'' | \mathbf{k} - \mathbf{k}_0|, \\
    \label{cos_hatA^1/2 (k) - cos_hatA^1/2 (k_0)}
\begin{aligned}
  \| \cos(\tau \widehat{\mathcal{A}}(\mathbf{k})^{1/2}) \widehat{F} (\mathbf{k}) - \cos(\tau \widehat{\mathcal{A}}(\mathbf{k}_0)^{1/2}) \widehat{F} (\mathbf{k}_0)\|_{L_2(\Omega) \to L_2 (\Omega) }
\\
\le (2 \widehat{C}' + \widehat{C}'') | \mathbf{k} - \mathbf{k}_0|.
\end{aligned}
    \end{gather}
\end{lemma}

\begin{proof}
Estimate~(\ref{F(k) - F(k_0)}) has been proved in~\cite[Lemma~9.10]{Su4}.

Let us prove~\eqref{hatA^1/2 (k) - hatA^1/2 (k_0)}. Let $\mathbf{k}, \mathbf{k}_0 \ne 0$.
We use the following representation
\begin{equation*}
    \widehat{\mathcal{A}}(\mathbf{k})^{1/2} \widehat{F} (\mathbf{k})  = \frac{1}{\pi} \int_{0}^{\infty} s^{-1/2} (\widehat{\mathcal{A}}(\mathbf{k}) \widehat{F} (\mathbf{k}) + sI)^{-1} \widehat{\mathcal{A}}(\mathbf{k}) \widehat{F} (\mathbf{k}) \, ds.
\end{equation*}
Then
$$
    \widehat{\mathcal{A}}(\mathbf{k})^{1/2} \widehat{F} (\mathbf{k}) - \widehat{\mathcal{A}}(\mathbf{k}_0)^{1/2} \widehat{F} (\mathbf{k}_0) =  \frac{1}{\pi} \int_{0}^{\infty} s^{-1/2}\Upsilon(s,\mathbf{k},\mathbf{k}_0)\, ds,
$$
where
$$
\Upsilon(s,\mathbf{k},\mathbf{k}_0):=
(\widehat{\mathcal{A}}(\mathbf{k}) \widehat{F} (\mathbf{k}) + sI)^{-1} \widehat{\mathcal{A}}(\mathbf{k}) \widehat{F} (\mathbf{k}) - (\widehat{\mathcal{A}}(\mathbf{k}_0) \widehat{F} (\mathbf{k}_0) + sI)^{-1} \widehat{\mathcal{A}}(\mathbf{k}_0) \widehat{F} (\mathbf{k}_0).
$$
It is easily seen that
\begin{multline*}
\Upsilon(s,\mathbf{k},\mathbf{k}_0)
\\
=
 s (\widehat{\mathcal{A}}(\mathbf{k}) \widehat{F} (\mathbf{k}) + sI)^{-1} \left( \widehat{\mathcal{A}}(\mathbf{k}) \widehat{F} (\mathbf{k}) - \widehat{\mathcal{A}}(\mathbf{k}_0) \widehat{F} (\mathbf{k}_0)\right)  (\widehat{\mathcal{A}}(\mathbf{k}_0) \widehat{F} (\mathbf{k}_0) + sI)^{-1}
\\
 =    \Upsilon_1(s,\mathbf{k}, \mathbf{k}_0) + \Upsilon_2(s,\mathbf{k}, \mathbf{k}_0) + \Upsilon_3(s, \mathbf{k}, \mathbf{k}_0),
\end{multline*}
where
\begin{align*}
 \Upsilon_1(s,\mathbf{k}, \mathbf{k}_0) =& \,s
(\widehat{\mathcal{A}}(\mathbf{k})  + sI)^{-1}
\\
&\times \left( \widehat{F} (\mathbf{k}) \widehat{\mathcal{A}}(\mathbf{k}) \widehat{F} (\mathbf{k}_0)  - \widehat{F} (\mathbf{k}) \widehat{\mathcal{A}}(\mathbf{k}_0) \widehat{F} (\mathbf{k}_0)\right)  (\widehat{\mathcal{A}}(\mathbf{k}_0) + sI)^{-1},
\\
 \Upsilon_2(s,\mathbf{k}, \mathbf{k}_0) =&
 - (I - \widehat{F} (\mathbf{k})) \widehat{\mathcal{A}}(\mathbf{k}_0) \widehat{F} (\mathbf{k}_0) (\widehat{\mathcal{A}}(\mathbf{k}_0) + sI)^{-1},
\\
\Upsilon_3(s,\mathbf{k}, \mathbf{k}_0) =&\,
 (\widehat{\mathcal{A}}(\mathbf{k})  + sI)^{-1} \widehat{F} (\mathbf{k}) \widehat{\mathcal{A}}(\mathbf{k}) (I-\widehat{F} (\mathbf{k}_0)).
\end{align*}
Hence,
$$
\widehat{\mathcal{A}}(\mathbf{k})^{1/2} \widehat{F} (\mathbf{k}) - \widehat{\mathcal{A}}(\mathbf{k}_0)^{1/2} \widehat{F} (\mathbf{k}_0) =
\Omega_1(\mathbf{k}, \mathbf{k}_0)+ \Omega_2(\mathbf{k}, \mathbf{k}_0) + \Omega_3(\mathbf{k}, \mathbf{k}_0),
$$
where
$$
\Omega_j(\mathbf{k}, \mathbf{k}_0) = \frac{1}{\pi} \int_{0}^{\infty} s^{-1/2} \Upsilon_j(s,\mathbf{k}, \mathbf{k}_0)\,ds,\quad j=1,2,3.
$$

To estimate $\Omega_1(\mathbf{k}, \mathbf{k}_0)$,
consider the difference of the sesquilinear forms of the operators
$\widehat{\mathcal{A}}(\mathbf{k})$ and $\widehat{\mathcal{A}}(\mathbf{k}_0)$
on the elements  $\mathbf{u}, \mathbf{v} \in \widetilde{H}^1 (\Omega; \mathbb{C}^n)$:
\begin{multline*}
\widehat{\mathfrak{a}} (\mathbf{k}) [\mathbf{u}, \mathbf{v}] - \widehat{\mathfrak{a}} (\mathbf{k}_0) [\mathbf{u}, \mathbf{v}] = \\ =
( g b(\mathbf{k} - \mathbf{k}_0) \mathbf{u}, b(\mathbf{D} + \mathbf{k}_0) \mathbf{v} )_{L_2 (\Omega)} + ( g b(\mathbf{D} + \mathbf{k}) \mathbf{u}, b(\mathbf{k} - \mathbf{k}_0) \mathbf{v})_{L_2 (\Omega)}.
\end{multline*}
Combining this with~\eqref{rank_alpha_ineq}, we see that
\begin{multline}
\label{9.23a}
|\widehat{\mathfrak{a}} (\mathbf{k}) [\mathbf{u}, \mathbf{v}] - \widehat{\mathfrak{a}} (\mathbf{k}_0) [\mathbf{u}, \mathbf{v}]| \le \\ \le \alpha_1^{1/2} \| g \|_{L_\infty}^{1/2}  |\mathbf{k}-\mathbf{k}_0| \left( \| \mathbf{u} \|_{L_2}  \| \widehat{\mathcal{A}}(\mathbf{k}_0)^{1/2} \mathbf{v} \|_{L_2} +  \| \widehat{\mathcal{A}}(\mathbf{k})^{1/2} \mathbf{u} \|_{L_2}  \|\mathbf{v} \|_{L_2}\right).
\end{multline}
Substituting $\mathbf{u} = \widehat{F} (\mathbf{k}) \boldsymbol{\varphi}$ and $\mathbf{v} =  \widehat{F} (\mathbf{k}_0) \boldsymbol{\psi}$
with $\boldsymbol{\varphi}, \boldsymbol{\psi} \in L_2 (\Omega; \mathbb{C}^n)$ in \eqref{9.23a},
and taking into account that
\begin{equation}
\label{hat_sqrtA(k)F(k)_up_est}
\| \widehat{\mathcal{A}}(\mathbf{k})^{1/2} \widehat{F} (\mathbf{k}) \|_{L_2(\Omega) \to L_2 (\Omega) } \le
\beta_3^{1/2} \alpha_1^{1/2} \| g \|_{L_{\infty}}^{1/2} |\mathbf{k}|, \qquad |\mathbf{k}| \le \widehat{t}^{\,0},
\end{equation}
 (which follows from~\eqref{abstr_A(t)F(t)_est}, \eqref{abstr_C3}, and~\eqref{hatX_1_estmate}),
we have
\begin{multline*}
\| \widehat{F} (\mathbf{k}) \widehat{\mathcal{A}}(\mathbf{k}) \widehat{F} (\mathbf{k}_0)  - \widehat{F} (\mathbf{k}) \widehat{\mathcal{A}}(\mathbf{k}_0) \widehat{F} (\mathbf{k}_0) \|_{L_2(\Omega) \to L_2 (\Omega) }
\\
\le  \beta_3^{1/2} \alpha_1 \| g \|_{L_{\infty}}   (|\mathbf{k}| + |\mathbf{k}_0|) |\mathbf{k} - \mathbf{k}_0|, \quad |\mathbf{k}|,|\mathbf{k}_0|  \le \widehat{t}^{\,0}.
\end{multline*}
Combining this with~\eqref{A(k)_nondegenerated_and_c_*}, we obtain
\begin{equation*}
\| \Omega_1(\mathbf{k}, \mathbf{k}_0) \|_{L_2(\Omega) \to L_2 (\Omega) } \le
\beta_3^{1/2}  \alpha_1 \| g \|_{L_{\infty}} \widehat{c}_*^{-1/2} |\mathbf{k} - \mathbf{k}_0|, \quad  0 <  |\mathbf{k}|,|\mathbf{k}_0|  \le \widehat{t}^{\,0}.
\end{equation*}
Now, from~\eqref{A(k)_nondegenerated_and_c_*}, \eqref{F(k) - F(k_0)}, and~\eqref{hat_sqrtA(k)F(k)_up_est} we deduce
\begin{equation*}
\| \Omega_2(\mathbf{k}, \mathbf{k}_0) \|_{L_2(\Omega) \to L_2 (\Omega) } \le
\widehat{C}' \beta_3 \alpha_1 \| g \|_{L_{\infty}} \widehat{c}_*^{-1/2} \widehat{t}^{\,0} |\mathbf{k} - \mathbf{k}_0|, \quad  0 <  |\mathbf{k}|,|\mathbf{k}_0|  \le \widehat{t}^{\,0}.
\end{equation*}
 The term $\Omega_3(\mathbf{k}, \mathbf{k}_0)$
satisfies the same estimate. Thus, estimate~\eqref{hatA^1/2 (k) - hatA^1/2 (k_0)} is proved in the case where $\mathbf{k}, \mathbf{k}_0 \ne 0$.
If, for instance, $\mathbf{k}_0 = 0$, then~\eqref{hatA^1/2 (k) - hatA^1/2 (k_0)} follows directly from~\eqref{hat_sqrtA(k)F(k)_up_est} and  the relation $\widehat{\mathcal{A}}(0)^{1/2} \widehat{F} (0) = 0$.

Let us prove estimate~\eqref{cos_hatA^1/2 (k) - cos_hatA^1/2 (k_0)}. We have
\begin{multline}
\label{cos_k_k0_f1}
e^{i \tau \widehat{\mathcal{A}}(\mathbf{k})^{1/2}} \widehat{F} (\mathbf{k}) - e^{i\tau \widehat{\mathcal{A}}(\mathbf{k}_0)^{1/2}} \widehat{F} (\mathbf{k}_0) = e^{i \tau \widehat{\mathcal{A}}(\mathbf{k})^{1/2}} \widehat{F} (\mathbf{k}) (\widehat{F} (\mathbf{k}) - \widehat{F} (\mathbf{k}_0)) + \\ + (\widehat{F} (\mathbf{k}) - \widehat{F} (\mathbf{k}_0)) e^{i \tau \widehat{\mathcal{A}}(\mathbf{k}_0)^{1/2}} \widehat{F} (\mathbf{k}_0) + \Xi (\tau, \mathbf{k}, \mathbf{k}_0),
\end{multline}
where
\begin{equation*}
\Xi (\tau, \mathbf{k}, \mathbf{k}_0) = e^{i \tau \widehat{\mathcal{A}}(\mathbf{k})^{1/2}} \widehat{F} (\mathbf{k}) \widehat{F} (\mathbf{k}_0) - \widehat{F} (\mathbf{k}) e^{i\tau \widehat{\mathcal{A}}(\mathbf{k}_0)^{1/2}} \widehat{F} (\mathbf{k}_0).
\end{equation*}
The sum of the first two terms in~\eqref{cos_k_k0_f1} does not exceed $2 \widehat{C}' | \mathbf{k}- \mathbf{k}_0|$,
in view of~(\ref{F(k) - F(k_0)}). The third term can be written as
\begin{align*}
\Xi (\tau, \mathbf{k}, \mathbf{k}_0) &= e^{i \tau \widehat{\mathcal{A}}(\mathbf{k})^{1/2}} \Sigma(\tau, \mathbf{k}, \mathbf{k}_0),\\
\Sigma(\tau, \mathbf{k}, \mathbf{k}_0) &= \widehat{F} (\mathbf{k}) \widehat{F} (\mathbf{k}_0) - e^{-i\tau \widehat{\mathcal{A}}(\mathbf{k})^{1/2}} \widehat{F} (\mathbf{k}) e^{i \tau \widehat{\mathcal{A}}(\mathbf{k}_0)^{1/2}} \widehat{F}(\mathbf{k}_0).
\end{align*}
Obviously, $\Sigma(0, \mathbf{k}, \mathbf{k}_0) = 0$, and
 $\Sigma'(\tau, \mathbf{k}, \mathbf{k}_0) = \frac{d \Sigma(\tau, \mathbf{k}, \mathbf{k}_0)}{d \tau}$ is given by
\begin{multline*}
\Sigma'(\tau, \mathbf{k}, \mathbf{k}_0)
\\
= i \widehat{F} (\mathbf{k}) e^{-i\tau \widehat{\mathcal{A}}(\mathbf{k})^{1/2}} (\widehat{\mathcal{A}}(\mathbf{k})^{1/2} \widehat{F} (\mathbf{k}) - \widehat{\mathcal{A}}(\mathbf{k_0})^{1/2} \widehat{F} (\mathbf{k_0}))
e^{i\tau \widehat{\mathcal{A}}(\mathbf{k_0})^{1/2}} \widehat{F} (\mathbf{k_0}).
\end{multline*}
Integrating over the interval $[0, \tau]$ and taking~\eqref{hatA^1/2 (k) - hatA^1/2 (k_0)} into account, we obtain
\begin{equation*}
\| \Xi (\tau, \mathbf{k}, \mathbf{k}_0) \|_{L_2(\Omega) \to L_2 (\Omega) } = \|\Sigma(\tau, \mathbf{k}, \mathbf{k}_0)\|_{L_2(\Omega) \to L_2 (\Omega) } \le \widehat{C}'' |\tau| | \mathbf{k} - \mathbf{k}_0|.
\end{equation*}
We arrive at~(\ref{cos_hatA^1/2 (k) - cos_hatA^1/2 (k_0)}).
\end{proof}

\noindent \textbf{Proof of Theorem~\ref{hat_s<2_thrm}}. \, It suffices to assume that $1 \le s < 2 $.
We prove by contradiction.  Fix $\tau \ne 0$.
Assume that for some $1 \le s < 2$ there exists a constant $\mathcal{C}(\tau) > 0$  such that estimate~\eqref{hat_s<2_est_imp}
holds for almost every $\mathbf{k} \in \widetilde{\Omega}$ and sufficiently small $\varepsilon > 0$.
By~\eqref{R_P} and~\eqref{cos_main_est_(I - hatP)},   it follows that there exists a constant  $\widetilde{\mathcal{C}}(\tau) > 0$
such that
\begin{multline}
\label{hat_s<2_proof_f1}
\| ( \cos (\varepsilon^{-1} \tau \widehat{\mathcal{A}}(\mathbf{k})^{1/2})  - \cos (\varepsilon^{-1} \tau \widehat{\mathcal{A}}^0(\mathbf{k})^{1/2})) \widehat{P}\|_{L_2(\Omega) \to L_2 (\Omega) } \varepsilon^s (|\mathbf{k}|^2 + \varepsilon^2)^{-s/2}
\\
 \le \widetilde{\mathcal{C}}(\tau) \varepsilon
\end{multline}
for almost every $\mathbf{k} \in \widetilde{\Omega}$ and sufficiently small $\varepsilon > 0$.

Now, let $|\mathbf{k}| \le \widehat{t}^{\,0}$. By \eqref{abstr_F(t)_threshold_1},
\begin{equation}
\label{hat_s<2_proof_f2}
\| \widehat{F} (\mathbf{k}) - \widehat{P} \|_{L_2 (\Omega) \to L_2 (\Omega)} \le \widehat{C}_1 |\mathbf{k}|, \quad |\mathbf{k}| \le \widehat{t}^{\,0}.
\end{equation}
From~\eqref{hat_s<2_proof_f1} and~\eqref{hat_s<2_proof_f2}
 it follows that there exists a constant $\check{\mathcal C}(\tau)>0$ such that
\begin{multline}
\label{hat_s<2_proof_f3}
\|  \cos (\varepsilon^{-1} \tau \widehat{\mathcal{A}}(\mathbf{k})^{1/2}) \widehat{F} (\mathbf{k})  - \cos (\varepsilon^{-1} \tau \widehat{\mathcal{A}}^0(\mathbf{k})^{1/2}) \widehat{P}\|_{L_2(\Omega) \to L_2 (\Omega) } \varepsilon^s (|\mathbf{k}|^2 + \varepsilon^2)^{-s/2}
\\
\le \check{\mathcal{C}}(\tau) \varepsilon
\end{multline}
for almost every $\mathbf{k}$ in the ball $|\mathbf{k}| \le \widehat{t}^{\,0}$
 and sufficiently small $\varepsilon > 0$.

Observe that  $\widehat{P}$ is the spectral projection of the operator $\widehat{\mathcal{A}}^0 (\mathbf{k})$ for the interval $[0, \widehat{\delta}]$. Applying Lemma~\ref{Lipschitz_lemma} to $\widehat{\mathcal{A}} (\mathbf{k})$ and $\widehat{\mathcal{A}}^0 (\mathbf{k})$,
 we conclude that for fixed $\tau$ and $\varepsilon$ the operator under the norm sign in~\eqref{hat_s<2_proof_f3} is continuous with respect to $\mathbf{k}$ in the ball $|\mathbf{k}| \le \widehat{t}^{\,0}$. Consequently, estimate~\eqref{hat_s<2_proof_f3} holds for
all $\mathbf{k}$ in that ball. In particular, it holds at the point $\mathbf{k} = t\boldsymbol{\theta}_0$ if $t \le \widehat{t}^{\,0}$. Applying~\eqref{hat_s<2_proof_f2} once more, we see that there exists a constant  $\check{\mathcal{C}}'(\tau)>0$ such that
\begin{multline}
\label{hat_s<2_proof_f4}
\| ( \cos (\varepsilon^{-1} \tau \widehat{\mathcal{A}}(t\boldsymbol{\theta}_0)^{1/2}) - \cos (\varepsilon^{-1} \tau \widehat{\mathcal{A}}^0(t\boldsymbol{\theta}_0)^{1/2})) \widehat{P}\|_{L_2(\Omega) \to L_2 (\Omega) } \varepsilon^s (t^2 + \varepsilon^2)^{-s/2}
 \\
\le \check{\mathcal{C}}'(\tau) \varepsilon
\end{multline}
for all $t \le \widehat{t}^{\,0}$ and sufficiently small $\varepsilon$.

Estimate~\eqref{hat_s<2_proof_f4} corresponds to the abstract estimate~\eqref{abstr_s<2_est_imp}.
Since $\widehat{N}_0 (\boldsymbol{\theta}_0) \ne 0$, applying Theorem~\ref{abstr_s<2_general_thrm}, we
arrive at a contradiction.   $\square$

\section{The operator $\mathcal{A} (\mathbf{k})$. Application of the scheme of Section~\ref{abstr_sandwiched_section}}

\subsection{The operator $\mathcal{A} (\mathbf{k})$}

We apply the scheme of Section~\ref{abstr_sandwiched_section} to study the operator
$\mathcal{A} (\mathbf{k}) = f^* \widehat{\mathcal{A}} (\mathbf{k}) f$.
Now  $\mathfrak{H} = \widehat{\mathfrak{H}} = L_2 (\Omega; \mathbb{C}^n)$, $\mathfrak{H}_* = L_2 (\Omega; \mathbb{C}^m)$, the role of  $A(t)$ is played by $A(t, \boldsymbol{\theta}) = \mathcal{A}(\mathbf{k})$, the role of $\widehat{A}(t)$ is played by $\widehat{A}(t, \boldsymbol{\theta}) = \widehat{\mathcal{A}}(\mathbf{k})$. Next, the  isomorphism $M$
is the operator of multiplication by the matrix-valued function $f(\mathbf{x})$.
The operator $Q$ is the operator of multiplication by the matrix-valued function
\begin{equation*}
Q(\mathbf{x}) = (f (\mathbf{x}) f (\mathbf{x})^*)^{-1}.
\end{equation*}

The block of the operator $Q$ in the subspace $\widehat{\mathfrak{N}}$ (see~\eqref{Ker3}) is the operator of multiplication by the constant matrix
\begin{equation*}
\overline{Q} = (\underline{f f^*})^{-1} = |\Omega|^{-1} \int_{\Omega} (f (\mathbf{x}) f (\mathbf{x})^*)^{-1} d \mathbf{x}.
\end{equation*}
Next, $M_0$ is the operator of multiplication by the constant matrix
\begin{equation}
\label{f_0}
f_0 = (\overline{Q})^{-1/2} = (\underline{f f^*})^{1/2}.
\end{equation}
Note that
\begin{equation}
\label{f0}
| f_0 | \le \| f \|_{L_{\infty}}, \qquad | f_0^{-1} | \le \| f^{-1} \|_{L_{\infty}}.
\end{equation}

In $L_2 (\mathbb{R}^d; \mathbb{C}^n)$,  we define the operator
\begin{equation}
\label{A0}
\mathcal{A}^0 := f_0 \widehat{\mathcal{A}}^0 f_0 = f_0 b(\mathbf{D})^* g^0 b(\mathbf{D}) f_0.
\end{equation}
Let $\mathcal{A}^0 (\mathbf{k})$~be the corresponding family of operators in $L_2 (\Omega; \mathbb{C}^n)$. Then
\begin{equation*}
\mathcal{A}^0 (\mathbf{k}) = f_0 \widehat{\mathcal{A}}^0 (\mathbf{k}) f_0.
\end{equation*}
By~\eqref{Ker3} and~\eqref{hatS_P=hatA^0_P}, we have
\begin{equation}
\label{f_0 hatS f_0 P = A^0}
f_0 \widehat{S} (\mathbf{k}) f_0 \widehat{P} = \mathcal{A}^0 (\mathbf{k}) \widehat{P}.
\end{equation}

\subsection{The analytic branches of eigenvalues and eigenvectors}
\label{sndw_eigenvalues_and_eigenvectors_section}
According to~\eqref{abstr_S_and_S_hat_relation},  the spectral germ $S(\boldsymbol{\theta})$ of the operator $A (t, \boldsymbol{\theta})$  acting in the subspace $\mathfrak{N}$ (see \eqref{Ker2}) is represented as
\begin{equation*}
S(\boldsymbol{\theta}) = P f^* b(\boldsymbol{\theta})^* g^0 b(\boldsymbol{\theta}) f|_{\mathfrak{N}},
\end{equation*}
where $P$~is the orthogonal projection of $L_2 (\Omega; \mathbb{C}^n)$ onto $\mathfrak{N}$.

The analytic (in $t$) branches of the eigenvalues
$\lambda_l (t, \boldsymbol{\theta})$ and the branches of the eigenvectors $\varphi_l (t, \boldsymbol{\theta})$ of $A (t, \boldsymbol{\theta})$
 admit the power series expansions of the form~\eqref{abstr_A(t)_eigenvalues_series}, \eqref{abstr_A(t)_eigenvectors_series}
 with the coefficients depending on $\boldsymbol{\theta}$:
\begin{gather}
\label{A_eigenvalues_series}
\lambda_l (t, \boldsymbol{\theta}) = \gamma_l (\boldsymbol{\theta}) t^2 + \mu_l (\boldsymbol{\theta}) t^3 + \ldots, \qquad l = 1, \ldots, n,
\\
\label{A_eigenvectors_series}
\varphi_l (t, \boldsymbol{\theta}) = \omega_l (\boldsymbol{\theta}) + t \psi^{(1)}_l (\boldsymbol{\theta}) + \ldots, \qquad l = 1, \ldots, n.
\end{gather}
The vectors $\omega_1 (\boldsymbol{\theta}), \ldots, \omega_n (\boldsymbol{\theta})$ form an orthonormal basis in the subspace~$\mathfrak{N}$
(see \eqref{Ker2}), and the vectors
\begin{equation*}
\zeta_l (\boldsymbol{\theta}) = f \omega_l (\boldsymbol{\theta}), \qquad l = 1, \ldots, n,
\end{equation*}
form a basis in $\widehat{\mathfrak{N}}$~(see~\eqref{Ker3})  orthonormal with the weight $\overline{Q}$, i.~e.,
$(\overline{Q} \zeta_l (\boldsymbol{\theta}), \zeta_j (\boldsymbol{\theta})) = \delta_{jl}, \; j, l = 1, \ldots,n$.

The numbers $\gamma_l (\boldsymbol{\theta})$ and the elements $\omega_l (\boldsymbol{\theta})$
are eigenvalues and eigenvectors of the spectral germ $S(\boldsymbol{\theta})$.
However, it is more convenient to turn to the generalized spectral problem for $\widehat{S}(\boldsymbol{\theta})$. According to~\eqref{abstr_hatS_gener_spec_problem}, the numbers $\gamma_l (\boldsymbol{\theta})$ and the elements $\zeta_l (\boldsymbol{\theta})$
 are eigenvalues and eigenvectors of the following generalized spectral problem:
\begin{equation}
\label{hatS_gener_spec_problem}
b(\boldsymbol{\theta})^* g^0 b(\boldsymbol{\theta}) \zeta_l (\boldsymbol{\theta}) = \gamma_l (\boldsymbol{\theta}) \overline{Q} \zeta_l (\boldsymbol{\theta}), \qquad l = 1, \ldots, n.
\end{equation}

\subsection{The operator $\widehat{N}_Q (\boldsymbol{\theta})$}
We need to describe the operator $\widehat{N}_Q$ (in abstract terms it was defined in Subsection~\ref{abstr_hatZ_Q_and_hatN_Q_section}).
Let $\Lambda_Q(\mathbf{x})$ be the $\Gamma$-periodic solution of the problem
\begin{equation}
\label{equation_for_Lambda_Q}
b(\mathbf{D})^* g(\mathbf{x}) (b(\mathbf{D}) \Lambda_Q(\mathbf{x}) + \mathbf{1}_m) = 0, \qquad \int_{\Omega} Q(\mathbf{x}) \Lambda_Q(\mathbf{x}) \, d \mathbf{x} = 0.
\end{equation}
Clearly, $\Lambda_Q(\mathbf{x})$ differs from the periodic solution  $\Lambda(\mathbf{x})$ of the problem~\eqref{equation_for_Lambda}
by a constant summand:
\begin{equation}
\label{Lambda_Q=Lambda+Lambda_Q^0}
\Lambda_Q(\mathbf{x}) = \Lambda(\mathbf{x}) + \Lambda_Q^0, \qquad \Lambda_Q^0 = -(\overline{Q})^{-1} (\overline{Q \Lambda})
\end{equation}

As shown in~\cite[Section~5]{BSu3}, the operator $\widehat{N}_Q (\boldsymbol{\theta})$ takes the form
\begin{equation}
\label{N_Q(theta)}
\widehat{N}_Q (\boldsymbol{\theta}) = b(\boldsymbol{\theta})^* L_Q (\boldsymbol{\theta}) b(\boldsymbol{\theta}) \widehat{P},
\end{equation}
where $L_Q (\boldsymbol{\theta})$~is an ($m \times m$)-matrix given by
\begin{equation}
\label{L_Q(theta)}
L_Q (\boldsymbol{\theta}) = | \Omega |^{-1} \int_{\Omega} (\Lambda_Q(\mathbf{x})^*b(\boldsymbol{\theta})^* \widetilde{g} (\mathbf{x}) + \widetilde{g} (\mathbf{x})^* b(\boldsymbol{\theta}) \Lambda_Q(\mathbf{x}))\, d \mathbf{x}.
\end{equation}
Combining~\eqref{Lambda_Q=Lambda+Lambda_Q^0}, \eqref{L_Q(theta)}, and~\eqref{g0}, \eqref{L(theta)},
we see that
\begin{equation*}
L_Q (\boldsymbol{\theta}) = L (\boldsymbol{\theta}) + L_Q^0 (\boldsymbol{\theta}), \qquad L_Q^0 (\boldsymbol{\theta}) = (\Lambda_Q^0)^* b(\boldsymbol{\theta})^* g^0 + g^0 b(\boldsymbol{\theta}) \Lambda_Q^0.
\end{equation*}

Observe that $L_Q (\mathbf{k}) := t L_Q (\boldsymbol{\theta}), \mathbf{k} \in \mathbb{R}^d$,  is  a Hermitian matrix-valued function first order homogeneous in $\mathbf{k}$. We put $\widehat{N}_Q (\mathbf{k}) := t^3 \widehat{N}_Q (\boldsymbol{\theta}), \mathbf{k} \in \mathbb{R}^d$. Then $\widehat{N}_Q (\mathbf{k}) = b(\mathbf{k})^* L_Q (\mathbf{k}) b(\mathbf{k})\widehat{P}$.
The matrix-valued function $b(\mathbf{k})^* L_Q (\mathbf{k}) b(\mathbf{k})$ is a homogeneous polynomial of the third degree  in~$\mathbf{k} \in \mathbb{R}^d$. It follows that either $\widehat{N}_Q (\boldsymbol{\theta}) = 0$ for any $\boldsymbol{\theta} \in \mathbb{S}^{d-1}$,
or $\widehat{N}_Q (\boldsymbol{\theta}) \ne 0$ for \textquotedblleft most\textquotedblright \ points $\boldsymbol{\theta}$
(except for the zeroes of this polynomial).

Some cases where the operator~\eqref{N_Q(theta)} is equal to zero were distinguished
in~\cite[Section~5]{BSu3}.

\begin{proposition}
    \label{N_Q=0_proposit}
   Suppose that at least one of the following conditions is satisfied:

$1^\circ$.
The operator $\mathcal{A}$ has the form $\mathcal{A} = f(\mathbf{x})^*\mathbf{D}^* g(\mathbf{x}) \mathbf{D}f(\mathbf{x})$, where $g(\mathbf{x})$~ is a symmetric matrix with real entries.

$2^\circ$. Relations~\emph{\eqref{g0=overline_g_relat}} are satisfied, i.~e., $g^0 = \overline{g}$.

\noindent
    Then $\widehat{N}_Q (\boldsymbol{\theta}) = 0$ for any $\boldsymbol{\theta} \in \mathbb{S}^{d-1}$.
\end{proposition}

Recall that (see Subsection~\ref{abstr_hatZ_Q_and_hatN_Q_section}) $\widehat{N}_Q (\boldsymbol{\theta}) = \widehat{N}_{0, Q} (\boldsymbol{\theta}) + \widehat{N}_{*,Q} (\boldsymbol{\theta})$. By~\eqref{abstr_hatN_0Q_N_*Q},
\begin{equation*}
\widehat{N}_{0, Q} (\boldsymbol{\theta}) = \sum_{l=1}^{n} \mu_l (\boldsymbol{\theta}) (\cdot, \overline{Q} \zeta_l(\boldsymbol{\theta}))_{L_2(\Omega)} \overline{Q} \zeta_l(\boldsymbol{\theta}).
\end{equation*}
We have
\begin{equation}
\label{hatN_0Q(theta)_matrix_elem}
(\widehat{N}_Q (\boldsymbol{\theta}) \zeta_l (\boldsymbol{\theta}), \zeta_l (\boldsymbol{\theta}))_{L_2 (\Omega)} = (\widehat{N}_{0,Q} (\boldsymbol{\theta}) \zeta_l (\boldsymbol{\theta}), \zeta_l (\boldsymbol{\theta}))_{L_2 (\Omega)} = \mu_l (\boldsymbol{\theta}), \  l=1, \ldots, n.
\end{equation}

Now, suppose that the matrices $b(\boldsymbol{\theta})$, $g (\mathbf{x})$, and $Q(\mathbf{x})$~have \emph{real entries}.
Then the matrix $\Lambda_Q (\mathbf{x})$ (see~\eqref{equation_for_Lambda_Q}) has purely imaginary entries, and
$\widetilde{g} (\mathbf{x})$ and $g^0$~have real entries. In this case $L_Q (\boldsymbol{\theta})$ (see~\eqref{L_Q(theta)}) and  $b(\boldsymbol{\theta})^* L_Q (\boldsymbol{\theta}) b(\boldsymbol{\theta})$~are Hermitian matrices with purely imaginary entries. If the analytic branches of the eigenvalues  $\lambda_l (t, \boldsymbol{\theta})$ and the analytic branches of the eigenvectors $\varphi_l (t, \boldsymbol{\theta})$ of the operator $A (t, \boldsymbol{\theta})$ can be chosen so that the vectors $\zeta_l (\boldsymbol{\theta}) = f \omega_l (\boldsymbol{\theta}), \; l = 1, \ldots,n,$ are real, then, by~\eqref{hatN_0Q(theta)_matrix_elem}, we have $\mu_l (\boldsymbol{\theta}) = 0, \; l = 1, \ldots, n$, i.~e., $\widehat{N}_{0,Q} (\boldsymbol{\theta}) = 0$. We arrive at the following statement.

\begin{proposition}
    Suppose that the matrices $b(\boldsymbol{\theta})$, $g (\mathbf{x})$, and $Q(\mathbf{x})$~have real entries. Suppose that in the expansions~\emph{\eqref{A_eigenvectors_series}}  the \textquotedblleft embrios\textquotedblright \ $\omega_l (\boldsymbol{\theta}), \; l = 1, \ldots, n,$ can be chosen so that the vectors  $\zeta_l (\boldsymbol{\theta}) = f \omega_l (\boldsymbol{\theta})$ are real. Then in~\emph{\eqref{A_eigenvalues_series}} we have $\mu_l (\boldsymbol{\theta}) = 0$,  $l=1, \ldots, n,$
i.~e., $\widehat{N}_{0,Q} (\boldsymbol{\theta}) = 0$ for any $\boldsymbol{\theta} \in \mathbb{S}^{d-1}$.
\end{proposition}

In the \textquotedblleft real\textquotedblright \ case under consideration, the operator $\widehat{S} (\boldsymbol{\theta})$  is a symmetric matrix with real entries; $\overline{Q}$ is also a symmetric matrix with real entries.
 Clearly,  if the eigenvalue $\gamma_j (\boldsymbol{\theta})$ of the generalized problem~\eqref{hatS_gener_spec_problem}
is simple, then the eigenvector $\zeta_j (\boldsymbol{\theta}) = f \omega_j (\boldsymbol{\theta})$
 is defined uniquely up to a phase factor, and we always can choose it to be real.  We arrive at the following corollary.

\begin{corollary}
    \label{sndw_simple_spec_N0Q=0_cor}
Suppose that   the matrices $b(\boldsymbol{\theta})$, $g (\mathbf{x})$, and $Q(\mathbf{x})$~ have real entries.
Suppose that the spectrum of the generalized spectral problem~\emph{\eqref{hatS_gener_spec_problem}}
is simple. Then $\widehat{N}_{0,Q} (\boldsymbol{\theta}) = 0$ for any $\boldsymbol{\theta} \in \mathbb{S}^{d-1}$.
\end{corollary}

\subsection{Multiplicities of the eigenvalues of the germ}
\label{sndw_eigenval_multipl_section}

This subsection concerns the case where $n \ge 2$. We return to the notation of Section~\ref{abstr_cluster_section},  tracing the multiplicities of the eigenvalues of the spectral germ $S (\boldsymbol{\theta})$. As has been mentioned in Subsection~\ref{sndw_eigenvalues_and_eigenvectors_section}, 
these eigenvalues are also the eigenvalues of the generalized problem~\eqref{hatS_gener_spec_problem}.
In general, the number $p(\boldsymbol{\theta})$ of different eigenvalues $\gamma^{\circ}_1 (\boldsymbol{\theta}), \ldots, \gamma^{\circ}_{p(\boldsymbol{\theta})} (\boldsymbol{\theta})$ of this problem  and their multiplicities $k_1 (\boldsymbol{\theta}), \ldots, k_{p(\boldsymbol{\theta})} (\boldsymbol{\theta})$ depend on the parameter $\boldsymbol{\theta} \in \mathbb{S}^{d-1}$.  For a fixed
$\boldsymbol{\theta}$, let $\mathfrak{N}_j (\boldsymbol{\theta})$ be the eigenspace  of the germ $S (\boldsymbol{\theta})$ corresponding to the eigenvalue $\gamma^{\circ}_j (\boldsymbol{\theta})$. Then $f \mathfrak{N}_j (\boldsymbol{\theta})$~ is the eigenspace of the problem~\eqref{hatS_gener_spec_problem}  corresponding to the same eigenvalue $\gamma^{\circ}_j (\boldsymbol{\theta})$.
Let $\mathcal{P}_j (\boldsymbol{\theta})$ denote the \textquotedblleft skew\textquotedblright \  projection of $L_2(\Omega; \mathbb{C}^n)$
onto the subspace $f \mathfrak{N}_j (\boldsymbol{\theta})$; $\mathcal{P}_j (\boldsymbol{\theta})$  is orthogonal with respect to the inner product with the weight $\overline{Q}$. Then, by~\eqref{abstr_hatN_0Q_N_*Q_invar_repr},  we have the following invariant representations for the operators
$\widehat{N}_{0,Q} (\boldsymbol{\theta})$ and $\widehat{N}_{*,Q} (\boldsymbol{\theta})$:
\begin{gather}
\label{N0Q_invar_repr}
\widehat{N}_{0,Q} (\boldsymbol{\theta}) = \sum_{j =1}^{p(\boldsymbol{\theta})} \mathcal{P}_j (\boldsymbol{\theta})^* \widehat{N}_Q (\boldsymbol{\theta}) \mathcal{P}_j (\boldsymbol{\theta}), \\
\notag
\widehat{N}_{*,Q} (\boldsymbol{\theta}) = \sum_{\substack{1 \le j, l \le p(\boldsymbol{\theta}):\\ j \ne l}} \mathcal{P}_j (\boldsymbol{\theta})^* \widehat{N}_Q (\boldsymbol{\theta}) \mathcal{P}_l (\boldsymbol{\theta}).
\end{gather}

\section{ Approximation of the sandwiched operator $\cos(\varepsilon^{-1} \tau \mathcal{A}(\mathbf{k}))$}

\subsection{The general case}

We apply theorems of Section~\ref{abstr_sandwiched_section} to the operator $\mathcal{A}(\mathbf{k})$.
 First we apply Theorem~\ref{abstr_cos_sandwiched_general_thrm}.
Taking~\eqref{c*}, \eqref{delta_fixation}, and~\eqref{X_1_estimate} into account,
 instead of the precise values of the constants~\eqref{abstr_C1_C2} and \eqref{abstr_C7}
which now depend on $\boldsymbol{\theta}$, we take the larger values
\begin{align}
    \label{C1}
    & C_1 = 2 \beta_1 r_0^{-1} \alpha_1^{1/2} \alpha_0^{-1/2} \|g\|_{L_\infty}^{1/2} \|g^{-1}\|_{L_\infty}^{1/2}\|f\|_{L_\infty} \|f^{-1}\|_{L_\infty},\\
    \label{C7}
 &\begin{aligned}
 C_7 = 2 \beta_7 r_0^{-1} \alpha_1 &\alpha_0^{-1/2} \|g\|_{L_\infty} \|g^{-1}\|_{L_\infty}^{1/2} \|f\|_{L_\infty}^2 \|f^{-1}\|_{L_\infty} +\\ &+ \beta_2 r_0^{-1} \alpha_1^{3/2} \alpha_0^{-1} \|g\|_{L_\infty}^{3/2} \|g^{-1}\|_{L_\infty} \|f\|_{L_\infty}^3 \|f^{-1}\|_{L_\infty}^2.
\end{aligned}
\end{align}

Denote
\begin{equation}
\label{JJ}
{\mathcal J}(\mathbf{k},\tau):= f \cos (\tau \mathcal{A}(\mathbf{k})^{1/2}) f^{-1} -
 f_0\cos (\tau \mathcal{A}^0(\mathbf{k})^{1/2})f_0^{-1}.
\end{equation}
Applying~\eqref{abstr_cos_sandwiched_general_est} and using~\eqref{R_P} and~\eqref{f_0 hatS f_0 P = A^0}, we obtain:
\begin{multline}
    \label{sndw_cos_main_est_k<hat_t^0}
    \| {\mathcal J}(\mathbf{k},\varepsilon^{-1}\tau) \mathcal{R}(\mathbf{k}, \varepsilon)  \widehat{P}
 \|_{L_2(\Omega) \to L_2 (\Omega) }
 \le \|f\|_{L_\infty}^2 \|f^{-1}\|_{L_\infty}^2 (C_1 + C_7 |\tau|) \varepsilon,
\\
    \tau \in \mathbb{R}, \; \varepsilon > 0, \; |\mathbf{k}| \le t^0.
\end{multline}

For $| \mathbf{k} | > t^0$ estimates are trivial. Similarly to~\eqref{R_hatP_est}, taking \eqref{f0} into account, we have
\begin{multline}
    \label{sndw_cos_main_est_k>hat_t^0}
    \|  {\mathcal J}(\mathbf{k},\varepsilon^{-1}\tau)\mathcal{R}(\mathbf{k}, \varepsilon)^{1/2}  \widehat{P}
 \|_{L_2(\Omega) \to L_2 (\Omega) }  \le 2 \|f\|_{L_\infty} \|f^{-1}\|_{L_\infty} (t^0)^{-1} \varepsilon,
\\     \tau \in \mathbb{R}, \; \varepsilon > 0, \; \mathbf{k} \in \widetilde{\Omega}, \; |\mathbf{k}| > t^0.
\end{multline}

Next, by~\eqref{R(k,eps)(I-P)_est},
\begin{multline}
\label{sndw_cos_main_est_(I - hatP)}
\|  {\mathcal J}(\mathbf{k},\varepsilon^{-1}\tau) \mathcal{R}(\mathbf{k}, \varepsilon)^{1/2}  (I - \widehat{P}) \|_{L_2(\Omega) \to L_2 (\Omega) }  \le 2 \|f\|_{L_\infty} \|f^{-1}\|_{L_\infty} r_0^{-1} \varepsilon,
\\
\tau \in \mathbb{R}, \; \varepsilon > 0, \; \mathbf{k} \in \widetilde{\Omega}.
\end{multline}

Finally, relations~\eqref{sndw_cos_main_est_k<hat_t^0}--\eqref{sndw_cos_main_est_(I - hatP)} (and expressions for the constants) imply the following result (which has been proved before in \cite[Theorem~8.2]{BSu5}).

\begin{theorem}
    \label{sndw_cos_general_thrm}
Let  ${\mathcal J}(\mathbf{k},\tau)$ be the operator defined by \emph{\eqref{JJ}}.
 For $\tau \in \mathbb{R}$, $\varepsilon > 0$, and $\mathbf{k} \in \widetilde{\Omega}$ we have
    \begin{equation*}
    \|  {\mathcal J}(\mathbf{k},\varepsilon^{-1}\tau)
\mathcal{R}(\mathbf{k}, \varepsilon)\|_{L_2(\Omega) \to L_2 (\Omega) }  \le (\mathcal{C}_1 + \mathcal{C}_2 |\tau|) \varepsilon,
    \end{equation*}
    where
    \begin{equation*}
    \begin{aligned}
    \mathcal{C}_1 &= \widetilde{\beta}_1 r_0^{-1} \alpha_1^{1/2} \alpha_0^{-1/2} \|g\|_{L_\infty}^{1/2} \|g^{-1}\|_{L_\infty}^{1/2} \|f\|_{L_\infty}^3 \|f^{-1}\|_{L_\infty}^3, \\
    \mathcal{C}_2 &= \|f\|_{L_\infty}^2\|f^{-1}\|_{L_\infty}^2 C_7,
    \end{aligned}
    \end{equation*}
    and $C_7$ is given by~\emph{\eqref{C7}}.
\end{theorem}

\subsection{The case where $\widehat{N}_Q(\boldsymbol{\theta}) = 0$}
Now we assume that $\widehat{N}_Q(\boldsymbol{\theta}) = 0$ for any $\boldsymbol{\theta} \in \mathbb{S}^{d-1}$.
Then, applying Theorem~\ref{abstr_cos_sandwiched_ench_thrm_1} and taking~\eqref{R_P} and~\eqref{f_0 hatS f_0 P = A^0}
into account, we have
\begin{multline*}
    \|  {\mathcal J}(\mathbf{k},\varepsilon^{-1}\tau)
\mathcal{R}(\mathbf{k}, \varepsilon)^{3/4} \widehat{P} \|_{L_2 (\Omega) \to L_2 (\Omega) } \le \|f\|_{L_\infty}^2 \|f^{-1}\|_{L_\infty}^2 ( C'_1 + C_{11} | \tau | ) \varepsilon,
\\
\tau \in \mathbb{R}, \; \varepsilon > 0, \; |\mathbf{k}| \le t^0.
\end{multline*}
Here $C'_1 = \max \{2, 2C_1\}$, and $C_{11}$ is given by
\begin{multline}
\label{C11}
C_{11} = 2^{1/2} \beta_1 r^{-2}_0 \alpha_1^{3/2} \alpha_0^{-1} \|g\|_{L_\infty}^{3/2}  \|g^{-1} \|_{L_\infty} \|f\|_{L_\infty}^{3}  \|f^{-1} \|_{L_\infty}^2
\\
+  4 \beta_{10} r^{-2}_0 \alpha_1^{2} \alpha_0^{-3/2} \|g\|_{L_\infty}^{2}  \|g^{-1}\|_{L_\infty}^{3/2} \|f\|_{L_\infty}^{4}  \|f^{-1} \|_{L_\infty}^3
\\
+ 4 \beta_{11} r^{-2}_0 \alpha_1^{3} \alpha_0^{-5/2} \|g\|_{L_\infty}^{3} \|g^{-1}\|_{L_\infty}^{5/2} \|f\|_{L_\infty}^{6}  \|f^{-1} \|_{L_\infty}^5 
\\
+ 4 \beta_{12} r^{-2}_0 \alpha_1^{4} \alpha_0^{-7/2} \|g\|_{L_\infty}^{4} \|g^{-1}\|_{L_\infty}^{7/2} \|f\|_{L_\infty}^{8}  \|f^{-1} \|_{L_\infty}^7.
\end{multline}

Together with~\eqref{sndw_cos_main_est_k>hat_t^0} and~\eqref{sndw_cos_main_est_(I - hatP)} this implies the following result.

\begin{theorem}
    \label{sndw_cos_enchanced_thrm_1}
Let  ${\mathcal J}(\mathbf{k},\tau)$ be the operator defined by \emph{\eqref{JJ}}.
 Let $\widehat{N}_Q(\boldsymbol{\theta})$ be the operator defined by~\emph{\eqref{N_Q(theta)}}. Suppose that  $\widehat{N}_Q(\boldsymbol{\theta})=0$ for any $\boldsymbol{\theta} \in \mathbb{S}^{d-1}$.
Then for $\tau \in \mathbb{R}$, $\varepsilon > 0$, and $\mathbf{k} \in \widetilde{\Omega}$ we have
    \begin{equation*}
        \|   {\mathcal J}(\mathbf{k},\varepsilon^{-1}\tau)
\mathcal{R}(\mathbf{k}, \varepsilon)^{3/4}\|_{L_2(\Omega) \to L_2 (\Omega) }  \le (\mathcal{C}_3 + \mathcal{C}_4 |\tau|) \varepsilon,
    \end{equation*}
  where
    \begin{align*}
    \mathcal{C}_3 &= 2 \|f\|_{L_\infty}  \|f^{-1} \|_{L_\infty}  \left( \max\{ \|f\|_{L_\infty}  \|f^{-1} \|_{L_\infty}, C_1 \|f\|_{L_\infty}  \|f^{-1} \|_{L_\infty} \}  + r_0^{-1}\right), \\
    \mathcal{C}_4 &= \|f\|_{L_\infty}^2  \|f^{-1} \|_{L_\infty}^2 C_{11}.
    \end{align*}
The constants $C_1$ and $C_{11}$ are given by~\emph{\eqref{C1}} and~\emph{\eqref{C11}}, respectively.
\end{theorem}

Recall that some sufficient conditions ensuring that $\widehat{N}_Q(\boldsymbol{\theta}) = 0$ for any
$\boldsymbol{\theta} \in \mathbb{S}^{d-1}$ are given in Proposition~\ref{N_Q=0_proposit}.

\subsection{The case where $\widehat{N}_{0,Q}(\boldsymbol{\theta}) = 0$}
Now we reject the assumption of Theorem~\ref{sndw_cos_enchanced_thrm_1}, but instead we assume that $\widehat{N}_{0,Q}(\boldsymbol{\theta}) = 0$ for any $\boldsymbol{\theta}$. We may also assume that $\widehat{N}(\boldsymbol{\theta}) = \widehat{N}_{*,Q}(\boldsymbol{\theta}) \ne 0$ for some $\boldsymbol{\theta}$,  and then for \textquotedblleft most\textquotedblright \ points $\boldsymbol{\theta}$.
(Otherwise, one can apply Theorem~\ref{sndw_cos_enchanced_thrm_1}.)
As in Subsection~\ref{ench_approx2_section}, in order to apply \textquotedblleft abstract\textquotedblright \ Theorem~\ref{abstr_cos_sandwiched_ench_thrm_2}, we have to impose an additional condition. We use the initial enumeration of the eigenvalues
$\gamma_1 (\boldsymbol{\theta}), \ldots , \gamma_n (\boldsymbol{\theta})$ of the germ $S (\boldsymbol{\theta})$
 (each eigenvalue is repeated corresponding to its multiplicity)
and enumerate them in the nondecreasing order:

\begin{equation}
\label{gamma(theta)}
    \gamma_1 (\boldsymbol{\theta}) \le \gamma_2 (\boldsymbol{\theta}) \le \ldots \le \gamma_n (\boldsymbol{\theta}).
\end{equation}
As has been already mentioned, the numbers~\eqref{gamma(theta)}  are also the eigenvalues of the generalized spectral problem~\eqref{hatS_gener_spec_problem}. For each $\boldsymbol{\theta}$, let $\mathcal{P}^{(k)} (\boldsymbol{\theta})$ be the
\textquotedblleft skew\textquotedblright \ projection (orthogonal with the weight $\overline{Q}$) of $L_2 (\Omega; \mathbb{C}^n)$
 onto the eigenspace of the problem~\eqref{hatS_gener_spec_problem} corresponding to the eigenvalue
 $\gamma_k (\boldsymbol{\theta})$. Clearly,  $\mathcal{P}^{(k)} (\boldsymbol{\theta})$ coincides with one of the projections
$\mathcal{P}_j (\boldsymbol{\theta})$   introduced in Subsection~\ref{sndw_eigenval_multipl_section} (but the number $j$ may depend on $\boldsymbol{\theta}$).

\begin{condition}
    \label{sndw_cond1}

$1^\circ$.
The operator $\widehat{N}_{0,Q}(\boldsymbol{\theta})$ defined by~\emph{\eqref{N0Q_invar_repr}} is equal to zero: $\widehat{N}_{0,Q}(\boldsymbol{\theta})=0$ for any $\boldsymbol{\theta} \in \mathbb{S}^{d-1}$.

 $2^\circ$. For any pair of indices $(k,r)$, $1\le k,r \le n$, $k\ne r$, such that
$\gamma_k (\boldsymbol{\theta}_0) = \gamma_r (\boldsymbol{\theta}_0) $ for some $\boldsymbol{\theta}_0 \in \mathbb{S}^{d-1}$, we have  $(\mathcal{P}^{(k)} (\boldsymbol{\theta}))^* \widehat{N}_Q (\boldsymbol{\theta}) \mathcal{P}^{(r)} (\boldsymbol{\theta}) = 0$ for any $\boldsymbol{\theta} \in \mathbb{S}^{d-1}$.
\end{condition}

   Condition~$2^\circ$ can be reformulated as follows: it is assumed that, for the \textquotedblleft blocks\textquotedblright \ $(\mathcal{P}^{(k)} (\boldsymbol{\theta}))^* \widehat{N}_Q (\boldsymbol{\theta}) \mathcal{P}^{(r)} (\boldsymbol{\theta})$ of the operator $\widehat{N}_Q (\boldsymbol{\theta})$ that are not identically zero,
    the corresponding branches of the eigenvalues $\gamma_k (\boldsymbol{\theta})$ and  $\gamma_r (\boldsymbol{\theta})$ do not intersect.

 Condition~\ref{sndw_cond1} is ensured by the following more restrictive condition.

\begin{condition}
    \label{sndw_cond2}

        $1^\circ$. The operator $\widehat{N}_{0,Q}(\boldsymbol{\theta})$ defined by~\emph{\eqref{N0Q_invar_repr}}
is equal to zero: $\widehat{N}_{0,Q}(\boldsymbol{\theta})=0$ при всех $\boldsymbol{\theta} \in \mathbb{S}^{d-1}$.

$2^\circ$. Suppose that the number $p$ of different eigenvalues of the generalized spectral problem~\emph{\eqref{hatS_gener_spec_problem}} does not depend on $\boldsymbol{\theta} \in \mathbb{S}^{d-1}$.
    \end{condition}

Under Condition~\ref{sndw_cond2},
denote the different eigenvalues of the germ enumerated in the increasing order by
$\gamma^{\circ}_1(\boldsymbol{\theta}), \ldots, \gamma^{\circ}_p(\boldsymbol{\theta})$.
Then their multiplicities $k_1, \ldots, k_p$ do not depend on $\boldsymbol{\theta} \in \mathbb{S}^{d-1}$.

\begin{remark}
    \label{sndw_simple_spec_remark}
    Assumption $2^\circ$ of Condition~\emph{\ref{sndw_cond2}} is a fortiori valid if the spectrum of the problem~\emph{\eqref{hatS_gener_spec_problem}} is simple for any $\boldsymbol{\theta} \in \mathbb{S}^{d-1}$.
\end{remark}

So, we assume that Condition~\ref{sndw_cond1}  is satisfied. We have to take care only about the pairs of indices from the set
\begin{equation*}
    \mathcal{K} := \{ (k,r) \colon 1 \le k,r \le n, \; k \ne r, \;  (\mathcal{P}^{(k)} (\boldsymbol{\theta}))^* \widehat{N}_Q (\boldsymbol{\theta}) \mathcal{P}^{(r)} (\boldsymbol{\theta}) \not\equiv 0 \}.
\end{equation*}
Denote
\begin{equation*}
    c^{\circ}_{kr} (\boldsymbol{\theta}) := \min \{c_*, n^{-1} |\gamma_k (\boldsymbol{\theta}) - \gamma_r (\boldsymbol{\theta})| \}, \quad (k,r) \in \mathcal{K}.
\end{equation*}
Since the operator-valued function $S (\boldsymbol{\theta})$  is continuous with respect to \hbox{$\boldsymbol{\theta} \in \mathbb{S}^{d-1}$},
then $\gamma_j (\boldsymbol{\theta})$~are continuous functions on the sphere $\mathbb{S}^{d-1}$.
By Condition~\ref{sndw_cond1}($2^\circ$), for $(k,r) \in \mathcal{K}$ we have~$|\gamma_k (\boldsymbol{\theta}) - \gamma_r (\boldsymbol{\theta})| > 0$ for any $\boldsymbol{\theta} \in \mathbb{S}^{d-1}$, whence
\begin{equation*}
    c^{\circ}_{kr} := \min_{\boldsymbol{\theta} \in \mathbb{S}^{d-1}} c^{\circ}_{kr} (\boldsymbol{\theta}) > 0, \quad (k,r) \in \mathcal{K}.
\end{equation*}

We put
\begin{equation}
    \label{c^circ}
    c^{\circ} := \min_{(k,r) \in \mathcal{K}} c^{\circ}_{kr}.
\end{equation}
Clearly, the number~\eqref{c^circ}~ plays the role of the number~\eqref{abstr_c^circ};
 it is important that, due to Condition~\ref{sndw_cond1}, we have chosen $c^{\circ}$ independently of $\boldsymbol{\theta}$.
Under Condition~\ref{sndw_cond1}, the number $t^{00}$ subject to~\eqref{abstr_t00}
 also can be chosen independently of $\boldsymbol{\theta} \in \mathbb{S}^{d-1}$.
Taking~\eqref{delta_fixation} and~\eqref{X_1_estimate} into account, we put
\begin{equation}
    \label{t00_fixation}
    t^{00} = (8 \beta_2)^{-1} r_0 \alpha_1^{-3/2} \alpha_0^{1/2} \| g\|_{L_{\infty}}^{-3/2} \| g^{-1}\|_{L_{\infty}}^{-1/2} \|f\|_{L_\infty}^{-3} \|f^{-1}\|_{L_\infty}^{-1} c^{\circ},
\end{equation}
where $c^{\circ}$ is given by~\eqref{c^circ}. (The condition $t^{00} \le t^{0}$ is valid automatically, since $c^{\circ} \le \| S (\boldsymbol{\theta}) \| \le \alpha_1 \|g\|_{L_{\infty}}\|f\|_{L_\infty}^2$.)

Under Condition~\ref{sndw_cond1}, we apply Theorem~\ref{abstr_cos_sandwiched_ench_thrm_2} and obtain
\begin{multline}
    \label{sndw_cos_enched_est_k<hat_t^00}
    \|    {\mathcal J}(\mathbf{k},\varepsilon^{-1}\tau)
\mathcal{R}(\mathbf{k}, \varepsilon)^{3/4} \widehat{P} \|_{L_2 (\Omega) \to L_2 (\Omega) }\le  \|f\|_{L_\infty}^2 \|f^{-1}\|_{L_\infty}^2 (C'_{15} + C_{16} | \tau |) \varepsilon,
\\
    \tau \in \mathbb{R}, \; \varepsilon > 0, \; |\mathbf{k}| \le t^{00},
\end{multline}
where
\begin{equation}
\label{sndw_C'_15_16}
\begin{split}
C'_{15} = \max\{2, 2 C_1 + C_{14}\}, \quad
C_{16} = C_{11} + C_{13}.
\end{split}
\end{equation}
The constants $C_{13}$ and $C_{14}$ are given by
\begin{align*}
   & \begin{array}{ll}
   C_{13} = 4 \beta_{13} n^2   \alpha_0^{-3/2} \alpha_1^{4}  \|g\|_{L_\infty}^{4}  \|g^{-1} \|_{L_\infty}^{3/2}  \|f\|_{L_\infty}^{8} \|f^{-1}\|_{L_\infty}^{3} r^{-2}_0 (c^{\circ})^{-2},
   \end{array}
   \\
   & \begin{array}{ll}
   C_{14} = 2 \beta_{14} n^2  \alpha_0^{-1} \alpha_1^{2}  \|g\|_{L_\infty}^{2}
\|g^{-1}\|_{L_\infty} \|f\|_{L_\infty}^{4} \|f^{-1}\|_{L_\infty}^{2} r_0^{-1}  ( c^{\circ})^{-1},
   \end{array}
\end{align*}
the constants $C_{1}$ and $C_{11}$ are defined by~\eqref{C1}, \eqref{C11}.

Similarly to~\eqref{sndw_cos_main_est_k>hat_t^0}, we have
\begin{multline}
    \label{sndw_cos_enched_est_k>hat_t^00}
  \|  {\mathcal J}(\mathbf{k},\varepsilon^{-1}\tau)
 \mathcal{R}(\mathbf{k}, \varepsilon)^{1/2}  \widehat{P} \|_{L_2(\Omega) \to L_2 (\Omega) }  \le  2 \|f\|_{L_\infty} \|f^{-1}\|_{L_\infty} (t^{00})^{-1} \varepsilon,
\\
    \tau \in \mathbb{R}, \; \varepsilon > 0, \; \mathbf{k} \in \widetilde{\Omega}, \; |\mathbf{k}| > t^{00}.
\end{multline}

Now relations~\eqref{sndw_cos_main_est_(I - hatP)}, \eqref{sndw_cos_enched_est_k<hat_t^00}, and \eqref{sndw_cos_enched_est_k>hat_t^00}  directly imply the following result.

\begin{theorem}
    \label{sndw_cos_enchanced_thrm_2}
Let  ${\mathcal J}(\mathbf{k},\tau)$ be the operator defined by \emph{\eqref{JJ}}.
   Suppose that Condition~\emph{\ref{sndw_cond1}} \emph{(}or more restrictive Condition~\emph{\ref{sndw_cond2}}\emph{)}
is satisfied. Then for $\tau \in \mathbb{R}$, $\varepsilon > 0$, and $\mathbf{k} \in \widetilde{\Omega}$ we have
    \begin{equation}
        \label{sndw_cos_enchanced_est_2}
    \|  {\mathcal J}(\mathbf{k},\varepsilon^{-1}\tau)
 \mathcal{R}(\mathbf{k}, \varepsilon)^{3/4}\|_{L_2(\Omega) \to L_2 (\Omega) }  \le (\mathcal{C}_5 + \mathcal{C}_6 |\tau|) \varepsilon,
    \end{equation}
    where
    \begin{align*}
    \mathcal{C}_5 &= \|f\|_{L_\infty} \|f^{-1}\|_{L_\infty} \left( \max\{ \|f\|_{L_\infty} \|f^{-1}\|_{L_\infty} C'_{15}, 2(t^{00})^{-1}\} + 2 r_0^{-1}\right),\\
     \mathcal{C}_6 &= \|f\|_{L_\infty}^2 \|f^{-1}\|_{L_\infty}^2 C_{16}.
    \end{align*}
    The constants $C'_{15}$, $C_{16}$, and $t^{00}$ are defined by~\emph{\eqref{sndw_C'_15_16}} and~\emph{\eqref{t00_fixation}}.
\end{theorem}

The assumptions of Theorem~\ref{sndw_cos_enchanced_thrm_2}  are a fortiori satisfied in the \textquotedblleft real\textquotedblright \ case, if the spectrum of the problem~\eqref{hatS_gener_spec_problem} is simple (see~Corollary~\ref{sndw_simple_spec_N0Q=0_cor} and Remark~\ref{sndw_simple_spec_remark}). We arrive at the following corollary.

\begin{corollary}
    \label{sndw_cos_enchanced_2_coroll}
    Suppose that the matrices $b (\boldsymbol{\theta})$, $g (\mathbf{x})$, and $Q (\mathbf{x})$~have real entries. Suppose that the spectrum of the problem~\emph{\eqref{hatS_gener_spec_problem}} is simple for any $\boldsymbol{\theta} \in \mathbb{S}^{d-1}$. Then estimate~\emph{\eqref{sndw_cos_enchanced_est_2}} holds for $\tau \in \mathbb{R}, \varepsilon > 0$, and $\mathbf{k} \in \widetilde{\Omega}$.
\end{corollary}

\subsection{ The sharpness of the result in the general case}

Application of Theorem~\ref{abstr_sndwchd_s<2_general_thrm}  allows us to confirm the sharpness of the result of Theorem~\ref{sndw_cos_general_thrm}  in the general case.

\begin{theorem}
    \label{s<2_thrm}
Let  ${\mathcal J}(\mathbf{k},\tau)$ be the operator defined by \emph{\eqref{JJ}}.
 Let $\widehat{N}_{0,Q} (\boldsymbol{\theta})$ be the operator defined by~\emph{\eqref{N0Q_invar_repr}}.
Suppose that $\widehat{N}_{0,Q} (\boldsymbol{\theta}_0) \ne 0$ for some $\boldsymbol{\theta}_0 \in \mathbb{S}^{d-1}$.
Let $0 \ne \tau \in \mathbb{R}$ and $0 \le s < 2$. Then there does not exist a constant $\mathcal{C} (\tau) > 0$ such that
the estimate
    \begin{equation}
        \label{sndw_s<2_est_imp}
        \| {\mathcal J}(\mathbf{k},\varepsilon^{-1}\tau)
\mathcal{R}(\mathbf{k}, \varepsilon)^{s/2}\|_{L_2(\Omega) \to L_2 (\Omega) }  \le \mathcal{C} (\tau) \varepsilon
    \end{equation}
   holds for almost every $\mathbf{k} = t \boldsymbol{\theta} \in \widetilde{\Omega}$ and sufficiently small $\varepsilon > 0$.
\end{theorem}

For the proof we need the following lemma which can be easily checked by analogy with the proof of  Lemma~\ref{Lipschitz_lemma}.

\begin{lemma}
    \label{sndw_Lipschitz_lemma}
 Let $\delta$ and $t^0$ be given by~\emph{\eqref{delta_fixation}} and~\emph{\eqref{t0_fixation}}, respectively.
Let $F (\mathbf{k}) = F (t, \boldsymbol{\theta})$~be the spectral projection of the operator $\mathcal{A}(\mathbf{k})$ for the interval $[0, \delta]$. Then for $|\mathbf{k}| \le t^0$ and $|\mathbf{k}_0| \le t^0$ we have
    \begin{gather*}
        \| F (\mathbf{k}) - F (\mathbf{k}_0)\|_{L_2(\Omega) \to L_2 (\Omega) } \le C' | \mathbf{k} - \mathbf{k}_0|, 
\\
        \| \mathcal{A}(\mathbf{k})^{1/2} F (\mathbf{k}) - \mathcal{A}(\mathbf{k}_0)^{1/2} F (\mathbf{k}_0)\|_{L_2(\Omega) \to L_2 (\Omega) } \le C'' | \mathbf{k} - \mathbf{k}_0|, 
\\
        \| \cos(\tau \mathcal{A}(\mathbf{k})^{1/2}) F (\mathbf{k}) - \cos(\tau \mathcal{A}(\mathbf{k}_0)^{1/2}) F (\mathbf{k}_0)\|_{L_2(\Omega) \to L_2 (\Omega) } 
\\
\le (2 C' + C'') | \mathbf{k} - \mathbf{k}_0|.
    \end{gather*}
\end{lemma}

\noindent \textbf{Proof of Theorem~\ref{s<2_thrm}}.
It suffices to assume that $1 \le s < 2$. We prove by contradiction. Fix $\tau \ne 0$.  Suppose that for some $1 \le s < 2$ there exists a constant   $\mathcal{C}(\tau) > 0$  such that estimate~\eqref{sndw_s<2_est_imp} holds for almost every $\mathbf{k} \in \widetilde{\Omega}$ and sufficiently small $\varepsilon > 0$. By~\eqref{R_P} and~\eqref{sndw_cos_main_est_(I - hatP)}, it follows that  there exists a constant
 $\widetilde{\mathcal{C}}(\tau) > 0$ such that
\begin{equation}
    \label{s<2_proof_f1}
    \|  {\mathcal J}(\mathbf{k},\varepsilon^{-1}\tau)
\widehat{P}\|_{L_2(\Omega) \to L_2 (\Omega) } \varepsilon^s (|\mathbf{k}|^2 + \varepsilon^2)^{-s/2}  \le \widetilde{\mathcal{C}}(\tau) \varepsilon
\end{equation}
for almost every  $\mathbf{k} \in \widetilde{\Omega}$  and sufficiently small $\varepsilon > 0$.

By~\eqref{abstr_P_and_P_hat_relation}, we have $f^{-1} \widehat{P} = P f^* \overline{Q}$, where $P$~ is the orthogonal projection of  $L_2 (\Omega; \mathbb{C}^n)$ onto the subspace~$\mathfrak{N}$ (see~\eqref{Ker2}). Then the operator under the norm sign in~\eqref{s<2_proof_f1} can be written as  $ f \cos (\varepsilon^{-1} \tau {\mathcal{A}}(\mathbf{k})^{1/2}) P f^* \overline{Q} - 
f_0 \cos (\varepsilon^{-1} \tau {\mathcal{A}}^0(\mathbf{k})^{1/2}) f_0^{-1} \widehat{P}$.

Now, let $|\mathbf{k}| \le t^0$. By \eqref{abstr_F(t)_threshold_1},
\begin{equation}
    \label{s<2_proof_f2}
    \| F (\mathbf{k}) - P \|_{L_2 (\Omega) \to L_2 (\Omega)} \le C_1 |\mathbf{k}|, \quad |\mathbf{k}| \le t^0.
\end{equation}
From~\eqref{s<2_proof_f1} and~\eqref{s<2_proof_f2} it follows that
 there exists a constant $\check{\mathcal C}(\tau)$ such that
\begin{multline}
    \label{s<2_proof_f3}
    \|  f \cos (\varepsilon^{-1} \tau \mathcal{A}(\mathbf{k})^{1/2}) F (\mathbf{k}) f^* \overline{Q} - f_0 \cos (\varepsilon^{-1} \tau \mathcal{A}^0(\mathbf{k})^{1/2}) f_0^{-1} \widehat{P}\|_{L_2(\Omega) \to L_2 (\Omega) }
\\
\times \varepsilon^s (|\mathbf{k}|^2 + \varepsilon^2)^{-s/2}  \le \check{\mathcal{C}}(\tau) \varepsilon
\end{multline}
for almost every $\mathbf{k} $ in the ball $|\mathbf{k}| \le t^0$ and sufficiently small $\varepsilon > 0$.

Observe that $\widehat{P}$  is the spectral projection of the operator $\mathcal{A}^0 (\mathbf{k})$ for the interval $[0, \delta]$.
Therefore, Lemma~\ref{sndw_Lipschitz_lemma} (applied to $\mathcal{A} (\mathbf{k})$ and $\mathcal{A}^0 (\mathbf{k})$)
 implies that for fixed $\tau$ and $\varepsilon$  the operator under the norm sign in~\eqref{s<2_proof_f3}
 is continuous with respect to $\mathbf{k}$ in the ball $|\mathbf{k}| \le t^0$. Hence, estimate~\eqref{s<2_proof_f3} holds for all $\mathbf{k}$
in this ball. In particular, it is satisfied at the point $\mathbf{k} = t\boldsymbol{\theta}_0$ if $t \le t^0$.
Applying once more the inequality~\eqref{s<2_proof_f2} and the identity $P f^* \overline{Q} = f^{-1} \widehat{P}$,
we see that there exists a constant $\check{\mathcal{C}}'(\tau)>0$ such that
\begin{equation}
    \label{s<2_proof_f4}
    \| {\mathcal J}(t\boldsymbol{\theta}_0,\varepsilon^{-1}\tau)
  \widehat{P}\|_{L_2(\Omega) \to L_2 (\Omega) } \varepsilon^s (t^2 + \varepsilon^2)^{-s/2}  \le \check{\mathcal{C}}'(\tau) \varepsilon
\end{equation}
for all $t \le t^0$  and sufficiently small $\varepsilon$.

In abstract terms, estimate~\eqref{s<2_proof_f4} corresponds to the inequality~\eqref{abstr_sndwchd_s<2_est_imp}.
Since $\widehat{N}_{0,Q} (\boldsymbol{\theta}_0) \ne 0$, applying Theorem~\ref{abstr_sndwchd_s<2_general_thrm},
 we arrive at a contradiction.  $\square$

\section{Approximation of the operator $\cos(\varepsilon^{-1} \tau \mathcal{A}^{1/2})$}

\subsection{Approximation of the operator $\cos(\varepsilon^{-1} \tau \widehat{\mathcal{A}}^{1/2})$}

In $L_2 (\mathbb{R}^d; \mathbb{C}^n)$, we consider the operator $\widehat{\mathcal{A}} = b(\mathbf{D})^* g(\mathbf{x}) b(\mathbf{D})$ (see~(\ref{hatA})). Let $\widehat{\mathcal{A}}^0$~ be the effective operator (see~(\ref{hatA0})).
Recall the notation $\mathcal{H}_0 = - \Delta$ and put
\begin{equation}
\label{R(epsilon)}
\mathcal{R} (\varepsilon) := \varepsilon^2 (\mathcal{H}_0 + \varepsilon^2 I)^{-1}.
\end{equation}

Expansion~\eqref{decompose} for $\widehat{\mathcal{A}}$ yields
\begin{equation*}
\cos(\varepsilon^{-1} \tau \widehat{\mathcal{A}}^{1/2}) = \mathcal{U}^{-1} \left( \int_{\widetilde{\Omega}} \oplus  \cos(\varepsilon^{-1} \tau \widehat{\mathcal{A}} (\mathbf{k})^{1/2}) \, d \mathbf{k}  \right) \mathcal{U}.
\end{equation*}
The operator $\cos(\varepsilon^{-1} \tau (\widehat{\mathcal{A}}^0)^{1/2})$
admits a similar representation. The operator $\mathcal{R} (\varepsilon)$  expands in the direct integral of the operators~\eqref{R(k, epsilon)}:
\begin{equation}
\label{R_decompose}
\mathcal{R} (\varepsilon) = \mathcal{U}^{-1} \left( \int_{\widetilde{\Omega}} \oplus  \mathcal{R} (\mathbf{k}, \varepsilon) \, d \mathbf{k}  \right) \mathcal{U}.
\end{equation}
It follows that the operator $(\cos(\varepsilon^{-1} \tau \widehat{\mathcal{A}}^{1/2}) - \cos(\varepsilon^{-1} \tau (\widehat{\mathcal{A}}^0)^{1/2})) \mathcal{R} (\varepsilon)^{s/2}$ expands in the direct integral of the operators
$$(\cos(\varepsilon^{-1} \tau \widehat{\mathcal{A}}(\mathbf{k})^{1/2}) - \cos(\varepsilon^{-1} \tau \widehat{\mathcal{A}}^0(\mathbf{k})^{1/2})) \mathcal{R} (\mathbf{k}, \varepsilon)^{s/2}.$$
Hence,
\begin{multline}
\label{norms_and_Gelfand_transf}
\| (\cos(\varepsilon^{-1} \tau \widehat{\mathcal{A}}^{1/2}) - \cos(\varepsilon^{-1} \tau (\widehat{\mathcal{A}}^0)^{1/2})) \mathcal{R} (\varepsilon)^{s/2} \|_{L_2(\mathbb{R}^d) \to L_2(\mathbb{R}^d)} =
\\
 = \underset{\mathbf{k} \in \widetilde{\Omega}}{\esssup} \| (\cos(\varepsilon^{-1} \tau \widehat{\mathcal{A}}(\mathbf{k})^{1/2}) - \cos(\varepsilon^{-1} \tau \widehat{\mathcal{A}}^0(\mathbf{k})^{1/2})) \mathcal{R} (\mathbf{k},\varepsilon)^{s/2} \|_{L_2(\Omega) \to L_2(\Omega)}.
\end{multline}
Therefore, Theorem~\ref{cos_general_thrm} directly implies the following statement (which has been proved before in \cite[Theorem~9.2]{BSu5}).

\begin{theorem}
    \label{cos_thrm_1}
    Let $\widehat{\mathcal{A}}$~ be the operator in $L_2 (\mathbb{R}^d; \mathbb{C}^n)$
given by $\widehat{\mathcal{A}} = b(\mathbf{D})^* g(\mathbf{x}) b(\mathbf{D})$, where $g(\mathbf{x})$ and $b(\mathbf{D})$ satisfy the assumptions of Subsection~\emph{\ref{A_oper_subsect}}. Let $\widehat{\mathcal{A}}^0 = b(\mathbf{D})^* g^0 b(\mathbf{D})$~be the effective operator, where $g^0$ is given by~\emph{\eqref{g0}}. Let $\mathcal{R} (\varepsilon)$~be defined by~\emph{(\ref{R(epsilon)})}. Then for $\tau \in \mathbb{R}$ and $\varepsilon > 0$ we have
    \begin{equation*}
    \| (\cos(\varepsilon^{-1} \tau \widehat{\mathcal{A}}^{1/2}) - \cos(\varepsilon^{-1} \tau (\widehat{\mathcal{A}}^0)^{1/2})) \mathcal{R} (\varepsilon) \|_{L_2(\mathbb{R}^d) \to L_2(\mathbb{R}^d)} \le (\widehat{\mathcal{C}}_1 + \widehat{\mathcal{C}}_2 |\tau|) \varepsilon.
    \end{equation*}
    The constants $\widehat{\mathcal{C}}_1$ and $\widehat{\mathcal{C}}_2$ are defined by~\emph{\eqref{calC1_C2}} and depend only on $r_0$, $\alpha_0$, $\alpha_1$, $\|g\|_{L_\infty}$, and $\|g^{-1}\|_{L_\infty}$.
\end{theorem}

Similarly,  Theorem~\ref{cos_enchanced_thrm_1} implies the following result.

\begin{theorem}
     \label{cos_thrm_2}
    Suppose that the assumptions of Theorem~\emph{\ref{cos_thrm_1}} are satisfied. Let $\widehat{N}(\boldsymbol{\theta})$
be the operator defined by~\emph{\eqref{N(theta)}}. Suppose that $\widehat{N}(\boldsymbol{\theta}) = 0$ for any $\boldsymbol{\theta} \in \mathbb{S}^{d-1}$. Then for $\tau \in \mathbb{R}$ and $\varepsilon > 0$ we have
    \begin{equation*}
    \| ( \cos (\varepsilon^{-1} \tau \widehat{\mathcal{A}}^{1/2})  - \cos (\varepsilon^{-1} \tau (\widehat{\mathcal{A}}^0)^{1/2}) )\mathcal{R}(\varepsilon)^{3/4}\|_{L_2(\mathbb{R}^d) \to L_2 (\mathbb{R}^d) }  \le (\widehat{\mathcal{C}}_3 + \widehat{\mathcal{C}}_4 |\tau|) \varepsilon.
    \end{equation*}
   The constants $\widehat{\mathcal{C}}_3$ and $\widehat{\mathcal{C}}_4$ are defined in Theorem~\emph{\ref{cos_enchanced_thrm_1}}
and depend only on $r_0$, $\alpha_0$, $\alpha_1$,  $\|g\|_{L_\infty}$, and $\|g^{-1}\|_{L_\infty}$.
\end{theorem}

Recall that some sufficient conditions ensuring  that the assumptions of Theorem~\ref{cos_thrm_2}  are satisfied are given in Proposition~\ref{N=0_proposit}.

Finally, applying  Theorem~\ref{cos_enchanced_thrm_2} and using the direct integral expansion, we obtain the following statement.

\begin{theorem}
    \label{cos_thrm_3}
   Suppose that the assumptions of Theorem~\emph{\ref{cos_thrm_1}} are satisfied. Suppose also that Condition~\emph{\ref{cond9}} \emph{(}or more restrictive Condition~\emph{\ref{cond99}}\emph{)} is satisfied. Then for $\tau \in \mathbb{R}$ and $\varepsilon > 0$ we have
    \begin{equation*}
    \| ( \cos (\varepsilon^{-1} \tau \widehat{\mathcal{A}}^{1/2})  - \cos (\varepsilon^{-1} \tau (\widehat{\mathcal{A}}^0)^{1/2}) )\mathcal{R}(\varepsilon)^{3/4}\|_{L_2(\mathbb{R}^d) \to L_2 (\mathbb{R}^d) }  \le (\widehat{\mathcal{C}}_5 + \widehat{\mathcal{C}}_6 |\tau|) \varepsilon.
    \end{equation*}
    The constants $\widehat{\mathcal{C}}_5$ and $\widehat{\mathcal{C}}_6$ are defined in Theorem~\emph{\ref{cos_enchanced_thrm_2}} and depend  only on $r_0$, $\alpha_0$, $\alpha_1$,  $\|g\|_{L_\infty}$, $\|g^{-1}\|_{L_\infty}$, and also on the number $\widehat{c}^{\circ}$ defined by~\emph{(\ref{hatc^circ})}.
    \end{theorem}

Recall that some sufficient conditions ensuring
that the assumptions of Theorem~\ref{cos_thrm_3} are satisfied are given in Corollary~\ref{cos_enchanced_2_coroll}.

Applying Theorem~\ref{hat_s<2_thrm},  we confirm the sharpness of the result of Theorem~\ref{cos_thrm_1}.

\begin{theorem}
     \label{s<2_cos_thrm}
 Suppose that the assumptions of Theorem~\emph{\ref{cos_thrm_1}} are satisfied. Let $\widehat{N}_0 (\boldsymbol{\theta})$
be the operator defined by~\emph{(\ref{N0_invar_repr})}. Suppose that $\widehat{N}_0 (\boldsymbol{\theta}_0) \ne 0$ for some  $\boldsymbol{\theta}_0 \in \mathbb{S}^{d-1}$. Let $0 \ne \tau \in \mathbb{R}$ and $0 \le s < 2$.
Then there does not exist a constant $\mathcal{C}(\tau) > 0$ such that estimate
    \begin{equation}
    \label{s<2_cos_est}
    \| ( \cos (\varepsilon^{-1} \tau \widehat{\mathcal{A}}^{1/2})  - \cos (\varepsilon^{-1} \tau (\widehat{\mathcal{A}}^0)^{1/2}) )\mathcal{R}(\varepsilon)^{s/2}\|_{L_2(\mathbb{R}^d) \to L_2 (\mathbb{R}^d) }  \le \mathcal{C}(\tau) \varepsilon
    \end{equation}
  holds for all sufficiently small $\varepsilon > 0$.
\end{theorem}

\begin{proof}
We prove by contradiction. Let us fix $\tau \ne 0$.
Suppose that for some  $0 \le s < 2$  there exists a constant $\mathcal{C}(\tau) > 0$ such that \eqref{s<2_cos_est}
 holds for all sufficiently small $\varepsilon > 0$. By~\eqref{norms_and_Gelfand_transf},  this means that for almost every $\mathbf{k} \in \widetilde{\Omega}$ and sufficiently small $\varepsilon$ estimate~\eqref{hat_s<2_est_imp} holds.
But this contradicts the statement of Theorem~\ref{hat_s<2_thrm}.
\end{proof}

\subsection{ Approximation of the sandwiched operator $\cos(\varepsilon^{-1} \tau \mathcal{A}^{1/2})$}
In $L_2 (\mathbb{R}^d; \mathbb{C}^n)$, consider the operator $\mathcal{A} \! = \! f(\mathbf{x})^* \widehat{\mathcal{A}} f(\mathbf{x}) \! = \! f(\mathbf{x})^* b(\mathbf{D})^* g(\mathbf{x}) b(\mathbf{D}) f(\mathbf{x})$ (see~\eqref{A}). Let $\widehat{\mathcal{A}}^0$~ be the operator~\eqref{hatA0}, and let $f_0$~be the matrix~\eqref{f_0}. Let $\mathcal{A}^0 = f_0 \widehat{\mathcal{A}}^0 f_0 = f_0 b(\mathbf{D})^* g^0 b(\mathbf{D}) f_0$ (see~\eqref{A0}).

Let ${\mathcal J}(\mathbf{k},  \tau)$ be the operator defined by \eqref{JJ}.
Similarly to ~\eqref{norms_and_Gelfand_transf}, using~\eqref{decompose} and~\eqref{R_decompose}, we have
\begin{multline}
\label{sndw_norms_and_Gelfand_transf}
\| (f \cos(\varepsilon^{-1} \tau \mathcal{A}^{1/2}) f^{-1} - f_0 \cos(\varepsilon^{-1} \tau (\mathcal{A}^0)^{1/2}) f_0^{-1}) \mathcal{R} (\varepsilon)^{s/2} \|_{L_2(\mathbb{R}^d) \to L_2(\mathbb{R}^d)}
\\
= \underset{\mathbf{k} \in \widetilde{\Omega}}{\esssup} \|  {\mathcal J}(\mathbf{k}, \varepsilon^{-1} \tau)
 \mathcal{R} (\mathbf{k},\varepsilon)^{s/2} \|_{L_2(\Omega) \to L_2(\Omega)}.
\end{multline}

Theorem~\ref{sndw_cos_general_thrm} together with~\eqref{sndw_norms_and_Gelfand_transf} implies the following result
(which has been proved before in \cite[Theorem~10.2]{BSu5}).

\begin{theorem}
    \label{sndw_cos_thrm_1}
    Let ${\mathcal A}$ be the operator in $L_2(\mathbb{R}^d;\mathbb{C}^n)$ given by
$\mathcal{A} = f(\mathbf{x})^* b(\mathbf{D})^* g(\mathbf{x}) b(\mathbf{D}) f(\mathbf{x})$, where
$g(\mathbf{x})$, $f(\mathbf{x})$, and $b(\mathbf{D})$ satisfy the assumptions of Subsection~\emph{\ref{A_oper_subsect}}.
Let $\mathcal{A}^0 = f_0 b(\mathbf{D})^* g^0 b(\mathbf{D}) f_0$, where $g^0$~is the effective matrix~\emph{\eqref{g0}},
and $f_0 =  (\underline{f f^*})^{1/2}$. Let $\mathcal{R}(\varepsilon)$~be defined by~\emph{\eqref{R(epsilon)}}.
Then for $\tau \in \mathbb{R}$ and $\varepsilon > 0$ we have
  \begin{multline*}
    \| (f \cos(\varepsilon^{-1} \tau \mathcal{A}^{1/2}) f^{-1} - f_0 \cos(\varepsilon^{-1} \tau (\mathcal{A}^0)^{1/2}) f_0^{-1}) \mathcal{R} (\varepsilon) \|_{L_2(\mathbb{R}^d) \to L_2(\mathbb{R}^d)}
\\
 \le (\mathcal{C}_1 + \mathcal{C}_2 |\tau|) \varepsilon.
    \end{multline*}
  The constants $\mathcal{C}_1$ and $\mathcal{C}_2$ are defined in Theorem~\emph{\ref{sndw_cos_general_thrm}}
 and depend only on $r_0$, $\alpha_0$, $\alpha_1$,  $\|g\|_{L_\infty}$, $\|g^{-1}\|_{L_\infty}$,  $\|f\|_{L_\infty}$, and $\|f^{-1}\|_{L_\infty}$.
\end{theorem}

Similarly, Theorem~\ref{sndw_cos_enchanced_thrm_1} leads to the following statement.

\begin{theorem}
   \label{sndw_cos_thrm_2}
Suppose that the assumptions of Theorem~\emph{\ref{sndw_cos_thrm_1}} are satisfied. Let $\widehat{N}_Q(\boldsymbol{\theta})$
be the operator defined by~\emph{\eqref{N_Q(theta)}}. Suppose that $\widehat{N}_Q(\boldsymbol{\theta}) = 0$ for any
$\boldsymbol{\theta} \in \mathbb{S}^{d-1}$. Then for $\tau \in \mathbb{R}$ and $\varepsilon > 0$ we have
   \begin{multline*}
   \| (f \cos(\varepsilon^{-1} \tau \mathcal{A}^{1/2}) f^{-1} - f_0 \cos(\varepsilon^{-1} \tau (\mathcal{A}^0)^{1/2}) f_0^{-1}) \mathcal{R} (\varepsilon)^{3/4} \|_{L_2(\mathbb{R}^d) \to L_2(\mathbb{R}^d)}
\\
\le (\mathcal{C}_3 + \mathcal{C}_4 |\tau|) \varepsilon.
   \end{multline*}
  The constants $\mathcal{C}_3$ and $\mathcal{C}_4$ are defined in Theorem~\emph{\ref{sndw_cos_enchanced_thrm_1}} and depend only on $r_0$, $\alpha_0$, $\alpha_1$,  $\|g\|_{L_\infty}$, $\|g^{-1}\|_{L_\infty}$,  $\|f\|_{L_\infty}$, and $\|f^{-1}\|_{L_\infty}$.
\end{theorem}

Recall that some sufficient conditions ensuring that the assumptions of Theorem~\ref{sndw_cos_thrm_2} are satisfiied are given in
Proposition~\ref{N_Q=0_proposit}.

Finally, from Theorem~\ref{sndw_cos_enchanced_thrm_2}  and the direct integral expansion  we deduce the following result.

\begin{theorem}
    \label{sndw_cos_thrm_3}
  Suppose that the assumptions of Theorem~\emph{\ref{sndw_cos_thrm_1}} are satisfied. Suppose also that Condition~\emph{\ref{sndw_cond1}} \emph{(}or more restrictive Condition~\emph{\ref{sndw_cond2}}\emph{)} is satisfied. Then for $\tau \in \mathbb{R}$ and $\varepsilon > 0$ we have
     \begin{multline*}
     \| (f \cos(\varepsilon^{-1} \tau \mathcal{A}^{1/2}) f^{-1} - f_0 \cos(\varepsilon^{-1} \tau (\mathcal{A}^0)^{1/2}) f_0^{-1}) \mathcal{R} (\varepsilon)^{3/4} \|_{L_2(\mathbb{R}^d) \to L_2(\mathbb{R}^d)}
\\
\le (\mathcal{C}_5 + \mathcal{C}_6 |\tau|) \varepsilon.
     \end{multline*}
    The constants $\mathcal{C}_5$ and $\mathcal{C}_6$ are defined in Theorem~\emph{\ref{sndw_cos_enchanced_thrm_2}} and depend only on $r_0$, $\alpha_0$, $\alpha_1$,  $\|g\|_{L_\infty}$, $\|g^{-1}\|_{L_\infty}$,  $\|f\|_{L_\infty}$, $\|f^{-1}\|_{L_\infty}$, and also on the number $c^{\circ}$
defined by~\emph{\eqref{c^circ}}.
\end{theorem}

Recall that some sufficient conditions ensuring that the assumptions of Theorem~\ref{sndw_cos_thrm_3} are satisfied are given in Corollary~\ref{cos_enchanced_2_coroll}.

By analogy with the proof of Theorem~\ref{s<2_cos_thrm},   we deduce the following result from Theorem~\ref{s<2_thrm};
this confirms that the result of Theorem~\ref{sndw_cos_thrm_1} is sharp.

\begin{theorem}
    \label{sndw_s<2_cos_thrm}
   Suppose that the assumptions of Theorem~\emph{\ref{sndw_cos_thrm_1}} are satisfied. Let $\widehat{N}_{0,Q}(\boldsymbol{\theta})$  be the operator defined by~\emph{\eqref{N0Q_invar_repr}}. Suppose that $\widehat{N}_{0,Q}(\boldsymbol{\theta}_0) \ne 0$ for some $\boldsymbol{\theta}_0 \in \mathbb{S}^{d-1}$. Let $0 \ne \tau \in \mathbb{R}$ and $0 \le s < 2$. Then there does not exist a constant $\mathcal{C}(\tau) > 0$ such that estimate
    \begin{multline}
    \label{sndw_s<2_cos_est}
    \| (f \cos(\varepsilon^{-1} \tau \mathcal{A}^{1/2}) f^{-1} - f_0 \cos(\varepsilon^{-1} \tau (\mathcal{A}^0)^{1/2}) f_0^{-1}) \mathcal{R} (\varepsilon)^{s/2} \|_{L_2(\mathbb{R}^d) \to L_2(\mathbb{R}^d)}
\\
 \le \mathcal{C}(\tau) \varepsilon
    \end{multline}
holds  for all sufficiently small $\varepsilon > 0$.
\end{theorem}

\section*{Chapter 3. Homogenization problems for hyperbolic equations}

\section{Approximation of the operators\\
$\cos(\tau \mathcal{A}_\varepsilon^{1/2})$ and $\mathcal{A}_\varepsilon^{-1/2} \sin(\tau \mathcal{A}_\varepsilon^{1/2})$}

\subsection{Operators $\widehat{\mathcal{A}}_\varepsilon$ and $\mathcal{A}_\varepsilon$. Statement of the problem}

If $\psi(\mathbf{x})$~is a $\Gamma$-periodic measurable function in  $\mathbb{R}^d$, we  denote $\psi^{\varepsilon}(\mathbf{x}) := \psi(\varepsilon^{-1} \mathbf{x}), \; \varepsilon > 0$. \emph{Our main objects} are  the operators $\widehat{\mathcal{A}}_\varepsilon$ and $\mathcal{A}_\varepsilon$ acting in $L_2 (\mathbb{R}^d; \mathbb{C}^n)$ and formally given by
\begin{gather}
\label{Ahat_eps}
\widehat{\mathcal{A}}_\varepsilon := b(\mathbf{D})^* g^{\varepsilon}(\mathbf{x}) b(\mathbf{D}), \\
\label{A_eps}
\mathcal{A}_\varepsilon := (f^{\varepsilon}(\mathbf{x}))^* b(\mathbf{D})^* g^{\varepsilon}(\mathbf{x}) b(\mathbf{D}) f^{\varepsilon}(\mathbf{x}).
\end{gather}
The precise definitions are given in terms of the corresponding quadratic forms (cf. Subsection~\ref{A_oper_subsect}).
 The coefficients of the operators~~\eqref{Ahat_eps} and \eqref{A_eps}  oscillate rapidly as $\varepsilon \to 0$.

\emph{Our goal} is to find approximations for the operators $\cos(\tau \mathcal{A}_\varepsilon^{1/2})$
and $\mathcal{A}_\varepsilon^{-1/2} \sin(\tau \mathcal{A}_\varepsilon^{1/2})$ and
to apply the results to homogenization of the Cauchy problem for hyperbolic equations.

\subsection{ The scaling transformation}
Let $T_{\varepsilon}$~be the \emph{unitary scaling transformation in $L_2 (\mathbb{R}^d; \mathbb{C}^n)$} defined by
\begin{equation*}
(T_{\varepsilon} \mathbf{u})(\mathbf{x}) = \varepsilon^{d/2} \mathbf{u} (\varepsilon \mathbf{x}), \; \varepsilon > 0.
\end{equation*}
Then $\mathcal{A}_\varepsilon = \varepsilon^{-2}T_{\varepsilon}^* \mathcal{A} T_{\varepsilon}$. Hence,
\begin{equation}
\label{cos_and_scale_transform}
\cos(\tau \mathcal{A}_\varepsilon^{1/2}) = T_{\varepsilon}^* \cos(\varepsilon^{-1} \tau \mathcal{A}^{1/2}) T_{\varepsilon}.
\end{equation}
The operator $\widehat{\mathcal{A}}_\varepsilon$ satisfies a similar relation.

Applying the scaling transformation to the resolvent of the operator $\mathcal{H}_0 = - \Delta$, we obtain
\begin{equation}
\label{H0_resolv_and_scale_transform}
(\mathcal{H}_0 + I)^{-1} = \varepsilon^2 T_\varepsilon^* (\mathcal{H}_0 + \varepsilon^2 I)^{-1} T_\varepsilon = T_\varepsilon^* \mathcal{R} (\varepsilon) T_\varepsilon.
\end{equation}
Here we have used the notation~\eqref{R(epsilon)}.

Finally, if $\psi(\mathbf{x})$~is a $\Gamma$-periodic function, then, under the scaling transformation, the operator $[\psi^{\varepsilon}]$
of multiplication by the function $\psi^{\varepsilon} (\mathbf{x}) = \psi (\varepsilon^{-1} \mathbf{x})$
turns into the operator $[\psi]$ of multiplication by $\psi (\mathbf{x})$:
\begin{equation}
\label{mult_op_and_scale_transform}
[\psi^{\varepsilon}] = T_\varepsilon^* [\psi] T_\varepsilon.
\end{equation}

\subsection{Approximation of the operators $\cos( \tau \widehat{\mathcal{A}}_\varepsilon^{1/2})$ and
$\widehat{\mathcal{A}}_\varepsilon^{-1/2}\sin( \tau \widehat{\mathcal{A}}_\varepsilon^{1/2})$}

We start with the simpler operator $\widehat{\mathcal{A}}_\varepsilon$. Let $\widehat{\mathcal{A}}^0$~be the effective operator~\eqref{hatA0}. Using relations of the form~\eqref{cos_and_scale_transform} (for the operators $\widehat{\mathcal{A}}_\varepsilon$ and $\widehat{\mathcal{A}}^0$)
and~\eqref{H0_resolv_and_scale_transform}, we obtain
\begin{multline}
\label{cos-cos_and_scale_transform}
(\cos( \tau \widehat{\mathcal{A}}_\varepsilon^{1/2}) - \cos( \tau (\widehat{\mathcal{A}}^0)^{1/2})) (\mathcal{H}_0 + I)^{-s/2} = \\ = T_{\varepsilon}^* (\cos(\varepsilon^{-1} \tau \widehat{\mathcal{A}}^{1/2}) - \cos(\varepsilon^{-1} \tau (\widehat{\mathcal{A}}^0)^{1/2})) \mathcal{R} (\varepsilon)^{s/2} T_{\varepsilon}, \quad \varepsilon > 0.
\end{multline}
Since $T_{\varepsilon}$  is unitary, combining this with Theorem~\ref{cos_thrm_1}, we deduce the following result (which has been proved before in~\cite[Theorem~13.1]{BSu5}).

\begin{theorem}
      \label{cos_thrm1_L2L2}
   Let $\widehat{\mathcal{A}}_{\varepsilon}$~be the operator~\emph{\eqref{Ahat_eps}} and let  $\widehat{\mathcal{A}}^0$~be
the effective operator~\emph{\eqref{hatA0}}. Then for $\tau \in \mathbb{R}$ and $\varepsilon > 0$ we have
      \begin{equation}
      \label{cos_est1_L2L2}
      \| (\cos( \tau \widehat{\mathcal{A}}_{\varepsilon}^{1/2}) - \cos( \tau (\widehat{\mathcal{A}}^0)^{1/2})) (\mathcal{H}_0 + I)^{-1} \|_{L_2(\mathbb{R}^d) \to L_2(\mathbb{R}^d)} \le (\widehat{\mathcal{C}}_1 + \widehat{\mathcal{C}}_2 |\tau|) \varepsilon.
      \end{equation}
     The constants $\widehat{\mathcal{C}}_1$ and $\widehat{\mathcal{C}}_2$ are given by~\emph{\eqref{calC1_C2}}
and depend only on $r_0$, $\alpha_0$, $\alpha_1$, $\|g\|_{L_\infty}$, and $\|g^{-1}\|_{L_\infty}$.
\end{theorem}

Obviously,
\begin{equation}
\label{cos-cos_le_2}
\| \cos( \tau \widehat{\mathcal{A}}_{\varepsilon}^{1/2}) - \cos( \tau (\widehat{\mathcal{A}}^0)^{1/2}) \|_{L_2 (\mathbb{R}^d) \to L_2 (\mathbb{R}^d)} \le 2, \quad  \tau \in \mathbb{R}, \;  \varepsilon > 0.
\end{equation}
Interpolating between~\eqref{cos-cos_le_2} and~\eqref{cos_est1_L2L2}, for $ 0 \le s \le 2$ we obtain
\begin{multline}
\label{cos_interp_1}
\| (\cos( \tau \widehat{\mathcal{A}}_{\varepsilon}^{1/2}) - \cos( \tau (\widehat{\mathcal{A}}^0)^{1/2}))  (\mathcal{H}_0 + I)^{-s/2} \|_{L_2 (\mathbb{R}^d) \to L_2 (\mathbb{R}^d)} \\
\le 2^{1-s/2} (\widehat{\mathcal{C}}_1 + \widehat{\mathcal{C}}_2  |\tau|)^{s/2} \varepsilon^{s/2}, \quad  \tau \in \mathbb{R}, \;  \varepsilon > 0.
\end{multline}
The operator $(\mathcal{H}_0 + I)^{s/2}$  is an isometric isomorphism of the Sobolev space $H^s(\mathbb{R}^d; \mathbb{C}^n)$ onto  $L_2(\mathbb{R}^d; \mathbb{C}^n)$, since for $\mathbf{u} \in H^s(\mathbb{R}^d; \mathbb{C}^n)$ we have
\begin{equation}
\label{H_L2_isometry}
\| (\mathcal{H}_0 + I)^{s/2} \mathbf{u} \|^2_{L_2 (\mathbb{R}^d)} = \int_{\mathbb{R}^d} ( 1 + |\boldsymbol{\xi}|^2)^s |\hat{\mathbf{u}} (\boldsymbol{\xi})|^2 d \boldsymbol{\xi} = \| \mathbf{u} \|^2_{ H^s(\mathbb{R}^d)},
\end{equation}
where $\hat{\mathbf{u}} (\boldsymbol{\xi})$ is the Fourier-image of $\mathbf{u} (\mathbf{x})$.
By \eqref{cos_interp_1} and~\eqref{H_L2_isometry},
\begin{multline}
\label{cos_est1_H^s_to_L2}
\| \cos( \tau \widehat{\mathcal{A}}_{\varepsilon}^{1/2}) - \cos( \tau (\widehat{\mathcal{A}}^0)^{1/2}) \|_{H^s (\mathbb{R}^d) \to L_2 (\mathbb{R}^d)} \le 2^{1-s/2} (\widehat{\mathcal{C}}_1 + \widehat{\mathcal{C}}_2  |\tau|)^{s/2} \varepsilon^{s/2}, \\ \tau \in \mathbb{R}, \;  \varepsilon > 0.
\end{multline}
In particular, for small $\varepsilon$ estimate~\eqref{cos_est1_H^s_to_L2} allows us  to consider large values of time
$\tau$, namely, we can take $\tau = O(\varepsilon^{-\alpha})$ with $0 < \alpha < 1$.

Next, using the identity
\begin{equation*}
\widehat{\mathcal{A}}_\varepsilon^{-1/2}\sin( \tau \widehat{\mathcal{A}}_\varepsilon^{1/2}) = \int_{0}^{\tau} \cos( \tilde{\tau} \widehat{\mathcal{A}}_\varepsilon^{1/2}) \, d \tilde{\tau}
\end{equation*}
and a similar identity for the operator $\widehat{\mathcal{A}}^0$, from~(\ref{cos_est1_H^s_to_L2}) we deduce
\begin{multline}
\label{sin_est1_H^s_to_L2}
\| \widehat{\mathcal{A}}_\varepsilon^{-1/2}\sin( \tau \widehat{\mathcal{A}}_\varepsilon^{1/2}) - (\widehat{\mathcal{A}}^0)^{-1/2}\sin( \tau (\widehat{\mathcal{A}}^0)^{1/2}) \|_{H^s (\mathbb{R}^d) \to L_2 (\mathbb{R}^d)}
\\
\le 2^{1-s/2} (1 + s/2)^{-1} \widehat{\mathcal{C}}_{2}^{-1} (\widehat{\mathcal{C}}_1 + \widehat{\mathcal{C}}_2  |\tau|)^{1+s/2} \varepsilon^{s/2}, \quad  \tau \in \mathbb{R}, \;  \varepsilon > 0.
\end{multline}
Note that, for large $|\tau|$, the coefficient at $\varepsilon^{s/2}$ in~\eqref{cos_est1_H^s_to_L2} is of order $O(|\tau|^{s/2})$,
while the coefficient in~\eqref{sin_est1_H^s_to_L2} is of order $O(|\tau|^{1+s/2})$.

We arrive at the following result (which has been proved before in~\cite[Theorem 13.2]{BSu5}).

\begin{theorem}
    \label{cos_sin_thrm1_H^s_L2}
    Suppose that the assumptions of Theorem~\emph{\ref{cos_thrm1_L2L2}} are satisfied. Then for
$0 \le s \le 2$, $\tau \in \mathbb{R}$, and $\varepsilon > 0$ we have
    \begin{gather}
    \label{cos_thrm1_H^s_L2_est}
    \| \cos( \tau \widehat{\mathcal{A}}_{\varepsilon}^{1/2}) - \cos( \tau (\widehat{\mathcal{A}}^0)^{1/2}) \|_{H^s (\mathbb{R}^d) \to L_2 (\mathbb{R}^d)} \le \widehat{\mathfrak{C}}_1 (s; \tau) \varepsilon^{s/2},
    \\
    \label{sin_thrm1_H^s_L2_est}
    \|\widehat{\mathcal{A}}_\varepsilon^{-1/2} \sin( \tau \widehat{\mathcal{A}}_\varepsilon^{1/2}) - (\widehat{\mathcal{A}}^0)^{-1/2} \sin( \tau (\widehat{\mathcal{A}}^0)^{1/2}) \|_{H^s (\mathbb{R}^d) \to L_2 (\mathbb{R}^d)} \le \widehat{\mathfrak{C}}_2 (s; \tau)  \varepsilon^{s/2},
    \end{gather}
   where
    \begin{equation}
    \label{hat_frakC_1_2}
    \begin{split}
    \widehat{\mathfrak{C}}_1 (s; \tau) &= 2^{1-s/2} (\widehat{\mathcal{C}}_1 + \widehat{\mathcal{C}}_2  |\tau|)^{s/2}, \\
    \widehat{\mathfrak{C}}_2 (s; \tau) &=  \widehat{\mathfrak{C}}_1 (s; \tau) (1 + s/2)^{-1}  (\widehat{\mathcal{C}}_1 \widehat{\mathcal{C}}_{2}^{-1} +  |\tau|).
    \end{split}
    \end{equation}
In particular, for $0 < \varepsilon \le 1$ and $|\tau| = \varepsilon^{-\alpha}$ we have
     \begin{gather}
     \label{cos_thrm1_H^s_L2_est_largetime}
     \begin{split}
     \| \cos( \tau \widehat{\mathcal{A}}_{\varepsilon}^{1/2}) - \cos( \tau (\widehat{\mathcal{A}}^0)^{1/2}) \|_{H^s (\mathbb{R}^d) \to L_2 (\mathbb{R}^d)} \le \widehat{\mathfrak{C}}_1 (s; 1) \varepsilon^{s(1-\alpha)/2}, \\ |\tau| = \varepsilon^{-\alpha}, \; 0 < \alpha < 1, \; 0 < \varepsilon \le 1,
     \end{split}
     \\
     \label{sin_thrm1_H^s_L2_est_largetime}
     \begin{split}
     \|&\mathcal{A}_\varepsilon^{-1/2} \sin( \tau \widehat{\mathcal{A}}_\varepsilon^{1/2}) - (\widehat{\mathcal{A}}^0)^{-1/2} \sin( \tau (\widehat{\mathcal{A}}^0)^{1/2}) \|_{H^s (\mathbb{R}^d) \to L_2 (\mathbb{R}^d)}
     \\
     &\le \widehat{\mathfrak{C}}_2 (s; 1)  \varepsilon^{s(1-\alpha)/2-\alpha}, \quad  |\tau| = \varepsilon^{-\alpha}, \; 0 < \alpha < s(s+2)^{-1}, \; 0 < \varepsilon \le 1.
     \end{split}
     \end{gather}
\end{theorem}

\subsection{Refinement of approximation
under the additional assumptions}

Similarly, from Theorem~\ref{cos_thrm_2} and~\eqref{cos-cos_and_scale_transform},
we deduce the following result.

\begin{theorem}
    \label{cos_thrm2_L2L2}
    Suppose that the assumptions of Theorem~\emph{\ref{cos_thrm1_L2L2}} are satisfied. Let $\widehat{N}(\boldsymbol{\theta})$
be the operator defined by~\emph{\eqref{N(theta)}}. Suppose that $\widehat{N}(\boldsymbol{\theta}) = 0$ for any $\boldsymbol{\theta} \in \mathbb{S}^{d-1}$. Then for $\tau \in \mathbb{R}$ и $\varepsilon > 0$ we have
    \begin{equation}
    \label{cos_est2_L2L2}
    \| (\cos( \tau \widehat{\mathcal{A}}_{\varepsilon}^{1/2}) - \cos( \tau (\widehat{\mathcal{A}}^0)^{1/2})) (\mathcal{H}_0 + I)^{-3/4} \|_{L_2(\mathbb{R}^d) \to L_2(\mathbb{R}^d)} \le (\widehat{\mathcal{C}}_3 + \widehat{\mathcal{C}}_4 |\tau|) \varepsilon.
    \end{equation}
    The constants $\widehat{\mathcal{C}}_3$ and $\widehat{\mathcal{C}}_4$ are defined in Theorem~\emph{\ref{cos_enchanced_thrm_1}}
and depend only on $r_0$, $\alpha_0$, $\alpha_1$, $\|g\|_{L_\infty}$, and $\|g^{-1}\|_{L_\infty}$.
\end{theorem}

Interpolating between~~\eqref{cos-cos_le_2} and \eqref{cos_est2_L2L2},  we obtain the following statement.

\begin{theorem}
    \label{cos_sin_thrm2_H^s_L2}
   Suppose that the assumptions of Theorem~\emph{\ref{cos_thrm2_L2L2}} are satisfied. Then for $0 \le s \le 3/2$,
$\tau \in \mathbb{R}$, and $\varepsilon > 0$ we have
    \begin{gather*}
    \| \cos( \tau \widehat{\mathcal{A}}_{\varepsilon}^{1/2}) - \cos( \tau (\widehat{\mathcal{A}}^0)^{1/2}) \|_{H^s (\mathbb{R}^d) \to L_2 (\mathbb{R}^d)} \le \widehat{\mathfrak{C}}_3 (s; \tau) \varepsilon^{2s/3},
    \\
    \|\widehat{\mathcal{A}}_\varepsilon^{-1/2} \sin( \tau \widehat{\mathcal{A}}_\varepsilon^{1/2}) - (\widehat{\mathcal{A}}^0)^{-1/2} \sin( \tau (\widehat{\mathcal{A}}^0)^{1/2}) \|_{H^s (\mathbb{R}^d) \to L_2 (\mathbb{R}^d)} \le \widehat{\mathfrak{C}}_4 (s; \tau)  \varepsilon^{2s/3},
    \end{gather*}
where
    \begin{equation}
    \label{hat_frakC_3_4}
    \begin{split}
    \widehat{\mathfrak{C}}_3 (s; \tau) &= 2^{1-2s/3} (\widehat{\mathcal{C}}_3 + \widehat{\mathcal{C}}_4  |\tau|)^{2s/3}, \\
    \widehat{\mathfrak{C}}_4 (s; \tau) &=  \widehat{\mathfrak{C}}_3 (s; \tau) (1 + 2s/3)^{-1}  (\widehat{\mathcal{C}}_3 \widehat{\mathcal{C}}_{4}^{-1} +  |\tau|).
    \end{split}
    \end{equation}
    In particular, for $0 < \varepsilon \le 1$ and $|\tau| = \varepsilon^{-\alpha}$ we have
    \begin{gather*}
    \begin{split}
    \| \cos( \tau \widehat{\mathcal{A}}_{\varepsilon}^{1/2}) - \cos( \tau (\widehat{\mathcal{A}}^0)^{1/2}) \|_{H^s (\mathbb{R}^d) \to L_2 (\mathbb{R}^d)} \le \widehat{\mathfrak{C}}_3 (s; 1) \varepsilon^{2s(1-\alpha)/3}, \\ |\tau| = \varepsilon^{-\alpha}, \; 0 < \alpha < 1, \; 0 < \varepsilon \le 1,
    \end{split}
    \\
    \begin{split}
    \|&\widehat{\mathcal{A}}_\varepsilon^{-1/2} \sin( \tau \widehat{\mathcal{A}}_\varepsilon^{1/2}) - (\widehat{\mathcal{A}}^0)^{-1/2} \sin( \tau (\widehat{\mathcal{A}}^0)^{1/2}) \|_{H^s (\mathbb{R}^d) \to L_2 (\mathbb{R}^d)}
    \\
&\le \widehat{\mathfrak{C}}_4 (s; 1)  \varepsilon^{2s(1-\alpha)/3-\alpha}, \quad |\tau| = \varepsilon^{-\alpha}, \; 0 < \alpha < 2s(2s+3)^{-1}, \; 0 < \varepsilon \le 1.
    \end{split}
    \end{gather*}
\end{theorem}

Theorem~\ref{cos_sin_thrm2_H^s_L2} and Proposition~\ref{N=0_proposit}  imply the following statement.

\begin{corollary}
   Suppose that at least one of the following conditions is fulfilled:

$1^{\circ}$. The operator $\widehat{\mathcal{A}}_\varepsilon$ has the form
$\widehat{\mathcal{A}}_\varepsilon = \mathbf{D}^* g^\varepsilon(\mathbf{x}) \mathbf{D}$, where $g(\mathbf{x})$~is a
symmetric matrix with real entries.

$2^{\circ}$. Relations~\emph{\eqref{g0=overline_g_relat}} are satisfied, i.~e., $g^0 = \overline{g}$.

$3^{\circ}$. Relations~\emph{\eqref{g0=underline_g_relat}} are satisfied, i.~e., $g^0 = \underline{g}$.
\emph{(}In particular, this is true if $m = n$.\emph{)}

\noindent
    Then the estimates of Theorem~\emph{\ref{cos_sin_thrm2_H^s_L2}} hold.
\end{corollary}

Finally, Theorem~\ref{cos_thrm_3} together with~\eqref{cos-cos_and_scale_transform}  lead to the following result.

\begin{theorem}
    \label{cos_thrm3_L2L2}
  Suppose that the assumptions of Theorem~\emph{\ref{cos_thrm1_L2L2}} are satisfied. Suppose that Condition~\emph{\ref{cond9}}
\emph{(}or more restrictive Condition~\emph{\ref{cond99})} is satisfied. Then for $\tau \in \mathbb{R}$ and $\varepsilon > 0$ we have
    \begin{equation}
    \label{cos_est3_L2L2}
    \| (\cos( \tau \widehat{\mathcal{A}}_{\varepsilon}^{1/2}) - \cos( \tau (\widehat{\mathcal{A}}^0)^{1/2})) (\mathcal{H}_0 + I)^{-3/4} \|_{L_2(\mathbb{R}^d) \to L_2(\mathbb{R}^d)} \le (\widehat{\mathcal{C}}_5 + \widehat{\mathcal{C}}_6 |\tau|) \varepsilon.
    \end{equation}
    The constants $\widehat{\mathcal{C}}_5$ and $\widehat{\mathcal{C}}_6$  are defined in Theorem~\emph{\ref{cos_enchanced_thrm_2}}
and depend only on $r_0$, $\alpha_0$, $\alpha_1$,  $\|g\|_{L_\infty}$, $\|g^{-1}\|_{L_\infty}$, and also on the number $\widehat{c}^{\circ}$
given by~\emph{\eqref{hatc^circ}}.
\end{theorem}

Interpolating between~\eqref{cos-cos_le_2} and~\eqref{cos_est3_L2L2}, we obtain the following result.

\begin{theorem}
    \label{cos_sin_thrm3_H^s_L2}
  Suppose that the assumptions of Theorem~\emph{\ref{cos_thrm3_L2L2}} are satisfied. Then for $0 \le s \le 3/2$, $\tau \in \mathbb{R}$,
and $\varepsilon > 0$ we have
    \begin{gather*}
    \| \cos( \tau \widehat{\mathcal{A}}_{\varepsilon}^{1/2}) - \cos( \tau (\widehat{\mathcal{A}}^0)^{1/2}) \|_{H^s (\mathbb{R}^d) \to L_2 (\mathbb{R}^d)} \le \widehat{\mathfrak{C}}_5 (s; \tau) \varepsilon^{2s/3},
    \\
    \|\widehat{\mathcal{A}}_\varepsilon^{-1/2} \sin( \tau \widehat{\mathcal{A}}_\varepsilon^{1/2}) - (\widehat{\mathcal{A}}^0)^{-1/2} \sin( \tau (\widehat{\mathcal{A}}^0)^{1/2}) \|_{H^s (\mathbb{R}^d) \to L_2 (\mathbb{R}^d)} \le \widehat{\mathfrak{C}}_6 (s; \tau)  \varepsilon^{2s/3},
    \end{gather*}
where
    \begin{equation}
    \label{hat_frakC_5_6}
    \begin{split}
    \widehat{\mathfrak{C}}_5 (s; \tau) &= 2^{1-2s/3} (\widehat{\mathcal{C}}_5 + \widehat{\mathcal{C}}_6  |\tau|)^{2s/3}, \\
    \widehat{\mathfrak{C}}_6 (s; \tau) &=  \widehat{\mathfrak{C}}_5 (s; \tau) (1 + 2s/3)^{-1}  (\widehat{\mathcal{C}}_5 \widehat{\mathcal{C}}_{6}^{-1} +  |\tau|).
    \end{split}
    \end{equation}
    In particular, for $0 < \varepsilon \le 1$ and $|\tau| = \varepsilon^{-\alpha}$ we have
    \begin{gather*}
    \begin{split}
    \| \cos( \tau \widehat{\mathcal{A}}_{\varepsilon}^{1/2}) - \cos( \tau (\widehat{\mathcal{A}}^0)^{1/2}) \|_{H^s (\mathbb{R}^d) \to L_2 (\mathbb{R}^d)} \le \widehat{\mathfrak{C}}_5 (s; 1) \varepsilon^{2s(1-\alpha)/3}, \\ |\tau| = \varepsilon^{-\alpha}, \; 0 < \alpha < 1, \; 0 < \varepsilon \le 1,
    \end{split}
    \\
    \begin{split}
    \|&\widehat{\mathcal{A}}_\varepsilon^{-1/2} \sin( \tau \widehat{\mathcal{A}}_\varepsilon^{1/2}) - (\widehat{\mathcal{A}}^0)^{-1/2} \sin( \tau (\widehat{\mathcal{A}}^0)^{1/2}) \|_{H^s (\mathbb{R}^d) \to L_2 (\mathbb{R}^d)}
     \\
     &\le \widehat{\mathfrak{C}}_6 (s; 1)  \varepsilon^{2s(1-\alpha)/3-\alpha}, \quad
     |\tau| = \varepsilon^{-\alpha}, \; 0 < \alpha < 2s(2s+3)^{-1}, \; 0 < \varepsilon \le 1.
    \end{split}
    \end{gather*}
\end{theorem}

Combining Theorem~\ref{cos_sin_thrm3_H^s_L2} and Corollary~\ref{cos_enchanced_2_coroll}, we arrive at the following corollary.

\begin{corollary}
  Suppose that the matrices $b (\boldsymbol{\theta})$ and $g (\mathbf{x})$~have real entries.
 Suppose that the spectrum of the germ $\widehat{S}(\boldsymbol{\theta})$ is simple for any $\boldsymbol{\theta} \in \mathbb{S}^{d-1}$.
Then the estimates of Theorem~\emph{\ref{cos_sin_thrm3_H^s_L2}} hold.
\end{corollary}

\subsection{ The sharpness of the result}

Applying Theorem~\ref{s<2_cos_thrm},  we confirm the sharpness of the result of Theorem~\ref{cos_thrm1_L2L2}
in the general case.

\begin{theorem}
    \label{s<2_cos_thrm_Rd}
 Suppose that the assumptions of Theorem~\emph{\ref{cos_thrm1_L2L2}} are satisfied. Let $\widehat{N}_0 (\boldsymbol{\theta})$
 be the operator defined by~\emph{\eqref{N0_invar_repr}}. Suppose that $\widehat{N}_0 (\boldsymbol{\theta}_0) \ne 0$ for some $\boldsymbol{\theta}_0 \in \mathbb{S}^{d-1}$. Let $0 \ne \tau \in \mathbb{R}$ and $0 \le s < 2$.
Then there does not exist a constant $\mathcal{C}(\tau) > 0$ such that estimate
    \begin{equation}
    \label{s<2_cos_est1}
    \| ( \cos (\tau \widehat{\mathcal{A}}_\varepsilon^{1/2})  - \cos (\tau (\widehat{\mathcal{A}}^0)^{1/2}) )(\mathcal{H}_0  + I)^{-s/2}\|_{L_2(\mathbb{R}^d) \to L_2 (\mathbb{R}^d) }  \le \mathcal{C}(\tau) \varepsilon
    \end{equation}
holds for all sufficiently small $\varepsilon > 0$.
\end{theorem}

\begin{proof} We prove by contradiction. Let us fix $\tau \ne 0$. Suppose that for some $0 \le s < 2$ there exists a constant
   $\mathcal{C}(\tau) > 0$ such that estimate~\eqref{s<2_cos_est1}  holds for all sufficiently small $\varepsilon > 0$.
   Then, by~\eqref{cos-cos_and_scale_transform}, for sufficiently small $\varepsilon$ estimate~\eqref{s<2_cos_est} is also valid.
   But this contradicts the statement of Theorem~\ref{s<2_cos_thrm}.
    \end{proof}

\subsection{Approximation of the sandwiched operators $\cos(\tau \mathcal{A}_\varepsilon^{1/2})$
and $\mathcal{A}_\varepsilon^{-1/2} \sin(\tau \mathcal{A}_\varepsilon^{1/2})$}
Now we proceed to the study of the operator $\mathcal{A}_\varepsilon$ (see~\eqref{A_eps}).
Let $\mathcal{A}^0$~be defined by~\eqref{A0}. Using relations of the form~\eqref{cos_and_scale_transform}
for the operators $\mathcal{A}_\varepsilon$ and $\mathcal{A}^0$, and taking~\eqref{H0_resolv_and_scale_transform} and \eqref{mult_op_and_scale_transform} into account, we obtain
\begin{multline}
\label{sndw_cos-cos_and_scale_transform}
(f^\varepsilon \cos( \tau \mathcal{A}_\varepsilon^{1/2}) (f^\varepsilon)^{-1} - f_0 \cos( \tau (\mathcal{A}^0)^{1/2}) f_0^{-1} ) (\mathcal{H}_0 + I)^{-s/2} = \\ = T_{\varepsilon}^* (f \cos(\varepsilon^{-1} \tau \mathcal{A}^{1/2}) f^{-1} - f_0 \cos(\varepsilon^{-1} \tau (\mathcal{A}^0)^{1/2}) f_0^{-1}) \mathcal{R} (\varepsilon)^{s/2} T_{\varepsilon}, \quad \varepsilon > 0.
\end{multline}

Since $T_{\varepsilon}$ is unitary, combining this with Theorem~\ref{sndw_cos_thrm_1},
 we arrive at the following result (which has been proved before in~\cite[Theorem~13.3]{BSu5}).

\begin{theorem}
    \label{sndw_cos_thrm1_L2L2}
    Let $\mathcal{A}_{\varepsilon}$~and~$\mathcal{A}^0$ be the operators defined by~\emph{\eqref{A_eps}} and~\emph{\eqref{A0}}, respectively.
    Then for $\tau \in \mathbb{R}$ and $\varepsilon > 0$ we have
    \begin{multline}
    \label{sndw_cos_est1_L2L2}
    \| ( f^\varepsilon \cos( \tau \mathcal{A}_\varepsilon^{1/2}) (f^\varepsilon)^{-1} - f_0 \cos( \tau (\mathcal{A}^0)^{1/2}) f_0^{-1} ) (\mathcal{H}_0 + I)^{-1} \|_{L_2(\mathbb{R}^d) \to L_2(\mathbb{R}^d)} \\
    \le (\mathcal{C}_1 + \mathcal{C}_2 |\tau|) \varepsilon.
    \end{multline}
    The constants $\mathcal{C}_1$ and $\mathcal{C}_2$ are defined in Theorem~\emph{\ref{sndw_cos_general_thrm}}
    and depend only on $r_0$, $\alpha_0$, $\alpha_1$, $\|g\|_{L_\infty}$, $\|g^{-1}\|_{L_\infty}$, $\| f \|_{L_\infty}$, and $\| f^{-1} \|_{L_\infty}$.
\end{theorem}

Obviously, taking \eqref{f0} into account, we have
\begin{multline}
\label{sndw_cos-cos_le_2}
\| f^\varepsilon \cos( \tau \mathcal{A}_\varepsilon^{1/2})  (f^\varepsilon)^{-1}  - f_0 \cos( \tau (\mathcal{A}^0)^{1/2}) f_0^{-1} \|_{L_2 (\mathbb{R}^d) \to L_2 (\mathbb{R}^d)}
\\
\le 2 \| f \|_{L_\infty} \| f^{-1} \|_{L_\infty}, \quad  \tau \in \mathbb{R}, \;  \varepsilon > 0.
\end{multline}
Interpolating between~\eqref{sndw_cos-cos_le_2} and~\eqref{sndw_cos_est1_L2L2},
we arrive at the following result which has been proved before in~\cite[Theorem~13.4]{BSu5}.

\begin{theorem}
    \label{sndw_cos_sin_thrm1_H^s_L2}
    Suppose that the assumptions of Theorem~\emph{\ref{sndw_cos_thrm1_L2L2}} are satisfied. Then for $0 \le s \le 2$,
    $\tau \in \mathbb{R}$, and $\varepsilon > 0$ we have
    \begin{align*}
    \| f^\varepsilon \cos( \tau \mathcal{A}_\varepsilon^{1/2})  (f^\varepsilon)^{-1}  - f_0 \cos( \tau (\mathcal{A}^0)^{1/2}) f_0^{-1}  \|_{H^s (\mathbb{R}^d) \to L_2 (\mathbb{R}^d)} \le \mathfrak{C}_1 (s; \tau) \varepsilon^{s/2},
    \\
    \| f^\varepsilon \mathcal{A}_\varepsilon^{-1/2} \sin( \tau \mathcal{A}_\varepsilon^{1/2}) (f^\varepsilon)^{-1} -  f_0 (\mathcal{A}^0)^{-1/2} \sin( \tau (\mathcal{A}^0)^{1/2}) f_0^{-1} \|_{H^s (\mathbb{R}^d) \to L_2 (\mathbb{R}^d)}
    \\
    \le \mathfrak{C}_2 (s; \tau)  \varepsilon^{s/2},
    \end{align*}
    where
    \begin{equation}
    \label{frakC_1_2}
    \begin{split}
    \mathfrak{C}_1 (s; \tau) &= (2 \| f \|_{L_\infty} \| f^{-1} \|_{L_\infty})^{1-s/2} (\mathcal{C}_1 + \mathcal{C}_2  |\tau|)^{s/2}, \\
    \mathfrak{C}_2 (s; \tau) &=  \mathfrak{C}_1 (s; \tau) (1 + s/2)^{-1}  (\mathcal{C}_1 \mathcal{C}_{2}^{-1} +  |\tau|).
    \end{split}
    \end{equation}
    In particular, for $0 < \varepsilon \le 1$ and $|\tau| = \varepsilon^{-\alpha}$ we have
    \begin{gather*}
    \begin{split}
    \| f^\varepsilon \cos( \tau \mathcal{A}^{1/2})  (f^\varepsilon)^{-1}  - f_0 \cos( \tau (\mathcal{A}^0)^{1/2}) f_0^{-1} \|_{H^s (\mathbb{R}^d) \to L_2 (\mathbb{R}^d)}
    \\
    \le \widehat{\mathfrak{C}}_1 (s; 1) \varepsilon^{s(1-\alpha)/2}, \quad |\tau| = \varepsilon^{-\alpha}, \; 0 < \alpha < 1, \; 0 < \varepsilon \le 1,
    \end{split}
    \\
    \begin{split}
    \|f^\varepsilon \mathcal{A}_\varepsilon^{-1/2} \sin( \tau \mathcal{A}_\varepsilon^{1/2}) (f^\varepsilon)^{-1} -  f_0 (\mathcal{A}^0)^{-1/2} \sin( \tau (\mathcal{A}^0)^{1/2}) f_0^{-1} \|_{H^s (\mathbb{R}^d) \to L_2 (\mathbb{R}^d)}
    \\
    \le \widehat{\mathfrak{C}}_2 (s; 1)  \varepsilon^{s(1-\alpha)/2-\alpha}, \quad
    |\tau| = \varepsilon^{-\alpha}, \; 0 < \alpha < s(s+2)^{-1}, \; 0 < \varepsilon \le 1.
    \end{split}
    \end{gather*}
\end{theorem}

\subsection{Refinement of approximation under the additional assumptions}

Using~\eqref{sndw_cos-cos_and_scale_transform} and Theorem~\ref{sndw_cos_thrm_2}, we deduce  the following result.

\begin{theorem}
    \label{sndw_cos_thrm2_L2L2}
    Suppose that the assumptions of Theorem~\emph{\ref{sndw_cos_thrm1_L2L2}} are satisfied. Let
    $\widehat{N}_Q (\boldsymbol{\theta})$  be the operator defined by~\emph{\eqref{N_Q(theta)}}. Suppose that
    $\widehat{N}_Q (\boldsymbol{\theta}) = 0$ for any $\boldsymbol{\theta} \in \mathbb{S}^{d-1}$.
    Then for $\tau \in \mathbb{R}$ and $\varepsilon > 0$ we have
    \begin{multline}
    \label{sndw_cos_est2_L2L2}
    \| (f^\varepsilon \cos( \tau \mathcal{A}_\varepsilon^{1/2})  (f^\varepsilon)^{-1}  - f_0 \cos( \tau (\mathcal{A}^0)^{1/2}) f_0^{-1} ) (\mathcal{H}_0 + I)^{-3/4} \|_{L_2(\mathbb{R}^d) \to L_2(\mathbb{R}^d)}
    \\
    \le (\mathcal{C}_3 + \mathcal{C}_4 |\tau|) \varepsilon.
    \end{multline}
    The constants $\mathcal{C}_3$ and $\mathcal{C}_4$ are defined in Theorem~\emph{\ref{sndw_cos_enchanced_thrm_1}} and depend only on $r_0$, $\alpha_0$, $\alpha_1$, $\|g\|_{L_\infty}$, $\|g^{-1}\|_{L_\infty}$, $\|f\|_{L_\infty}$, and $\|f^{-1}\|_{L_\infty}$.
\end{theorem}

Interpolating between~\eqref{sndw_cos-cos_le_2} and~\eqref{sndw_cos_est2_L2L2}, we obtain the following result.

\begin{theorem}
    \label{sndw_cos_sin_thrm2_H^s_L2}
    Suppose that the assumptions of Theorem~\emph{\ref{sndw_cos_thrm2_L2L2}} are satisfied. Then for $0 \le s \le 3/2$, $\tau \in \mathbb{R}$, and $\varepsilon > 0$ we have
    \begin{align*}
    \| f^\varepsilon \cos( \tau \mathcal{A}_\varepsilon^{1/2})  (f^\varepsilon)^{-1}  - f_0 \cos( \tau (\mathcal{A}^0)^{1/2}) f_0^{-1}  \|_{H^s (\mathbb{R}^d) \to L_2 (\mathbb{R}^d)} \le \mathfrak{C}_3 (s; \tau) \varepsilon^{2s/3},
    \\
    \| f^\varepsilon \mathcal{A}_\varepsilon^{-1/2} \sin( \tau \mathcal{A}_\varepsilon^{1/2}) (f^\varepsilon)^{-1} - f_0 (\mathcal{A}^0)^{-1/2} \sin( \tau (\mathcal{A}^0)^{1/2}) f_0^{-1} \|_{H^s (\mathbb{R}^d) \to L_2 (\mathbb{R}^d)}
    \\
    \le \mathfrak{C}_4 (s; \tau)  \varepsilon^{2s/3},
    \end{align*}
    where
    \begin{equation}
    \label{frakC_3_4}
    \begin{split}
    \mathfrak{C}_3 (s; \tau) &= (2 \| f \|_{L_\infty} \| f^{-1} \|_{L_\infty})^{1-2s/3} (\mathcal{C}_3 + \mathcal{C}_4  |\tau|)^{2s/3}, \\
    \mathfrak{C}_4 (s; \tau) &=  \mathfrak{C}_3 (s; \tau) (1 + 2s/3)^{-1}  (\mathcal{C}_3 \mathcal{C}_{4}^{-1} +  |\tau|).
    \end{split}
    \end{equation}
    In particular, for $0 < \varepsilon \le 1$ and $|\tau| = \varepsilon^{-\alpha}$ we have
    \begin{gather*}
    \begin{split}
    \|f^\varepsilon \cos( \tau \mathcal{A}_\varepsilon^{1/2})  (f^\varepsilon)^{-1}  - f_0 \cos( \tau (\mathcal{A}^0)^{1/2}) f_0^{-1} \|_{H^s (\mathbb{R}^d) \to L_2 (\mathbb{R}^d)}
     \\
     \le \mathfrak{C}_3 (s; 1) \varepsilon^{2s(1-\alpha)/3}, \quad |\tau| = \varepsilon^{-\alpha}, \; 0 < \alpha < 1, \; 0 < \varepsilon \le 1,
    \end{split}
    \\
    \begin{split}
    \|f^\varepsilon \mathcal{A}_\varepsilon^{-1/2} \sin( \tau \mathcal{A}_\varepsilon^{1/2}) (f^\varepsilon)^{-1} - f_0 (\mathcal{A}^0)^{-1/2} \sin( \tau (\mathcal{A}^0)^{1/2}) f_0^{-1} \|_{H^s (\mathbb{R}^d) \to L_2 (\mathbb{R}^d)}
    \\
    \le \mathfrak{C}_4 (s; 1)  \varepsilon^{2s(1-\alpha)/3-\alpha},
    \quad |\tau| = \varepsilon^{-\alpha}, \; 0 < \alpha < 2s(2s+3)^{-1}, \; 0 < \varepsilon \le 1.
    \end{split}
    \end{gather*}
\end{theorem}

From Theorem~\ref{sndw_cos_sin_thrm2_H^s_L2} and Proposition~\ref{N_Q=0_proposit} we deduce the following corollary.

\begin{corollary}
    Suppose that at least one of the following conditions is fulfilled:

    $1^\circ$.
         The operator $\mathcal{A}_\varepsilon$ has the form
         $\mathcal{A}_\varepsilon = (f^\varepsilon (\mathbf{x}))^* \mathbf{D}^* g^\varepsilon(\mathbf{x}) \mathbf{D} f^\varepsilon (\mathbf{x})$,
         where $g(\mathbf{x})$~is a symmetric matrix with real entries.

          $2^\circ$.
        Relations~\emph{\eqref{g0=overline_g_relat}} are satisfied, i.~e., $g^0 = \overline{g}$.

\noindent   Then the estimates of Theorem~\emph{\ref{sndw_cos_sin_thrm2_H^s_L2}} hold.
\end{corollary}

Finally, Theorem~\ref{sndw_cos_thrm_3} and~\eqref{sndw_cos-cos_and_scale_transform} imply the following result.

\begin{theorem}
    \label{sndw_cos_thrm3_L2L2}
    Suppose that the assumptions of Theorem~\emph{\ref{sndw_cos_thrm1_L2L2}} are satisfied. Suppose that Condition~\emph{\ref{sndw_cond1}} \emph{(}or more restrictive Condition~\emph{\ref{sndw_cond2})} is satisfied. Then for $\tau \in \mathbb{R}$ and $\varepsilon > 0$
    we have
    \begin{multline}
    \label{sndw_cos_est3_L2L2}
    \| (f^\varepsilon \cos( \tau \mathcal{A}_\varepsilon^{1/2})  (f^\varepsilon)^{-1}  - f_0 \cos( \tau (\mathcal{A}^0)^{1/2}) f_0^{-1} ) (\mathcal{H}_0 + I)^{-3/4} \|_{L_2(\mathbb{R}^d) \to L_2(\mathbb{R}^d)}
    \\
    \le (\mathcal{C}_5 + \mathcal{C}_6 |\tau|) \varepsilon.
    \end{multline}
    The constants $\mathcal{C}_5$ and $\mathcal{C}_6$ are defined in Theorem~\emph{\ref{sndw_cos_enchanced_thrm_2}}
    and depend only on $r_0$, $\alpha_0$, $\alpha_1$,  $\|g\|_{L_\infty}$, $\|g^{-1}\|_{L_\infty}$,  $\|f\|_{L_\infty}$, $\|f^{-1}\|_{L_\infty}$,
    and also on the number $c^{\circ}$ defined by~\emph{\eqref{c^circ}}.
\end{theorem}

Interpolating between~~\eqref{sndw_cos-cos_le_2} and \eqref{sndw_cos_est3_L2L2}, we obtain the following result.

\begin{theorem}
    \label{sndw_cos_sin_thrm3_H^s_L2}
   Suppose that the assumptions of Theorem~\emph{\ref{sndw_cos_thrm3_L2L2}} are satisfied. Then for $0 \le s \le 3/2$, $\tau \in \mathbb{R}$, and $\varepsilon > 0$ we have
   \begin{align*}
   \| f^\varepsilon \cos( \tau \mathcal{A}_\varepsilon^{1/2})  (f^\varepsilon)^{-1}  - f_0 \cos( \tau (\mathcal{A}^0)^{1/2}) f_0^{-1}  \|_{H^s (\mathbb{R}^d) \to L_2 (\mathbb{R}^d)} \le \mathfrak{C}_5 (s; \tau) \varepsilon^{2s/3},
   \\
   \| f^\varepsilon \mathcal{A}_\varepsilon^{-1/2} \sin( \tau \mathcal{A}_\varepsilon^{1/2}) (f^\varepsilon)^{-1} - f_0 (\mathcal{A}^0)^{-1/2} \sin( \tau (\mathcal{A}^0)^{1/2}) f_0^{-1} \|_{H^s (\mathbb{R}^d) \to L_2 (\mathbb{R}^d)}
   \\
   \le \mathfrak{C}_6 (s; \tau)  \varepsilon^{2s/3},
   \end{align*}
   where
   \begin{equation}
   \label{frakC_5_6}
   \begin{split}
   \mathfrak{C}_5 (s; \tau) &= (2 \| f \|_{L_\infty} \| f^{-1} \|_{L_\infty})^{1-2s/3} (\mathcal{C}_5 + \mathcal{C}_6  |\tau|)^{2s/3}, \\
   \mathfrak{C}_6 (s; \tau) &=  \mathfrak{C}_5 (s; \tau) (1 + 2s/3)^{-1}  (\mathcal{C}_5 \mathcal{C}_{6}^{-1} +  |\tau|).
   \end{split}
   \end{equation}
   In particular, for $0 < \varepsilon \le 1$ and $|\tau| = \varepsilon^{-\alpha}$ we have
   \begin{gather*}
   \begin{split}
   \|f^\varepsilon \cos( \tau \mathcal{A}_\varepsilon^{1/2})  (f^\varepsilon)^{-1}  - f_0 \cos( \tau (\mathcal{A}^0)^{1/2}) f_0^{-1} \|_{H^s (\mathbb{R}^d) \to L_2 (\mathbb{R}^d)} \\
   \le \mathfrak{C}_5 (s; 1) \varepsilon^{2s(1-\alpha)/3}, \quad |\tau| = \varepsilon^{-\alpha}, \; 0 < \alpha < 1, \; 0 < \varepsilon \le 1,
   \end{split}
   \\
   \begin{split}
   \|f^\varepsilon \mathcal{A}_\varepsilon^{-1/2} \sin( \tau \mathcal{A}_\varepsilon^{1/2}) (f^\varepsilon)^{-1} - f_0 (\mathcal{A}^0)^{-1/2} \sin( \tau (\mathcal{A}^0)^{1/2}) f_0^{-1} \|_{H^s (\mathbb{R}^d) \to L_2 (\mathbb{R}^d)}
   \\
   \le \mathfrak{C}_6 (s; 1)  \varepsilon^{2s(1-\alpha)/3-\alpha}, \quad |\tau| = \varepsilon^{-\alpha}, \; 0 < \alpha < 2s(2s+3)^{-1}, \; 0 < \varepsilon \le 1.
   \end{split}
   \end{gather*}
\end{theorem}

Combining Theorem~\ref{sndw_cos_sin_thrm3_H^s_L2} and Corollary~\ref{sndw_cos_enchanced_2_coroll}, we arrive at the following corollary.

\begin{corollary}
    Suppose that the matrices $b (\boldsymbol{\theta})$, $g (\mathbf{x})$, and
    $Q(\mathbf{x}) = (f(\mathbf{x})f(\mathbf{x})^*)^{-1}$~have real entries.
Suppose that the spectrum of the generalized spectral problem~\emph{\eqref{hatS_gener_spec_problem}}
is simple for any $\boldsymbol{\theta} \in \mathbb{S}^{d-1}$.
Then the estimates of Theorem~\emph{\ref{sndw_cos_sin_thrm3_H^s_L2}} hold.
\end{corollary}

\subsection{The sharpness of the result.}

Applying Theorem~\ref{sndw_s<2_cos_thrm},  we confirm the sharpness of the result of Theorem~\ref{sndw_cos_thrm1_L2L2} in the general case.

\begin{theorem}
    Suppose that the assumptions of Theorem~\emph{\ref{sndw_cos_thrm1_L2L2}} are satisfied. Let $\widehat{N}_{0,Q} (\boldsymbol{\theta})$
     be the operator defined by~\emph{\eqref{N0Q_invar_repr}}. Suppose that \hbox{$\widehat{N}_{0,Q} (\boldsymbol{\theta}_0) \ne 0$} for some  $\boldsymbol{\theta}_0 \in \mathbb{S}^{d-1}$. Let $0 \ne \tau \in \mathbb{R}$ and $0 \le s < 2$.
     Then there does not exist a constant $\mathcal{C}(\tau) > 0$ such that estimate
    \begin{multline}
    \label{sndw_s<2_cos_est1}
    \| ( f^\varepsilon \cos( \tau \mathcal{A}_\varepsilon^{1/2})  (f^\varepsilon)^{-1}  - f_0 \cos( \tau (\mathcal{A}^0)^{1/2}) f_0^{-1}  )(\mathcal{H}_0  + I)^{-s/2}\|_{L_2(\mathbb{R}^d) \to L_2 (\mathbb{R}^d) }
    \\
    \le \mathcal{C}(\tau) \varepsilon
    \end{multline}
    holds for all sufficiently small $\varepsilon > 0$.
\end{theorem}

 \begin{proof} We prove by contradiction. Let us fix $\tau \ne 0$.
 Suppose that for some $0 \le s < 2$ there exists a constant $\mathcal{C}(\tau) > 0$ such that
    estimate~\eqref{sndw_s<2_cos_est1} holds for all sufficiently small $\varepsilon > 0$.
    Then, by~\eqref{sndw_cos-cos_and_scale_transform}, estimate~\eqref{sndw_s<2_cos_est} also holds for sufficiently small $\varepsilon$.
    But this contradicts the statement of Theorem~\ref{sndw_s<2_cos_thrm}.
\end{proof}

\section{Homogenization of the Cauchy problem \\ for  hyperbolic equations}

\subsection{The Cauchy problem for the homogeneous equation with the operator $\widehat{\mathcal{A}}_\varepsilon$}

Let $\widehat{\mathcal{A}}_\varepsilon$~be the operator~\eqref{Ahat_eps}.
Let $\mathbf{v}_\varepsilon (\mathbf{x}, \tau), \, \mathbf{x} \in \mathbb{R}^d, \, \tau \in \mathbb{R}$,
be the solution of the Cauchy problem
\begin{equation}
\label{homog_Cauchy_hatA_eps}
\left\{
\begin{aligned}
&\frac{\partial^2 \mathbf{v}_\varepsilon (\mathbf{x}, \tau)}{\partial \tau^2} = - b(\mathbf{D})^* g^\varepsilon (\mathbf{x}) b(\mathbf{D}) \mathbf{v}_\varepsilon (\mathbf{x}, \tau), \\
& \mathbf{v}_\varepsilon (\mathbf{x}, 0) = \boldsymbol{\phi} (\mathbf{x}), \quad \frac{\partial \mathbf{v}_\varepsilon }{\partial \tau} (\mathbf{x}, 0) = \boldsymbol{\psi}(\mathbf{x}),
\end{aligned}
\right.
\end{equation}
where $\boldsymbol{\phi}, \boldsymbol{\psi} \in L_2 (\mathbb{R}^d; \mathbb{C}^n)$~are given functions.
The solution can be represented as
\begin{equation*}
\mathbf{v}_\varepsilon (\cdot, \tau) = \cos(\tau \widehat{\mathcal{A}}_\varepsilon^{1/2}) \boldsymbol{\phi} + \widehat{\mathcal{A}}_\varepsilon^{-1/2} \sin(\tau \widehat{\mathcal{A}}_\varepsilon^{1/2}) \boldsymbol{\psi}.
\end{equation*}
Let $\mathbf{v}_0 (\mathbf{x}, \tau)$~be the solution of the
\textquotedblleft homogenized\textquotedblright \ Cauchy problem
\begin{equation}
\label{homog_Cauchy_hatA_0}
\left\{
\begin{aligned}
&\frac{\partial^2 \mathbf{v}_0 (\mathbf{x}, \tau)}{\partial \tau^2} = - b(\mathbf{D})^* g^0   b(\mathbf{D}) \mathbf{v}_0 (\mathbf{x}, \tau), \\
& \mathbf{v}_0 (\mathbf{x}, 0) = \boldsymbol{\phi} (\mathbf{x}), \quad \frac{\partial \mathbf{v}_0 }{\partial \tau} (\mathbf{x}, 0) = \boldsymbol{\psi}(\mathbf{x}).
\end{aligned}
\right.
\end{equation}
Here $g^0$ is the effective matrix given by~\eqref{g0}. Then
\begin{equation*}
\mathbf{v}_0 (\cdot, \tau) = \cos(\tau (\widehat{\mathcal{A}}^0)^{1/2}) \boldsymbol{\phi} + (\widehat{\mathcal{A}}^0)^{-1/2} \sin(\tau (\widehat{\mathcal{A}}^0)^{1/2}) \boldsymbol{\psi}.
\end{equation*}

Theorem~\ref{cos_sin_thrm1_H^s_L2} directly implies the following result which has been proved before in~\cite[Theorem~15.1]{BSu5}.

\begin{theorem}
    \label{homog_Cauchy_hatA_eps_thrm}
    Let $\mathbf{v}_\varepsilon$~be the solution of problem~\emph{\eqref{homog_Cauchy_hatA_eps}}, and let $\mathbf{v}_0$~be the solution of problem~\emph{\eqref{homog_Cauchy_hatA_0}}.

    $1^{\circ}.$
           If $\boldsymbol{\phi}, \boldsymbol{\psi} \in H^s (\mathbb{R}^d; \mathbb{C}^n)$, $0 \le s \le 2$, then for $\tau \in \mathbb{R}$ and $\varepsilon > 0$ we have
         \begin{equation*}
         \| \mathbf{v}_\varepsilon (\cdot, \tau) - \mathbf{v}_0 (\cdot, \tau) \|_{L_2(\mathbb{R}^d)} \le \varepsilon^{s/2} \left( \widehat{\mathfrak{C}}_1 (s; \tau) \| \boldsymbol{\phi} \|_{H^s(\mathbb{R}^d)} + \widehat{\mathfrak{C}}_2 (s; \tau) \| \boldsymbol{\psi} \|_{H^s(\mathbb{R}^d)} \right).
         \end{equation*}
         In particular, for $0 < \varepsilon \le 1$ and $\tau = \pm \varepsilon^{-\alpha}$
         \begin{multline}
         \label{homog_Cauchy_hatA_eps_est_largetime1}
         \| \mathbf{v}_\varepsilon (\cdot, \pm \varepsilon^{-\alpha}) - \mathbf{v}_0 (\cdot, \pm \varepsilon^{-\alpha}) \|_{L_2(\mathbb{R}^d)}
\\
\le \varepsilon^{s(1 - \alpha)/2} \left( \widehat{\mathfrak{C}}_1 (s; 1) \| \boldsymbol{\phi} \|_{H^s(\mathbb{R}^d)} + \varepsilon^{-\alpha} \widehat{\mathfrak{C}}_2 (s; 1) \| \boldsymbol{\psi} \|_{H^s(\mathbb{R}^d)} \right).
         \end{multline}
         Here $0 < \alpha < s(s+2)^{-1}$ if $\boldsymbol{\psi} \ne 0$, and $0 < \alpha < 1$ if $\boldsymbol{\psi} = 0$.
         The constants $\widehat{\mathfrak{C}}_1 (s; \tau), \widehat{\mathfrak{C}}_2 (s; \tau)$
         are defined by~\emph{\eqref{hat_frakC_1_2}}.

      $2^{\circ}.$
  If $\boldsymbol{\phi}, \boldsymbol{\psi} \in L_2 (\mathbb{R}^d; \mathbb{C}^n)$, then
         \begin{equation*}
         \lim\limits_{\varepsilon \to 0} \| \mathbf{v}_\varepsilon (\cdot, \tau) - \mathbf{v}_0 (\cdot, \tau) \|_{L_2(\mathbb{R}^d)} = 0, \quad \tau \in \mathbb{R}.
         \end{equation*}

    $3^{\circ}.$
         If $\boldsymbol{\phi} \in L_2 (\mathbb{R}^d; \mathbb{C}^n)$ and $\boldsymbol{\psi} = 0$, then
         \begin{equation*}
         \lim\limits_{\varepsilon \to 0} \| \mathbf{v}_\varepsilon (\cdot, \pm \varepsilon^{-\alpha}) - \mathbf{v}_0 (\cdot, \pm \varepsilon^{-\alpha}) \|_{L_2(\mathbb{R}^d)} = 0, \quad 0 < \alpha < 1.
         \end{equation*}
\end{theorem}

Statement $2^\circ$ follows directly from statement $1^\circ$ and the obvious estimate
 \begin{equation*}
\| \mathbf{v}_\varepsilon (\cdot, \tau) - \mathbf{v}_0 (\cdot, \tau) \|_{L_2(\mathbb{R}^d)} \le 2 \| \boldsymbol{\phi} \|_{L_2 (\mathbb{R}^d)} +  2 |\tau| \| \boldsymbol{\psi} \|_{L_2(\mathbb{R}^d)},
\end{equation*}
by the Banach-Steinhaus theorem.
Statement~$3^\circ$ is deduced from~\eqref{homog_Cauchy_hatA_eps_est_largetime1}
(with $\boldsymbol{\psi}=0$ and $0 < \alpha < 1$) and the obvious estimate
\begin{equation*}
\| \mathbf{v}_\varepsilon (\cdot, \pm \varepsilon^{-\alpha}) - \mathbf{v}_0 (\cdot, \pm \varepsilon^{-\alpha}) \|_{L_2(\mathbb{R}^d)} \le 2 \| \boldsymbol{\phi} \|_{L_2 (\mathbb{R}^d)},
\end{equation*}
also by the Banach-Steinhaus theorem.

Statement $1^\circ$ can be refined under the additional assumptions.
Theorem~\ref{cos_sin_thrm2_H^s_L2} implies the following result.

\begin{theorem}
    \label{homog_Cauchy_hatA_eps_ench_thrm_1}
    Suppose that the assumptions of Theorem~\emph{\ref{homog_Cauchy_hatA_eps_thrm}} are satisfied. Let $\widehat{N}(\boldsymbol{\theta})$ be the operator defined by~\emph{\eqref{N(theta)}}. Suppose that $\widehat{N}(\boldsymbol{\theta}) = 0$ for any $\boldsymbol{\theta} \in \mathbb{S}^{d-1}$. If $\boldsymbol{\phi}, \boldsymbol{\psi} \in H^s (\mathbb{R}^d, \mathbb{C}^n)$, $0 \le s \le 3/2$,
    then for $\tau \in \mathbb{R}$ and $\varepsilon > 0$ we have
    \begin{equation*}
     \| \mathbf{v}_\varepsilon (\cdot, \tau) - \mathbf{v}_0 (\cdot, \tau) \|_{L_2(\mathbb{R}^d)} \le \varepsilon^{2s/3} \left( \widehat{\mathfrak{C}}_3 (s; \tau) \| \boldsymbol{\phi} \|_{H^s(\mathbb{R}^d)} + \widehat{\mathfrak{C}}_4 (s; \tau) \| \boldsymbol{\psi} \|_{H^s(\mathbb{R}^d)} \right).
    \end{equation*}
    In particular, for $0 < \varepsilon \le 1$ and $\tau = \pm \varepsilon^{-\alpha}$ we have
    \begin{multline}
    \label{homog_Cauchy_hatA_eps_est_largetime2}
    \| \mathbf{v}_\varepsilon (\cdot, \pm \varepsilon^{-\alpha}) - \mathbf{v}_0 (\cdot, \pm \varepsilon^{-\alpha}) \|_{L_2(\mathbb{R}^d)}
\\
\le \varepsilon^{2s(1 - \alpha)/3} \left( \widehat{\mathfrak{C}}_3 (s; 1) \| \boldsymbol{\phi} \|_{H^s(\mathbb{R}^d)} + \varepsilon^{-\alpha} \widehat{\mathfrak{C}}_4 (s; 1) \| \boldsymbol{\psi} \|_{H^s(\mathbb{R}^d)} \right).
    \end{multline}
    Here $0 < \alpha < 2s(2s+3)^{-1}$ if $\boldsymbol{\psi} \ne 0$, and $0 < \alpha < 1$ if $\boldsymbol{\psi} = 0$. The constants  $\widehat{\mathfrak{C}}_3 (s; \tau), \widehat{\mathfrak{C}}_4 (s; \tau)$ are defined by~\emph{\eqref{hat_frakC_3_4}}.
\end{theorem}

Finally, from Theorem~\ref{cos_sin_thrm3_H^s_L2} we deduce the following statement.

\begin{theorem}
    Suppose that the assumptions of Theorem~\emph{\ref{homog_Cauchy_hatA_eps_thrm}} are satisfied. Suppose also that Condition~\emph{\ref{cond9}} \emph{(}or more restrictive Condition~\emph{\ref{cond99})} is satisfied. If $\boldsymbol{\phi}, \boldsymbol{\psi} \in H^s (\mathbb{R}^d, \mathbb{C}^n)$, $0 \le s \le 3/2$, then for $\tau \in \mathbb{R}$ и $\varepsilon > 0$ we have
    \begin{equation*}
    \| \mathbf{v}_\varepsilon (\cdot, \tau) - \mathbf{v}_0 (\cdot, \tau) \|_{L_2(\mathbb{R}^d)} \le \varepsilon^{2s/3} \left( \widehat{\mathfrak{C}}_5 (s; \tau) \| \boldsymbol{\phi} \|_{H^s(\mathbb{R}^d)} + \widehat{\mathfrak{C}}_6 (s; \tau) \| \boldsymbol{\psi} \|_{H^s(\mathbb{R}^d)} \right).
    \end{equation*}
    In particular, for $0 < \varepsilon \le 1$ and $\tau = \pm \varepsilon^{-\alpha}$ we have
    \begin{multline*}
    \| \mathbf{v}_\varepsilon (\cdot, \pm \varepsilon^{-\alpha}) - \mathbf{v}_0 (\cdot, \pm \varepsilon^{-\alpha}) \|_{L_2(\mathbb{R}^d)}
\\
 \le \varepsilon^{2s(1 - \alpha)/3} \left( \widehat{\mathfrak{C}}_5 (s; 1) \| \boldsymbol{\phi} \|_{H^s(\mathbb{R}^d)} + \varepsilon^{-\alpha} \widehat{\mathfrak{C}}_6(s; 1) \| \boldsymbol{\psi} \|_{H^s(\mathbb{R}^d)} \right).
    \end{multline*}
    Here $0 < \alpha < 2s(2s+3)^{-1}$ if $\boldsymbol{\psi} \ne 0$, and $0 < \alpha < 1$ if $\boldsymbol{\psi} = 0$.
    The constants $\widehat{\mathfrak{C}}_5 (s; \tau), \widehat{\mathfrak{C}}_6 (s; \tau)$ are defined by~\emph{\eqref{hat_frakC_5_6}}.
\end{theorem}

\subsection{The Cauchy problem for the nonhomogeneous equation with the operator $\widehat{\mathcal{A}}_\varepsilon$}

Now we consider the Cauchy problem for the nonhomogeneous equation
\begin{equation}
\label{nonhomog_Cauchy_hatA_eps}
\left\{
\begin{aligned}
&\frac{\partial^2 \mathbf{v}_\varepsilon (\mathbf{x}, \tau)}{\partial \tau^2} = - b(\mathbf{D})^* g^\varepsilon (\mathbf{x}) b(\mathbf{D}) \mathbf{v}_\varepsilon (\mathbf{x}, \tau) + \mathbf{F} (\mathbf{x}, \tau), \\
& \mathbf{v}_\varepsilon (\mathbf{x}, 0) = \boldsymbol{\phi} (\mathbf{x}), \quad \frac{\partial \mathbf{v}_\varepsilon }{\partial \tau} (\mathbf{x}, 0) = \boldsymbol{\psi}(\mathbf{x}),
\end{aligned}
\right.
\end{equation}
where $\boldsymbol{\phi}, \boldsymbol{\psi} \in L_2 (\mathbb{R}^d, \mathbb{C}^n), \, \mathbf{F} \in L_{1, \mathrm{loc}} (\mathbb{R}; L_2 (\mathbb{R}^d, \mathbb{C}^n) )$~are given functions.
 The solution of this problem can be represented as
 \begin{multline*}
\mathbf{v}_\varepsilon (\cdot, \tau) = \cos(\tau \widehat{\mathcal{A}}_\varepsilon^{1/2}) \boldsymbol{\phi} + \widehat{\mathcal{A}}_\varepsilon^{-1/2} \sin(\tau \widehat{\mathcal{A}}_\varepsilon^{1/2}) \boldsymbol{\psi}
\\
+ \int_{0}^{\tau} \widehat{\mathcal{A}}_\varepsilon^{-1/2} \sin((\tau - \tilde{\tau}) \widehat{\mathcal{A}}_\varepsilon^{1/2}) \mathbf{F} (\cdot, \tilde{\tau}) \, d \tilde{\tau}.
\end{multline*}
Let $\mathbf{v}_0 (\mathbf{x}, \tau)$~be the solution of the homogenized problem
\begin{equation}
\label{nonhomog_Cauchy_hatA_0}
\left\{
\begin{aligned}
&\frac{\partial^2 \mathbf{v}_0 (\mathbf{x}, \tau)}{\partial \tau^2} = - b(\mathbf{D})^* g^0 b(\mathbf{D}) \mathbf{v}_0 (\mathbf{x}, \tau) + \mathbf{F} (\mathbf{x}, \tau) , \\
& \mathbf{v}_0 (\mathbf{x}, 0) = \boldsymbol{\phi} (\mathbf{x}), \quad \frac{\partial \mathbf{v}_0 }{\partial \tau} (\mathbf{x}, 0) = \boldsymbol{\psi}(\mathbf{x}).
\end{aligned}
\right.
\end{equation}
Then
\begin{multline*}
\mathbf{v}_0 (\cdot, \tau) = \cos(\tau (\widehat{\mathcal{A}}^0)^{1/2}) \boldsymbol{\phi} + (\widehat{\mathcal{A}}^0)^{-1/2} \sin(\tau (\widehat{\mathcal{A}}^0)^{1/2}) \boldsymbol{\psi}
\\
+ \int_{0}^{\tau} (\widehat{\mathcal{A}}^0)^{-1/2} \sin((\tau - \tilde{\tau}) (\widehat{\mathcal{A}}^0)^{1/2}) \mathbf{F} (\cdot, \tilde{\tau}) \, d \tilde{\tau}.
\end{multline*}

From Theorem~\ref{cos_sin_thrm1_H^s_L2}  we deduce the following result (which has been proved before in~\cite[Theorem~15.2]{BSu5}).

\begin{theorem}
    \label{nonhomog_Cauchy_hatA_eps_thrm}
    Let $\mathbf{v}_\varepsilon$~be the solution of problem~\emph{\eqref{nonhomog_Cauchy_hatA_eps}}, and let $\mathbf{v}_0$~be the solution of problem~\emph{\eqref{nonhomog_Cauchy_hatA_0}}.

    $1^{\circ}$.
         If $\boldsymbol{\phi}, \boldsymbol{\psi} \in H^s (\mathbb{R}^d; \mathbb{C}^n)$ and
$\mathbf{F} \in L_{1, \mathrm{loc}} (\mathbb{R}; H^s (\mathbb{R}^d; \mathbb{C}^n))$, where $0 \le s \le 2$, then for $\tau \in \mathbb{R}$ и $\varepsilon > 0$ we have
        \begin{multline}
        \label{nonhomog_Cauchy_hatA_eps_est1}
        \| \mathbf{v}_\varepsilon(\cdot, \tau) - \mathbf{v}_0 (\cdot, \tau) \|_{L_2 (\mathbb{R}^d)} \le
        \varepsilon^{s/2} \left( \widehat{\mathfrak{C}}_1 (s; \tau) \| \boldsymbol{\phi} \|_{H^s(\mathbb{R}^d)} + \widehat{\mathfrak{C}}_2 (s; \tau) \| \boldsymbol{\psi} \|_{H^s(\mathbb{R}^d)} \right)
\\
+  \varepsilon^{s/2} \widehat{\mathfrak{C}}_2 (s; \tau) \|\mathbf{F} \|_{L_1((0,\tau);H^s(\mathbb{R}^d))}.
        \end{multline}
        Under the additional assumption that $\mathbf{F} \in L_p (\mathbb{R}_\pm; H^s (\mathbb{R}^d; \mathbb{C}^n) )$ \emph{(}where \hbox{$1 \le p \le \infty$}\emph{),} for $\tau = \pm \varepsilon^{-\alpha}$ and $0 < \varepsilon \le 1$ we have
        \begin{multline}
        \label{nonhomog_Cauchy_hatA_eps_est_largetime1}
        \| \mathbf{v}_\varepsilon (\cdot, \pm \varepsilon^{-\alpha}) - \mathbf{v}_0 (\cdot, \pm \varepsilon^{-\alpha}) \|_{L_2(\mathbb{R}^d)}
\\
\le \varepsilon^{s(1 - \alpha)/2} \left( \widehat{\mathfrak{C}}_1 (s; 1) \| \boldsymbol{\phi} \|_{H^s(\mathbb{R}^d)} +  \varepsilon^{-\alpha} \widehat{\mathfrak{C}}_2 (s; 1) \| \boldsymbol{\psi} \|_{H^s(\mathbb{R}^d)} \right)
\\
+ \varepsilon^{s(1 - \alpha)/2-\alpha-\alpha/p'} \widehat{\mathfrak{C}}_2 (s; 1) \|\mathbf{F} \|_{L_p(\mathbb{R}_\pm; H^s(\mathbb{R}^d))}.
        \end{multline}
        Here $p^{-1} + (p')^{-1} = 1$ and $0 < \alpha < s(s + 2 + 2/p')^{-1}$.
        The constants $\widehat{\mathfrak{C}}_1 (s; \tau)$ and $\widehat{\mathfrak{C}}_2 (s; \tau)$ are defined by~\emph{\eqref{hat_frakC_1_2}}.

    $2^{\circ}$.  If $\boldsymbol{\phi}, \boldsymbol{\psi} \in L_2 (\mathbb{R}^d; \mathbb{C}^n)$ and $\mathbf{F} \in L_{1, \mathrm{loc}} (\mathbb{R}; L_2 (\mathbb{R}^d; \mathbb{C}^n) )$, then
        \begin{equation*}
        \lim\limits_{\varepsilon \to 0} \| \mathbf{v}_\varepsilon (\cdot, \tau) - \mathbf{v}_0 (\cdot, \tau) \|_{L_2(\mathbb{R}^d)} = 0, \quad \tau \in \mathbb{R}.
        \end{equation*}
\end{theorem}

\begin{proof}
Estimate~\eqref{nonhomog_Cauchy_hatA_eps_est1} follows from ~\eqref{cos_thrm1_H^s_L2_est} and \eqref{sin_thrm1_H^s_L2_est}.

Under the assumption that $\mathbf{F} \in L_p (\mathbb{R}_\pm; H^s (\mathbb{R}^d; \mathbb{C}^n) )$, estimate~\eqref{nonhomog_Cauchy_hatA_eps_est_largetime1}
is deduced from~\eqref{cos_thrm1_H^s_L2_est_largetime} and \eqref{sin_thrm1_H^s_L2_est_largetime}.

Statement~$2^\circ$ follows from~\eqref{nonhomog_Cauchy_hatA_eps_est1} and the obvious estimate
\begin{equation*}
\| \mathbf{v}_\varepsilon (\cdot, \tau) - \mathbf{v}_0 (\cdot, \tau) \|_{L_2(\mathbb{R}^d)} \le 2 \| \boldsymbol{\phi} \|_{L_2 (\mathbb{R}^d)} +
  2 |\tau| \bigl(\| \boldsymbol{\psi} \|_{L_2(\mathbb{R}^d)} + \| \mathbf{F} \|_{L_{1} ((0,\tau); L_2 (\mathbb{R}^d) )}\bigr),
\end{equation*}
by the Banach-Steinhaus theorem.
\end{proof}

Statement $1^\circ$ of Theorem~\ref{nonhomog_Cauchy_hatA_eps_thrm} can be refined under the additional assumptions.
Theorem~\ref{cos_sin_thrm2_H^s_L2}  implies the following result.

\begin{theorem}
    \label{nonhomog_Cauchy_hatA_eps_ench_thrm_1}
    Suppose that the assumptions of Theorem~\emph{\ref{nonhomog_Cauchy_hatA_eps_thrm}} are satisfied. Let $\widehat{N}(\boldsymbol{\theta})$ be the operator defined by~\emph{\eqref{N(theta)}}. Suppose that $\widehat{N}(\boldsymbol{\theta}) = 0$ for any $\boldsymbol{\theta} \in \mathbb{S}^{d-1}$. If $\boldsymbol{\phi}, \boldsymbol{\psi} \in H^s (\mathbb{R}^d; \mathbb{C}^n)$ and  $\mathbf{F} \in L_{1, \mathrm{loc}} (\mathbb{R}; H^s (\mathbb{R}^d; \mathbb{C}^n))$, $0 \le s \le 3/2$, then for $\tau \in \mathbb{R}$ and $\varepsilon > 0$ we have
    \begin{multline*}
    \| \mathbf{v}_\varepsilon (\cdot, \tau) - \mathbf{v}_0 (\cdot, \tau) \|_{L_2(\mathbb{R}^d)} \le \varepsilon^{2s/3} \left( \widehat{\mathfrak{C}}_3 (s; \tau) \| \boldsymbol{\phi} \|_{H^s(\mathbb{R}^d)} + \widehat{\mathfrak{C}}_4 (s; \tau) \| \boldsymbol{\psi} \|_{H^s(\mathbb{R}^d)} \right)
\\
+ \varepsilon^{2s/3} \widehat{\mathfrak{C}}_4 (s; \tau) \|\mathbf{F} \|_{L_1((0,\tau);H^s(\mathbb{R}^d))}.
    \end{multline*}
    Under the additional assumption that $\mathbf{F} \in L_p (\mathbb{R}_\pm; H^s (\mathbb{R}^d; \mathbb{C}^n) )$ \emph{(}where
\hbox{$1 \le p \le \infty$}\emph{)}, for $\tau = \pm \varepsilon^{-\alpha}$ and $0 < \varepsilon \le 1$ we have
    \begin{multline*}
    \| \mathbf{v}_\varepsilon (\cdot, \pm \varepsilon^{-\alpha}) - \mathbf{v}_0 (\cdot, \pm \varepsilon^{-\alpha}) \|_{L_2(\mathbb{R}^d)}
\\
\le \varepsilon^{2s(1 - \alpha)/3} \left( \widehat{\mathfrak{C}}_3 (s; 1) \| \boldsymbol{\phi} \|_{H^s(\mathbb{R}^d)} +  \varepsilon^{-\alpha} \widehat{\mathfrak{C}}_4 (s; 1) \| \boldsymbol{\psi} \|_{H^s(\mathbb{R}^d)} \right)
\\
+ \varepsilon^{2s(1 - \alpha)/3-\alpha-\alpha/p'} \widehat{\mathfrak{C}}_4 (s; 1) \|\mathbf{F} \|_{L_p(\mathbb{R}_\pm; H^s(\mathbb{R}^d))}.
    \end{multline*}
    Here $p^{-1} + (p')^{-1} = 1$ and $0 < \alpha < 2s(2s + 3 + 3/p')^{-1}$. The constants $\widehat{\mathfrak{C}}_3 (s; \tau), \widehat{\mathfrak{C}}_4 (s; \tau)$ are defined by~\emph{\eqref{hat_frakC_3_4}}.
\end{theorem}

Finally, Theorem~\ref{cos_sin_thrm3_H^s_L2} implies the following result.

\begin{theorem}
    Suppose that the assumptions of Theorem~\emph{\ref{nonhomog_Cauchy_hatA_eps_thrm}} are satisfied. Suppose also that
    Condition~\emph{\ref{cond9}} \emph{(}or more restrictive condition~\emph{\ref{cond99})} is satisfied. If $\boldsymbol{\phi}, \boldsymbol{\psi} \in H^s (\mathbb{R}^d; \mathbb{C}^n)$ and $\mathbf{F} \in L_{1, \mathrm{loc}} (\mathbb{R}; H^s (\mathbb{R}^d; \mathbb{C}^n))$, $0~\le~s~\le 3/2$, then for $\tau \in \mathbb{R}$ and $\varepsilon > 0$ we have
    \begin{multline*}
    \| \mathbf{v}_\varepsilon (\cdot, \tau) - \mathbf{v}_0 (\cdot, \tau) \|_{L_2(\mathbb{R}^d)}  \le
\varepsilon^{2s/3} \left( \widehat{\mathfrak{C}}_5 (s; \tau) \| \boldsymbol{\phi} \|_{H^s(\mathbb{R}^d)} + \widehat{\mathfrak{C}}_6 (s; \tau) \| \boldsymbol{\psi} \|_{H^s(\mathbb{R}^d)}\right)
\\
+ \varepsilon^{2s/3} \widehat{\mathfrak{C}}_6 (s; \tau) \|\mathbf{F} \|_{L_1((0,\tau);H^s(\mathbb{R}^d))}.
    \end{multline*}
    Under the additional assumption that $\mathbf{F} \in L_p (\mathbb{R}_\pm; H^s (\mathbb{R}^d; \mathbb{C}^n) )$ \emph{(}where
\hbox{$1 \le p \le \infty$}\emph{)}, for $\tau = \pm \varepsilon^{-\alpha}$ and $0 < \varepsilon \le 1$ we have
    \begin{multline*}
    \| \mathbf{v}_\varepsilon (\cdot, \pm \varepsilon^{-\alpha}) - \mathbf{v}_0 (\cdot, \pm \varepsilon^{-\alpha}) \|_{L_2(\mathbb{R}^d)} \le \\ \le \varepsilon^{2s(1 - \alpha)/3} \left( \widehat{\mathfrak{C}}_5 (s; 1) \| \boldsymbol{\phi} \|_{H^s(\mathbb{R}^d)} +  \varepsilon^{-\alpha} \widehat{\mathfrak{C}}_6 (s; 1) \| \boldsymbol{\psi} \|_{H^s(\mathbb{R}^d)} \right)
\\
+ \varepsilon^{2s(1 - \alpha)/3-\alpha-\alpha/p'} \widehat{\mathfrak{C}}_6 (s; 1) \|\mathbf{F} \|_{L_p(\mathbb{R}_\pm; H^s(\mathbb{R}^d))}.
    \end{multline*}
    Here $p^{-1} + (p')^{-1} = 1$ and $0 < \alpha < 2s(2s + 3 + 3/p')^{-1}$. The constants $\widehat{\mathfrak{C}}_5 (s; \tau), \widehat{\mathfrak{C}}_6 (s; \tau)$ are defined by~\emph{(\ref{hat_frakC_5_6})}.
\end{theorem}

\subsection{The Cauchy problem for the homogeneous equation with the operator $\mathcal{A}_\varepsilon$}

Let $\mathcal{A}_\varepsilon$~be the operator given by~\eqref{A_eps}.
Let $\mathbf{v}_\varepsilon (\mathbf{x}, \tau)$,  $\mathbf{x} \in \mathbb{R}^d$, $\tau \in \mathbb{R}$, be the solution of
the Cauchy problem
\begin{equation}
    \label{homog_Cauchy_A_eps}
    \left\{
    \begin{aligned}
        &\frac{\partial^2 \mathbf{v}_\varepsilon (\mathbf{x}, \tau)}{\partial \tau^2} = -(f^\varepsilon (\mathbf{x}))^* b(\mathbf{D})^* g^\varepsilon (\mathbf{x}) b(\mathbf{D}) f^\varepsilon (\mathbf{x}) \mathbf{v}_\varepsilon (\mathbf{x}, \tau), \\
        & f^\varepsilon (\mathbf{x}) \mathbf{v}_\varepsilon (\mathbf{x}, 0) = \boldsymbol{\phi} (\mathbf{x}), \quad f^\varepsilon (\mathbf{x}) \frac{\partial \mathbf{v}_\varepsilon }{\partial \tau} (\mathbf{x}, 0) = \boldsymbol{\psi}(\mathbf{x}),
    \end{aligned}
    \right.
\end{equation}
where $\boldsymbol{\phi}, \boldsymbol{\psi} \in L_2 (\mathbb{R}^d; \mathbb{C}^n)$~are given functions.
The solution can be represented as
\begin{equation*}
    \mathbf{v}_\varepsilon (\cdot, \tau) = \cos(\tau \mathcal{A}_\varepsilon^{1/2}) (f^\varepsilon)^{-1} \boldsymbol{\phi} + \mathcal{A}_\varepsilon^{-1/2} \sin(\tau \mathcal{A}_\varepsilon^{1/2}) (f^\varepsilon)^{-1} \boldsymbol{\psi}.
\end{equation*}
Let $\mathbf{v}_0 (\mathbf{x}, \tau)$~be the solution of the ``homogenized'' Cauchy problem
\begin{equation}
    \label{homog_Cauchy_A_0}
    \left\{
    \begin{aligned}
        &\frac{\partial^2 \mathbf{v}_0 (\mathbf{x}, \tau)}{\partial \tau^2} = - f_0 b(\mathbf{D})^* g^0 b(\mathbf{D}) f_0 \mathbf{v}_0 (\mathbf{x}, \tau), \\
        & f_0 \mathbf{v}_0 (\mathbf{x}, 0) = \boldsymbol{\phi} (\mathbf{x}), \quad f_0 \frac{\partial \mathbf{v}_0 }{\partial \tau} (\mathbf{x}, 0) = \boldsymbol{\psi}(\mathbf{x}),
    \end{aligned}
    \right.
\end{equation}
where $g^0$ is the effective matrix~\eqref{g0}, and $f_0$ is defined by~\eqref{f_0}. Then
\begin{equation*}
    \mathbf{v}_0 (\cdot, \tau) = \cos(\tau (\mathcal{A}^0)^{1/2}) f_0^{-1} \boldsymbol{\phi} + (\mathcal{A}^0)^{-1/2} \sin(\tau (\mathcal{A}^0)^{1/2}) f_0^{-1} \boldsymbol{\psi}.
\end{equation*}

Theorem~\ref{sndw_cos_sin_thrm1_H^s_L2} implies the following result (which has been proved before in~\cite[Theorem~15.3]{BSu5}).

\begin{theorem}
    \label{homog_Cauchy_A_eps_thrm_1}
    Let $\mathbf{v}_\varepsilon$~be the solution of problem~\emph{\eqref{homog_Cauchy_A_eps}}, and let $\mathbf{v}_0$~be the solution of
    problem~\emph{\eqref{homog_Cauchy_A_0}}.

    $1^{\circ}$.
        If $\boldsymbol{\phi}, \boldsymbol{\psi} \in H^s (\mathbb{R}^d; \mathbb{C}^n)$, $0 \le s \le 2$, then for
        $\tau \in \mathbb{R}$ and $\varepsilon > 0$ we have
        \begin{multline*}
        \| f^\varepsilon \mathbf{v}_\varepsilon (\cdot, \tau) - f_0 \mathbf{v}_0 (\cdot, \tau) \|_{L_2(\mathbb{R}^d)}
\\
\le \varepsilon^{s/2} \left( \mathfrak{C}_1 (s; \tau) \| \boldsymbol{\phi} \|_{H^s(\mathbb{R}^d)} + \mathfrak{C}_2 (s; \tau) \| \boldsymbol{\psi} \|_{H^s(\mathbb{R}^d)} \right).
        \end{multline*}
        In particular, for $0 < \varepsilon \le 1$ and $\tau = \pm \varepsilon^{-\alpha}$
        \begin{multline*}
        \| f^\varepsilon \mathbf{v}_\varepsilon (\cdot, \pm \varepsilon^{-\alpha}) - f_0 \mathbf{v}_0 (\cdot, \pm \varepsilon^{-\alpha}) \|_{L_2(\mathbb{R}^d)} \le \\ \le \varepsilon^{s(1 - \alpha)/2} \left( \mathfrak{C}_1 (s; 1) \| \boldsymbol{\phi} \|_{H^s(\mathbb{R}^d)} + \varepsilon^{-\alpha} \mathfrak{C}_2 (s; 1) \| \boldsymbol{\psi} \|_{H^s(\mathbb{R}^d)} \right).
        \end{multline*}
        Here $0 < \alpha < s(s+2)^{-1}$ if $\boldsymbol{\psi} \ne 0$, and $0 < \alpha < 1$ if $\boldsymbol{\psi} = 0$.
        The constants $\mathfrak{C}_1 (s; \tau), \mathfrak{C}_2 (s; \tau)$ are defined by~\emph{\eqref{frakC_1_2}}.

   $2^{\circ}$.
        If $\boldsymbol{\phi}, \boldsymbol{\psi} \in L_2 (\mathbb{R}^d; \mathbb{C}^n)$, then
        \begin{equation*}
            \lim\limits_{\varepsilon \to 0} \| f^\varepsilon \mathbf{v}_\varepsilon (\cdot, \tau) - f_0 \mathbf{v}_0 (\cdot, \tau) \|_{L_2(\mathbb{R}^d)} = 0, \quad \tau \in \mathbb{R}.
        \end{equation*}

  $3^{\circ}$.
        If $\boldsymbol{\phi} \in L_2 (\mathbb{R}^d; \mathbb{C}^n)$ and $\boldsymbol{\psi} = 0$, then
        \begin{equation*}
            \lim\limits_{\varepsilon \to 0} \| f^\varepsilon \mathbf{v}_\varepsilon (\cdot, \pm \varepsilon^{-\alpha}) - f_0 \mathbf{v}_0 (\cdot, \pm \varepsilon^{-\alpha}) \|_{L_2(\mathbb{R}^d)} = 0, \quad 0 < \alpha < 1.
        \end{equation*}
\end{theorem}

Statement $1^\circ$ can be refined under the additional assumptions. Theorem~\ref{sndw_cos_sin_thrm2_H^s_L2}
implies the following result.

\begin{theorem}
    Suppose that the assumptions of Theorem~\emph{\ref{homog_Cauchy_A_eps_thrm_1}} are satisfied. Let $\widehat{N}_Q (\boldsymbol{\theta})$
    be the operator defined by~\emph{\eqref{N_Q(theta)}}. Suppose that $\widehat{N}_Q (\boldsymbol{\theta}) = 0$ for any $\boldsymbol{\theta} \in \mathbb{S}^{d-1}$. If $\boldsymbol{\phi}, \boldsymbol{\psi} \in H^s (\mathbb{R}^d; \mathbb{C}^n)$, $0 \le s \le 3/2$, then for $\tau \in \mathbb{R}$ and $\varepsilon > 0$ we have
    \begin{multline*}
        \| f^\varepsilon \mathbf{v}_\varepsilon (\cdot, \tau) - f_0 \mathbf{v}_0 (\cdot, \tau) \|_{L_2(\mathbb{R}^d)}
\\
\le \varepsilon^{2s/3} \left( \mathfrak{C}_3 (s; \tau) \| \boldsymbol{\phi} \|_{H^s(\mathbb{R}^d)} + \mathfrak{C}_4 (s; \tau) \| \boldsymbol{\psi} \|_{H^s(\mathbb{R}^d)} \right).
    \end{multline*}
    In particular, for $0 < \varepsilon \le 1$ and $\tau = \pm \varepsilon^{-\alpha}$
    \begin{multline*}
        \|  f^\varepsilon \mathbf{v}_\varepsilon (\cdot, \pm \varepsilon^{-\alpha}) - f_0 \mathbf{v}_0 (\cdot, \pm \varepsilon^{-\alpha}) \|_{L_2(\mathbb{R}^d)}
\\
\le \varepsilon^{2s(1 - \alpha)/3} \left( \mathfrak{C}_3 (s; 1) \| \boldsymbol{\phi} \|_{H^s(\mathbb{R}^d)} + \varepsilon^{-\alpha} \mathfrak{C}_4 (s; 1) \| \boldsymbol{\psi} \|_{H^s(\mathbb{R}^d)} \right).
    \end{multline*}
    Here $0 < \alpha < 2s(2s+3)^{-1}$ if $\boldsymbol{\psi} \ne 0$, and $0 < \alpha < 1$ if $\boldsymbol{\psi} = 0$.
    The constants $\mathfrak{C}_3 (s; \tau), \mathfrak{C}_4 (s; \tau)$ are defined by~\emph{\eqref{frakC_3_4}}.
\end{theorem}

Finally, Theorem~\ref{sndw_cos_sin_thrm3_H^s_L2} implies the following result.

\begin{theorem}
    \label{homog_Cauchy_A_eps_thrm_3}
    Suppose that the assumptions of Theorem~\emph{\ref{homog_Cauchy_A_eps_thrm_1}} are satisfied. Suppose also that Condition~\emph{\ref{sndw_cond1}} \emph{(}or more restrictive Condition~\emph{\ref{sndw_cond2})} is satisfied. If $\boldsymbol{\phi}, \boldsymbol{\psi} \in H^s (\mathbb{R}^d; \mathbb{C}^n)$, $0 \le s \le 3/2$, then for $\tau \in \mathbb{R}$ and $\varepsilon > 0$ we have
    \begin{multline*}
        \| f^\varepsilon \mathbf{v}_\varepsilon (\cdot, \tau) - f_0 \mathbf{v}_0 (\cdot, \tau) \|_{L_2(\mathbb{R}^d)}
\\
\le \varepsilon^{2s/3} \left( \mathfrak{C}_5 (s; \tau) \| \boldsymbol{\phi} \|_{H^s(\mathbb{R}^d)} + \mathfrak{C}_6 (s; \tau) \| \boldsymbol{\psi} \|_{H^s(\mathbb{R}^d)} \right).
    \end{multline*}
    In particular, for $0 < \varepsilon \le 1$ and $\tau = \pm \varepsilon^{-\alpha}$
    \begin{multline*}
        \| f^\varepsilon \mathbf{v}_\varepsilon (\cdot, \pm \varepsilon^{-\alpha}) - f_0 \mathbf{v}_0 (\cdot, \pm \varepsilon^{-\alpha}) \|_{L_2(\mathbb{R}^d)}
\\
\le \varepsilon^{2s(1 - \alpha)/3} \left( \mathfrak{C}_5 (s; 1) \| \boldsymbol{\phi} \|_{H^s(\mathbb{R}^d)} + \varepsilon^{-\alpha} \mathfrak{C}_6(s; 1) \| \boldsymbol{\psi} \|_{H^s(\mathbb{R}^d)} \right).
    \end{multline*}
    Here $0 < \alpha < 2s(2s+3)^{-1}$ if $\boldsymbol{\psi} \ne 0$, and $0 < \alpha < 1$ if $\boldsymbol{\psi} = 0$.
    The constants $\mathfrak{C}_5 (s; \tau), \mathfrak{C}_6 (s; \tau)$ are defined by~\emph{\eqref{frakC_5_6}}.
\end{theorem}

If $\mathbf{v}_\varepsilon$ is the solution of problem~\eqref{homog_Cauchy_A_eps}, 
then the function $f^\varepsilon \mathbf{v}_\varepsilon := \mathbf{z}_\varepsilon$
is the solution of the problem
\begin{equation}
\label{homog_Cauchy_wQ}
\left\{
\begin{aligned}
& Q^\varepsilon (\mathbf{x}) \frac{\partial^2 \mathbf{z}_\varepsilon (\mathbf{x}, \tau)}{\partial \tau^2} = - b(\mathbf{D})^* g^\varepsilon (\mathbf{x}) b(\mathbf{D}) \mathbf{z}_\varepsilon (\mathbf{x}, \tau), \\
& \mathbf{z}_\varepsilon (\mathbf{x}, 0) = \boldsymbol{\phi} (\mathbf{x}), \quad \frac{\partial \mathbf{z}_\varepsilon }{\partial \tau} (\mathbf{x}, 0) = \boldsymbol{\psi}(\mathbf{x}),
\end{aligned}
\right.
\end{equation}
where $Q^\varepsilon (\mathbf{x}) = (f^\varepsilon(\mathbf{x}) (f^\varepsilon(\mathbf{x}))^* )^{-1}$.
Similarly, if $\mathbf{v}_0$ is the solution of problem~\eqref{homog_Cauchy_A_0}, then 
 the function $f_0 \mathbf{v}_0 := \mathbf{z}_0$ satisfies
\begin{equation}
\label{homog_Cauchy_wQ_eff}
\left\{
\begin{aligned}
& \overline{Q}  \frac{\partial^2 \mathbf{z}_0 (\mathbf{x}, \tau)}{\partial \tau^2} = - b(\mathbf{D})^* g^0 b(\mathbf{D}) \mathbf{z}_0 (\mathbf{x}, \tau), \\
& \mathbf{z}_0 (\mathbf{x}, 0) = \boldsymbol{\phi} (\mathbf{x}), \quad \frac{\partial \mathbf{z}_0 }{\partial \tau} (\mathbf{x}, 0) = \boldsymbol{\psi}(\mathbf{x}),
\end{aligned}
\right.
\end{equation}
where $\overline{Q}$~is the mean value of $Q(\mathbf{x})$.

If \eqref{homog_Cauchy_wQ} is treated as the initial statement of the problem, then it is assumed that $Q(\mathbf{x})$~is a $\Gamma$-periodic
 matrix-valued function such that
\begin{equation}
\label{Q_condition}
Q(\mathbf{x}) > 0; \quad Q, Q^{-1} \in L_\infty.
\end{equation}
Representation $Q(\mathbf{x}) = (f(\mathbf{x}) f(\mathbf{x})^*)^{-1}$ is valid, for instance, with $f = Q^{-1/2}$.

Theorems~\ref{homog_Cauchy_A_eps_thrm_1}--\ref{homog_Cauchy_A_eps_thrm_3} admit the following equivalent formulations.

\begin{theorem}
    \label{homog_Cauchy_wQ_thrm_1}
    Suppose that $\mathbf{z}_\varepsilon$~is the solution of problem~\emph{\eqref{homog_Cauchy_wQ}}. Let $\mathbf{z}_0$~be the solution of problem~\emph{\eqref{homog_Cauchy_wQ_eff}}.

        $1^\circ$. If $\boldsymbol{\phi}, \boldsymbol{\psi} \in H^s (\mathbb{R}^d; \mathbb{C}^n)$, $0 \le s \le 2$,
        then for $\tau \in \mathbb{R}$ and $\varepsilon > 0$ we have
        \begin{equation*}
        \| \mathbf{z}_\varepsilon (\cdot, \tau) - \mathbf{z}_0 (\cdot, \tau) \|_{L_2(\mathbb{R}^d)} \le \varepsilon^{s/2} \left( \mathfrak{C}_1 (s; \tau) \| \boldsymbol{\phi} \|_{H^s(\mathbb{R}^d)} + \mathfrak{C}_2 (s; \tau) \| \boldsymbol{\psi} \|_{H^s(\mathbb{R}^d)} \right).
        \end{equation*}
        In particular, for $0 < \varepsilon \le 1$ and $\tau = \pm \varepsilon^{-\alpha}$
        \begin{multline}
        \label{homog_Cauchy_wQ_est_largetime1}
        \| \mathbf{z}_\varepsilon (\cdot, \pm \varepsilon^{-\alpha}) - \mathbf{z}_0 (\cdot, \pm \varepsilon^{-\alpha}) \|_{L_2(\mathbb{R}^d)}
\\
\le \varepsilon^{s(1 - \alpha)/2} \left( \mathfrak{C}_1 (s; 1) \| \boldsymbol{\phi} \|_{H^s(\mathbb{R}^d)} + \varepsilon^{-\alpha} \mathfrak{C}_2 (s; 1) \| \boldsymbol{\psi} \|_{H^s(\mathbb{R}^d)} \right).
        \end{multline}
        Here $0 < \alpha < s(s+2)^{-1}$ if $\boldsymbol{\psi} \ne 0$, and $0 < \alpha < 1$ if $\boldsymbol{\psi} = 0$.
        The constants $\mathfrak{C}_1 (s; \tau), \mathfrak{C}_2 (s; \tau)$ are given by~\emph{\eqref{frakC_1_2}}.

        $2^\circ$.  If $\boldsymbol{\phi}, \boldsymbol{\psi} \in L_2 (\mathbb{R}^d; \mathbb{C}^n)$, then
        \begin{equation*}
        \lim\limits_{\varepsilon \to 0} \| \mathbf{z}_\varepsilon (\cdot, \tau) - \mathbf{z}_0 (\cdot, \tau) \|_{L_2(\mathbb{R}^d)} = 0, \quad \tau \in \mathbb{R}.
        \end{equation*}

        $3^\circ$. If $\boldsymbol{\phi} \in L_2 (\mathbb{R}^d; \mathbb{C}^n)$ and $\boldsymbol{\psi} = 0$, then
        \begin{equation*}
        \lim\limits_{\varepsilon \to 0} \| \mathbf{z}_\varepsilon (\cdot, \pm \varepsilon^{-\alpha}) - \mathbf{z}_0 (\cdot, \pm \varepsilon^{-\alpha}) \|_{L_2(\mathbb{R}^d)} = 0, \quad 0 < \alpha < 1.
        \end{equation*}
\end{theorem}

\begin{theorem}
    \label{homog_Cauchy_wQ_thrm_2}
    Suppose that the assumptions of Theorem~\emph{\ref{homog_Cauchy_wQ_thrm_1}} are satisfied.
    Let $\widehat{N}_Q (\boldsymbol{\theta})$ be the operator defined by~\emph{\eqref{N_Q(theta)}}. Suppose that $\widehat{N}_Q (\boldsymbol{\theta}) = 0$ for any $\boldsymbol{\theta} \in \mathbb{S}^{d-1}$. If $\boldsymbol{\phi}, \boldsymbol{\psi} \in
H^s (\mathbb{R}^d; \mathbb{C}^n)$, $0 \le s \le 3/2$, then for $\tau \in \mathbb{R}$ and $\varepsilon > 0$ we have
    \begin{equation*}
    \| \mathbf{z}_\varepsilon (\cdot, \tau) - \mathbf{z}_0 (\cdot, \tau) \|_{L_2(\mathbb{R}^d)} \le \varepsilon^{2s/3} \left( \mathfrak{C}_3 (s; \tau) \| \boldsymbol{\phi} \|_{H^s(\mathbb{R}^d)} + \mathfrak{C}_4 (s; \tau) \| \boldsymbol{\psi} \|_{H^s(\mathbb{R}^d)} \right).
    \end{equation*}
    In particular, for $0 < \varepsilon \le 1$ and $\tau = \pm \varepsilon^{-\alpha}$
    \begin{multline}
    \label{homog_Cauchy_wQ_est_largetime2}
    \|  \mathbf{z}_\varepsilon (\cdot, \pm \varepsilon^{-\alpha}) - \mathbf{z}_0 (\cdot, \pm \varepsilon^{-\alpha}) \|_{L_2(\mathbb{R}^d)}
\\
 \le \varepsilon^{2s(1 - \alpha)/3} \left( \mathfrak{C}_3 (s; 1) \| \boldsymbol{\phi} \|_{H^s(\mathbb{R}^d)} + \varepsilon^{-\alpha} \mathfrak{C}_4 (s; 1) \| \boldsymbol{\psi} \|_{H^s(\mathbb{R}^d)} \right).
    \end{multline}
    Here $0 < \alpha < 2s(2s+3)^{-1}$ if $\boldsymbol{\psi} \ne 0$, and $0 < \alpha < 1$ if $\boldsymbol{\psi} = 0$.
    The constants $\mathfrak{C}_3 (s; \tau), \mathfrak{C}_4 (s; \tau)$ are defined by~\emph{\eqref{frakC_3_4}}.
\end{theorem}

\begin{theorem}
    Suppose that the assumptions of Theorem~\emph{\ref{homog_Cauchy_wQ_thrm_1}} are satisfied. Suppose also that Condition~\emph{\ref{sndw_cond1}} \emph{(}or more restrictive Condition~\emph{\ref{sndw_cond2})} is satisfied. If $\boldsymbol{\phi}, \boldsymbol{\psi} \in H^s (\mathbb{R}^d; \mathbb{C}^n)$, $0 \le s \le 3/2$, then for $\tau \in \mathbb{R}$ and $\varepsilon > 0$ we have
    \begin{equation*}
    \| \mathbf{z}_\varepsilon (\cdot, \tau) - \mathbf{z}_0 (\cdot, \tau) \|_{L_2(\mathbb{R}^d)} \le \varepsilon^{2s/3} \left( \mathfrak{C}_5 (s; \tau) \| \boldsymbol{\phi} \|_{H^s(\mathbb{R}^d)} + \mathfrak{C}_6 (s; \tau) \| \boldsymbol{\psi} \|_{H^s(\mathbb{R}^d)} \right).
    \end{equation*}
    In particular, for $0 < \varepsilon \le 1$ and $\tau = \pm \varepsilon^{-\alpha}$ we have
    \begin{multline*}
    \| \mathbf{z}_\varepsilon (\cdot, \pm \varepsilon^{-\alpha}) - \mathbf{z}_0 (\cdot, \pm \varepsilon^{-\alpha}) \|_{L_2(\mathbb{R}^d)}
\\
\le \varepsilon^{2s(1 - \alpha)/3} \left( \mathfrak{C}_5 (s; 1) \| \boldsymbol{\phi} \|_{H^s(\mathbb{R}^d)} + \varepsilon^{-\alpha} \mathfrak{C}_6(s; 1) \| \boldsymbol{\psi} \|_{H^s(\mathbb{R}^d)} \right).
    \end{multline*}
    Here $0 < \alpha < 2s(2s+3)^{-1}$ if $\boldsymbol{\psi} \ne 0$, and $0 < \alpha < 1$ if $\boldsymbol{\psi} = 0$.
    The constants $\mathfrak{C}_5 (s; \tau), \mathfrak{C}_6 (s; \tau)$ are defined by~\emph{\eqref{frakC_5_6}}.
\end{theorem}

\subsection{The Cauchy problem for the nonhomogeneous equation with the operator $\mathcal{A}_\varepsilon$}

Now we consider more general problem
\begin{equation}
    \label{nonhomog_Cauchy_A_eps}
    \left\{
    \begin{aligned}
        &\frac{\partial^2 \mathbf{v}_\varepsilon (\mathbf{x}, \tau)}{\partial \tau^2} = - (f^\varepsilon (\mathbf{x}))^* b(\mathbf{D})^* g^\varepsilon (\mathbf{x}) b(\mathbf{D}) f^\varepsilon (\mathbf{x}) \mathbf{v}_\varepsilon (\mathbf{x}, \tau) + (f^\varepsilon (\mathbf{x}))^{-1} \mathbf{F} (\mathbf{x}, \tau), \\
        & f^\varepsilon (\mathbf{x}) \mathbf{v}_\varepsilon (\mathbf{x}, 0) = \boldsymbol{\phi} (\mathbf{x}), \quad f^\varepsilon (\mathbf{x}) \frac{\partial \mathbf{v}_\varepsilon }{\partial \tau} (\mathbf{x}, 0) = \boldsymbol{\psi}(\mathbf{x}),
    \end{aligned}
    \right.
\end{equation}
where $\boldsymbol{\phi}, \boldsymbol{\psi} \in L_2 (\mathbb{R}^d, \mathbb{C}^n)$ and $\mathbf{F} \in L_{1, \mathrm{loc}} (\mathbb{R}; L_2 (\mathbb{R}^d, \mathbb{C}^n) )$~are given functions. The solution of this problem can be represented as
 \begin{multline*}
    \mathbf{v}_\varepsilon (\cdot, \tau) = \cos(\tau \mathcal{A}_\varepsilon^{1/2}) (f^\varepsilon)^{-1} \boldsymbol{\phi} + \mathcal{A}_\varepsilon^{-1/2} \sin(\tau \mathcal{A}_\varepsilon^{1/2}) (f^\varepsilon)^{-1} \boldsymbol{\psi}
\\
+ \int_{0}^{\tau} \mathcal{A}_\varepsilon^{-1/2} \sin((\tau - \tilde{\tau}) \mathcal{A}_\varepsilon^{1/2}) (f^\varepsilon)^{-1} \mathbf{F} (\cdot, \tilde{\tau}) \, d \tilde{\tau} .
\end{multline*}
Let $\mathbf{v}_0 (\mathbf{x}, \tau)$~be  the solution of the homogenized problem
\begin{equation}
    \label{nonhomog_Cauchy_A_0}
    \left\{
    \begin{aligned}
        &\frac{\partial^2 \mathbf{v}_0 (\mathbf{x}, \tau)}{\partial \tau^2} = - f_0 b(\mathbf{D})^* g^0  b(\mathbf{D}) f_0 \mathbf{v}_0 (\mathbf{x}, \tau) + f_0^{-1} \mathbf{F} (\mathbf{x}, \tau) , \\
        & f_0 \mathbf{v}_0 (\mathbf{x}, 0) = \boldsymbol{\phi} (\mathbf{x}), \quad f_0 \frac{\partial \mathbf{v}_0 }{\partial \tau} (\mathbf{x}, 0) = \boldsymbol{\psi}(\mathbf{x}).
    \end{aligned}
    \right.
\end{equation}
Then
\begin{multline*}
    \mathbf{v}_0 (\cdot, \tau) = \cos(\tau (\mathcal{A}^0)^{1/2}) f_0^{-1} \boldsymbol{\phi} + (\mathcal{A}^0)^{-1/2} \sin(\tau (\mathcal{A}^0)^{1/2}) f_0^{-1} \boldsymbol{\psi}
\\
+ \int_{0}^{\tau} (\mathcal{A}^0)^{-1/2} \sin((\tau - \tilde{\tau}) (\mathcal{A}^0)^{1/2}) f_0^{-1} \mathbf{F} (\cdot, \tilde{\tau}) \, d \tilde{\tau}.
\end{multline*}

By analogy with the proof of Theorem~\ref{nonhomog_Cauchy_hatA_eps_thrm},
from Theorem~\ref{sndw_cos_sin_thrm1_H^s_L2} we deduce the
following result (which has been proved before in~\cite[Theorem~15.5]{BSu5}).

\begin{theorem}
    \label{nonhomog_Cauchy_A_eps_thrm_1}
    Let $\mathbf{v}_\varepsilon$~be the solution of problem~\emph{\eqref{nonhomog_Cauchy_A_eps}}, 
and let $\mathbf{v}_0$~be the solution of problem~\emph{\eqref{nonhomog_Cauchy_A_0}}.

     $1^{\circ}.$
         If $\boldsymbol{\phi}, \boldsymbol{\psi} \in H^s (\mathbb{R}^d; \mathbb{C}^n)$ and  $\mathbf{F} \in L_{1, \mathrm{loc}} (\mathbb{R}; H^s (\mathbb{R}^d; \mathbb{C}^n))$, $0 \le s \le 2$, then for $\tau \in \mathbb{R}$ and $\varepsilon > 0$ we have
        \begin{multline*}
      \| f^\varepsilon \mathbf{v}_\varepsilon(\cdot, \tau) - f_0 \mathbf{v}_0 (\cdot, \tau) \|_{L_2 (\mathbb{R}^d)}
 \\
\le  \varepsilon^{s/2} \left( \mathfrak{C}_1 (s; \tau) \| \boldsymbol{\phi} \|_{H^s(\mathbb{R}^d)} + \mathfrak{C}_2 (s; \tau) \| \boldsymbol{\psi} \|_{H^s(\mathbb{R}^d)} \right)
\\
+  \varepsilon^{s/2} \mathfrak{C}_2 (s; \tau) \|\mathbf{F} \|_{L_1((0,\tau);H^s(\mathbb{R}^d))}.
        \end{multline*}
        Under the additional assumption that $\mathbf{F} \in L_p (\mathbb{R}_\pm; H^s (\mathbb{R}^d; \mathbb{C}^n) )$ \emph{(}where
\hbox{$1 \le p \le \infty$}\emph{)}, for $\tau = \pm \varepsilon^{-\alpha}$ and $0 < \varepsilon \le 1$ we have
        \begin{multline*}
             \| f^\varepsilon \mathbf{v}_\varepsilon (\cdot, \pm \varepsilon^{-\alpha}) - f_0 \mathbf{v}_0 (\cdot, \pm \varepsilon^{-\alpha}) \|_{L_2(\mathbb{R}^d)}
 \\
\le \varepsilon^{s(1 - \alpha)/2} \left( \mathfrak{C}_1 (s; 1) \| \boldsymbol{\phi} \|_{H^s(\mathbb{R}^d)} +  \varepsilon^{-\alpha} \mathfrak{C}_2 (s; 1) \| \boldsymbol{\psi} \|_{H^s(\mathbb{R}^d)}\right)
\\
 + \varepsilon^{s(1 - \alpha)/2 -\alpha-\alpha/p'} \mathfrak{C}_2 (s; 1) \|\mathbf{F} \|_{L_p(\mathbb{R}_\pm; H^s(\mathbb{R}^d))}.
        \end{multline*}
        Here $p^{-1} + (p')^{-1} = 1$ and $0 < \alpha < s(s + 2 + 2/p')^{-1}$. The constants $\mathfrak{C}_1 (s; \tau), \mathfrak{C}_2 (s; \tau)$ are defined by~\emph{(\ref{frakC_1_2})}.

         $2^{\circ}.$
        If $\boldsymbol{\phi}, \boldsymbol{\psi} \in L_2 (\mathbb{R}^d; \mathbb{C}^n)$ and $\mathbf{F} \in L_{1, \mathrm{loc}} (\mathbb{R}; L_2 (\mathbb{R}^d; \mathbb{C}^n) )$, then
        \begin{equation*}
            \lim\limits_{\varepsilon \to 0} \| f^\varepsilon  \mathbf{v}_\varepsilon (\cdot, \tau) -  f_0 \mathbf{v}_0 (\cdot, \tau) \|_{L_2(\mathbb{R}^d)} = 0, \quad \tau \in \mathbb{R}.
        \end{equation*}
\end{theorem}

Statement $1^\circ$  can be refined under the additional assumptions.
 Theorem~\ref{sndw_cos_sin_thrm2_H^s_L2} implies the following result.

\begin{theorem}
    Suppose that the assumptions of Theorem~\emph{\ref{nonhomog_Cauchy_A_eps_thrm_1}} are satisfied. Let $\widehat{N}_Q (\boldsymbol{\theta})$, be the operator defined by~\emph{\eqref{N_Q(theta)}}. Suppose that $\widehat{N}_Q (\boldsymbol{\theta}) = 0$ for any $\boldsymbol{\theta} \in \mathbb{S}^{d-1}$. If $\boldsymbol{\phi}, \boldsymbol{\psi} \in H^s (\mathbb{R}^d; \mathbb{C}^n)$ and  $\mathbf{F} \in L_{1, \mathrm{loc}} (\mathbb{R}; H^s (\mathbb{R}^d; \mathbb{C}^n))$, $0 \le s \le 3/2$, then for $\tau \in \mathbb{R}$ and $\varepsilon > 0$ we have
    \begin{multline*}
        \| f^\varepsilon \mathbf{v}_\varepsilon (\cdot, \tau) - f_0 \mathbf{v}_0 (\cdot, \tau) \|_{L_2(\mathbb{R}^d)}
\\
\le \varepsilon^{2s/3} \left( \mathfrak{C}_3 (s; \tau) \| \boldsymbol{\phi} \|_{H^s(\mathbb{R}^d)}
+ \mathfrak{C}_4 (s; \tau) \| \boldsymbol{\psi} \|_{H^s(\mathbb{R}^d)} \right)
\\
+ \varepsilon^{2s/3} \mathfrak{C}_4 (s; \tau) \|\mathbf{F} \|_{L_1((0,\tau);H^s(\mathbb{R}^d))}.
    \end{multline*}
    Under the additional assumption that $\mathbf{F} \in L_p (\mathbb{R}_\pm; H^s (\mathbb{R}^d; \mathbb{C}^n) )$ \emph{(}where
\hbox{$1 \le p \le \infty$}\emph{)}, for $\tau = \pm \varepsilon^{-\alpha}$ and $0 < \varepsilon \le 1$
    \begin{multline*}
        \| f^\varepsilon \mathbf{v}_\varepsilon (\cdot, \pm \varepsilon^{-\alpha}) - f_0 \mathbf{v}_0 (\cdot, \pm \varepsilon^{-\alpha}) \|_{L_2(\mathbb{R}^d)}
\\
\le \varepsilon^{2s(1 - \alpha)/3} \left( \mathfrak{C}_3 (s; 1) \| \boldsymbol{\phi} \|_{H^s(\mathbb{R}^d)} +  \varepsilon^{-\alpha} \mathfrak{C}_4 (s; 1) \| \boldsymbol{\psi} \|_{H^s(\mathbb{R}^d)} \right)
\\
+ \varepsilon^{2s(1 - \alpha)/3 -\alpha-\alpha/p'} \mathfrak{C}_4 (s; 1) \|\mathbf{F} \|_{L_p(\mathbb{R}_\pm; H^s(\mathbb{R}^d))}.
    \end{multline*}
    Here $p^{-1} + (p')^{-1} = 1$ and $0 < \alpha < 2s(2s + 3 + 3/p')^{-1}$. The constants $\mathfrak{C}_3 (s; \tau), \mathfrak{C}_4 (s; \tau)$ are defined by~\emph{(\ref{frakC_3_4})}.
\end{theorem}

Finally, applying Theorem~\ref{sndw_cos_sin_thrm3_H^s_L2},  we deduce the following result.

\begin{theorem}
    \label{nonhomog_Cauchy_A_eps_thrm_3}
    Suppose that the assumptions of Theorem~\emph{\ref{nonhomog_Cauchy_A_eps_thrm_1}} are satisfied. Suppose also that Condition~\emph{\ref{sndw_cond1}} \emph{(}or more restrictive Condition~\emph{\ref{sndw_cond2})} is satisfied. If $\boldsymbol{\phi}, \boldsymbol{\psi} \in H^s (\mathbb{R}^d; \mathbb{C}^n)$ and  $\mathbf{F} \in L_{1, \mathrm{loc}} (\mathbb{R}; H^s (\mathbb{R}^d; \mathbb{C}^n))$, $0 \le s \le 3/2$, then for $\tau \in \mathbb{R}$ and $\varepsilon > 0$ we have
    \begin{multline*}
        \| f^\varepsilon \mathbf{v}_\varepsilon (\cdot, \tau) - f_0 \mathbf{v}_0 (\cdot, \tau) \|_{L_2(\mathbb{R}^d)}
\\
\le \varepsilon^{2s/3} \left( \mathfrak{C}_5 (s; \tau) \| \boldsymbol{\phi} \|_{H^s(\mathbb{R}^d)} + \mathfrak{C}_6 (s; \tau) \| \boldsymbol{\psi} \|_{H^s(\mathbb{R}^d)} \right)
\\
+ \varepsilon^{2s/3} \mathfrak{C}_6 (s; \tau) \|\mathbf{F} \|_{L_1((0,\tau);H^s(\mathbb{R}^d))}.
    \end{multline*}
    Under the additional assumption that $\mathbf{F} \in L_p (\mathbb{R}_\pm; H^s (\mathbb{R}^d; \mathbb{C}^n) )$ \emph{(}where
\hbox{$1 \le p \le \infty$}\emph{)}, for $\tau = \pm \varepsilon^{-\alpha}$ and $0 < \varepsilon \le 1$
    \begin{multline*}
        \| f^\varepsilon \mathbf{v}_\varepsilon (\cdot, \pm \varepsilon^{-\alpha}) - f_0 \mathbf{v}_0 (\cdot, \pm \varepsilon^{-\alpha}) \|_{L_2(\mathbb{R}^d)}
\\
\le \varepsilon^{2s(1 - \alpha)/3} \left( \mathfrak{C}_5 (s; 1) \| \boldsymbol{\phi} \|_{H^s(\mathbb{R}^d)} +  \varepsilon^{-\alpha} \mathfrak{C}_6 (s; 1) \| \boldsymbol{\psi} \|_{H^s(\mathbb{R}^d)} \right)
\\
+ \varepsilon^{2s(1 - \alpha)/3-\alpha-\alpha/p'} \mathfrak{C}_6 (s; 1) \|\mathbf{F} \|_{L_p(\mathbb{R}_\pm; H^s(\mathbb{R}^d))}.
    \end{multline*}
    Here $p^{-1} + (p')^{-1} = 1$ and $0 < \alpha < 2s(2s + 3 + 3/p')^{-1}$. The constants $\mathfrak{C}_5 (s; \tau), \mathfrak{C}_6 (s; \tau)$ are defined by~\emph{(\ref{frakC_5_6})}.
\end{theorem}

If $\mathbf{v}_\varepsilon$  satisfies~\eqref{nonhomog_Cauchy_A_eps}, then 
the function $f^\varepsilon \mathbf{v}_\varepsilon =: \mathbf{z}_\varepsilon$ is the solution of the problem
\begin{equation}
\label{nonhomog_Cauchy_wQ}
\left\{
\begin{aligned}
& Q^\varepsilon (\mathbf{x}) \frac{\partial^2 \mathbf{z}_\varepsilon (\mathbf{x}, \tau)}{\partial \tau^2} = - b(\mathbf{D})^* g^\varepsilon (\mathbf{x}) b(\mathbf{D}) \mathbf{z}_\varepsilon (\mathbf{x}, \tau) + Q^\varepsilon (\mathbf{x}) \mathbf{F} (\mathbf{x}, \tau), \\
& \mathbf{z}_\varepsilon (\mathbf{x}, 0) = \boldsymbol{\phi} (\mathbf{x}), \quad \frac{\partial \mathbf{z}_\varepsilon }{\partial \tau} (\mathbf{x}, 0) = \boldsymbol{\psi}(\mathbf{x}),
\end{aligned}
\right.
\end{equation}
where $Q^\varepsilon (\mathbf{x}) = (f^\varepsilon(\mathbf{x}) (f^\varepsilon(\mathbf{x}))^* )^{-1}$.
Similarly, if $\mathbf{v}_0$ is the solution of problem~\eqref{nonhomog_Cauchy_A_0}, then 
the function $f_0 \mathbf{v}_0 =: \mathbf{z}_0$ satisfies
\begin{equation}
\label{nonhomog_Cauchy_wQ_eff}
\left\{
\begin{aligned}
& \overline{Q} \frac{\partial^2 \mathbf{z}_0 (\mathbf{x}, \tau)}{\partial \tau^2} = - b(\mathbf{D})^* g^0 b(\mathbf{D}) \mathbf{z}_0 (\mathbf{x}, \tau) + \overline{Q}  \mathbf{F} (\mathbf{x}, \tau), \\
& \mathbf{z}_0 (\mathbf{x}, 0) = \boldsymbol{\phi} (\mathbf{x}), \quad \frac{\partial \mathbf{z}_0 }{\partial \tau} (\mathbf{x}, 0) = \boldsymbol{\psi}(\mathbf{x}).
\end{aligned}
\right.
\end{equation}

If~\eqref{nonhomog_Cauchy_wQ} is treated as the initial statement of the problem, then it is assumed that $Q(\mathbf{x})$~is a  $\Gamma$-periodic
matrix-valued function satisfying~\eqref{Q_condition}.

Theorems~\ref{nonhomog_Cauchy_A_eps_thrm_1}--\ref{nonhomog_Cauchy_A_eps_thrm_3} admit the following reformulations.

\begin{theorem}
    \label{nonhomog_Cauchy_wQ_thrm_1}
    Let $\mathbf{z}_\varepsilon$~be the solution of problem~\emph{\eqref{nonhomog_Cauchy_wQ}}, and let $\mathbf{z}_0$~be the
    solution of problem~\emph{\eqref{nonhomog_Cauchy_wQ_eff}}.

    $1^{\circ}.$
         If $\boldsymbol{\phi}, \boldsymbol{\psi} \in H^s (\mathbb{R}^d;\mathbb{C}^n)$ and $\mathbf{F} \in L_{1, \mathrm{loc}} (\mathbb{R}; H^s (\mathbb{R}^d; \mathbb{C}^n))$, $0 \le s \le 2$, then for $\tau \in \mathbb{R}$ and $\varepsilon > 0$ we have
        \begin{multline*}
        \| \mathbf{z}_\varepsilon(\cdot, \tau) - \mathbf{z}_0 (\cdot, \tau) \|_{L_2 (\mathbb{R}^d)}
\le   \varepsilon^{s/2} \left( \mathfrak{C}_1 (s; \tau) \| \boldsymbol{\phi} \|_{H^s(\mathbb{R}^d)} + \mathfrak{C}_2 (s; \tau) \| \boldsymbol{\psi} \|_{H^s(\mathbb{R}^d)}\right)
\\
 +  \varepsilon^{s/2} \mathfrak{C}_2 (s; \tau) \|\mathbf{F} \|_{L_1((0,\tau);H^s(\mathbb{R}^d))}.
        \end{multline*}
        Under the additional assumption that $\mathbf{F} \in L_p (\mathbb{R}_\pm; H^s (\mathbb{R}^d; \mathbb{C}^n) )$ \emph{(}where
\hbox{$1 \le p \le \infty$}\emph{)}, for $\tau = \pm \varepsilon^{-\alpha}$ and $0 < \varepsilon \le 1$
        \begin{multline*}
        \| \mathbf{z}_\varepsilon (\cdot, \pm \varepsilon^{-\alpha}) - \mathbf{z}_0 (\cdot, \pm \varepsilon^{-\alpha}) \|_{L_2(\mathbb{R}^d)}
 \\
\le \varepsilon^{s(1 - \alpha)/2} \left( \mathfrak{C}_1 (s; 1) \| \boldsymbol{\phi} \|_{H^s(\mathbb{R}^d)} +  \varepsilon^{-\alpha} \mathfrak{C}_2 (s; 1) \| \boldsymbol{\psi} \|_{H^s(\mathbb{R}^d)} \right)
\\
+ \varepsilon^{s(1 - \alpha)/2-\alpha-\alpha/p'} \mathfrak{C}_2 (s; 1) \|\mathbf{F} \|_{L_p(\mathbb{R}_\pm; H^s(\mathbb{R}^d))}.
        \end{multline*}
        Here $p^{-1} + (p')^{-1} = 1$ and $0 < \alpha < s(s + 2 + 2/p')^{-1}$. The constants $\mathfrak{C}_1 (s; \tau), \mathfrak{C}_2 (s; \tau)$ are defined by~\emph{(\ref{frakC_1_2})}.

         $2^{\circ}.$
        If $\boldsymbol{\phi}, \boldsymbol{\psi} \in L_2 (\mathbb{R}^d; \mathbb{C}^n)$ and $\mathbf{F} \in L_{1, \mathrm{loc}} (\mathbb{R}; L_2 (\mathbb{R}^d; \mathbb{C}^n) )$, then
        \begin{equation*}
        \lim\limits_{\varepsilon \to 0} \| \mathbf{z}_\varepsilon (\cdot, \tau) -  \mathbf{z}_0 (\cdot, \tau) \|_{L_2(\mathbb{R}^d)} = 0, \quad \tau \in \mathbb{R}.
        \end{equation*}
\end{theorem}

\begin{theorem}
    \label{nonhomog_Cauchy_wQ_thrm_2}
    Suppose that the assumptions of Theorem~\emph{\ref{nonhomog_Cauchy_wQ_thrm_1}} are satisfied. Let $\widehat{N}_Q (\boldsymbol{\theta})$ be
    the operator defined by~\emph{\eqref{N_Q(theta)}}. Suppose that $\widehat{N}_Q (\boldsymbol{\theta}) = 0$ for any $\boldsymbol{\theta} \in \mathbb{S}^{d-1}$. If $\boldsymbol{\phi}, \boldsymbol{\psi} \in H^s (\mathbb{R}^d; \mathbb{C}^n)$ and  $\mathbf{F} \in L_{1, \mathrm{loc}} (\mathbb{R}; H^s (\mathbb{R}^d; \mathbb{C}^n))$, $0 \le s \le 3/2$, then for $\tau \in \mathbb{R}$ and $\varepsilon > 0$ we have
    \begin{multline*}
    \| \mathbf{z}_\varepsilon (\cdot, \tau) - \mathbf{z}_0 (\cdot, \tau) \|_{L_2(\mathbb{R}^d)}
\le \varepsilon^{2s/3} \left( \mathfrak{C}_3 (s; \tau) \| \boldsymbol{\phi} \|_{H^s(\mathbb{R}^d)} + \mathfrak{C}_4 (s; \tau) \| \boldsymbol{\psi} \|_{H^s(\mathbb{R}^d)} \right)
\\
+ \varepsilon^{2s/3} \mathfrak{C}_4 (s; \tau) \|\mathbf{F} \|_{L_1((0,\tau);H^s(\mathbb{R}^d))}.
    \end{multline*}
    Under the additional assumption that $\mathbf{F} \in L_p (\mathbb{R}_\pm; H^s (\mathbb{R}^d; \mathbb{C}^n) )$ \emph{(}where
\hbox{$1 \le p \le \infty$}\emph{)}, for $\tau = \pm \varepsilon^{-\alpha}$ and $0 < \varepsilon \le 1$ we have
    \begin{multline*}
    \| \mathbf{z}_\varepsilon (\cdot, \pm \varepsilon^{-\alpha}) - \mathbf{z}_0 (\cdot, \pm \varepsilon^{-\alpha}) \|_{L_2(\mathbb{R}^d)}
\\
\le \varepsilon^{2s(1 - \alpha)/3} \left( \mathfrak{C}_3 (s; 1) \| \boldsymbol{\phi} \|_{H^s(\mathbb{R}^d)} +  \varepsilon^{-\alpha} \mathfrak{C}_4 (s; 1) \| \boldsymbol{\psi} \|_{H^s(\mathbb{R}^d)} \right)
\\
+ \varepsilon^{2s(1 - \alpha)/3 -\alpha-\alpha/p'} \mathfrak{C}_4 (s; 1) \|\mathbf{F} \|_{L_p(\mathbb{R}_\pm; H^s(\mathbb{R}^d))}.
    \end{multline*}
    Here $p^{-1} + (p')^{-1} = 1$ and $0 < \alpha < 2s(2s + 3 + 3/p')^{-1}$. The constants $\mathfrak{C}_3 (s; \tau), \mathfrak{C}_4 (s; \tau)$
    are defined by~\emph{\eqref{frakC_3_4}}.
\end{theorem}

\begin{theorem}
    Suppose that the assumptions of Theorem~\emph{\ref{nonhomog_Cauchy_wQ_thrm_1}} are satisfied.
    Suppose also that Condition~\emph{\ref{sndw_cond1}} \emph{(}or more restrictive Condition~\emph{\ref{sndw_cond2})} is satisfied. If $\boldsymbol{\phi}, \boldsymbol{\psi} \in H^s (\mathbb{R}^d; \mathbb{C}^n)$ and  $\mathbf{F} \in L_{1, \mathrm{loc}} (\mathbb{R}; H^s (\mathbb{R}^d; \mathbb{C}^n))$, $0 \le s \le 3/2$, then for $\tau \in \mathbb{R}$ and $\varepsilon > 0$ we have
    \begin{multline*}
    \| \mathbf{z}_\varepsilon (\cdot, \tau) - \mathbf{z}_0 (\cdot, \tau) \|_{L_2(\mathbb{R}^d)} \le \varepsilon^{2s/3} \left( \mathfrak{C}_5 (s; \tau) \| \boldsymbol{\phi} \|_{H^s(\mathbb{R}^d)} + \mathfrak{C}_6 (s; \tau) \| \boldsymbol{\psi} \|_{H^s(\mathbb{R}^d)}\right)
\\
+ \varepsilon^{2s/3}  \mathfrak{C}_6 (s; \tau) \|\mathbf{F} \|_{L_1((0,\tau);H^s(\mathbb{R}^d))}.
    \end{multline*}
    Under the additional assumption that $\mathbf{F} \in L_p (\mathbb{R}_\pm; H^s (\mathbb{R}^d; \mathbb{C}^n) )$ \emph{(}where
\hbox{$1 \le p \le \infty$}\emph{)}, for $\tau = \pm \varepsilon^{-\alpha}$ and $0 < \varepsilon \le 1$ we have
    \begin{multline*}
    \| \mathbf{z}_\varepsilon (\cdot, \pm \varepsilon^{-\alpha}) - \mathbf{z}_0 (\cdot, \pm \varepsilon^{-\alpha}) \|_{L_2(\mathbb{R}^d)}
\\
\le \varepsilon^{2s(1 - \alpha)/3} \left( \mathfrak{C}_5 (s; 1) \| \boldsymbol{\phi} \|_{H^s(\mathbb{R}^d)} +  \varepsilon^{-\alpha} \mathfrak{C}_6 (s; 1) \| \boldsymbol{\psi} \|_{H^s(\mathbb{R}^d)} \right)
\\
+ \varepsilon^{2s(1 - \alpha)/3-\alpha-\alpha/p'} \mathfrak{C}_6 (s; 1) \|\mathbf{F} \|_{L_p(\mathbb{R}_\pm; H^s(\mathbb{R}^d))}.
    \end{multline*}
    Here $p^{-1} + (p')^{-1} = 1$ and $0 < \alpha < 2s(2s + 3 + 3/p')^{-1}$. The constants $\mathfrak{C}_5 (s; \tau), \mathfrak{C}_6 (s; \tau)$
    are defined by~\emph{(\ref{frakC_5_6})}.
\end{theorem}

\section{Application of the general results: the acoustics equation}

\subsection{The model example}

In $L_2 (\mathbb{R}^d)$, $d \ge 1$, we consider the operator
\begin{equation}
\label{model_operator}
\widehat{\mathcal{A}} = \mathbf{D}^* g(\mathbf{x}) \mathbf{D} =  - \operatorname{div} g(\mathbf{x}) \nabla.
\end{equation}
Here $g(\mathbf{x})$~is a $\Gamma$-periodic Hermitian $(d \times d)$-matrix-valued function such that
\begin{equation}
  \label{model_operator_g_cond}
  g(\mathbf{x}) > 0; \qquad g, g^{-1} \in L_\infty.
\end{equation}
The operator~\eqref{model_operator} is a particular case of the operator~\eqref{hatA}.
Now we have $n=1$, $m=d$, and $b(\mathbf{D}) = \mathbf{D}$.
Obviously, condition~\eqref{rank_alpha_ineq} is satisfied with $\alpha_0 = \alpha_1 = 1$.
Representation $g(\mathbf{x}) = h(\mathbf{x})^* h(\mathbf{x})$ is valid, for instance, with $h = g^{1/2}$.

By~\eqref{hatA0}, the effective operator for the operator~\eqref{model_operator} is given by
\begin{equation*}
\widehat{\mathcal{A}}^0 = \mathbf{D}^* g^0 \mathbf{D} =  - \operatorname{div} g^0 \nabla.
\end{equation*}
According to~\eqref{equation_for_Lambda} and~\eqref{g0}, the effective matrix is defined as follows.
Let $\mathbf{e}_1, \ldots, \mathbf{e}_d$~be the standard orthonormal basis in $\mathbb{R}^d$.
Let $\Phi_j \in \widetilde{H}^1 (\Omega)$~be the weak  $\Gamma$-periodic solution of the problem
\begin{equation}
\label{model_eff_1}
\operatorname{div} g(\mathbf{x}) (\nabla \Phi_j(\mathbf{x}) + \mathbf{e}_j) = 0, \quad \int_{\Omega} \Phi_j(\mathbf{x}) \, d \mathbf{x} = 0.
\end{equation}
Then $g^0$~is the $(d \times d)$-matrix with the columns
\begin{equation}
\label{model_eff_2}
\mathbf{g}_j^0 = |\Omega|^{-1} \int_{\Omega} g(\mathbf{x}) (\nabla \Phi_j(\mathbf{x}) + \mathbf{e}_j) \, d \mathbf{x}, \quad j = 1, \ldots, d.
\end{equation}
If $d=1$, then $m = n = 1$, whence $g^0 = \underline{g}$.

If $g(\mathbf{x})$~is a symmetric matrix with real entries, then, by Proposition~\ref{N=0_proposit}($1^\circ$),
  $\widehat{N} (\boldsymbol{\theta}) = 0$ for any $\boldsymbol{\theta} \in \mathbb{S}^{d-1}$. If $g(\mathbf{x})$~is a Hermitian matrix with complex entries, then, in general, $\widehat{N} (\boldsymbol{\theta})$ is not zero. Since $n=1$, then $\widehat{N} (\boldsymbol{\theta}) = \widehat{N}_0 (\boldsymbol{\theta})$ is the operator of multiplication by $\widehat{\mu}(\boldsymbol{\theta})$, where  $\widehat{\mu}(\boldsymbol{\theta})$~is the coefficient at $t^3$ in the expansion 
\begin{equation*}
    \widehat{\lambda}(t, \boldsymbol{\theta}) = \widehat{\gamma} (\boldsymbol{\theta}) t^2 + \widehat{\mu} (\boldsymbol{\theta}) t^3 + \ldots
\end{equation*}
 for the first eigenvalue of the operator $\widehat{\mathcal{A}} (\mathbf{k})$. Calculation (see~\cite[Subsection~10.3]{BSu3}) shows that
\begin{align*}
    \widehat{N} (\boldsymbol{\theta}) &= \widehat{\mu} (\boldsymbol{\theta}) = -i \sum_{j,l,k=1}^{d} (a_{jlk} - a_{jlk}^*) \theta_j \theta_l \theta_k, \quad \boldsymbol{\theta} \in \mathbb{S}^{d-1}, \\
    a_{jlk} &= |\Omega|^{-1} \int_{\Omega} \Phi_j (\mathbf{x})^* \left\langle g(\mathbf{x}) (\nabla \Phi_l (\mathbf{x}) + \mathbf{e}_l), \mathbf{e}_k \right\rangle \, d \mathbf{x}, \quad j, l, k = 1, \ldots, d.
\end{align*}

The following example is borrowed from~\cite[Subsection~10.4]{BSu3}.

\begin{example}[\cite{BSu3}]
    \label{model_exmpl}
    Let $d=2$ and $\Gamma = (2 \pi \mathbb{Z})^2$. Suppose that the matrix $g(\mathbf{x})$ is given by
    \begin{equation*}
    g(\mathbf{x}) = \begin{pmatrix}
    1 & i \beta'(x_1) \\
    - i \beta' (x_1) & 1
    \end{pmatrix},
    \end{equation*}
    where $\beta(x_1)$~is a smooth $(2 \pi)$-periodic real-valued function such that \hbox{$1 - (\beta'(x_1))^2 > 0$} and
    $\int_{0}^{2 \pi} \beta(x_1)\, d x_1 = 0$. Then $\widehat{N} (\boldsymbol{\theta}) = - \alpha \pi^{-1} \theta_2^3$, where
    $\alpha = \int_{0}^{2 \pi} \beta (x_1) (\beta' (x_1))^2 dx_1$.
    It is easy to give a concrete example where \hbox{$\alpha \ne 0$}: if $\beta (x_1) = c (\sin x_1 + \cos 2x_1)$ with $0 < c < 1/3$, then
    \hbox{$\alpha = - (3 \pi/2) c^3 \ne 0$}. In this example, $\widehat{N} (\boldsymbol{\theta}) = \widehat{\mu} (\boldsymbol{\theta}) \ne 0$ for all $\boldsymbol{\theta} \in \mathbb{S}^1$ except for the points $(\pm 1 ,0)$.
\end{example}

Consider the Cauchy problem
\begin{equation}
\label{model_Cauchy}
\left\{
\begin{aligned}
& \frac{\partial^2 v_\varepsilon (\mathbf{x}, \tau)}{\partial \tau^2} = - \mathbf{D}^* g^\varepsilon (\mathbf{x}) \mathbf{D} v_\varepsilon (\mathbf{x}, \tau), \\
& v_\varepsilon (\mathbf{x}, 0) = \phi (\mathbf{x}), \quad  \frac{\partial v_\varepsilon }{\partial \tau} (\mathbf{x}, 0) = \psi (\mathbf{x}),
\end{aligned}
\right.
\end{equation}
where $\phi, \psi$~are given functions. Problem~\eqref{model_Cauchy} is of the form~\eqref{homog_Cauchy_hatA_eps}.
The homogenized problem~(see~\eqref{homog_Cauchy_hatA_0}) takes the form
\begin{equation*}
\left\{
\begin{aligned}
& \frac{\partial^2 v_0 (\mathbf{x}, \tau)}{\partial \tau^2} = - \mathbf{D}^* g^0  \mathbf{D} v_0 (\mathbf{x}, \tau), \\
& v_0 (\mathbf{x}, 0) = \phi (\mathbf{x}), \quad  \frac{\partial v_0 }{\partial \tau} (\mathbf{x}, 0) = \psi (\mathbf{x}).
\end{aligned}
\right.
\end{equation*}
By Theorem~\ref{homog_Cauchy_hatA_eps_thrm}, if $\phi, \psi \in H^s (\mathbb{R}^d)$, $0 \le s \le 2$, then,
for fixed  $\tau \in \mathbb{R}$,
the solution $v_\varepsilon$ converges to $v_0$ in $L_2(\mathbb{R}^d)$, as $\varepsilon \to 0$. We have
\begin{equation}
\label{model_general_est}
\| v_\varepsilon (\cdot, \tau) - v_0 (\cdot, \tau) \|_{L_2(\mathbb{R}^d)} \le \varepsilon^{s/2} \left( \widehat{\mathfrak{C}}_1 (s; \tau) \| \phi \|_{H^s(\mathbb{R}^d)} + \widehat{\mathfrak{C}}_2 (s; \tau) \| \psi \|_{H^s(\mathbb{R}^d)} \right),
\end{equation}
where $\widehat{\mathfrak{C}}_1 (s; \tau)$ and $\widehat{\mathfrak{C}}_2 (s; \tau)$ are given by~\eqref{hat_frakC_1_2},
and $\widehat{\mathcal{C}}_1$, $\widehat{\mathcal{C}}_2$ depend only on the norms $\|g\|_{L_\infty}$, $\|g^{-1}\|_{L_\infty}$,
and the parameters of the lattice~$\Gamma$. Example~\ref{model_exmpl} confirms that the result~\eqref{model_general_est} is sharp in the
  general case; see Theorem~\ref{s<2_cos_thrm_Rd}.

If the matrix $g(\mathbf{x})$ has real entries, by Proposition~\ref{N=0_proposit}($1^\circ$), we have
\hbox{$\widehat{N}(\boldsymbol{\theta})=0$} for any $\boldsymbol{\theta} \in {\mathbb S}^{d-1}$.
Then, by Theorem~\ref{homog_Cauchy_hatA_eps_ench_thrm_1},
for $\phi, \psi \in H^s (\mathbb{R}^d)$, $0 \le s \le 3/2$, we have
\begin{equation*}
    \| v_\varepsilon (\cdot, \tau) - v_0 (\cdot, \tau) \|_{L_2(\mathbb{R}^d)} \le \varepsilon^{2s/3} \left( \widehat{\mathfrak{C}}_3 (s; \tau) \| \phi \|_{H^s(\mathbb{R}^d)} + \widehat{\mathfrak{C}}_4 (s; \tau) \| \psi \|_{H^s(\mathbb{R}^d)} \right),
\end{equation*}
where $\widehat{\mathfrak{C}}_3 (s; \tau)$ and $\widehat{\mathfrak{C}}_4 (s; \tau)$ are of the form~\eqref{hat_frakC_3_4}, and  $\widehat{\mathcal{C}}_3$, $\widehat{\mathcal{C}}_4$ depend only on the norms $\|g\|_{L_\infty}$, $\|g^{-1}\|_{L_\infty}$
and the parameters of the lattice $\Gamma$.

One can also apply the statements about the behavior of the solution for large $|\tau|$
(estimates of the form~\eqref{homog_Cauchy_hatA_eps_est_largetime1} in the general case and~\eqref{homog_Cauchy_hatA_eps_est_largetime2}
in the ``real'' case). It is possible to consider more general problem for the nonhomogeneous equation and apply
Theorem~\ref{nonhomog_Cauchy_hatA_eps_thrm} in the general case and Theorem~\ref{nonhomog_Cauchy_hatA_eps_ench_thrm_1}
in the ``real'' case.

\subsection{The acoustics equation}
In $L_2 (\mathbb{R}^d)$, $d \ge 1$, consider the operator $\widehat{\mathcal{A}} = \mathbf{D}^* g(\mathbf{x}) \mathbf{D}$~(see~\eqref{model_operator}), where $g(\mathbf{x})$ is a $\Gamma$-periodic symmetric matrix-valued function with real entries
 satisfying condition~\eqref{model_operator_g_cond}. The matrix  $g(\mathbf{x})$ characterizes the parameters of the acoustical (in general,
anisotropic) medium under study. The effective matrix $g^0$ is defined according to~\eqref{model_eff_1},
 \eqref{model_eff_2}. The effective operator is given by $\widehat{\mathcal{A}}^0 = \mathbf{D}^* g^0 \mathbf{D}$.
 Let $Q (\mathbf{x})$~be a $\Gamma$-periodic function in $\mathbb{R}^d$ such that
\begin{equation*}
    Q(\mathbf{x}) > 0; \qquad Q, Q^{-1} \in L_\infty.
\end{equation*}
The function $Q(\mathbf{x})$ stands for the density of the medium.
Consider the operator $\widehat{\mathcal{A}}_\varepsilon = \mathbf{D}^* g^{\varepsilon} \mathbf{D}$
whose coefficients oscillate rapidly for small $\varepsilon$.

Consider the Cauchy problem for the acoustics equation in a rapidly oscillating medium:
\begin{equation}
    \label{acoustics_Cauchy}
    \left\{
    \begin{aligned}
        & Q^\varepsilon (\mathbf{x}) \frac{\partial^2 z_\varepsilon (\mathbf{x}, \tau)}{\partial \tau^2} = - \widehat{\mathcal{A}}_\varepsilon z_\varepsilon (\mathbf{x}, \tau), \\
        & z_\varepsilon (\mathbf{x}, 0) = \phi (\mathbf{x}), \quad  \frac{\partial z_\varepsilon }{\partial \tau} (\mathbf{x}, 0) = \psi (\mathbf{x}),
    \end{aligned}
    \right.
\end{equation}
where $\phi, \psi$~are given functions.
Problem~\eqref{acoustics_Cauchy} is of the form~\eqref{homog_Cauchy_wQ}.
Then the homogenized problem (see~\eqref{homog_Cauchy_wQ_eff}) takes the form
\begin{equation*}
        \left\{
    \begin{aligned}
        & \overline{Q} \frac{\partial^2 z_0 (\mathbf{x}, \tau)}{\partial \tau^2} = - \widehat{\mathcal{A}}^0 z_0 (\mathbf{x}, \tau), \\
        & z_0 (\mathbf{x}, 0) = \phi (\mathbf{x}), \quad  \frac{\partial z_0 }{\partial \tau} (\mathbf{x}, 0) = \psi (\mathbf{x}).
    \end{aligned}
    \right.
\end{equation*}

By Proposition~\ref{N_Q=0_proposit}($1^\circ$), we have $\widehat{N}_Q (\boldsymbol{\theta}) = 0$ for any $\boldsymbol{\theta} \in \mathbb{S}^{d-1}$. (We put $f = Q^{-1/2}$.)
Then, by Theorem~\ref{homog_Cauchy_wQ_thrm_2}, for $\phi, \psi \in H^s (\mathbb{R}^d)$, $0 \le s \le 3/2$, we have
\begin{equation*}
        \| z_\varepsilon (\cdot, \tau) - z_0 (\cdot, \tau) \|_{L_2(\mathbb{R}^d)} \le \varepsilon^{2s/3} \left( \mathfrak{C}_3 (s; \tau) \| \phi \|_{H^s(\mathbb{R}^d)} + \mathfrak{C}_4 (s; \tau) \| \psi \|_{H^s(\mathbb{R}^d)} \right),
\end{equation*}
where $\mathfrak{C}_3 (s; \tau)$ and $\mathfrak{C}_4 (s; \tau)$ are of the form~\eqref{frakC_3_4},
and $\mathcal{C}_3$, $\mathcal{C}_4$ depend only on the norms $\|g\|_{L_\infty}$, $\|g^{-1}\|_{L_\infty}$, $\|Q\|_{L_\infty}$, $\|Q^{-1}\|_{L_\infty}$, and the parameters of the lattice $\Gamma$.

One can also apply the statement about the behavior of the solution for large $|\tau|$
(estimate~\eqref{homog_Cauchy_wQ_est_largetime2}).
It is also possible to consider more general Cauchy problem for the nonhomogeneous equation and
apply Theorem~\ref{nonhomog_Cauchy_wQ_thrm_2}.

\section{Application of the general results: \\
the system of elasticity theory}

\subsection{The operator of elasticity theory}
Let $d \ge 2$. We represent the operator of elasticity theory as in~\cite[Chapter~5, Section~2]{BSu1}.
Let $\zeta$~be an orthogonal second rank tensor in $\mathbb{R}^d$; in the standard orthonormal basis in $\mathbb{R}^d$,
 it can be represented by a matrix $\zeta = \{\zeta_{jl}\}_{j,l = 1}^d$.
 We shall consider \emph{symmetric} tensors $\zeta$, which will be identified with vectors $\zeta_* \in \mathbb{C}^m$, $2m = d(d+1)$,
 by the following rule. The vector $\zeta_*$ is formed by all components $\zeta_{jl}$,  $j \le l$,
 and the pairs $(j, l)$ are put in order in some fixed way.

For the \emph{displacement vector} $\mathbf{u} \in H^1 (\mathbb{R}^d; \mathbb{C}^n)$, we introduce the deformation tensor
\begin{equation*}
e (\mathbf{u}) = \frac{1}{2} \left\lbrace \frac{\partial u_j}{\partial x_l} + \frac{\partial u_l}{\partial x_j} \right\rbrace.
\end{equation*}
Let $e_* (\mathbf{u})$~be the vector corresponding to the tensor
$e (\mathbf{u})$ in accordance with the rule described above. The relation
\begin{equation*}
b(\mathbf{D}) \mathbf{u} = -i e_* (\mathbf{u})
\end{equation*}
determines an $(m \times d)$-matrix homogeneous DO $b(\mathbf{D})$ uniquely; the symbol of this DO
is a matrix with real entries. For instance, with an appropriate ordering, we have
\begin{equation*}
b(\boldsymbol{\xi}) = \begin{pmatrix}
\xi_1 & 0 \\
\frac{1}{2} \xi_2 & \frac{1}{2} \xi_1 \\
0 & \xi_2
\end{pmatrix}, \quad d=2; \quad
b(\boldsymbol{\xi}) = \begin{pmatrix}
\xi_1 & 0 & 0 \\
\frac{1}{2} \xi_2 & \frac{1}{2} \xi_1 & 0 \\
0 & \xi_2 & 0 \\
0 & \frac{1}{2} \xi_3 & \frac{1}{2} \xi_2 \\
0 & 0 & \xi_3 \\
\frac{1}{2} \xi_3 & 0 & \frac{1}{2} \xi_1
\end{pmatrix}, \quad d=3.
\end{equation*}

Let $\sigma(\mathbf{u})$~be the \emph{stress tensor}, and let $\sigma_*(\mathbf{u})$~be the corresponding vector.
In the accepted way of writing, the \emph{Hooke law} about proportionality of stresses and deformations can be expressed by
the relation
\begin{equation*}
\sigma_*(\mathbf{u}) = g (\mathbf{x}) e_* (\mathbf{u}),
\end{equation*}
where $g (\mathbf{x})$~is an $(m \times m)$-matrix-valued function with real entries
(which gives a \textquotedblleft concise\textquotedblright \ description of the Hooke tensor).
The matrix $g (\mathbf{x})$ characterizes the parameters of the elastic (in general, anisotropic) medium under study.
We assume that the matrix-valued function $g (\mathbf{x})$~is $\Gamma$-periodic and such that
$g (\mathbf{x}) > 0$, and $g, g^{-1} \in L_\infty$.

The energy of elastic deformations is given by the quadratic form
\begin{multline}
\label{elast_energy}
w[\mathbf{u},\mathbf{u}] = \frac{1}{2} \int_{\mathbb{R}^d} \left\langle \sigma_*(\mathbf{u}), e_* (\mathbf{u}) \right\rangle_{\mathbb{C}^m} d \mathbf{x} = \frac{1}{2} \int_{\mathbb{R}^d} \left\langle g(\mathbf{x}) b(\mathbf{D}) \mathbf{u}, b(\mathbf{D}) \mathbf{u} \right\rangle_{\mathbb{C}^m} d \mathbf{x},
\\
\mathbf{u} \in H^1 (\mathbb{R}^d; \mathbb{C}^d).
\end{multline}
The operator $\mathcal{W}$ generated by this form is
the \emph{operator of elasticity theory}. Thus, the operator
\begin{equation*}
2 \mathcal{W} = b(\mathbf{D})^* g b(\mathbf{D}) = \widehat{\mathcal{A}}
\end{equation*}
is of the form~\eqref{hatA} with $n = d$ and $m = d(d + 1)/2$.

In the case of isotropic medium, the matrix $g(\mathbf{x})$ depends only on two functional \emph{Lame parameters} $\lambda(\mathbf{x})$ and $\mu(\mathbf{x})$. The parameter $\mu(\mathbf{x})$~is the \emph{shear modulus}.
Often, another parameter $K(\mathbf{x})$ is introduced instead of $\lambda(\mathbf{x})$;
$K(\mathbf{x})$ is called the  \emph{modulus of volume compression}. We need yet another modulus $\beta(\mathbf{x})$.
Here are the relations:
\begin{equation*}
K(\mathbf{x}) = \lambda(\mathbf{x}) + \frac{2 \mu(\mathbf{x})}{d}, \quad \beta(\mathbf{x}) = \mu(\mathbf{x}) + \frac{\lambda(\mathbf{x})}{2}.
\end{equation*}
The modulus $\lambda(\mathbf{x})$ may be negative. In the isotropic case, the conditions that ensure the positive definiteness of the matrix  $g(\mathbf{x})$ are as follows: $\mu(\mathbf{x}) \ge \mu_0 > 0$, $K(\mathbf{x}) \ge K_0 > 0$.
As an example, we write down the matrix $g$ in the isotropic case for $d = 2, 3$:
\begin{gather*}
g_{\mu, K} (\mathbf{x}) = \begin{pmatrix}
K + \mu & 0 & K - \mu \\
0 & 4 \mu & 0 \\
K - \mu & 0 & K + \mu
\end{pmatrix}, \quad d = 2, \\
g_{\mu, K} (\mathbf{x}) = \frac{1}{3} \begin{pmatrix}
3K + 4\mu & 0 & 3K - 2\mu & 0 & 3K - 2\mu & 0 \\
0 & 12\mu & 0 & 0 & 0 & 0 \\
3K - 2\mu & 0 & 3K + 4\mu & 0 & 3K - 2\mu & 0 \\
0 & 0 & 0 & 12 \mu & 0 & 0 \\
3K - 2\mu & 0 & 3K - 2\mu & 0 & 3K + 4\mu & 0 \\
0 & 0 & 0 & 0 & 0 & 12 \mu
\end{pmatrix}, \quad d = 3.
\end{gather*}

\subsection{Homogenization of the Cauchy problem for the elasticity system}
\label{elast_homog_section}

Consider the operator $\mathcal{W}_\varepsilon = \frac{1}{2} \widehat{\mathcal{A}}_\varepsilon$ with rapidly oscillating
 coefficients. The effective matrix $g^0$ and the effective operator $\mathcal{W}^0 = \frac{1}{2} \widehat{\mathcal{A}}^0$
 are defined by the general rules (see~\eqref{g0}, \eqref{g_tilde},  \eqref{hatA0}).

Let $Q(\mathbf{x})$~be a $\Gamma$-periodic symmetric $(d \times d)$-matrix-valued function with real entries such that
\begin{equation*}
Q(\mathbf{x}) > 0; \quad Q, Q^{-1} \in L_\infty.
\end{equation*}
Usually $Q(\mathbf{x})$~is a scalar function (the density of the medium).
We assume that $Q(\mathbf{x})$ is a matrix-valued function.
Consider the following \emph{Cauchy problem for the elasticity system}:
\begin{equation}
\label{elasticity_Cauchy}
\left\{
\begin{aligned}
& Q^\varepsilon (\mathbf{x}) \frac{\partial^2 \mathbf{u}_\varepsilon (\mathbf{x}, \tau)}{\partial \tau^2} = - \mathcal{W}_\varepsilon \mathbf{u}_\varepsilon (\mathbf{x}, \tau), \\
& \mathbf{u}_\varepsilon (\mathbf{x}, 0) = \boldsymbol{\phi} (\mathbf{x}), \quad  \frac{\partial \mathbf{u}_\varepsilon }{\partial \tau} (\mathbf{x}, 0) = \boldsymbol{\psi}(\mathbf{x}),
\end{aligned}
\right.
\end{equation}
where $\boldsymbol{\phi}, \boldsymbol{\psi} \in L_2 (\mathbb{R}^d; \mathbb{C}^d)$~are given functions.
Problem~\eqref{elasticity_Cauchy} is of the form~\eqref{homog_Cauchy_wQ}. The homogenized problem takes the form
\begin{equation}
\label{elasticity_Cauchy_eff}
\left\{
\begin{aligned}
& \overline{Q} \frac{\partial^2 \mathbf{u}_0 (\mathbf{x}, \tau)}{\partial \tau^2} = - \mathcal{W}^0 \mathbf{u}_0 (\mathbf{x}, \tau), \\
& \mathbf{u}_0 (\mathbf{x}, 0) = \boldsymbol{\phi} (\mathbf{x}), \quad  \frac{\partial \mathbf{u}_0 }{\partial \tau} (\mathbf{x}, 0) = \boldsymbol{\psi}(\mathbf{x}).
\end{aligned}
\right.
\end{equation}

By Theorem~\ref{homog_Cauchy_wQ_thrm_1}, if $\boldsymbol{\phi}, \boldsymbol{\psi} \in H^s (\mathbb{R}^d; \mathbb{C}^d)$, $0 \le s \le 2$,
then, for fixed $\tau \in \mathbb{R}$,  the solution $\mathbf{u}_\varepsilon$ converges to $\mathbf{u}_0$
in $L_2 (\mathbb{R}^d; \mathbb{C}^d)$, as $\varepsilon \to 0$. We have
\begin{equation}
\label{elast_general_est}
\| \mathbf{u}_\varepsilon (\cdot, \tau) - \mathbf{u}_0 (\cdot, \tau) \|_{L_2(\mathbb{R}^d)} \le \varepsilon^{s/2} \left( \mathfrak{C}_1 (s; \tau) \| \boldsymbol{\phi} \|_{H^s(\mathbb{R}^d)} + \mathfrak{C}_2 (s; \tau) \| \boldsymbol{\psi} \|_{H^s(\mathbb{R}^d)} \right),
\end{equation}
where $\mathfrak{C}_1 (s; \tau)$ and $\mathfrak{C}_2 (s; \tau)$ are of the form~\eqref{frakC_1_2}, $f = Q^{-1/2}$, and $\mathcal{C}_1$, $\mathcal{C}_2$ depend only on the norms $\|g\|_{L_\infty}$, $\|g^{-1}\|_{L_\infty}$, $\|Q\|_{L_\infty}$, $\|Q^{-1}\|_{L_\infty}$, and the
 parameters of the lattice $\Gamma$.

Note that, even in the isotropic case, in general, estimate~\eqref{elast_general_est} cannot be improved (see Subsection~\ref{example} below).
Recall also Example~\ref{elast_exmpl_N0_ne_0} which confirms the sharpness of the result in the anisotropic case.

One can also apply the statement about the behavior of the solution for large $|\tau|$
         (estimate~\eqref{homog_Cauchy_wQ_est_largetime1}).
         It is also possible to consider more general Cauchy problem for the nonhomogeneous equation and apply Theorem~\ref{nonhomog_Cauchy_wQ_thrm_1}.

 \subsection{Example}
\label{example}
    Consider the system of isotropic elasticity in the twodimensional case assuming that $K$ and $\mu$ are periodic and depend only on $x_1$.
    Now $d=2$, $m=3$, $n=2$, $\Gamma = (2 \pi \mathbb{Z})^2$, and $Q(\mathbf{x}) = \mathbf{1}$.

    The periodic solution $\Lambda(\mathbf{x})$ of equation~\eqref{equation_for_Lambda} is given by
    \begin{equation*}
    \Lambda (\mathbf{x}) = \begin{pmatrix}
    \Lambda_{11} (x_1) & 0 & \Lambda_{13} (x_1) \\
    0 & \Lambda_{22} (x_1) & 0
    \end{pmatrix},
    \end{equation*}
where $\Lambda_{11}$, $\Lambda_{22}$, and $\Lambda_{13}$ are solutions of the following equations
    \begin{align}
    \notag
    D_1 (K + \mu) (D_1 \Lambda_{11} + 1) &= 0, \qquad \int_{\Omega} \Lambda_{11}(x_1) \, d x_1 = 0, \\
    \label{isotr_elast_exmp_Lambda22_eq}
    D_1 \left(2 \mu \left(\frac{1}{2} D_1 \Lambda_{22} + 1 \right)  \right) &= 0, \qquad \int_{\Omega} \Lambda_{22}(x_1) \, d x_1 = 0,  \\
    \notag
    D_1 \left( (K + \mu) D_1 \Lambda_{13} + K - \mu \right) &= 0, \qquad \int_{\Omega} \Lambda_{13}(x_1) \, d x_1 = 0.
    \end{align}
    Hence,
    \begin{gather*}
    (K + \mu) (D_1 \Lambda_{11} + 1) = \underline{(K + \mu)}, \\
    \mu \left( \frac{1}{2} D_1 \Lambda_{22} + 1 \right) = \underline{\mu}, \\
    (K + \mu) D_1 \Lambda_{13} + K - \mu = \underline{(K + \mu)} \overline{\left(\frac{K - \mu}{K + \mu} \right) }.
    \end{gather*}
        Then the matrix $\widetilde{g} (\mathbf{x})$ (see~\eqref{g_tilde}) is given by
    \begin{equation*}
    \widetilde{g} (\mathbf{x}) = \begin{pmatrix}
    \underline{(K + \mu)} & 0 & \underline{(K + \mu)} \overline{\left(\frac{K - \mu}{K + \mu} \right) } \\
    0 & 4 \underline{\mu} & 0 \\
    \underline{(K + \mu)} \left(\frac{K - \mu}{K + \mu} \right) & 0 & \frac{4 K \mu}{K + \mu} + \frac{K - \mu}{K + \mu} \overline{\left(\frac{K - \mu}{K + \mu} \right) }  \underline{(K + \mu)}
    \end{pmatrix}.
    \end{equation*}
    The effective matrix $g^0$~(see~\eqref{g0}) takes the form
    \begin{equation*}
    g^0 = \begin{pmatrix}
    A & 0 & B \\
    0 & C & 0 \\
    B & 0 & E
    \end{pmatrix},
    \end{equation*}
    where
    \begin{align*}
    A &= \underline{(K + \mu)}, \  B = \underline{(K + \mu)} \overline{\left(\frac{K - \mu}{K + \mu} \right) }, \   C = 4 \underline{\mu},
\\
 E &= 4 \overline{\left( \frac{K \mu}{K + \mu} \right) } + \left(\overline{\frac{K - \mu}{K + \mu}} \right)^2 \underline{(K + \mu)}.
    \end{align*}
    The spectral germ is given by
    \begin{equation}
    \label{isotr_elast_exmp_germ}
    \widehat{S}(\boldsymbol{\theta}) = \begin{pmatrix}
    A\theta_1^2  + \frac{1}{4}C\theta_2^2   &   \left( B + \frac{1}{4}C \right)\theta_1 \theta_2  \\
     \left( B + \frac{1}{4}C \right)\theta_1 \theta_2  & E\theta_2^2  + \frac{1}{4}C\theta_1^2
    \end{pmatrix}.
    \end{equation}
    In order to construct an example where $B + \frac{1}{4}C = 0$, we put
    \begin{align*}
    K(x_1) &= a + 100 \cdot \left\{
    \begin{aligned}
    &-1, & &\quad \text{if} \  x_1 < \frac{\pi}{2}\\
    & 1, & &\quad \text{if} \  x_1 \ge \frac{\pi}{2},
    \end{aligned} \right.\\
    \mu(x_1) &= 1 + 624 \cdot \cos^2 x_1.
    \end{align*}
    Then
    \begin{align*}
    A^{-1} &=  \frac{1}{4 q} + \frac{3}{4 r},
\quad
    B =  \frac{6 (a + b) q + 2(a - b) r - 4 q r}{ r + 3 q },
    \\
    C &= 4\sqrt{c+1},\quad
    E = \frac{6 b r - 6 b q  - 12 b^2 + 4 qr}{r  + 3q }.
    \end{align*}
    Here $b = 100$, $c=624$, $q= \sqrt{(a-b+c+1)(a-b+1)}$,
and $r= \sqrt{(a+b+c+1)(a+b+1)}$. Thus, relation $B + \frac{1}{4} C = 0$ holds if $a$ satisfies the equation
    \begin{equation*}
    \frac{6 (a + b) q  + 2(a - b) r - 4 qr}{ r + 3 q}  = -25.
    \end{equation*}
    The root exists, since the left-hand side is continuous in $a$ and, by calculations,
    is (approximately) equal to $-34.4$ for $a= 130$, and is equal to $-22.1$ for $a=150$.
    The approximate value of the root is $a \approx 145.6581$.
    For such $a$ the eigenvalues of the germ~\eqref{isotr_elast_exmp_germ} coincide if
    \begin{equation*}
    A \theta_1^2 + \frac{1}{4}C \theta_2^2 = E \theta_2^2 + \frac{1}{4} C \theta_1^2.
    \end{equation*}
    Since $ \theta_1^2 +  \theta_2^2 = 1$, this is valid for
    \begin{equation}
    \label{isotr_elast_exmp_theta}
    \theta_1^2 = \frac{E - \frac{1}{4}C}{A + E - \frac{1}{2}C} \approx 0.5394.
    \end{equation}

    Next, we calculate $L(\boldsymbol{\theta})$ and $\widehat{N} (\boldsymbol{\theta})$ (see~\eqref{L(theta)}, \eqref{N(theta)}):
    \begin{gather*}
    L(\boldsymbol{\theta}) =
    \begin{pmatrix}
    0 & S \theta_2 & 0 \\
    S^* \theta_2 & 0 & T^* \theta_2 \\
    0 & T \theta_2 & 0
    \end{pmatrix}, \\
    \widehat{N} (\boldsymbol{\theta}) = \frac{1}{2}\begin{pmatrix}
    0 &  S \theta_1^2 \theta_2 + T^* \theta_2^3 \\
    S^* \theta_1^2 \theta_2 + T \theta_2^3 & 0
    \end{pmatrix},
    \end{gather*}
    where
    \begin{equation*}
    S = \underline{K + \mu} \overline{\left(\frac{K - \mu}{K + \mu} \right) \Lambda_{22}}, \quad T = \overline{\left( \frac{4 K \mu}{K + \mu} +  \frac{K - \mu}{K + \mu} \overline{\left(\frac{K - \mu}{K + \mu} \right)} \underline{K + \mu} \right) \Lambda_{22} },
    \end{equation*}
    and $\Lambda_{22}$ is the solution of the equation~\eqref{isotr_elast_exmp_Lambda22_eq}:
    \begin{equation*}
    \Lambda_{22}(x_1) =  \left\{
    \begin{aligned}
    & 2i \arctan\left(\frac{1}{25} \tan(x_1) \right) - 2 i x_1 , & &\quad \text{if} \; 0 \le x_1 < \frac{\pi}{2}\\
    & 2i \arctan\left(\frac{1}{25} \tan(x_1) \right) - 2 i x_1 + 2 \pi i , & &\quad \text{if} \; \frac{\pi}{2} \le x_1 < \frac{3\pi}{2}\\
    & 2i \arctan\left(\frac{1}{25} \tan(x_1) \right) - 2 i x_1 + 4 \pi i , & &\quad \text{if} \; \frac{3 \pi}{2} \le  x_1 \le 2 \pi
    \end{aligned} \right. .
    \end{equation*}
    Approximately, we have $S \approx 65.6650i$, $T \approx 76.2833i$.

    For the points $\boldsymbol{\theta}^{(j)}, j=1,2,3,4$, satisfying~\eqref{isotr_elast_exmp_theta}, we have $\widehat{\gamma}_1 (\boldsymbol{\theta}^{(j)}) = \widehat{\gamma}_2 (\boldsymbol{\theta}^{(j)})$ and $\widehat{N} (\boldsymbol{\theta}^{(j)}) = \widehat{N}_0 (\boldsymbol{\theta}^{(j)}) \ne 0$. The numbers $\pm \widehat{\mu}$, where $\widehat{\mu}$ is approximately equal to $0.09850$, are the
eigenvalues of the operator  $\widehat{N} (\boldsymbol{\theta}^{(j)})$.
    Applying Theorem~\ref{s<2_cos_thrm_Rd}, we confirm that the result~\eqref{elast_general_est} is sharp.

\subsection{The Hill body}
In mechanics (see,~e.g.,~\cite{ZhKO}), the elastic isotropic medium with constant shear modulus
$\mu(\mathbf{x}) = \mu_0 = \mathrm{const}$ is called the Hill body.
In this case, a simpler factorization for the operator $\mathcal{W}$ is possible (see~\cite[Chapter~5, Subsection~2.3]{BSu1}).
Now the energy form~\eqref{elast_energy} can be written as
\begin{equation}
\label{elast_energy_Hill}
w[\mathbf{u},\mathbf{u}] = \int_{\mathbb{R}^d} \left\langle g_\wedge(\mathbf{x}) b_\wedge(\mathbf{D}) \mathbf{u}, b_\wedge(\mathbf{D}) \mathbf{u} \right\rangle_{\mathbb{C}^{m_\wedge}} d \mathbf{x}.
\end{equation}
Here $m_\wedge = 1 + d(d-1)/2$. The $(m_\wedge \times d)$-matrix $b_\wedge(\boldsymbol{\xi})$ can be described as follows.
The first row of $b_\wedge(\boldsymbol{\xi})$ is $(\xi_1, \xi_2, \ldots, \xi_d)$.
The other rows correspond to pairs of indices $(j, l), 1 \le j <l \le d$.
The element in the $(j, l)$th row and the $j$th column is $\xi_l$, and
the element in the $(j, l)$th row and the $l$th column is $(-\xi_j)$; all other elements
of the $(j, l)$th row are equal to zero. The order of the rows is irrelevant. Finally,
\begin{equation}
g_\wedge (\mathbf{x}) = \mathrm{diag} \{\beta(\mathbf{x}), \mu_0/2, \mu_0/2, \ldots, \mu_0/2\}.
\end{equation}
Thus, by~\eqref{elast_energy_Hill},
\begin{equation*}
\mathcal{W} = b_\wedge(\mathbf{D})^* g_\wedge (\mathbf{x}) b_\wedge(\mathbf{D}).
\end{equation*}
According to~\cite[Chapter~5, Subsection~2.3]{BSu1}, the effective matrix $g_\wedge^0$ coincides with $\underline{g_\wedge}$:
\begin{equation}
\label{g_wedge^0}
g_\wedge^0 = \underline{g_\wedge} = \mathrm{diag} \{\underline{\beta}, \mu_0/2, \mu_0/2, \ldots, \mu_0/2\}.
\end{equation}
The effective operator is given by
\begin{equation}
\label{elast_W^0_Hill}
\mathcal{W}^0 = b_\wedge(\mathbf{D})^* g_\wedge^0 b_\wedge(\mathbf{D}).
\end{equation}
For the new factorization of the operator $\mathcal{W}_\varepsilon = b_\wedge(\mathbf{D})^* g_\wedge^\varepsilon (\mathbf{x}) b_\wedge(\mathbf{D})$, the statement of the Cauchy problem~\eqref{elasticity_Cauchy} remains the same.
In the homogenized problem~\eqref{elasticity_Cauchy_eff}, the effective operator $\mathcal{W}^0$ is given by~\eqref{elast_W^0_Hill}
 and, due to~\eqref{g_wedge^0}, is calculated explicitly.
 For the Hill body, the solutions of problem~\eqref{elasticity_Cauchy}
 satisfy all the statements of Subsection~\ref{elast_homog_section}.

Now, suppose that $Q(\mathbf{x}) = \mathbf{1}$. Since $g_\wedge^0 = \underline{g_\wedge}$, using Proposition~\ref{N=0_proposit}($3^\circ$),
we see that $\widehat{N}(\boldsymbol{\theta}) = 0$ for any $\boldsymbol{\theta} \in \mathbb{S}^{d-1}$.
Then, by Theorem~\ref{homog_Cauchy_hatA_eps_ench_thrm_1}, for
$\boldsymbol{\phi}, \boldsymbol{\psi} \in H^s (\mathbb{R}^d; \mathbb{C}^d)$, $0 \le s \le 3/2$, we have
\begin{equation}
\| \mathbf{u}_\varepsilon (\cdot, \tau) - \mathbf{u}_0 (\cdot, \tau) \|_{L_2(\mathbb{R}^d)} \le \varepsilon^{2s/3} \left( \widehat{\mathfrak{C}}_3 (s; \tau) \| \boldsymbol{\phi} \|_{H^s(\mathbb{R}^d)} + \widehat{\mathfrak{C}}_4 (s; \tau) \| \boldsymbol{\psi} \|_{H^s(\mathbb{R}^d)} \right),
\end{equation}
where $\widehat{\mathfrak{C}}_3 (s; \tau)$ and $\widehat{\mathfrak{C}}_4 (s; \tau)$ are given by~\eqref{hat_frakC_3_4}, and  $\widehat{\mathcal{C}}_3$, $\widehat{\mathcal{C}}_4$ depend on $\|\beta\|_{L_\infty}$, $\|\beta^{-1}\|_{L_\infty}$, $\mu_0$, and the
 parameters of the lattice $\Gamma$.

\end{document}